\documentclass[reqno,11pt]{amsart}
\usepackage[colorlinks=true, linkcolor=blue, citecolor=blue]{hyperref}

\usepackage{amssymb}   
\usepackage{amsmath, graphicx, rotating}
\usepackage{color}     
\usepackage{soul}
\usepackage[dvipsnames]{xcolor}     
\usepackage{tcolorbox}

\usepackage{ifthen}
\usepackage{xkeyval}
\usepackage{todonotes}
\usepackage[T1]{fontenc}
\usepackage{lmodern}
\usepackage[english]{babel}

\usepackage{ upgreek }
\usepackage{stmaryrd}
\SetSymbolFont{stmry}{bold}{U}{stmry}{m}{n}
\usepackage{amsthm}
\usepackage{float}

\usepackage{ bbm }
\usepackage{ stmaryrd }
\usepackage{ mathrsfs }
\usepackage{ frcursive }
\usepackage{ comment }

\usepackage{pgf, tikz}
\usetikzlibrary{shapes}
\usepackage{varioref}
\usepackage{enumitem}
\usepackage{longtable}

\usepackage{mathtools}

\usepackage{dsfont}

\definecolor{rouge}{rgb}{0.7,0.00,0.00}
\definecolor{vert}{rgb}{0.00,0.5,0.00}
\definecolor{bleu}{rgb}{0.00,0.00,0.8}

\usepackage[textheight=7.61in,textwidth=4.98in]{geometry} % forAOP
\newtheorem{theorem}{Theorem}[section]
\newtheorem*{theorem*}{Theorem}
\newtheorem{lemma}[theorem]{Lemma}

\newtheorem{corollary}[theorem]{Corollary}

\labelformat{hypothesis}{\textbf{M\kern-0.1mm#1}}

\newtheorem{condition}{Condition}
\newtheorem{conditionA}{A\kern-0.1mm}
\labelformat{conditionA}{\textbf{A\kern-0.1mm#1}}

\renewcommand\dots{\hbox to 1em{.\hss.\hss.}}

\theoremstyle{definition}

\newtheorem{remark}[theorem]{Remark}

\def \eref#1{\hbox{(\ref{#1})}}

\numberwithin{equation}{section}

\newcommand*{\abs}[1]{\left\lvert#1\right\rvert}

\def\bb#1{\mathbb{#1}}

\def\geq{\geqslant}
\def\leq{\leqslant}

\newcommand\ee{\varepsilon}

\def\geq{\geqslant}
\def\leq{\leqslant}
\def\Rd {\bb R ^d}
\def\Rd*{(\bb R ^d)^*}
\def\Pd{{\mathbb{P}}^{d-1}}
\def\Pd*{(\mathbb{P}^{d-1})^*}

\def\bb#1{\mathbb{#1}}

\newcommand\veps{\ee}

\begin{document}
\title[Conditioned local limit theorems]{Conditioned local limit theorems for random walks on the real line} 
%with moving boundaries}
\author{ Ion~Grama}
\curraddr[I. Grama]{ Universit\'{e} de Bretagne Sud, Laboratoire de Math\'ematiques de Bretagne Atlantique, UMR CNRS 6205, 
     Centre Yves Coppens, Campus de Tohannic, 56017 Vannes, France}
%Université de Bretagne-Sud & UMR CNRS 6205
%Centre Yves Coppens, Campus de Tohannic, BP573, 56017 Vannes, France
\email[I. Grama]{ion.grama@univ-ubs.fr}
\author{Hui Xiao}
\curraddr[H. Xiao]{Universit\"{a}t Hildesheim, Institut f\"{u}r Mathematik und Angewandte Informatik, 31141 Hildesheim, Germany}
\email[H. Xiao]{xiao@uni-hildesheim.de}
%\author{Hui Xiao}
%\curraddr[H. Xiao]{Universit\'{e} de Bretagne Sud, CRYC, 56017 Vannes, France}
%\email[H. Xiao]{hui.xiao@univ-ubs.fr}

\date{\today }

\begin{abstract}   
Consider a random walk $S_n=\sum_{i=1}^n X_i$ with independent 
and identically distributed real-valued increments $X_i$
of zero mean and finite variance. Assume that $X_i$ is non-lattice and has a moment of order $2+\delta$.
%Our results hold for random walks with
%We study that case of zero drift $\mu=0$ under the moment assumption as well as non-zero drifts $\mu\not=0.$ 
%Our results are stated under the assumption that the random variables satisfy the non-lattice condition. 
For any $x\geq 0$, let 
$\tau_x  = \inf \left\{ k\geq 1: x+S_{k} < 0 \right\}$
be the first time when the random walk $x+S_n$ leaves the half-line $[0,\infty)$.
We study the asymptotic behavior of the probability $\bb P (\tau_x >n)$  
%\textcolor{magenta}{and prove limit theorems} 
and that of the expectation 
$\mathbb{E} \left( f(x + S_n ),  \tau_x > n \right)$   
for a large class of target function $f$ and various values of $x$, $y$ possibly depending on $n$.  
This general setting implies limit theorems for the joint distribution 
$\mathbb{P} \left( x + S_n  \in y+ [0, \Delta],  \tau_x > n \right)$ where $\Delta>0$ may also depend on $n$.
In particular, the case of moderate deviations $y=\sigma \sqrt{q n\log n}$ %for some small constant $q>0$ 
is considered. 
We also deduce some new asymptotics for random walks with drift and
 give explicit constants in the asymptotic of the probability $\bb P (\tau_x =n)$. 
%We also prove conditioned local limit theorems with target function on the random walk and precise
%estimates of the error term in terms of the target function. 
%A local limit theorem for the stopping time $\tau_x$. 
For the proofs we establish new conditioned integral limit theorems with precise error terms.
%and on an asymptotic approach for the duality, which allow to establish 
%local limit theorems with rate $n^{3/2}$. 
  
%which allows to get precise bounds for the errors. 
\end{abstract}

\maketitle

\tableofcontents

%\section{Introduction}
%
%We prove two versions of Stone's local limit theorem for a random walk conditioned to stay positive. 
%The case of discrete valued random walks have been considered by Denisov and Wachtel \cite{Den Wachtel 2011}.
% Caravenna \cite{Carav05} has obtained \ldots among many other contributions. 
%We extend the result of Denisov and Wachtel to the case of sums of
% %to the case are stated under the assumption that the 
% random variables satisfying the non-lattice condition.
%Under our assumptions we also improve the results of Caravenna \cite{Carav05}. 

%We shall denote by $\mathbb{B}_{x,\ee }$ a ball of
%radius $\ee $ and center in $x$ in $\bb R _{+}:$ $\mathbb{B}%
%_{x,\ee }=\left\{ u\in \bb R _{+}:\left\vert x-u\right\vert \leq
%\ee \right\} .$ 

\section{Main results}\label{Sect-Mean0}

\subsection{Introduction and assumptions}\label{Subsec-Intro}
Assume that on the probability space $\left( \Omega ,\mathscr{F},\mathbb{P}\right)$ 
we are given a sequence of independent identically distributed real-valued random variables $(X_i)_{i\geq 1}$ 
with $\bb E X_1 = 0$ and  $\bb E X_1^2 = \sigma^2 \in (0, \infty)$.   
Let $S_n =\sum_{i=1}^{n}X_{i}$, $n \geq 1$. 
%\begin{equation}
%S_n =\sum_{i=1}^{n}X_{i}, \quad n \geq 1.  \label{iid-000}
%\end{equation}
%%Note that $X_{0}$ is not included in the sum $S_n .$ 
For any starting point $x\in \bb R_+: = [0,\infty)$, 
consider the first moment $\tau_x$ when the random walk $(x+S_n )_{n\geq 1}$ 
goes below the constant boundary $0$, 
%leaves the set $\bb R_+$, 
which is defined as  
\begin{align*}
\tau_x %%%%= \inf \left\{ k\geq 0: x+S_{k} < 0 \right\} 
= \inf \left\{ k \geq 1: x+S_{k} < 0 \right\}.
% \quad \mbox{with} \   \inf\emptyset = \infty.   
\end{align*}
%where $\inf\emptyset = \infty$.
%It is well known that, if $\mu\leq 0$ it holds $\bb P(\tau_x <\infty)=1,$ for all $x\geq 0.$
%We complete the above definition by setting $\tau_x=0$ for $x<0$. 
For $\Delta > 0$, $x\geq 0$ and $y\geq 0$ possibly depending on $n$, 
consider the probability
\begin{align} \label{Objective-proba001}
\mathbb{P} \left( x + S_n  \in y+ [0, \Delta],  \tau_x > n \right). 
\end{align} 
The asymptotic behavior of the probability \eqref{Objective-proba001} has been studied by many authors,
we refer to 
L\'{e}vy \cite{Levy37},    Borovkov \cite{Borovk62, Borovkov04a, Borovkov04b},   Feller \cite{Fel64}, Spitzer \cite{Spitzer},
Bolthausen \cite{Bolth}, Iglehart \cite{Igle74}, Eppel \cite{Eppel-1979},  Bertoin and Doney \cite{BertDoney94},  
Caravenna \cite{Carav05},  Vatutin and Wachtel \cite{VatWacht09},  
Denisov and Wachtel \cite{Den Wachtel 2011}, Kersting and Vatutin \cite{KV17},
Denisov, Sakhanenko and Wachtel \cite{DSW18} and to the references therein.  
In particular,  Eppel \cite{Eppel-1979} and later Vatutin and Wachtel \cite{VatWacht09} 
have found its asymptotic  in the case when  $\frac{x}{\sqrt{n}}\to 0$ and $\frac{y}{\sqrt{n}}\to 0.$ 
However, there are very few results dealing with the case when these conditions are not satisfied. % see \cite{Don12, DJ12}.
We refer to Doney \cite{Don12} for the case when 
$\frac{y}{\sqrt{n}} \sim c_1$ and $\frac{x}{\sqrt{n}} \sim c_2$
for some constants $c_1, c_2 >0$. 
Precise large deviations for the special class of heavy tailed distributions have been studied in 
Doney and Jones \cite{DJ12}. 
Note also that the results in \cite{Don12} and \cite{VatWacht09} are stated for the case of random variables 
in the domain of attraction of stable laws, which will not be considered here.

In this paper, we shall study the asymptotic of the probability \eqref{Objective-proba001}  when %$\mu \in \bb R$, and 
$x, y \geq 0$ are moving to $\infty$ under some moment assumptions on the increment $X_1$.
%  possibly faster than $\sqrt{n}$. 
Before  stating our main results we give the necessary definitions and  introduce some notation.

We first recall the definition of non-latticity: a random variable $X$ is said to be non-lattice 
if for any $h>0$ and $a\in [0,h)$ it holds that
$\bb P(X \in  h\bb Z +a )\not=1$, where $\bb Z$ is the set of integers. 
We shall make use of the following conditions: 

%We refer to \cite{Don95} for the case when these conditions are not satisfied. 

%%%%%%%%%%%%%%%%%%%%%%%%%%%%%%%%%%% 
% Throughout the paper we assume the following non-latticity condition:

\begin{conditionA}\label{A1}
The law of the random variable $X_1$ is non-lattice. 
% a density with respect to the Haar measure on $\bb G$. 
%{\rm (i) (Allowability)}
%Every $g\in\Gamma_{\mu}$ is allowable.
%
%{\rm (ii) (Positivity)}
%$\Gamma_{\mu}$ contains at least one strictly positive matrix.
\end{conditionA}

%We also need the following moment conditions:
% where it is assumed that $\delta \geq 0$ and $\lambda \in \bb R$.
\begin{conditionA}\label{SecondMoment}
$\bb E X_1 = 0$ and there exists $\delta>0$ such that $\bb E (|X_1|^{2+\delta})  < \infty.$
%the random variable $X_1$ has a moment of order $2+\delta$, i.e.\  
%\begin{align*} %\label{}
%\bb E |X_1|^{2+\delta}  < \infty.
%\end{align*}
\end{conditionA}

The limit behavior of the probability \eqref{Objective-proba001} depends on two harmonic functions 
related to the random walk $(x+S_n )_{n\geq 1}$, which we proceed to introduce. 
It is well known (see \cite{BertDoney94})  
that under condition \ref{SecondMoment}, 
%By the results of Bertoin and Doney \cite{BertDoney94}, % Doney \cite{Don95}, 
%under the stated assumption,  
the function
\begin{align*}
x \in \bb R_+ \mapsto V\left(x\right) = -\mathbb E S_{\tau_x} 
% =\lim_{n\rightarrow +\infty }\mathbb{E}\left(y+S_n ;\tau _{y}>n\right).
\end{align*}
is well-defined and strictly positive. % for all $x \geq 0$.
%Using the results of Tanaka \cite[Lemma 1]{Tanaka1999} and the renewal arguments, one can verify that (under the condition $\bb EX_1^2=1$) 
%the function $V$ is a Doob $h$-transform, i.e., for any $x\geq 0$,
%for any $x \geq 0$,
%%It is well known that the function $V$ is harmonic ????
Using the results of Tanaka \cite[Lemma 1]{Tanaka1999} and the renewal arguments, 
one can verify that %(under the condition $\bb EX_1^2=1$) 
%the function $V$ is harmonic (or equivalently is a Doob's $h$-transform), i.e., for any $x\geq 0$,
for any $x\geq 0$,
\begin{align} \label{Doob transf}
\bb E V(x+S_1)\mathds 1 _{\{ x+S_1\geq 0  \}} = V(x).   
\end{align} 
The function $V$ will be called harmonic function of the random walk $(x+S_n )_{n\geq 1}$ killed at $\tau_x$. 
%and has the following property, \textcolor{magenta}{see \cite{Tanaka1999}:} 
The harmonic function $V$ is uniquely defined up to a constant factor.
If $U$ is the renewal function  
in the strict increasing ladder process of $(S_n)_{n\geq 1}$   (\cite{KV17}), then, 
for any  $x\geq 0$ it holds $U(x)=V(x)/V(0)$. 
Due to identity \eqref{Doob transf}, the function $V$ can be uniquely extended to the whole real line $\bb R$
by setting $V(x)= \bb E V(x+S_1)\mathds 1 _{\{ x+S_1\geq 0  \}}$ for $x<0$, 
so that the harmonicity property \eqref{Doob transf} is preserved. 
Moreover, the support of $V$ is given by (\cite[Example 2.10]{GLL18Ann})
$\mathscr D = \{ x \in \bb R: \bb P(x + X_1 > 0) >0  \}$. 
%\begin{align*}%\label{}
%\mathscr D = \{ x \in \bb R: \bb P(x + X_1 > 0) >0  \}. 
%\end{align*}}
%We refer to \cite[Example 2.10]{GLL18Ann} for details. 

Let $\left( S^*_{n}\right) _{n\geq 1}$ be the dual random walk: $S^*_{n}=-\sum_{i=1}^{n}X_{i}$. 
Denote by $\tau^*_x$ the dual exit time: $\tau^*_x = \inf \left\{ k\geq 1: x+S^*_{k} < 0 \right\}$.  
For $x \geq 0$, let  %$V(x) : = -\bb E (S_{\tau_x})$ and 
$V^*(x) : = -\bb E (S_{\tau_x^*}^*)$
be  the harmonic function of the random walk $(x+S_n^* )_{n\geq 1}$
killed at $\tau_x^*$. % under the probability measure $\bb P$. 
%According to \cite[Example 2.10]{GLL18Ann}
The function $V^*$ can be extended to the whole real line in the same way as the function $V$. 

%Denote $\sigma_{\lambda}^2 : = \bb E  |X_1|^{2} e^{ \lambda X_1 }.$ 
Denote by $\Phi(x)=\frac{1}{\sqrt{2\pi}}\int_{-\infty}^x e^{-t^2/2} dt$, $x\in \bb R$ the standard normal distribution function. 
Let $\phi^+$ be the Rayleigh density function, i.e. 
\begin{align} \label{Def-Rayleighdens}
\phi^+(s) = s e^{-s^{2}/2} \mathds 1_{\{ s \geq 0 \} }, \quad s \in \bb R.   
\end{align} %\notag\\
When the starting point $x$ is far from the boundary we shall make use of the following function:
%For any $x \geq 0$ and $s\in \bb R$ denote
\begin{align} \label{Def-Levydens}
\psi(s,x)=\frac{1}{\sqrt{2\pi }}
   \left( e^{-\frac{(s-x)^2}{2}}- e^{-\frac{(s+x)^2}{2}} \right),  \quad  s, x \in \bb R.   
\end{align} 
Note that $\psi(s,x)>0$ for any $s, x>0$, and that $\psi(s,0)=0$ for any $s\in \bb R$.
It is useful to note that for any fixed $x > 0$, the function $s\in \bb R_+ \mapsto \psi(s,x)$  
is not a density on $\bb R_+$, 
but with the normalization factor $1/(2 \Phi( x ) -1)$ it becomes one.
By straightforward calculations, one can verify that uniformly in $s$ 
over compact sets of $\bb R_+$,
%\begin{align*} %\label{}
% \psi(s,x) \sim \phi^+(s) (2 \Phi( x ) -1).
%\end{align*}
\begin{align*} %\label{}
\lim_{x\to 0} \frac{\psi(s,x)}{2 \Phi( x ) -1} = \phi^+(s).
\end{align*}
%Note that for any $x>0$ we have $\psi(s,x)>0$ for $s\geq 0$ and that $\psi(s,0)=0$, for any $s\in \bb R$.
%It is also useful to note that for any fixed $x\geq 0$, the function $s\in \bb R_+ \mapsto \psi(s,x)$  is not a density, 
%but with the normalization $2 \Phi( x ) -1$, it becomes one: 
%for $s \in \bb R$ and $x >0$, define  
%\begin{align*}%\label{}
%\psi^+ (s, x) = \frac{1}{ 2 \Phi(x) -1  }   \psi(s,x).  
% % \left( e^{- \frac{(s-x)^2}{2 }} -  e^{- \frac{(s+x)^2}{2 }}  \right) \mathbf 1_{\{ s \geq 0 \} }. 
%\end{align*}
%%Denote for short 
%%$\psi (s,x) = \psi_{1}(s,x)$ and $\psi^+ (s,x) = \psi_{1}^+(s,x)$. 
%Note that $\lim_{x \to 0} \psi^+(s, x)$ exists for any $s \in \bb R$ and so the function $s\mapsto \psi^+(s, 0)$ is well defined
%and moreover $\psi^+(\cdot, 0)$ coincides with the Raleigh  density on $\bb R$.
%%For short, set $\psi(s,t)=\psi_{1}(s,x)$ and $\psi^+(s,t)=\psi^+_{1}(s,x)$, for $s,x\geq 0$.  
%\todos{Maybe we put this later}
%Note that as $\frac{x}{\sqrt{n}} \to 0$, we have 
%\begin{align*}%\label{}
% \psi \left( \frac{y}{\sqrt{n}}, \frac{x}{\sqrt{n}}  \right)   
%& \sim  \frac{1}{ \sqrt{2\pi } } \phi^+ \left( \frac{y}{\sqrt{n}} \right)  \frac{2 x}{\sqrt{n}}. 
%\end{align*}
%Moreover, if $\frac{x}{\sqrt{n}} \to 0$ and $\frac{y}{\sqrt{n}} \to 0$, then
%\begin{align*}%\label{}
%\psi^+ \left( \frac{y}{\sqrt{n}}, \frac{x}{\sqrt{n}} \right) 
%\sim  \frac{1}{ \sqrt{2\pi } } \phi^+ \left( \frac{y}{\sqrt{n}} \right)  \frac{2 x}{\sqrt{n}}
%\sim  \frac{2}{ \sqrt{2\pi } }  \frac{yx}{n}. 
%\end{align*}
%
%
%

In the sequel, the notation $f_{\alpha}(n) \sim g_{\alpha}(n)$, uniformly in $\alpha \in A$ as $n\to \infty$, 
 means that $\lim_{n \to \infty} \sup_{\alpha \in A}\frac{f_{\alpha}(n)}{g_{\alpha}(n)} = 1$. 
 For short, we will write $\bb E (X; B)$ for the expectation $\bb E (X \mathds 1_{B})$.

\subsection{Starting point near the boundary}\label{Sect-CLLT-AA}
In this section we formulate our results when the starting point $x$ is near the boundary $0$.
Below we fix a target function $f(\cdot)$ on the sum $x+S_n-y$ where $y$ is a drift variable.
The following result is effective when the drift $y$ moves to infinity. % with a speed of order at least  $\sqrt{n}$.
% and \ref{theorem-n3/2ChangeMeas}.  
It will be obtained as a consequence of the more general Theorem \ref{t-B 002}.

\begin{theorem} \label{Theorem-AA001}
Assume \ref{A1} and \ref{SecondMoment}. 
Let $f: \bb R \mapsto \bb R$ be a directly Riemann integrable function with support in $\bb R_+$ 
such that %$f(t)=0$ for $t < 0$ and 
$\int_{\bb R_+}  f(t) (1 + t)  dt < \infty$. 
Then, for any $\eta \in (0, 1]$ and any sequence of positive numbers $(\alpha_n)_{n \geq 1}$ satisfying
$\lim_{n\rightarrow \infty} \alpha_n =  0$, 
we have, as $n \to \infty$,   uniformly in $x \in [0, \alpha_n \sqrt{n}]$ and 
 $y \in [\eta \sqrt{n}, \eta^{-1} \sqrt{n}]$,
 \begin{align}\label{demoCLLT-003_Targetaa}
\mathbb E \left( f( x+S_n - y);  \,  \tau_x > n  \right) 
   \sim
 \frac{2V (x) }{\sqrt{2\pi } \sigma^2 n } 
   \phi^+  \left( \frac{y}{ \sigma \sqrt{n}} \right)  \int_{\bb R_+} f(t)  dt. 
\end{align}
%the asymptotic \eqref{demoCLLT-003_Targetaa} holds . 
Moreover, there exist constants $\ee_0,  \delta_0 >0$ such that for any $\ee \in (0, \ee_0)$, $q \in (0, \delta_0)$ and $\eta >0$,
 as $n \to \infty$,  
 the asymptotic \eqref{demoCLLT-003_Targetaa} holds
 uniformly in $x \in [0, n^{1/2 - \ee}]$ and 
 $y \in [\eta \sqrt{n}, \sigma \sqrt{q n \log n}]$. 
 %$y \in [%\alpha_n^{-1},  n^{1/2-q/2}, \sigma \sqrt{q n \log n}]$,
%Moreover, for any $\eta \in (0, 1]$ and any sequence of positive numbers $(\alpha_n)_{n \geq 1}$ satisfying
%$\lim_{n\rightarrow \infty} \alpha_n =  0$, 
%the asymptotic \eqref{demoCLLT-003_Targetaa} holds uniformly in $x \in [0, \alpha_n \sqrt{n}]$ and 
% $y \in [\eta \sqrt{n}, \eta^{-1} \sqrt{n}]$. 
\end{theorem}

%The presence of the target function $f$ in the asymptotic \eqref{demoCLLT-003_Targetaa}
%is very important since it allows to deal with random walks with drift 
%and to establish local theorems for the time $\tau_x$.
%Indeed, %for a random walk with drift,  
%with a suitable change of measure the drift vanishes 
%so that we can apply \eqref{demoCLLT-003_Targetaa} with an appropriate target function. 
%This will be used in Sections \ref{Sect-Application} and \ref{Sect-LLT-exit-taux}
%to obtain the exact asymptotics for the exit time of random walks with negative drift. 
%The precise asymptotic of the local probability $\bb P(\tau_x=n)$ 
%is given in Section \ref{Sect-LLT-exit-taux}.

% covers the case of the function. 
In the particular case when $f = \mathds 1_{[0, \Delta]}$ with $\Delta >0$, 
one can improve Theorem \ref{Theorem-AA001}
% we can improve \eqref{demoCLLT-003_Targetaa} 
by showing the uniformly in $\Delta$ in a certain range.    
% In the case when the function $f$ is of the form $f = \mathds 1_{[0, \Delta]}$,  
%we can improve the previous result as follows. 

\begin{theorem} \label{Theorem Delta-002}
Assume \ref{A1} and \ref{SecondMoment}. 
Then, for any $\eta \in (0, 1]$, $\Delta_0 >0$ and any sequence of positive numbers $(\alpha_n)_{n \geq 1}$ satisfying
$\lim_{n\rightarrow \infty} \alpha_n =  0$, 
we have, as $n \to \infty$,  uniformly in $x \in [0, \alpha_n \sqrt{n}]$,  
 $y \in [\eta \sqrt{n}, \eta^{-1} \sqrt{n}]$ and $\Delta \in [\Delta_0,   \alpha_n \sqrt{n}]$,
\begin{align} \label{demoCLLT-003aa}
\mathbb P \left( x+S_n  \in [0,  \Delta]  + y,  \,  \tau_x > n  \right)  
  \sim  
 \Delta  \frac{2  V (x) }{\sqrt{2\pi } \sigma^2 n }    %e^{\lambda (x - y)}
 \phi^+ \left( \frac{y}{ \sigma \sqrt{n}} \right). 
\end{align}
%the asymptotic \eqref{demoCLLT-003aa} holds uniformly in $x \in [0, \alpha_n \sqrt{n}]$,  
% $y \in [\eta \sqrt{n}, \eta^{-1} \sqrt{n}]$ and $\Delta \in [\Delta_0,   \alpha_n \sqrt{n}]$. 
%Then, %for any directly Riemann integrable function $f: \bb R \mapsto \bb R_+$, 
%there exists a constant $\delta_0 >0$ such that for any $q \in (0, \delta_0)$
Moreover, there exist constants $\ee_0,  \delta_0 >0$ such that
for any $\ee \in (0, \ee_0)$, $q \in (0, \delta_0)$, $\eta >0$ and $\Delta_0 >0$,   
as $n \to \infty$,  the asymptotic \eqref{demoCLLT-003aa} holds
 uniformly in $x \in [0, n^{1/2 - \ee}]$,  $y \in [\eta \sqrt{n}, \sigma \sqrt{q n \log n}]$
%$0 < \liminf_{n \to \infty} \frac{y}{\sqrt{n}} \leq \limsup_{n \to \infty} \frac{y}{\sqrt{n}} < \infty$,
and $\Delta \in [\Delta_0,   n^{1/2 - \ee} ]$. 
%Moreover, for any $\eta \in (0, 1]$, $\Delta_0 >0$ and any sequence of positive numbers $(\alpha_n)_{n \geq 1}$ satisfying
%$\lim_{n\rightarrow \infty} \alpha_n =  0$, 
%the asymptotic \eqref{demoCLLT-003aa} holds uniformly in $x \in [0, \alpha_n \sqrt{n}]$,  
% $y \in [\eta \sqrt{n}, \eta^{-1} \sqrt{n}]$ and $\Delta \in [\Delta_0,   \alpha_n \sqrt{n}]$. 
%%where we use the convention that $\frac{ 1 - e^{-\lambda \Delta} }{\lambda} = \Delta$ when $\lambda = 0$. 
%%\todos{Ask Doney if $\Delta$ is uniform in his result}
\end{theorem}

%The asymptotic \eqref{demoCLLT-003aa} is of independent interest and its 
%The proof will be given in the end of Section \ref{SectProofThm1}. 
In particular, from the second assertion of Theorem \ref{Theorem Delta-002}, 
 taking $y = \sigma \sqrt{q  n \log n}$ in \eqref{demoCLLT-003aa}, we get, 
as $n \to \infty$,  uniformly in $x \in [0, n^{1/2 - \ee}]$
and $\Delta \in [\Delta_0,  n^{1/2 - \ee} ]$, 
\begin{align}\label{CCLTMDsmallxaa}
\mathbb{P} \left( x+S_n  \in  [0, \Delta] +  \sigma \sqrt{q  n \log n},  \,  \tau_{x} > n \right) 
\sim  \frac{2 V(x) }{\sqrt{2\pi }  \sigma^2 }  
    \frac{ \Delta \sqrt{q \log n}}{n^{1+ q/2 }}. 
\end{align}
%The relation \eqref{CCLTMDsmallx} exhibits a novel asymptotic in the local limit theorem with moderate deviations
 %for random walks conditioned to stay positive. 
The asymptotic \eqref{CCLTMDsmallxaa} can be compared with the classical local limit theorem with moderate deviations (for non-killed random walks) 
which can be deduced from the results of Nagaev \cite{Nag65} and Amosova \cite{Amo72}: as $n \to \infty$, 
\begin{align}\label{CLTMD01aa}
\mathbb{P} \left( S_n  \in  [0, \Delta] + \sigma  \sqrt{q n \log n}  \right) 
\sim  \frac{ \Delta }{\sqrt{2\pi }  \sigma  n^{(1+ q)/2 }}. 
\end{align}
We refer to Breuillard \cite{Bre05}
for similar results for random walks satisfying the diophantine condition, 
and to Grama \cite{Gra97} -- for martingales. 

%The result \eqref{CLTMD01} has been generalized to martingales in \cite{Gra97}. 

%In particular, in \eqref{demoCLLT-003}, if we take $y = a \sqrt{n} + o(\sqrt{n})$ for any constant $a >0$, 
%then $\phi^+ ( \frac{y}{\sqrt{n}} ) \to \phi^+ ( a ) >0$ as $n \to \infty$.  

In the particular case when $x =0$, $y = O(\sqrt{n})$ and $\Delta >0$ is a fixed real number,
the asymptotic \eqref{demoCLLT-003aa} has been established earlier by Caravenna \cite{Carav05}
under the non-lattice condition \ref{A1} and the optimal second moment assumption \ref{SecondMoment} with $\delta =0$. 
Under the same conditions, Caravenna's result has been generalized  by Doney \cite{Don12}
to the case when $x$ can depend on $n$.    
Under the stronger moment condition of oder $2 + \delta$, 
%Theorem \ref{Theorem-AA001}  
our asymptotic \eqref{demoCLLT-003aa} improves on these results
in two aspects. 
Firstly, $\Delta \in [\Delta_0, n^{1/2 - \ee} ]$ is allowed to depend on $n$;
secondly, %$y$ is allowed to go beyond the range $\sqrt{n}$. 
$y$ can take values in the range $[\eta \sqrt{n}, \sigma \sqrt{q n \log n}]$. 
Indeed, the case when $y$ goes beyond the range $\sqrt{n}$ has not been considered in the literature so far
and our result gives a new asymptotic in this range. 
%We conjecture that our asymptotics 
%\eqref{demoCLLT-003_Targetaa} and \eqref{demoCLLT-003aa} should be valid uniformly in the larger interval 
%$y \in [\alpha_n, o(n^{1/6})]$ for any $\alpha_n \to \infty$. 
%
%As a potential application of Theorem \ref{Thm-C 002}, we can mention that it opens a way to study 
%large deviations for critical and subcritical branching processes in random environment. 
%

Our second result gives the exact asymptotics for the probability \eqref{Objective-proba001}
when $x$ is near the boundary and $\frac{y}{\sqrt{n}} \to 0$ as $n \to \infty$.
It is a consequence of a more general Theorem \ref{theorem-n3/2-upper-lower-bounds}.  
\begin{theorem} \label{Theorem-AA002}
Assume \ref{A1} and \ref{SecondMoment}. 
Let $(\alpha_n)_{n\geq 1}$ be any sequence of positive numbers satisfying
$\lim_{n\rightarrow \infty} \alpha_n =  0$ and $ n^{1/4}  \alpha_n \geq 1$.

1. Let $f: \bb R \mapsto \bb R$ be a directly Riemann integrable function with support in $\bb R_+$ 
satisfying %$f(t)=0$ for $t < 0$ and 
$\int_{\bb R_+}  f(t) (1 + t)^{\gamma} dt < \infty$ for some constant $\gamma >1$. 
Then,  for any $a >0$, we have,  
%\todos{we should introduce $y\in [0,c]$. This result is uniform in $y$}
 as $n \to \infty$,  uniformly in $x \in [0,  \alpha_n \sqrt{n}]$ and $y \in [0, a]$, 
\begin{align}\label{Asymn32Smallxaa}
\mathbb E \left( f(x+S_n-y);  \,   \tau_x > n  \right)
\sim      \frac{ 2 V(x) }{ \sqrt{2\pi }  \sigma^3  n^{3/2}  } 
    \int_{\bb R_+}   f( t-y)   V^{\ast } (t) dt. 
\end{align}

2. Let $f: \bb R \mapsto \bb R$ be a directly Riemann integrable function with support in $\bb R_+$ 
such that $\int_{\bb R_+}  f(t) (1 + t)  dt < \infty$. 
Then,  we have, as $n \to \infty$,  
 uniformly in $x \in [0, \alpha_n \sqrt{n}]$ and $y \in [\alpha_n^{-1}, \alpha_n \sqrt{n}]$,
 \begin{align}\label{demoCLLT-003_Targetbbaa}
\mathbb E \left( f( x+S_n - y);  \,  \tau_x > n  \right)  
    \sim  \frac{2 y V(x) }{\sqrt{2\pi } \sigma^3 n^{3/2} }    
    \int_{\bb R_+} f(t)  dt. 
\end{align}
\end{theorem}

The presence of the target function $f$ in the asymptotics \eqref{Asymn32Smallxaa} and  \eqref{demoCLLT-003_Targetbbaa}
is very important since it allows to deal with random walks with drift 
and to establish local theorems for the time $\tau_x$.
Indeed, %for a random walk with drift,  
with a suitable change of measure the drift vanishes 
so that we can apply \eqref{Asymn32Smallxaa} and  \eqref{demoCLLT-003_Targetbbaa} with an appropriate target function. 
This will be used in Sections \ref{Sect-Application} and \ref{Sect-LLT-exit-taux}
to obtain the exact asymptotics for the exit time of random walks with negative drift. 
The precise asymptotic of the local probability $\bb P(\tau_x=n)$ 
is given in Section \ref{Sect-LLT-exit-taux}.

It is easy to see that in the range $y \in [\alpha_n^{-1}, \alpha_n \sqrt{n}]$, 
the right hand side of the asymptotic \eqref{demoCLLT-003_Targetbbaa} is equivalent to that in \eqref{demoCLLT-003_Targetaa}.
We conjecture that our asymptotics 
\eqref{demoCLLT-003_Targetaa} and \eqref{demoCLLT-003aa} should be valid uniformly in the larger interval 
$y \in [\alpha_n, o(n^{1/6})]$ for any $\alpha_n \to \infty$. 
%We conjecture that \eqref{demoCLLT-003_Targetaa} also remains true uniformly in $y \in [\alpha_n \sqrt{n}, \eta \sqrt{n}]$.  

In particular, for any real number $a >0$, taking $f(t) = e^{-at}$ and $y = 0$ in \eqref{Asymn32Smallxaa}, 
we get that,  as $n \to \infty$,  uniformly in $x \in [0, \alpha_n \sqrt{n}]$, 
\begin{align}\label{AsymExponenFunaa}
\mathbb E \left( e^{-a (x+S_n) }; \tau_x >n\right)
\sim      \frac{ 2 V(x) }{ \sqrt{2\pi }  \sigma^3  n^{3/2}  } 
    \int_{\bb R_+}    e^{- a t}  V^{\ast } (t) dt. 
\end{align}
The result \eqref{AsymExponenFunaa} holds uniformly in $x \in [0, \alpha_n \sqrt{n}]$, 
thus improving that of \cite[Proposition 2.1]{ABKV12} (see also \cite[Theorem 4.10]{KV17}) 
which was proved for fixed $x \geq 0$. 
  Note that the results in \cite{ABKV12, KV17} are established using a completely different approach based on 
  the Wiener-Hopf factorization, in particular using the Sparre-Andersen and Spitzer identities. 
  These type of results turn out to be very useful for studying limit theorems for branching processes in random environment.
 % It is expected that using our results (Theorems \ref{Theorem-AA001} and \ref{Theorem-AA002})
 % one could obtain new limit theorems for branching processes in random environment. 

In the particular case when $f = \mathds 1_{[0, \Delta]}$ with $\Delta >0$, 
the following result improves the asymptotic \eqref{demoCLLT-003_Targetbbaa}
by showing the uniformly in $\Delta$. 
%The asymptotic \eqref{demoCLLT-003_Targetbbaa} of course covers the particular case 
%where $f = \mathds 1_{[0, \Delta]}$ for any fixed real number $\Delta >0$. 
%The following result which holds uniformly in $\Delta$. 
%%When $f = \mathds 1_{[0, \Delta] + y}$, we can provide a uniformity in $\Delta$.  

\begin{theorem}\label{Theorem-AA002bis}
Assume \ref{A1} and \ref{SecondMoment}. 
Let $(\alpha_n)_{n\geq 1}$ be any sequence of positive numbers satisfying
$\lim_{n\to \infty} \alpha_n =  0$ and $ n^{1/4}  \alpha_n \geq 1$.
Then, for any $\Delta_0 >0$, we have,  
as $n \to \infty$,  uniformly in $x \in [0, \alpha_n \sqrt{n}]$,  $y \in [\alpha_n^{-1}, \alpha_n \sqrt{n}]$
and $\Delta \in [\Delta_0,  o(y) ]$,  
\begin{align} \label{demoCLLT-003bbaa}
 \mathbb P \left( x+S_n  \in [0,  \Delta]  + y,  \,  \tau_x > n  \right)  
   \sim  
\Delta \frac{2 y  V(x) }{\sqrt{2\pi } \sigma^3 n^{3/2} }.  
\end{align}
\end{theorem}

Note that the asymptotic \eqref{demoCLLT-003bbaa} does not hold 
when $\frac{\Delta}{\sqrt{n}} \to \infty$. 
Therefore Theorem \ref{Theorem-AA002bis} corrects a misstatement in \cite[Proposition 18]{Don12}, 
where it is claimed, in particular, that the asymptotic \eqref{demoCLLT-003bbaa} holds uniformly in $\Delta>0$.

%The asymptotic \eqref{demoCLLT-003bbaa} is of independent interest 
%and its proof will be given in the end of Section \ref{SecProofThm2}. 

%%%%%%%%%%%%%%%%%%%%%%%%%%%%%%%%%%%%%%%
\subsection{Starting point far from the boundary} \label{Sect-CLLT-BB}
In this section we formulate our results when the starting point $x$ is far from the boundary, 
more precisely when $\frac{x}{\sqrt{n}} \in [\eta^{-1}, \eta]$ for any real number $\eta \geq 1$.  
The following result is the analog of Theorem \ref{Theorem-AA001} for large $x$.

\begin{theorem} \label{Theorem-BB001}
Assume \ref{A1} and \ref{SecondMoment}. 
Let $\eta \geq 1$  be any fixed real number.   
%Let $(\alpha_n)_{n\geq 1}$ be any sequence of positive numbers satisfying 
%$\lim_{n\rightarrow \infty} \alpha_n =  \infty$. 
Let $f: \bb R \mapsto \bb R$ be any directly Riemann integrable function with support in $\bb R_+$ satisfying 
$\int_{\bb R_+}  f(t) (1 + t)  dt < \infty$. 
Then, there exists a constant $q_0 >0$ such that for any $q \in (0, q_0)$,
 we have, as $n \to \infty$,  uniformly in $x \in [\eta^{-1} \sqrt{n}, \eta \sqrt{n}]$ 
 and $y \in  [\eta^{-1} \sqrt{n}, \sigma \sqrt{q n \log n}]$,
 \begin{align}\label{demoCLLTLargeffaa}
& \mathbb E ( f( x+S_n - y);   \tau_x > n )   
   \sim    
  \frac{1}{\sigma \sqrt{n}}   
  \psi  \left( \frac{y}{\sigma \sqrt{n}}, \frac{ x }{\sigma \sqrt{n}} \right) 
    \int_{\bb R_+} f(t)  dt. 
\end{align}
\end{theorem}

%Theorem {Theorem-BB001} and is deduced from Theorem \ref{ThmCaraLargeStarMeasure}. 

%The asymptotic \eqref{demoCLLTLargeffaa} of course covers 
In the particular case  when $f = \mathds 1_{[0, \Delta]}$ with $\Delta >0$, 
we are able to prove the uniformly in $\Delta$. 

\begin{theorem} \label{Theorem-BB001-Delta}
Under \ref{A1} and \ref{SecondMoment}, for any $\ee, \eta, \Delta_0 >0$, there exists a constant $q_0 >0$ such that
for any $q \in (0, q_0)$, we have,  
as $n \to \infty$,  uniformly in $x \in [\eta^{-1} \sqrt{n}, \eta \sqrt{n}]$,  
$y \in [\eta^{-1} \sqrt{n}, \sigma \sqrt{q n \log n}]$
%$0 < \liminf_{n \to \infty} \frac{y}{\sqrt{n}} \leq \limsup_{n \to \infty} \frac{y}{\sqrt{n}} < \infty$,
and $\Delta \in [\Delta_0,  n^{1/2 - \ee} ]$,  
%uniformly in $x \in [\alpha_n, \eta \sqrt{n}]$, $y \in [\alpha_n, \sigma_{\lambda} \sqrt{q n \log n}]$
%%$0 < \liminf_{n \to \infty} \frac{y}{\sqrt{n}} \leq \limsup_{n \to \infty} \frac{y}{\sqrt{n}} < \infty$,
%and $\Delta \in [\Delta_0, n^{1/2 - \ee} ]$, 
\begin{align} \label{ThmCaraLargeStartaaaa}
& \mathbb P ( x+S_n  \in [0,  \Delta]  + y,   \tau_x > n )  
\sim  
    %e^{\lambda (y - x)} 
     \frac{ \Delta }{\sigma \sqrt{n}}   
     \psi  \left( \frac{y}{\sigma \sqrt{n}}, \frac{ x }{\sigma \sqrt{n}} \right).  
\end{align}
\end{theorem}

%Theorem \ref{ThmCaraLargeStart} is coherent with Theorem \ref{Thm-C 002} when 
%$x \in [\alpha_n, n^{1/2 - \ee}]$ and $y \in [\alpha_n, \sigma_{\lambda}\sqrt{q n \log n}]$.
%Indeed, by elementary calculations, 
%we have, as $n \to \infty$, uniformly in $x \in [\alpha_n, n^{1/2 - \ee}]$ and $y \in [\alpha_n, \sigma_{\lambda} \sqrt{q n \log n}]$,
%\begin{align*}%\label{} 
%  \psi \left( \frac{y}{\sigma_{\lambda} \sqrt{n}}, \frac{ x }{\sigma_{\lambda} \sqrt{n}}  \right)
%  \sim  
%  \frac{2xy}{ \sqrt{2\pi } \sigma_{\lambda}^2 n } e^{- \frac{y^2}{2 \sigma_{\lambda}^2 n} }
%  \sim 
%  \frac{2V_{\lambda} \left( x\right) }{\sqrt{2\pi n} \sigma_{\lambda}  }   
% \phi^+ \left( \frac{y}{ \sigma_{\lambda}  \sqrt{n}} \right). 
%\end{align*}
In particular,  
taking $x = \sigma \eta \sqrt{n}$ and $y = \sigma \sqrt{q n \log n}$ in \eqref{ThmCaraLargeStartaaaa}, 
%In particular, in \eqref{ThmCaraLargeStartaa} taking $\lambda = 0$ and $y = \sigma \sqrt{q n \log n}$,
%where $a_n \leq \sqrt{q \log n}$ and $a_n \to \infty$ as $n \to \infty$,  
we have,  as $n \to \infty$,  uniformly in $\Delta \in [\Delta_0, n^{1/2 - \ee} ]$, 
\begin{align}\label{ThmCaraLargeStartbbaa}
\mathbb{P} \left( x+S_n  \in  [0, \Delta] + \sigma \sqrt{q n \log n},  \,  \tau_x > n  \right) 
\sim  \frac{ \Delta e^{- \frac{\eta^2}{2} + \eta \sqrt{q \log n}} }{\sqrt{2\pi } \sigma n^{(1 + q)/2}}. 
\end{align}

We continue with an analog of Theorem \ref{Theorem-AA002} for $x$ large.
\begin{theorem} \label{Theorem-BB002}
Assume \ref{A1} and \ref{SecondMoment}. Let $\eta \geq 1$ be any fixed real number. 
%Let $(\alpha_n)_{n\geq 1}$ be any 
%sequence of positive numbers satisfying $\lim_{n \to \infty} \alpha_n =  \infty$. 

1. Let $f: \bb R \mapsto \bb R$ be any directly Riemann integrable function with support in $\bb R_+$ satisfying
$\int_{\bb R_+}  f(t) (1 + t)^{\gamma}  dt < \infty$ for some constant $\gamma >1$. 
Then, for any $a >0$, 
%for any directly Riemann integrable function $f: \bb R_+ \mapsto \bb R$ 
%satisfying $\int_{\bb R_+} (1 + t) e^{-\lambda t} f(t) dt < \infty$,
we have, as $n \to \infty$,  uniformly in $x \in [\eta^{-1} \sqrt{n}, \eta \sqrt{n}]$ and $y \in [0, a]$,  
%\todos{we should introduce $y\in [0,c]$. This is uniform in $y$}
\begin{align}\label{LLT32LargeStartingaa}
\mathbb E \left( f(x+S_n - y);  \,  \tau_x > n  \right)  
  \sim  
    \frac{2}{ \sqrt{2\pi} \sigma^2 n }  \phi^+ \left( \frac{x}{\sigma \sqrt{n}} \right)
    \int_{\bb R_+}   f( t - y)  V^{\ast } (t) dt. 
\end{align}
%%Assume that $\bb E X_1 = \mu <0$ and  ,,,
%%Let $(\alpha_{n})_{n\geq 1}$ be any sequence of positive numbers such that $\lim_{n\rightarrow \infty} \alpha_{n} = 0$.  
%In particular,  for any fixed constants $y \geq 0$ and  $\Delta >0$, we have, 
%%there exists a constant $\ee >0$ such that,  
%as $n \to \infty$,  uniformly in $x \in [\eta^{-1} \sqrt{n}, \eta \sqrt{n}]$,  
%\begin{align*}%\label{}
%&  \mathbb P ( x + S_n \in [0, \Delta] + y,   \tau_x > n )   
%\sim   
%\frac{2  e^{n \Lambda(\lambda) + \lambda  x} }{ \sqrt{2\pi} \sigma_{\lambda}^2 n }  \phi^+ \Big( \frac{x}{\sigma_{\lambda} \sqrt{n}} \Big)
%    \int_{y}^{y + \Delta}  e^{-\lambda t} V_{\lambda}^{\ast } (t) dt. 
%\end{align*}

2. Let $f: \bb R \mapsto \bb R$ be any directly Riemann integrable function with support in $\bb R_+$ satisfying 
$\int_{\bb R_+}  f(t) (1 + t)  dt < \infty$.    
Let $(\alpha_n)_{n\geq 1}$ be any sequence of positive numbers satisfying 
$\lim_{n\to \infty} \alpha_n =  0$ and $ n^{1/4}  \alpha_n \geq 1$.
Then we have, as $n \to \infty$,  uniformly in $x \in [\eta^{-1} \sqrt{n}, \eta \sqrt{n}]$ 
 and $y \in  [\alpha_n^{-1}, \alpha_n \sqrt{n}]$,
 \begin{align}\label{demoCLLTLargeffbbaa}  
& \mathbb E ( f( x+S_n - y);   \tau_x > n )   
   \sim    
  \frac{ 2 y }{\sqrt{2 \pi} \sigma^2 n}   
  \phi^+  \left( \frac{ x }{\sigma \sqrt{n}} \right)  
    \int_{\bb R_+} f(t)  dt. 
\end{align}
\end{theorem}

%The asymptotic \eqref{demoCLLTLargeffbbaa} covers the particular case 
%where $f = \mathds 1_{[0, \Delta]}$ for any fixed real number $\Delta >0$. 
%In fact, we can establish 

The following version of \eqref{demoCLLTLargeffbbaa} with $f = \mathds 1_{[0, \Delta]}$ %, $\Delta >0$,    
holds uniformly in $\Delta$. 

\begin{theorem} \label{Theorem-BB002bis}
Assume \ref{A1} and \ref{SecondMoment}. 
Let $(\alpha_n)_{n\geq 1}$ be any sequence of positive numbers satisfying
$\lim_{n\to \infty} \alpha_n =  0$ and $ n^{1/4}  \alpha_n \geq 1$.
Then, for any $\Delta_0 >0$, 
we have,  as $n \to \infty$,  uniformly in $x \in [\eta^{-1} \sqrt{n}, \eta \sqrt{n}]$,  
$y \in  [\alpha_n^{-1}, \alpha_n \sqrt{n}]$ and $\Delta \in [\Delta_0,  o(y) ]$,  
%uniformly in $x \in [\alpha_n, \eta \sqrt{n}]$, $y \in [\alpha_n, \sigma_{\lambda} \sqrt{q n \log n}]$
%%$0 < \liminf_{n \to \infty} \frac{y}{\sqrt{n}} \leq \limsup_{n \to \infty} \frac{y}{\sqrt{n}} < \infty$,
%and $\Delta \in [\Delta_0, n^{1/2 - \ee} ]$, 
\begin{align} \label{ThmCaraLargeStartaabbaa}
& \mathbb P ( x+S_n  \in [0,  \Delta]  + y,   \tau_x > n )  
\sim  
 \Delta  \frac{ 2 y  }{\sqrt{2 \pi} \sigma^2  n }   
     \phi^+  \left( \frac{ x }{\sigma \sqrt{n}} \right).  
\end{align}
\end{theorem}
%The proof of \eqref{ThmCaraLargeStartaabbaa} will be given in the end of Section \ref{SecProofThm4}. 

%%%%%%%%%%%%%%%%%%%%%%%%%%%%%%%%%%%%%%%%%%%%%%
\subsection{Random walks with drift}\label{Sect-Application}
In this section we consider the case when the random walk has non-zero drift. % $\mu$ is not necessarily $0$. 
As before, assume  that on the probability space $\left( \Omega ,\mathscr{F},\mathbb{P}\right)$ 
the sequence $(X_i)_{i\geq 1}$ is independent identically distributed  
with $\bb E X_1 = \mu \in \bb R$ and  $\bb E X_1^2 = \sigma^2 \in (0, \infty)$. 
Let $S_n =\sum_{i=1}^{n}X_{i}$, $n \geq 1$. 
%\begin{equation}
%S_n =\sum_{i=1}^{n}X_{i}, \quad n \geq 1.  \label{iid-000}
%\end{equation}
%%Note that $X_{0}$ is not included in the sum $S_n .$ 
For any starting point $x\in \bb R_+: = [0,\infty)$,  define  
\begin{align*}
\tau_x %%%%= \inf \left\{ k\geq 0: x+S_{k} < 0 \right\} 
= \inf \left\{ k\geq 1: x+S_{k} < 0 \right\} \quad \mbox{with} \   \inf\emptyset = \infty.   
\end{align*}
%where $\inf\emptyset = \infty$.
%It is well known that, if $\mu\leq 0$ it holds $\bb P(\tau_x <\infty)=1,$ for all $x\geq 0.$

The case when random walk $(S_n)_{n\geq 1}$ has a negative drift, i.e. $\mu < 0$, 
was studied by 
%In the case when $\mu < 0$, i.e. the random walk $(S_n)$ has a negative drift, 
Iglehart \cite{Igl74}, where  the exact asymptotic of the probability $\bb P(\tau_0 >n)$ has been established
under an exponential moment condition on $X_1$.  
Further results of the probability \eqref{Objective-proba001} can be found in Doney \cite{Don89}, Keener \cite{Kee92} and, 
under the assumption that $X_1$ is heavy tailed and has finite second moment condition 
in Bansaye and Vatutin \cite{BV13}. %  have described the asymptotics of the probability \eqref{Objective-proba001}. 
When %Under the assumption that 
the random variable $X_1$ is absolutely continuous with subexponential density, 
 Denisov, Vatutin and Wachtel \cite{DVW14} investigated % have determined 
 the asymptotic behavior of %the probability  
 \eqref{Objective-proba001} for fixed $x \geq 0$ and  various ranges of $y$.

%Denote 
%\begin{align*} %\label{}
%I_{\mu} : = \left\{ s \in \bb R:  \bb E e^{s X_1} < \infty \right\}.
%\end{align*} %\notag\\
All the results in Sections \ref{Sect-CLLT-AA} and \ref{Sect-CLLT-BB} have analogs for random walks with drift,
by using an exponential change of probability measure, under the following additional Cram\'er-type condition: 
\begin{conditionA}\label{ExponentialMoment}
 %The random variable $X_1$ has an exponential moment, i.e.\  
 There exist constants $\lambda \in \bb R$ and  $\delta > 0$ % : = \{ s \in \bb R:  \bb E e^{s X_1} < \infty \}$ 
 such that 
\begin{align*} %\label{}
\bb E  X_1  e^{\lambda X_1} = 0 
%\qquad
%\bb E  |X_1|^{2} e^{ \lambda X_1 } = \sigma_{\lambda}^2  <\infty, 
\quad
\mbox{and}
\quad
\bb E  |X_1|^{2 + \delta} e^{ \lambda X_1 }  < \infty. 
\end{align*}
\end{conditionA}
It is easy to see that $\bb E e^{\lambda X_1} \leq \bb E |X_1|^{2+\delta} \mathds1 _{\{|X_1| >1\}} e^{\lambda X_1}< \infty$.

Note that from condition \ref{ExponentialMoment} we have that $\lambda > 0$ ($\lambda = 0$, $\lambda < 0$) 
when $\mu < 0$ ($\mu = 0$, $\mu > 0$).
%In the second case $\lim_{n\to\infty}\bb P (\tau_x  >n) =1 $.  
In the case when $\mu=0$ 
%the constant $\lambda$ in condition \ref{ExponentialMoment}  equals $0$ and therefore 
condition \ref{ExponentialMoment} becomes \ref{SecondMoment}.
% so that $\mu = \bb E X_1 = 0$.

%%%%%%%%%%%%%%%%%%%%%%%%%%%%%%%%%%%

As before, let $\left( S^*_{n}\right) _{n\geq 1}$ be the dual random walk: $S^*_{n}=-\sum_{i=1}^{n}X_{i}$. 
Denote by $\tau^*_x$ the dual exit time: $\tau^*_x = \inf \left\{ k\geq 1: x+S^*_{k} < 0 \right\}$
with $\inf\emptyset = \infty$.  
Let $\lambda \in I_{\mu}$ be from  condition \ref{ExponentialMoment}. 
Denote 
\begin{align*}%\label{}
\Lambda(\lambda) = \log \bb E e^{\lambda X_1}  \quad  \mbox{and}  \quad  \sigma_{\lambda}^2  = \bb E  |X_1|^{2} e^{ \lambda X_1 }.
\end{align*}
%$\Lambda(\lambda) = \log \bb E (e^{\lambda X_1})$ 
%and $\sigma_{\lambda}^2  = \bb E  |X_1|^{2} e^{ \lambda X_1 }.$ 
Define a change of probability measure by setting 
\begin{align*} %\label{}
\bb P_{\lambda} (X_1 \in dx) = e^{\lambda x - \Lambda(\lambda)} \bb P(X_1 \in dx) 
\end{align*}
and let $\bb E_{\lambda}$ be the expectation corresponding to $\bb P_{\lambda}$. 
From condition \ref{ExponentialMoment} it follows that $\Lambda(\lambda) <0$ whenever $\lambda \neq 0$.  
For $x \geq 0$, let $V_{\lambda}(x) : = -\bb E_{\lambda} (S_{\tau_x})$ 
and $V_{\lambda}^*(x) : = -\bb E_{\lambda} (S_{\tau_x^*}^*)$
be  the harmonic functions of random walks $S_n$ and $S_n^*$
killed at $\tau_x$ and $\tau_x^*$ under the probability measure $\bb P_{\lambda}$, respectively.  
The functions $V_{\lambda}$ and $V_{\lambda}^*$ can be extended to the whole real line $\bb R$
in the same way as the harmonic function $V$. 
% \label{Doob transf}
%\bb E V(x+S_1)\mathds 1 _{\{ x+S_1\geq 0  \}} = V(x).   
%\end{align} 

The results of this section are obtained from those in Sections \ref{Sect-CLLT-AA} and \ref{Sect-CLLT-BB}
by using an exponential change of probability measure.  
The key point in applying this change of measure is the presence of the target functions in 
Theorems \ref{Theorem-AA002} and \ref{Theorem-BB002}.

In the particular case when the random walk has negative drift, i.e.\, $\mu = \bb E X_1 < 0$ 
(in this case $\lambda > 0$ by condition \ref{ExponentialMoment}), 
we have the following result. 

\begin{theorem}\label{Thm-Iglehart001}
Assume \ref{A1} and \ref{ExponentialMoment} for some $\lambda > 0$. 
Let $(\alpha_n)_{n\geq 1}$ be any sequence of positive numbers satisfying
$\lim_{n\rightarrow \infty} \alpha_n =  0$. 
Then we have, as $n \to \infty$,  uniformly in $x \in [0, \alpha_n \sqrt{n}]$, 
\begin{align}\label{tauxNegativeDrift}
\bb P (\tau_x >n) \sim   \frac{ 2 V_{\lambda}(x) e^{ n \Lambda(\lambda) + \lambda x} }{ \sqrt{2\pi }  \sigma_{\lambda}^3  n^{3/2}  } 
    \int_{\bb R_+}  e^{-\lambda t}  V_{\lambda}^{\ast } (t) dt. 
\end{align}
\end{theorem}

The asymptotic \eqref{tauxNegativeDrift} improves the results of 
Iglehart \cite[Theorem 2.1]{Igl74} and Doney \cite[Theorem II]{Don89}. % in two aspects.
Firstly,  our result \eqref{tauxNegativeDrift} holds uniformly in $x \in [0, \alpha_n \sqrt{n}]$, 
while \cite{Igl74} deals with the particular case $x =0$
and \cite{Don89} deals with fixed $x \geq 0$. 
Secondly, our moment condition \ref{ExponentialMoment} is weaker than that used in \cite{Igl74}. 
%it is assumed in \cite{Igl74} that $\lambda$ should be in the interior of the interval $I_{\mu}$. 

%It is worth mentioning that limit theorems for random walks on groups have been established by 
%Bougerol \cite{Bou81} and Gou\"ezel \cite{Gou14}.  

%In the particular case when the random walk has negative drift, i.e.\, $\mu = \bb E X_1 < 0$ 
%(in this case $\lambda > 0$ by condition \ref{ExponentialMoment}), 
%%In the particular case when $\bb E X_1 < 0$ and $f = \mathds 1_{[0, \infty)}$, 
%with $f = \mathds 1_{[0, \infty)}$, 
The following theorem complements the results of Iglehart \cite{Igl74} and Doney \cite{Don89} when $x$ is large:  

\begin{theorem}\label{Thm-Iglehart002}
Assume \ref{A1} and \ref{ExponentialMoment} for some $\lambda > 0$.  %Let $\eta \geq 1$  be any fixed real number.   
Then, for any $\eta \geq 1$, we have, as $n \to \infty$,  
uniformly in $x \in [\eta^{-1} \sqrt{n}, \eta \sqrt{n}]$,  
\begin{align*}%\label{}
\bb P (\tau_x >n) \sim   
\frac{2 e^{n \Lambda(\lambda) + \lambda  x}  }{ \sqrt{2\pi} \sigma_{\lambda}^2 n }  \phi^+ \left( \frac{x}{\sigma_{\lambda} \sqrt{n}} \right) 
    \int_{\bb R_+}  e^{-\lambda t}  V_{\lambda}^{\ast } (t) dt. 
\end{align*}
\end{theorem}

%%%%%%%%%%%%%%%%%%%%%%%%%%%%%%%%%%%%%%%%%
\subsection{Local limit theorems for the exit time}\label{Sect-LLT-exit-taux}
In this section we formulate two local limit theorems for the exit time $\tau_x$.

Our first result gives the asymptotics of the % local behavior of $\tau_x$ 
probability $\mathbb P \left( \tau_x  =  n  \right)$ when the random walk is driftless. 
Set
\begin{align}\label{Def-varkappa}
\varkappa := \int_{- \infty}^0 V^{\ast } (t) dt = \int_{\bb R_+}   \bb P(t + X_1 < 0)   V^{\ast } (t) dt. %\in (0, \infty). 
\end{align}
The second equality follows from the harmonicity property \eqref{Doob transf}. 
By using condition \ref{A1} and Markov's inequality, 
one can check that the second integral in \eqref{Def-varkappa} is strictly positive and finite; see Lemma \ref{Lem-Equ-Harmonic}. 

\begin{theorem} \label{Thm-LLT-taux-xsmall}
Assume \ref{A1} and \ref{SecondMoment}. 
Let $(\alpha_n)_{n\geq 1}$ be any sequence of positive numbers satisfying
$\lim_{n\rightarrow \infty} \alpha_n =  0$. 
Then we have, as $n \to \infty$,  uniformly in $x \in [0, \alpha_n \sqrt{n}]$, 
\begin{align}\label{Asym-tau-n-xsmall}
\mathbb P \left( \tau_x  =  n  \right)
\sim      \frac{ 2 \varkappa V(x) %\int_{-\infty}^{0}  V^{\ast } (t) dt \in (0, \infty) 
}{ \sqrt{2\pi }  \sigma^3  n^{3/2}  }.
%    \int_{\bb R_+}   \bb P(t + X_1 < 0)   V^{\ast } (t) dt. 
\end{align}
Moreover, for any $\eta \geq 1$, it holds, as $n \to \infty$,  
uniformly in $x \in [\eta^{-1} \sqrt{n}, \eta \sqrt{n}]$,  
\begin{align}\label{Asym-tau-n-xlarge}
\mathbb P \left( \tau_x  =  n  \right)
\sim     \frac{2 \varkappa}{ \sqrt{2\pi} \sigma^2 n }  \phi^+ \left( \frac{x}{\sigma \sqrt{n}} \right). 
%    \int_{\bb R_+}   \bb P(t + X_1 < 0)   V^{\ast } (t) dt. 
\end{align}
\end{theorem}

For fixed $x \geq 0$, asymptotics of the probability 
$\mathbb P \left( \tau_x  =  n  \right)$ %asymptotic \eqref{Asym-tau-n-xsmall} 
have been obtained for instance by  
 Eppel \cite{Eppel-1979},  Borovkov \cite{Bor70} and   Vatutin and Wachtel \cite{VatWacht09}.
% however without explicit values of the involved constants.   
A uniform version has been obtained by Doney \cite{Don12}. 
The asymptotic \eqref{Asym-tau-n-xsmall} is in accordance with 
Theorem 2.8 of \cite{GLL20}, which is stated for Markov chains with finite state space and for fixed $x \geq 0$.
However, the setting in \cite{GLL20} does not cover the general case of random walks in $\bb R$.

To the best of our knowledge, the explicit formula \eqref{Def-varkappa} for the constant $\varkappa$ in 
the asymptotics \eqref{Asym-tau-n-xsmall} and \eqref{Asym-tau-n-xlarge} 
has not been known in the literature. 
Our formulas \eqref{Asym-tau-n-xsmall} and \eqref{Asym-tau-n-xlarge} (under assumptions \ref{A1} and \ref{SecondMoment}) 
correct some misstatement 
in the analogous results in papers \cite[Theorems 7]{VatWacht09} and \cite[Theorem 2 (A)]{Don12}. 
Under the setting of our paper ($X_1$ is in the domain of attraction of the normal law), 
these results claim that, % can be formulated in the following way: 
for any fixed $x \geq 0$, as $n \to \infty$,
\begin{align}\label{Result-Error}
\mathbb P \left( \tau_x  =  n  \right)
\sim      \frac{ V(x) }{ 2 \sqrt{2\pi }  \sigma  n^{3/2}  }.
\end{align}
Therefore, compared with \eqref{Asym-tau-n-xsmall},  
the factor $\frac{4 \varkappa}{\sigma^2}$ is missing in \eqref{Result-Error}. 
The same remark also applies to some other results, in particular to \cite[Theorems 10]{VatWacht09} and \cite[Theorem 2 (B)]{Don12}.

Our next result extends Theorem \ref{Thm-LLT-taux-xsmall} to the case of random walks with drift. 
Let 
\begin{align}\label{Def-varkappa-lamb}
\varkappa_{\lambda}  := \bb E e^{\lambda X_1} \int_{- \infty}^0  e^{-\lambda t}  V_{\lambda}^{\ast } (t) dt 
=  \int_{\bb R_+}   \bb P(t + X_1 < 0)  e^{-\lambda t}  V_{\lambda}^{\ast } (t) dt. %  \in (0, \infty). 
\end{align}
The second equality will be proved in Lemma \ref{Lem-Equ-Harmonic} where we also show that 
the constant $\varkappa_{\lambda}$ is strictly positive and finite. 

%will be shown in the proof of Theorem \ref{Thm-LLT-taux-xlarge} which is stated below.  
%As in \eqref{Def-varkappa}, the second integral in \eqref{Def-varkappa-lamb} is strictly positive and finite
%by using condition \ref{ExponentialMoment} and Markov's inequality. 

\begin{theorem}\label{Thm-LLT-taux-xlarge}
Assume \ref{A1} and \ref{ExponentialMoment}. 
Let $(\alpha_n)_{n\geq 1}$ be any sequence of positive numbers satisfying
$\lim_{n\rightarrow \infty} \alpha_n =  0$. 
Then we have, as $n \to \infty$,  uniformly in $x \in [0, \alpha_n \sqrt{n}]$, 
\begin{align}\label{Asym-tau-n-xsmall-changemeas}
\mathbb P \left( \tau_x  =  n  \right)
\sim      \frac{ 2 \varkappa_{\lambda}  V_{\lambda}(x) e^{n \Lambda(\lambda) + \lambda x} }{ \sqrt{2\pi }  \sigma_{\lambda}^3  n^{3/2}  }.   
%    \int_{\bb R_+}   \bb P(t + X_1 < 0)  e^{-\lambda t}  V_{\lambda}^{\ast } (t) dt
\end{align}
Moreover, for any $\eta \geq 1$, it holds, as $n \to \infty$,  
uniformly in $x \in [\eta^{-1} \sqrt{n}, \eta \sqrt{n}]$,  
\begin{align}\label{Asym-tau-n-xlarge-changemeas}
\mathbb P \left( \tau_x  =  n  \right)
\sim      \frac{2  \varkappa_{\lambda} e^{n \Lambda(\lambda) + \lambda  x}  }{ \sqrt{2\pi} \sigma_{\lambda}^2 n }  
      \phi^+ \left( \frac{x}{\sigma_{\lambda} \sqrt{n}} \right). 
%    \int_{\bb R_+}   \bb P(t + X_1 < 0)  e^{-\lambda t}  V_{\lambda}^{\ast } (t) dt. 
\end{align}
\end{theorem}

The asymptotics \eqref{Asym-tau-n-xsmall-changemeas} and \eqref{Asym-tau-n-xlarge-changemeas} 
for random walks with drift are new to our knowledge.

The proof of Theorems \ref{Thm-LLT-taux-xsmall} and \ref{Thm-LLT-taux-xlarge}
% \eqref{Asym-tau-n-xsmall}, \eqref{Asym-tau-n-xlarge}, 
%\eqref{Asym-tau-n-xsmall-changemeas} and \eqref{Asym-tau-n-xlarge-changemeas} 
is given in Section \ref{Sec: loc taux}.
Our proof is an elementary application of the local limit theorems with target functions, 
% and of the Markov property, which 
which is different from the proofs in \cite{Eppel-1979},  \cite{Bor70}, \cite{VatWacht09} and \cite{Don12}
based on more involved arguments.

%Theorem \ref{Thm-LLT-taux-xsmall} is deduced from the asymptotics
%\eqref{Asymn32Smallxaa} and \eqref{LLT32LargeStartingaa} of Theorems \ref{Theorem-AA002} and \ref{Theorem-BB002} respectively. 

%%%%%%%%%%%%%%%%%%%%%%%%%%%%%%%%%%%%%%%%%%%%%%
\subsection{Method of the proof}
Usually to study the asymptotic behavior of the probability 
%$\mathbb{P} \left( x + S_n  \in y+ [0, \Delta],  \tau_x > n \right)$ 
\eqref{Objective-proba001}
the Wiener-Hopf factorization is employed.
However for the proof of our results this technique does not seem appropriate.
In the  present paper we adopt an approach inspired by \cite{Den Wachtel 2011}.
To give a sketch of the proofs let us denote $f = \mathds 1_{y+[0, \Delta]}$ and let us
  assume first that in the expectation % \eqref{Objective-proba001}
$\mathbb{E} \left( f(x+S_n); \tau_x >n\right)$
 the starting point $x\geq 0$ is near the 
boundary (in our setting the level $0$). 
%and that the variable $y$  is arbitrary. %of the interval $[0,\Delta]$ 
%is moving to infinity with the speed $\sqrt{n}$ (this is a particular case of Theorem \ref{Thm-C 002}).
%We denote for short $f_y = \mathds 1_{y+[0, \Delta]}$.
%For simplicity we assume that $\sigma=1$.
Then splitting the trajectory of the walk $(x + S_n)_{n\geq 0}$ 
into two parts of length $k$ and $m$ with $k+m=n$ and using the Markov property after some transformations we can approximate 
the probability \eqref{Objective-proba001} as follows (see \eqref{JJJ-markov property} and 
\eqref{staring point for the lower-bound-001} for more details):
\begin{align} \label{methods-001}
\mathbb{E} \left( f(x+S_n); \tau_x >n\right) 
\approx \int_{\mathbb R_+}  \mathbb{E} f(t+S_{m}) \mathbb{P}\left( x+S_{k}\in dt,  \tau_x >k\right).
 \end{align} 
%where $r_n$ is some remainder therm which converges to $0$
%faster than the first term in the right hand side of \eqref{methods-001}.  
%There are two type of conditioned local limit theorems. One of them is effective when the support 
%is moving to 
% we consider the proof of conditioned local limit theorem with rate $n^{-1/2}$ 
%(see Theorem \ref{Thm-C 002}). An upper bound for 
The expectation 
$\mathbb{E} f(t+S_{m})$ 
corresponding to the second part of the trajectory 
is handled by using the ordinary local limit theorem (Theorem \ref{LLT-general}), which gives
\begin{align} \label{methods-001a}
\mathbb{E} \left( f(x+S_{n}); \tau_x >n\right) \approx
 \int_{\bb R _{+}}  \varphi_n(t)  \mathbb{P}\left(  \frac{x+S_{k}}{\sigma \sqrt{k}}\in dt, \tau_{x}>k\right), 
%\mathbb{P}\left( x+S_{k}\in dt;  \tau_x >k\right)
\end{align}
where $ \varphi_n(t) := \int_{\mathbb R} f(\sigma\sqrt{k} s) \phi_{m/k} (t-s) ds$
and $\phi_{v}$ is the normal density of mean zero and variance $v>0$.
%with some function $\varphi_n$ defined by \eqref{JJJ006b}.
For the integral in \eqref{methods-001a}, which corresponds to the first part of the trajectory, 
we apply the conditioned integral limit theorem (Theorem \ref{Theor-IntegrLimTh}). 
After some elementary calculations, this leads to the approximation
\begin{align} \label{methodproof-002}
\mathbb{E} \left( f(x+S_{n}); \tau_x >n\right) \approx
\frac{2V(x)}{ \sqrt{2\pi k} \sigma}  \int_{\mathbb R} f(\sigma\sqrt{n} t)  \phi_{\delta_n}*\phi^+_{1-\delta_n}(t) dt,
\end{align}
%where $\phi_{\delta_n}$ is the normal density with variance $\delta_n=\frac{m}{n}$ and 
where $\delta_n=\frac{m}{n}$ and $\phi^+_{v}$ is the Rayleigh density with scale parameter $\sqrt{v}$, $v>0$.
Using the convolution Lemma \ref{t-Aux lemma} as $\delta_n$ becomes small, we obtain, as $n\to \infty$,
\begin{align} \label{methodproof-003}
\mathbb{E} \left( f(x+S_{n}); \tau_x >n\right) \approx
\frac{2V(x)}{ \sqrt{2\pi} \sigma^2 n}  \int_{\mathbb R} f( t) \phi^+\left(\frac{t}{\sigma\sqrt{n}}\right) dt, 
\end{align}
from which we can deduce Theorem \ref{Theorem-AA001}.
Note that, the main term in the asymptotic \eqref{methodproof-003} becomes meaningful 
only when the support of the function moves to infinity as $n\to\infty$, 
which is the case for instance when $f = \mathds 1_{y+[0, \Delta]}$ and $y$ moves to $\infty$.
%The corresponding assertion of Theorem \ref{Thm-C 002} is obtained from \eqref{methodproof-003} 
%by applying the approximation \eqref{methodproof-003} to the random walk $(x + S_n)_{n\geq 0}$ 
%under an exponential change of measure for $y$ sufficiently large. 
The use of the convolution step \eqref{methodproof-002} is the crucial point of our approach, which 
makes the difference with that of \cite{Den Wachtel 2011}.
It also clarifies the idea of the method and greatly simplifies the computations
compared to \cite{Den Wachtel 2011}. 

The case when  $f = \mathds 1_{y+[0, \Delta]}$ and both $x$ and $y$ are near the boundary %(here the level $0$ -- 
(which is a particular case of Theorem \ref{Theorem-AA002}) is handled
in a different way.    
Using the Markov property we obtain the approximation 
\begin{align} \label{HHHHH-001}
\mathbb{E} \left( f(x+S_{n}); \tau_x >n\right) \approx 
 \int_{\mathbb R_+}   H_m(t) \mathbb{P}\left( x+S_{k}\in dt,  \tau_x >k\right),
\end{align} %\notag\\
where $H_m(t)=\mathbb{E} \left( f(t+S_{m}); \tau_t  > m\right)$.
First, by the asymptotic \eqref{methodproof-003} 
applied with the function $f=H_m$ and $n=k$, %for the first part of the trajectory,
%that we have established before 
we have
\begin{align*} %\label{HHHHH-002}
\mathbb{E} \left( f(x+S_{n}); \tau_x >n\right) \approx 
\frac{2V(x)}{ \sqrt{2\pi }\sigma^2k}  \int_{\mathbb R_+}  H_m(t) %\mathbb{E} \left( f(s+S_{m}); \tau_s  > m\right) 
  \phi^+\left(\frac{t}{\sigma\sqrt{k}}\right) dt,
%\mathbb{P}\left( x+S_{k}\in dt;  \tau_x >k\right),
\end{align*}  
from which by the reversibility we get 
% of the random walk $(x + S_n)_{n\geq 0}$ 
\begin{align*} %\label{}
\mathbb{E} \left( f(x+S_{n}); \tau_x >n\right) \approx 
\frac{2V(x)}{\sqrt{2\pi }\sigma^2 k}
 \int_{\mathbb R_+} f(t)  \mathbb{E} \left[ \phi^+\left(\frac{t+S_{m}^*}{\sigma \sqrt{m}}\right); \tau_t^*  > m \right]  dt.
\end{align*} %\notag\\
The expectation inside the integral (which corresponds to the second part of the trajectory)
is handled using the conditioned integral limit theorem 
(Theorem \ref{Theor-IntegrLimTh}) for the reversed walk 
$(t + S_n^*)_{n\geq 0}$, which after some elementary
transformations leads to the approximation
\begin{align} \label{methodproof-005}
\mathbb{E} \left( f(x+S_{n}); \tau_x >n\right) \approx
\frac{2V(x)}{ \sqrt{2\pi } \sigma^3 n^{3/2}}  \int_{\mathbb R} f( t) V^*(t) dt.
\end{align}
This gives the assertion \eqref{Asymn32Smallxaa} of Theorem \ref{Theorem-AA002}. 
%applying the approximation \eqref{methodproof-005} to a random walk under a change of measure.
It is worth mentioning that when handling the right hand side of \eqref{HHHHH-001}
there is a problem with the function $H_m$ which does not comply with the assumptions of 
Theorem \ref{Theorem-AA001}. 
%in order  with the assumptions of Theorem \label{Theorem-AA001}
 % we use the asymptotic \eqref{methodproof-003} 
%with the function $f=H_m$, which in general ????
We overcome this by using %we make use of 
a special smoothing technique developed 
in Section \ref{subsec-auxiliary}.   

The corresponding assertion of Theorem \ref{Thm-Iglehart001} is obtained %from \eqref{methodproof-005}
by applying the approximation \eqref{methodproof-005} to the random walk $(x + S_n)_{n\geq 0}$ 
under an exponential change of measure. At this stage a result with a target function $f$ is necessary. 

The case when the staring point $x$ is of order $\sqrt{n}$ is treated much in the same way, 
except that the product $\phi^+(t)V(x) $ is replaced by the function $\psi(t,x)$ defined by \eqref{Def-Levydens}.
In order to perform this program,  we establish 
effective local limit theorems for $S_n$ with target functions on $S_n$   
and with explicit remainders expressed in terms of the target function.
We will also establish conditioned integral limit theorems with a rate of convergence. 
% in % effective versions of 

Let us note that the method described above works also for sums of i.i.d. random variables taking values in lattices
and can also be applied for random walks in cones.  This last point is treated in a subsequent paper.  
We also would like to point out some other advantages of our method. 
With our approach we can avoid the use of the reversibility of the random walk  $(x + S_n)_{n\geq 0}$
(this requires a minor modification of the proof of the lower bound in section \ref{sec: proof lower bound}).
This is in contract to the previous work in the area \cite{Carav05, Don12, KV17, VatWacht09} where the reversibility is 
used essentially in the proof. 
Therefore with our approach we can extend Theorems \ref{Theorem-AA001}, \ref{Theorem-AA002},
\ref{Theorem-BB001} and \ref{Theorem-BB002}
to the case of  Markov random walks.  
We refer to Grama, Quint and Xiao \cite{GQX21} for a study on conditioned limit theorems for 
 hyperbolic dynamical systems and to Grama, Mentemeier and Xiao \cite{GMX21} where this type of results are 
applied for the study of branching products of random matrices.

%%%%%%%%%%%%%%%%%%%%%%%%%%%%%%%%%%%%%%%%%%
%%%%%%%%%%%%%%%%%%%%%%%%%%%%%%%%%%%%%%%%%%
\section{Effective limit theorems}

\subsection{Preliminary results} \label{subsec-auxiliary}

%We say that the function g ε-dominates the function f (or f ε-minorates g) if for any t ∈ R
%\subsection{Upper and lower $\ee$-envelopes} \label{subsec-envelopes}
%\subsection{Directly Riemann integrable functions} \label{subsec-envelopes}

%We shall consider directly Riemann integrable functions. 

%Then, for any $u\in\bb R$, we have $\sup_{|v| \leq \ee} f(u+v) 
%\leq \sup_{|v| \leq \ee} \overline f_{\delta}(u+v) =: \overline f_{\delta,\ee}(u),$ 
%so $\overline f_{\delta,\ee}$ $\ee$-dominates $f$
%%is an upper $\ee$-envelope of $f$.  %where $\overline h_{\delta,\ee}$ 
%and is measurable since $\overline f_{\delta}$ takes values in a countable set.
%In the same way, $\underline f_{\delta,-\ee}:= \inf_{|v| \leq \ee} \overline f_{\delta}(u+v)$ $\ee$-minorates $f$ 
%and is measurable. 
%% is a Borel measurable lower $\ee$-envelope of $f$. 
%By definition we have 
%$\underline f_{\delta,-\ee} \leq_{\ee} \underline f_{\delta} \leq f \leq \overline f_{\delta}\leq_{\ee} \overline f_{\delta,\ee}$.
%By \eqref{ladderfun002},  if $f: \bb R \to \bb R_+$ is a bounded measurable function
%then there exists a bounded measurable function $g: \bb R \to \bb R_+$
%such that $f\leq_{\ee} g$. 

%Recall that a non-negative function $f$ is d.R.i.\
%if its upper and lower Riemann sums over  the whole real line $\bb R$ 
%are finite and converge to the same finite limit, as the partition becomes finer. 
Let $f : \bb R \mapsto  \bb R_+$ be a non-negative Borel measurable function. 
Consider the upper and lower ladder functions: for $I_k = [k \delta, (k+1) \delta)$ with $\delta >0$, 
\begin{align*}
\overline f_{\delta}(u) = \sum_{k\in \bb Z}  \mathds 1_{I_k}(u)  \sup_{u' \in I_k} f(u'), \quad % \label{ladderfun001} \quad
\underline f_{\delta}(u) = \sum_{k\in \bb Z}  \mathds 1_{I_k}(u)  \inf_{u' \in I_k} f(u'),  \quad  u \in \bb R.   %\label{ladderfun002}
\end{align*}
The function $f$
%$f : \bb R \mapsto  \bb R_+$ % $g : \bb R \mapsto  [0, \infty)$ 
is called directly Riemann integrable if %(d.R.i.) if %the following holds:
$\int_{\bb R} \overline f_{\delta}(u) du <\infty$ for any $\delta>0$ small enough,
and 
\begin{align} \label{alternative_def_DRI001}
\lim_{\delta\to 0}  \int_{\bb R} \left( \overline f_{\delta}(u) - \underline f_{\delta}(u) \right) du = 0, 
\end{align}
here and below all the integrals are understood in the Lebesgue sense.
A Borel measurable function $f$ %$f : \bb R \mapsto  \bb R$ 
with values in $\bb R$ is directly Riemann integrable if 
both its positive and negative parts are directly Riemann integrable.
%1. For any $\delta>0$ small enough,
%\begin{align} \label{alternative_def_DRI000}
%\int_{\bb R} \overline f_{\delta}(u) du <\infty.
%\end{align}
%
%2. As $\delta\to 0$ it holds
%\begin{align} \label{alternative_def_DRI001}
%\int_{\bb R}| \overline f_{\delta}(u) - \underline f_{\delta}(u)  | du \to 0. 
%\end{align}
Since $\underline f_{\delta} \leq f \leq \overline f_{\delta}$, 
we have that \eqref{alternative_def_DRI001} is equivalent to the following property:
\begin{align} \label{alternative_def_DRI002}
\lim_{\delta\to 0} \int_{\bb R} \left( \overline f_{\delta}(u) - f(u)  \right) du =0,\quad 
\lim_{\delta\to 0} \int_{\bb R} \left( f(u) - \underline f_{\delta}(u)  \right) du =0. 
\end{align}
Every directly Riemann integrable function is necessarily integrable with respect to the Lebesgue measure on $\bb R$ 
 and vanishes at infinity. 
 But the converse may not hold because of the possible oscillations at infinity.
We refer to \cite[\S XI.1]{Fel64} and \cite[\S V.4]{Asm-2003}  for more details. 

Let $f, g:\mathbb R \mapsto \mathbb R_+$ 
be Borel measurable functions and $\ee>0$.
We say that $g$ $\ee$-dominates $f$ and we use notation $f \leq_{\ee} g$ if
\begin{align} \label{def_upper_envelope_001} f(u)  \leq  g(u+v), \ \forall u\in \bb R,\ \forall |v| \leq \ee.   
\end{align}
We say equivalently that $f$ $\ee$-minorates $g$  and we use notation $g \geq_{\ee} f$.
%Note that \eqref{def_upper_envelope_001} is equivalent to $f(u+v)  \leq  g(u)$ for any $u\in \bb R$ and $|v| \leq \ee$. 
%%\begin{align} \label{def_upper_envelope_002}
%% f(u+v)  \leq  g(u), \ \forall u\in \bb R,\ \forall |v| \leq \ee.   
%%\end{align}
%%Similarly, we we say that $f\geq_{\ee} g$ if $ f(u)  \geq  g(u+v)$,  $\forall u\in \bb R,$ $\forall |v| \leq \ee.$ 
%%Note that $f\leq_{\ee}g$ is equivalent to $g\geq_{\ee}f $.
%%The relations $f <_{\ee} g$ and $f >_{\ee} g$ are defined similarly.  
If $f\leq_{\ee} g$ we also say that $g$ is an upper $\ee$-envelope of $f$ 
or equivalently that $f$ is an lower $\ee$-envelope of $g$.
%%Similarly, a lower $\ee$-envelope of $h$ is a function $g$ such that:
%%\begin{align} \label{def_lower_envelope_001}
%% h(u)  \geq  g(u+v), \ \forall u\in \bb R,\ \forall |v| \leq \ee.   
%%\end{align}
%\todos{Use only one name: $\ee$-dominates or upper $\ee$-envelope ?}
In the following we are interested in (Borel) measurable $\ee$-envelopes.
Any function $f$ has measurable upper and lower $\ee$-envelopes
for any $\ee>0$ sufficiently small.
To see this consider the upper and lower $\ee$-envelopes of $\overline f_{\delta}$ and $\underline f_{\delta}$,
\begin{align}\label{Def_f_deltaee}
\overline f_{\delta,\ee}(u)=\sup_{|v - u| \leq \ee} \overline f_{\delta}(v),  \quad 
\overline f_{\delta,-\ee}(u)=\inf_{|v - u| \leq \ee} \underline f_{\delta}(v), \quad  u \in \bb R. 
\end{align}
Then, obviously, it holds that
\begin{align} \label{two-sided bound for f 001}
\underline f_{\delta,-\ee} \leq_{\ee} \underline f_{\delta} \leq f \leq \overline f_{\delta}\leq_{\ee} \overline f_{\delta,\ee}.
\end{align} 
In addition, both $\overline f_{\delta,\ee}$ and $\underline f_{\delta,-\ee}$ are Borel measurable functions 
since  $\overline f_{\delta}$ and $\underline f_{\delta}$ take values in a countable set.

The following lemma shows that if $f$ is directly Riemann integrable then it
can be approximated by the help of the functions $\overline f_{\delta, \ee}$
and $\underline f_{\delta, -\ee}$.
% from the upper and lower $\ee$-envelopes of the upper and lower ladder functions.

\begin{lemma}\label{lemma-DRI-convergence}
Let $f:\bb R \mapsto \bb R_+$ be a directly Riemann integrable function. 
Then we have %$\overline f_{\delta, \ee} \leq f \leq \underline f_{\delta, -\ee}$ and
\begin{align*} %\label{}
\lim_{\delta \to 0} \lim_{\ee\to 0}  \int_{\bb R}| \overline f_{\delta, \ee} (u) -  f (u)  | du =0,
\quad %\mbox{and} \quad 
\lim_{\delta \to 0} \lim_{\ee\to 0}  \int_{\bb R}| \underline f_{\delta, -\ee} (u) - f (u)  | du =0. 
\end{align*}
\end{lemma}

\begin{proof}
%The bounds $\overline f_{\delta, \ee} \leq f \leq \underline f_{\delta, -\ee}$ follows from \eqref{two-sided bound for f 001}.
By definition \eqref{Def_f_deltaee}, it holds that almost surely under the Lebesgue measure on $\bb R$, 
\begin{align} \label{AlmostSureCon}
\lim_{\ee\to 0} \overline f_{\delta,\ee} = \overline f_{\delta} \quad \mbox{and} \quad  
\lim_{\ee\to 0} \underline f_{\delta,-\ee} = \underline f_{\delta}.
\end{align}
Using the Lebesgue dominated convergence theorem, we conclude that for any $\delta>0$,
\begin{align} \label{conv to f_delta001}
\lim_{\ee\to 0}  \int_{\bb R}| \overline f_{\delta, \ee} (u) -  \overline f_{\delta}(u)  | du =0,
\quad %\mbox{and} \quad 
\lim_{\ee\to 0}  \int_{\bb R}| \underline f_{\delta, -\ee} (u) - \underline f_{\delta} (u)  | du =0. 
\end{align}
For this it is enough to prove that there exist $\ee_0 \in (0,\delta]$ 
and a Lebesgue integrable function $g$ on $\bb R$ such that 
\begin{align} \label{AlmostSureCon2}
0\leq \sup_{\ee\in (0,\ee_0)}\underline f_{\delta,-\ee} \leq \sup_{\ee\in (0,\ee_0)}\overline f_{\delta,\ee} \leq g.  %\notag\\
\end{align} 
Indeed, for $k \in \bb Z$, let 
\begin{align*} %\label{}
g(u)= \overline f_{\delta} (u) + \overline f_{\delta} (u+\delta) +\overline f_{\delta} (u-\delta),  \quad  u \in [k \delta, (k+1) \delta ). 
\end{align*}
We see that $\overline f_{\delta,-\ee} (u) \leq g(u)$ for any $u\in \bb R$ and $\ee\in (0,\delta]$. 
%Let $\ee\in (0,\delta]$. For any $u\in \bb R$,
%$\overline f_{\delta,\ee} (u) \leq g(u),$  
%where, for any $u\in I_k=[k\delta,(k+1)\delta )$, 
In addition, 
\begin{align*} %\label{}
& \sum_{k=-\infty}^{\infty}  \mathds 1_{[k\delta,(k+1)\delta)}(u) g (u)
= \sum_{k=-\infty}^{\infty}  \mathds 1_{[k\delta,(k+1)\delta)}(u) \overline f_{\delta} (u) \\
& \qquad\quad  + \sum_{k=-\infty}^{\infty}  \mathds 1_{[k\delta,(k+1)\delta)}(u) \overline f_{\delta} (u+\delta) 
+  \sum_{k=-\infty}^{\infty}  \mathds 1_{[k\delta,(k+1)\delta)}(u) \overline f_{\delta} (u-\delta).  
\end{align*}
Since the right hand side is an integrable function, 
we get that the function $g$ is Lebesgue integrable, which shows \eqref{AlmostSureCon2}.

The assertion of the lemma follows from \eqref{conv to f_delta001} and \eqref{alternative_def_DRI002}. 
\end{proof}

%\begin{remark}
%The usefulness of the $\ee$-envelopes is illustrated by the following example, 
%which will be used in the proofs. 
%Consider the function 
%\begin{align*} %\label{}
%H_{m}(x)=\bb E(g(x+S_m); \tau_x>m),\ x\in \bb R,
%\end{align*}
%where $g:\bb R \mapsto \bb R_+$ is a bounded measurable function.
%%We shall construct explicit measurable upper and lower $\ee$-envelopes for $H_m$ which will be used in the sequel.
%Note that the function $H_m$ is measurable and that the function
%$x\mapsto \sup_{|v|\leq \ee} H_{m}(x+v)$ is an $\ee$-envelope of $H_m$, but it is difficult to
%deal with and in particular it may not be measurable. 
% %and that
%%\begin{align*} %\label{}
%%H_{m}(x)=\bb E(g(x+S_m) \mathds 1_{(0,\infty)} (x+\min_{1\leq j\leq m}S_j ),\ x\in \bb R,
%%\end{align*}
%A measurable upper $\ee$-envelope is easily constructed by setting 
%\begin{align*} %\label{}
%H_{m,\ee}(x)=\bb E(g_\ee(x+S_m); \tau_{x+\ee}>m),\ x\in \bb R,
%\end{align*}
%where $g_\ee \geq_{\ee} g $ (is any measurable upper $\ee$-envelope of $g$).
%To verify that $H_{m,\ee}$ is an $\ee$-envelope of $H_{m}$ it is enough to note that $H_m(x+v) \leq H_{m,\ee}(x)$, for any $|v|\leq\ee$. 
%Similarly, a measurable lower $\ee$-envelope of $H_m$ is given by 
%\begin{align*} %\label{}
%H_{m,-\ee}(x)=\bb E(g_{-\ee}(x+S_m); \tau_{x-\ee}>m),\ x\in \bb R,
%\end{align*}
%where $g_{-\ee}$ is any measurable lower $\ee$-envelope of $g$.
%\end{remark}

%%%%%%%%%%%%%%%%%%%%%%%%%%%%%%%%%%%%%%%%%%%%%%%%%%%%%%%%%
%\subsection{Smoothing inequalities}

Now we state some smoothing inequalities.  
For any integrable function $f: \bb R \mapsto \bb R$, its Fourier transform is defined by 
$\widehat{f}(t) = \int_{\bb R} e^{-itu} f(u) du$, $t \in \bb R$. 
Note that 
\begin{align}\label{IdentiFourier}
\frac{1}{2\pi} \int_{\bb R} e^{-itu}   \left( \frac{\sin \left( u/2 \right)}{u/2} \right)^2 du = (1 - |t|) \mathds 1_{\{ |t| \leq 1 \}},
\quad  t \in \bb R, 
\end{align}
where we use the convention that $\frac{\sin 0}{0} = 1$. 
We introduce the density function $\kappa$ by setting
\begin{align*} %\label{}
\kappa (u) = 
\left[ \int_{\bb R} \left( \frac{\sin \left( v/4 \right)}{v/4} \right)^4 dv \right]^{-1} 
\left( \frac{\sin \left( u/4 \right)}{u/4} \right)^4, 
\quad  u \in \bb R.
\end{align*}
By \eqref{IdentiFourier}, its Fourier transform $\widehat{\kappa}$ 
is a non-negative even function with support on $[-1, 1]$. 
%given by
%\[
%\hat{\kappa}(t) = 1-\abs{t}, \quad  t \in [-1,1], \qquad \text{and} \qquad \hat{\kappa}(t) = 0 \quad \text{otherwise.}
%\]
%Note that
%\begin{align} \label{densequalto1-001}
%\int_{\bb R} \kappa (u) d u  = \hat{\kappa}(0) = 1 = \int_{\bb R} \hat{\kappa}(t) d t.
%\end{align}
For any $\ee >0$, we define the rescaled density function $\kappa_{\ee}$ by
\begin{align} \label{kernel-smoo-001}
\kappa_{\ee} (u) = \frac{1}{\ee} \kappa \left( \frac{u}{\ee} \right), \quad  u \in \bb R. 
\end{align}
Its Fourier transform $\widehat{\kappa}_{\ee}$ %is given by $\hat{\kappa}_{\ee} (t) = \hat{\kappa}(\ee t)$, 
is supported on $[-\veps^{-1}, \veps^{-1}]$. 
Note also that there exists a constant $c>0$ such that for any $\ee > 0$, 
\begin{equation}\label{smbd001}
\int_{\abs{u} \geq \frac{1}{\ee}} \kappa (u) d u 
 \leq c \int_{\frac{1}{\ee}}^{\infty} \frac{1}{u^4} d u \leq c \ee^3.
\end{equation}

We shall make use of the following smoothing inequalities, whose proofs can be carried out in the same way as that 
of \cite[Lemma 5.2]{GLL20}.  
Below let 
\begin{align*}%\label{}
f*g(u)=\int_{\bb R} f(u-v)g(v)dv, \quad  u \in \bb R, 
\end{align*}
%$f*g(u)=\int_{\bb R} f(u-v)g(v)dv$, $u \in \bb R$, 
be the convolution of $f$ and $g$, whenever the integral makes sense. 
%The convolution $f*g$ of two functions $f$ and $g$ is defined by

%\begin{align*} %\label{}
%f*g(x)=\int_{\bb R} f(x-u)g(u)du. % =\int_{\bb R} h(v)g(x-v)dv. 
%\end{align*}

%assertions (the first one is similar to Lemma ??? of \cite{GLL20}). 
%For the completeness we reproduce the proof here.
%For any $\ee > 0$, denote by $\scr H_{\ee}$ the set of non-negative and locally bounded functions $h$ such that $h$, $\overline{h}_{\ee}$ and $\underline{h}_{\ee}$ are measurable from $\left( \bb R, \scr B \left( \bb R \right) \right)$ to $\left( \bb R_+, \scr B \left( \bb R_+ \right) \right)$ and Lebesgue-integrable (where $\scr B \left( \bb R \right)$, $\scr B \left( \bb R_+ \right)$ are the Borel $\sigma$-algebras).
\begin{lemma}\label{smoothing-lemma-001}
There exists a constant $c>0$ such that for any $\ee \in (0,1/2)$ and any integrable functions $f,g: \bb R \to \bb R_+$
with $f\leq_{\ee} g$, and any $u \in \bb R$, 
%Let $\ee \in (0,1/4)$. 
%Let $f,g: \bb R \to \bb R_+$
% be integrable functions and that $f\leq_{\ee} g$. 
%%$h_\ee$ be a measurable upper $\ee$-envelope of $h$ and 
%%$h_{-\ee}$ be a measurable lower $\ee$-envelope of $h$. 
%Then for any $u \in \bb R$,
\begin{align*}
f(u) \leq  (1 + c \ee) g*\kappa_{\ee^2} (u), 
\quad % \mbox{and} \ 
g(u)  \geq f*\kappa_{\ee^2} (u) - \int_{\abs{v} > \ee} f \left( u- v \right) \kappa_{\ee^2} (v) d v.  
\end{align*}
%and
%\begin{align*}
%g(u)  \geq f*\kappa_{\ee^2} (u) - \int_{\abs{v} > \ee} f \left( u- v \right) \kappa_{\ee^2} (v) d v.
%\end{align*}
\end{lemma}
%\begin{proof}
%Using \eqref{densequalto1-001}, \eqref{kernel-smoo-001} and \eqref{smbd001}, we have
%\begin{align*}
%	f (u) &= \int_{\abs{v} \leq \ee} f(u) \kappa_{\ee^2} (v) d v + f(u) \int_{\abs{v} > \ee}  \kappa_{\ee^2} (v) d v \\
%	&\leq \int_{\abs{v} \leq \ee} g \left( u- v \right) \kappa_{\ee^2} (v) d v + 2 \ee f(u) ,
%\end{align*}
%from which the upper bound follows immediately.
%For the lower bound, using again \eqref{densequalto1-001}, \eqref{kernel-smoo-001} and \eqref{smbd001},
%\begin{align*}
%	g(u) &\geq \int_{\abs{v} \leq \ee} g(u) \kappa_{\ee^2} (v) d v \\
%	&\geq \int_{\abs{v} \leq \ee} f \left( u- v \right) \kappa_{\ee^2} (v) d v \\
%	&= f*\kappa_{\ee^2} (u) - \int_{\abs{v} > \ee} f \left( u- v \right) \kappa_{\ee^2} (v) d v.
%\end{align*}
%\end{proof}

\begin{remark}\label{f is bounded if g is integr -001}
%The domination property $\leq_{\ee} $ implies that 
If $f \leq_{\ee} g$ and the function $g$ is integrable, then $f$ is bounded. 
Indeed, since $f \leq_{\ee} g$ and $g$ is an integrable function, 
by Lemma \ref{smoothing-lemma-001} we get $f \leq (1 + c \ee) g * \kappa_{\ee^2}$.
Since the Fourier transform of $g * \kappa_{\ee^2}$ is compactly supported on $[- \frac{1}{\ee^2}, \frac{1}{\ee^2}]$, 
by the Fourier inversion formula, 
%\begin{align*}%\label{}
%| g * \kappa_{\ee^2} (x) | = \left| \frac{1}{2 \pi} \int_{\bb R} e^{itx}  \widehat{g}(t) \, \widehat \kappa_{\ee^2} (t) dt  \right| \leq c. 
%\end{align*}
the function $g * \kappa_{\ee^2}$ is bounded on $\bb R$, so that $f$ is bounded on $\bb R$. 
\end{remark}

\subsection{Local limit theorems with precise remainders} \label{subsec-nonasymptoticLLT}
The local limit theorem is a refinement of the central limit theorem and has been of considerable interest since 
the groundwork of Gnedenko \cite{Gned48},  Shepp \cite{{She64}} and  Stone \cite{Sto65} for sums of real-valued random variables. 
%These results have been improved by many authors, e.g. Breuillard \cite{Bre05}.
%The main results of this section are the following statements, which
In this section we establish effective local limit theorems for the random walk $(S_n)_{n \geq 1}$ 
with target functions and explicit convergence rates,   
which will be an important ingredient to obtain precise errors terms in the conditioned local limit theorems. 
%The results of this section holds for random walks with increments having only the second moment
%which is weaker than condition \ref{SecondMoment} assumed in the rest of the paper.
%By an integrable function $f:\bb R \mapsto \bb R$ we mean a Borel measurable function 
%whose Lebesgue integral $\int_{\bb R} f(x) dx$ exists and is finite.
%Recall that $\phi$ is the standard normal density on $\bb R$. 

We introduce the necessary notation. % which will be used in the sequel. 
The standard normal density is denoted by  $\phi (t) = \frac{1}{\sqrt{2\pi }} e^{-t^{2}/2},$ $t\in \bb R $.
Let $\Phi^+(t) = (1 - e^{-t^2/2}) \mathds 1_{\{ t \geq 0 \} }$ be the Rayleigh distribution function on $\bb R$. 
By an integrable function $f:\bb R \mapsto \bb R$ we mean a Borel measurable function 
whose Lebesgue integral $\| f \|_1 :=  \int_{\bb R} |f(x)| dx$ exists and is finite.
%For any integrable function $f: \bb R \mapsto \bb R_+$ we denote $\|f\|_1= \int_{\bb R} f(t)dt.$
In the sequel $c$ denotes a positive constant and $c_{\alpha}, c_{\alpha, \beta}$ denote positive constants 
depending only on their indices.  
All these constants are subject to change theirs values every occurence.

%We first fix some notation.
%We denote by $\bb R_+ = [0, \infty)$ the non-negative half-line
%and by $\bb R_+^* = (0, \infty)$ the positive half-line.  
%In the paper $\mathcal{N}_{a,\sigma ^{2}}$ stands for a normal random variable of mean $a$ and variance $\sigma ^{2}>0.$ 

\begin{theorem}
\label{Theor-Loc-cmpsupp}
Assume \ref{A1} and \ref{SecondMoment} for some constant $\delta \in (0,1]$. % Suppose that $X_{1}$ is non-lattice. 
Let $K\subset \bb R$ be a compact set.
Then there exists a constant $c_K >0$
%sequence $r_n(K,\mu)$ depending on the set $K$ and on the law $\mu$ 
%satisfying $\lim_{n\rightarrow \infty}r_n(K,\mu)=0$
such that for any  %\todos{Why this is necessary!!!}
integrable function $f:\bb R\mapsto \bb R$ whose Fourier transform has a compact support contained in $K$ 
 and any $n\geq 1$,
\begin{equation*}
\left|  \mathbb{E} f(S_n) 
- \frac{1}{ \sigma \sqrt{n}} \int_{\bb R }  f(t)  \phi \left( \frac{t}{ \sigma \sqrt{n}}\right) dt \right| 
\leq   \frac{c_K}{n^{(1+\delta)/2}} \left\Vert f\right\Vert_1. 
\end{equation*}
%where one can choose
%\begin{align} \label{sequence-r_n}
%r_n(K,\mu)= c_\mu \left(n^{-\delta/2} %+ e^{-\delta \sqrt{n}/8}  +  c'_{\mu,K}\frac{\sqrt{n}}{\veps^2 }e^{-c_{\mu,K} n} \right). 
%\end{align}
\end{theorem}

Theorem \ref{Theor-Loc-cmpsupp} can be established following the standard approach based on the Fourier transform
(we refer for example to \cite{Sto65, Bre05}) and therefore its proof will not be detailed here.

%%%%%%%%%%%%%%%%%%%%%%%%

%Note that in the formulation of Theorem \ref{Theor-Loc-cmpsupp} we did not include the condition that $f$ is continuous, 
%since this is granted by other assumptions. 
%Indeed, an integrable function $g:\bb R\mapsto \bb R_+$ whose Fourier transform has a compact support contained in $K$
%is also continuous on $\bb R$.

%\begin{align}\label{LLT-general001_Try}
%\sqrt{n} \mathbb{P} \left(  S_n  \in [a,b] + \sqrt{\beta n \log n}  \right)
%- \int_{a + \sqrt{ \beta n \log n}}^{b + \sqrt{ \beta n \log n}}  \phi \left( \frac{t}{\sigma \sqrt{n}}\right) dt 
%\leq  c (r_{n}(\veps)  +\ee ) \left\Vert g\right\Vert_1. 
%\end{align}
%Note that 
%\begin{align*}%\label{}
%\int_{a + \sqrt{ \beta n \log n}}^{b + \sqrt{ \beta n \log n}}  \phi \left( \frac{t}{\sigma \sqrt{n}}\right) dt 
%&  =  \frac{1}{\sqrt{2 \pi}}  \int_{a + \sqrt{ \beta n \log n}}^{b + \sqrt{ \beta n \log n}}  e^{- \frac{t^2}{2 \sigma^2 n}}  dt   \\
%&  =   \frac{\sqrt{n}}{\sqrt{2 \pi}}  \int_{\frac{a}{\sqrt{n}} + \sqrt{ \beta \log n}}^{\frac{b}{\sqrt{n}} + \sqrt{ \beta \log n}}  
%    e^{- \frac{u^2}{2 \sigma^2}}  du  \\
%& \sim  \frac{\sqrt{n}}{\sqrt{2 \pi}}   \frac{1}{ \sqrt{ \beta \log n} }  e^{- \frac{\beta \log n}{2 \sigma^2}}   \\
%& =  \frac{\sqrt{n}}{\sqrt{2 \pi}}   \frac{1}{ \sqrt{ \beta \log n} }  n^{- \frac{\beta}{2 \sigma^2}}.  
%\end{align*}

In Theorem \ref{Theor-Loc-cmpsupp} the target function $f$ is actually assumed to be infinitely differentiable on $\bb R$
since its Fourier transform is compactly supported. 
Below we shall deduce from Theorem \ref{Theor-Loc-cmpsupp} 
the following bounds in the local limit theorem for integrable functions $f$ %:\bb R \mapsto\bb R_+$.
which are not necessarily smooth. 
%Let us point out that %in Theorem \ref{LLT-general} 
In the formulation below we assume that the functions
$f,g: \bb R\mapsto \bb R_+$ are integrable and $f\leq_{\ee} g$. 
By Remark \ref{f is bounded if g is integr -001}, this implies that $f$ is bounded on $\bb R$.

 \begin{theorem} \label{LLT-general}
Assume \ref{A1} and \ref{SecondMoment} for some constant $\delta \in (0,1]$. % Suppose that $X_{1}$ is non-lattice. 
Then, for any $\ee \in (0,\frac{1}{2})$, 
there exists a constant $c>0$ not depending on $\ee$ and a constant $c_{\ee} >0$ such that 
%There exists a constant $c>0$ %and a function $\veps \in (0,\frac{1}{4}) \mapsto r_n(\ee)\in \bb R_+ $ 
%and positive sequences $(r_n(\ee))_{n\geq 1}$, where $\ee \in (0,\frac{1}{4}) $, 
%such that,  for any $\ee \in (0,\frac{1}{4})$, $\lim_{n\rightarrow \infty}r_n(\ee)=0$  and 
the following holds:

\noindent 1. For any  integrable functions $f,g: \bb R\mapsto \bb R_+$ satisfying $f\leq_{\ee} g$ and $n\geq 1$,
\begin{align}\label{LLT-general001}
  \mathbb{E} f (S_n) 
- \frac{ 1 + c \ee}{ \sigma \sqrt{n} }\int_{\bb R } g (t) \phi \left( \frac{t}{ \sigma \sqrt{n}}\right) dt 
\leq  \frac{c_{\ee}}{n^{ (1 + \delta)/2 }} \left\Vert g\right\Vert_1. 
\end{align}

\noindent  2. For any integrable functions $f,h: \bb R\mapsto \bb R_+$ satisfying $f\geq_{\ee} h$ and $n\geq 1$,
\begin{align}\label{LLT-general002}
  \mathbb{E} f (S_n) 
 - & \frac{ 1 }{ \sigma \sqrt{n} }
  \int_{\bb R }  \big[ h(t) -  c \ee f(t) \big] \phi \left( \frac{t}{ \sigma \sqrt{n}}\right) dt  
  \geq   -  \frac{c_{\ee}}{n^{ (1+\delta) /2 }}  \left\Vert f \right\Vert_1. 
%   \notag\\
%& \geq   -  \frac{ c \ee}{ \sigma \sqrt{n}}    \int_{\bb R }f (t) \phi \left( \frac{t}{ \sigma  \sqrt{n}}\right) dt
% -  \frac{c_{\ee}}{n^{ (1+\delta) /2 }}  \left\Vert f \right\Vert_1. 
\end{align}

%\noindent  3.  For any integrable functions $f,g: \bb R\mapsto \bb R_+$ satisfying $f\leq_{\ee} g$ and $n\geq 1$, 
%\begin{align}\label{LLT-general003}
%\mathbb{E} f\left( S_n \right)  \leq  \left( \frac{c}{ n^{1/2} } + \frac{c_{\ee}}{n^{(1+\delta)/2}} \right) \left\Vert g\right\Vert_1. 
%\end{align}
\end{theorem}

In the proof of Theorem \ref{LLT-general} we shall make use of the following inequality. 

\begin{lemma}\label{Lem_Deconvolution}
There exists a constant $c>0$ such that for any $\ee \in (0, \frac{1}{2})$, $n \geq 1$ 
and any integrable function $g: \bb R \mapsto \bb R_+$, 
\begin{align*}%\label{Inequa_Deconvolution}
\int_{\bb R }g*\kappa_{\ee^2}\left( t\right) \phi \left( \frac{t}{ \sigma \sqrt{n}}\right) dt 
\leq  (1 + c \ee) \int_{\bb R}  g(t) \phi \left( \frac{t}{ \sigma \sqrt{n}}\right) dt + \frac{ c \ee }{ \sqrt{n} } \| g \|_1. 
\end{align*}
\end{lemma}

\begin{proof}%\label{}
Denote by 
$\phi_{v}(x)=\frac{1}{\sqrt{2\pi v}}e^{-\frac{x^2}{2 v}}$, $x\in \mathbb R$ 
the normal density of variance $v$.  
By Fubini's theorem, we have
\begin{align*} %\label{gener33-02}
\int_{\bb R }g*\kappa_{\ee^2}\left( t\right) \phi \left( \frac{t}{ \sigma \sqrt{n}}\right) dt 
%&=\int_{\bb R } \int_{\bb R } g(v) \kappa_{\ee^2} ( t-v) dv  \phi \left( \frac{t}{ \sigma \sqrt{n}}\right) dt \notag\\
&= \sigma \sqrt{n} \int_{\bb R } \left(  \int_{\bb R } g(v) \kappa_{\ee^2} ( t-v) dv \right)  \phi_{\sigma^2 n} \left( t\right) dt \notag\\
&= \sigma \sqrt{n} \int_{\bb R } 
   \left( \int_{\mathbb R}  \kappa_{\ee^2} ( t-v) \phi_{\sigma^2 n} \left( t\right) dt \right) g(v) dv \notag\\
&= \sigma \sqrt{n} \int_{\bb R } \phi_{ \sigma^2 n } * \kappa_{\ee^2} ( t) g(t) dt.
\end{align*}
Using the second inequality in Lemma \ref{smoothing-lemma-001} gives that for any $t\in \bb R$, 
\begin{align*} %\label{gener33-03}
 \phi_{ \sigma^2 n } * \kappa_{\ee^2} (t) \leq \psi_{ \sigma^2 n }(t) 
    + \int_{\abs{v} \geq \ee} \phi_{ \sigma^2 n } \left( t- v \right) \kappa_{\ee^2} (v) d v,  
\end{align*}
where $\psi_{ \sigma^2 n }(t)=\sup_{|v|\leq\ee} \phi_{ \sigma^2 n }(t+v)$, $t \in \bb R$. 
%, is an upper $\ee$-envelope of $\phi_{ \sigma^2 n }$. 
Hence, using again Fubini's theorem, we get
\begin{align} \label{gener33-04}
& \int_{\bb R }g*\kappa_{\ee^2}\left( t\right) \phi \left( \frac{t}{ \sigma \sqrt{n}}\right) dt   \notag\\
& \leq  \sigma \sqrt{n}\int_{\bb R } \psi_{ \sigma^2 n }(t) g(t)  dt 
 + \int_{\abs{v} \geq \ee}  \left[ \int_{\bb R } \phi \left( \frac{t-v}{ \sigma \sqrt{n}} \right)   g(t) dt  \right]  \kappa_{\ee^2} (v) d v \notag\\
% + \sigma \sqrt{ n}\int_{\bb R } 
%   \left[ \int_{\abs{v} \geq \ee} \phi_{ \sigma^2 n } \left( t- v \right) \kappa_{\ee^2} (v) d v \right]  g(t)  dt    \notag\\
&=: J_1+J_2.
\end{align}
For $J_1$, by elementary calculations we see that
\begin{align} \label{gener33-05}
J_1&=  \frac{1}{\sqrt{2\pi}} 
 \left[ \int_{-\infty}^{-\ee} e^{-\frac{(t+\ee)^2}{ 2 \sigma^2 n }}  g(t)dt
 +   \int_{-\ee}^{\ee}   g(t) dt 
 +  \int_{\ee}^{\infty} e^{-\frac{(t-\ee)^2}{ 2 \sigma^2 n  }}  g(t)dt \right] \notag\\
& = \frac{1}{\sqrt{2\pi}} \Bigg[ \int_{\bb R}  e^{-\frac{ t^2}{ 2 \sigma^2 n  }}  g(t)dt
   + \int_{-\infty}^{-\ee}  \left( e^{-\frac{(t+\ee)^2}{ 2 \sigma^2 n }}-e^{-\frac{t^2}{ 2 \sigma^2 n }} \right)  g(t)dt \notag \\
& \qquad\qquad + \int_{\ee}^{\infty}  \left( e^{-\frac{(t-\ee)^2}{ 2 \sigma^2 n }}-e^{-\frac{t^2}{ 2 \sigma^2 n }} \right)  g(t)dt 
   + \int_{-\ee}^{\ee}  \left( e^{-\frac{t^2}{ 2 \sigma^2 n }}-1 \right)  g(t)dt  \Bigg] \notag \\
&\leq  \frac{1}{\sqrt{2\pi}} \int_{\bb R} e^{-\frac{ t^2}{ 2 \sigma^2 n }}  g(t)dt
+ \frac{c\ee}{\sqrt{n}} \| g \|_1.
\end{align}
For $J_2$, 
%by Fubini's theorem, we write 
%\begin{align*}%\label{}
%J_2 
%%& = \sqrt{ n} \int_{\bb R } 
%%   \left(  \int_{\abs{v} \geq \ee} \phi_{\sqrt{n}} \left( t- v \right) \kappa_{\ee^2} (v) d v  \right) g(t)  dt   \notag\\
%& =  \int_{\abs{v} \geq \ee}  \left[ \int_{\bb R } \phi \left( \frac{t-v}{ \sigma \sqrt{n}} \right)   g(t) dt  \right]  \kappa_{\ee^2} (v) d v.  
%\end{align*}
since the normal density $\phi$ is Lipschitz continuous on $\bb R$, we have 
\begin{align*}%\label{}
\sup_{t \in \bb R} \left| \phi \left( \frac{t-v}{ \sigma \sqrt{n}} \right) -   \phi \left( \frac{t}{ \sigma \sqrt{n}} \right)  \right|  
\leq  c \frac{|v|}{\sqrt{n}}. 
\end{align*}
Since $\int_{\abs{v} \geq \ee} \kappa_{\ee^2} (v) d v \leq c \ee$
and $\int_{\abs{v} \geq \ee}  |v|  \kappa_{\ee^2} (v) d v \leq c \ee^2$,  it follows that
\begin{align}\label{gener33-06}
J_2  
& \leq  \int_{\abs{v} \geq \ee}  \left[ \int_{\bb R } \phi \left( \frac{t}{ \sigma \sqrt{n}} \right)   g(t) dt  \right]  \kappa_{\ee^2} (v) d v
  +  \frac{c}{\sqrt{n}} \|g\|_1    \int_{\abs{v} \geq \ee}  |v|  \kappa_{\ee^2} (v) d v   \notag\\
% & \leq  c \ee  \int_{\bb R } \phi \left( \frac{t}{ \sigma \sqrt{n}} \right)   g(t) dt 
%    +   \frac{c \ee^2}{\sqrt{n}}  \|g\|_1  \int_{|u| \geq \frac{1}{\ee}} |u| \kappa(u) du  \notag\\
 &  \leq  c \ee  \int_{\bb R } \phi \left( \frac{t}{ \sigma  \sqrt{n}} \right)   g(t) dt 
    +   \frac{c \ee^2}{\sqrt{n}} \|g\|_1.  
\end{align}
Putting together \eqref{gener33-04}, \eqref{gener33-05} and \eqref{gener33-06}, 
the desired result follows. 
%we get \eqref{Inequa_Deconvolution}. 
\end{proof}

\begin{proof}[Proof of Theorem \ref{LLT-general}]
Let $\ee\in (0,\frac{1}{4})$ and $f\leq_{\ee} g$. 
Note that the function $g*\kappa_{\ee^2}$ is continuous and integrable on $\bb R$
(the integrability follows from $\|g*\kappa_{\ee^2}\|_1 =\| g\|_1 \|\kappa_{\ee^2}\|_1= \| g\|_1$). 
By Lemma \ref{smoothing-lemma-001}, it holds that $f \leq  (1 + c\veps) g*\kappa_{\ee^2}$. 
Moreover, the support of the function $\widehat \kappa_{\ee^2}$ is contained in the set $[-\ee^{-2},\ee^{-2}].$ 
Therefore, by Theorem \ref{Theor-Loc-cmpsupp},  we have that
%there exists a sequence $r_n(\ee)$ depending on $\ee$ and on the law $\mu$ 
%satisfying $\lim_{n\rightarrow \infty}r_n(\ee)=0$ such that,  
for any $n\geq 1$,
\begin{align} \label{gener33-01}
\mathbb{E} f\left(S_n \right)  
&\leq (1 + c\ee) \mathbb{E} g*\kappa_{\ee^2}\left(S_n \right)   \notag\\
& \leq   \frac{1 + c\ee}{\sigma \sqrt{n}}  
 \int_{\bb R }g*\kappa_{\ee^2}\left( t\right)  \phi \left( \frac{t}{ \sigma \sqrt{n}}\right) dt 
+ \frac{c_{\ee}}{n^{(1+\delta)/2}} \left\Vert g \right\Vert_1, 
\end{align}
where we used the fact that $\left\Vert g*\kappa_{\ee^2}\right\Vert_1=\left\Vert g\right\Vert_1$. 
By Lemma \ref{Lem_Deconvolution}, % and the fact that $\frac{1 + 2 \ee}{1- 2\ee} \leq (1+8\ee)$ for $\ee\in (0,1/4)$,
the upper bound \eqref{LLT-general001} follows.

Now we prove the lower bound \eqref{LLT-general002}. 
Set $\epsilon = \ee/2$ 
and let $\overline h: \bb R\mapsto \bb R_+$ be such that $h \leq_{\epsilon}  \overline h  \leq_{\epsilon}  f$. 
By Lemma \ref{smoothing-lemma-001}, 
we have
%for $u\in \bb R$, 
%\begin{align} \label{conv-lower-bound222-01}
%f(u) \geq  \overline h *\kappa_{\epsilon^2} (u) 
%  - \int_{\abs{v} \geq \epsilon}  \overline h \left( u- v \right) \kappa_{\epsilon^2} (v) d v.  %\notag\\
%\end{align}
%%where $g*\kappa_{\ee^2}$ is continuous and integrable 
%%(since $\|g*\kappa_{\ee^2}\1 =\| g\|_1 \|\kappa_{\ee^2}\|_1= \| g\|_1$) %\todos{NOT TRUE: TO VERIFY !!!}
%%and the support of $\widehat \kappa_{\ee^2}$ is contained in the set $[-\ee^{-2},\ee^{-2}].$ 
%Therefore, 
\begin{align} \label{Fstar001}
\mathbb{E} f\left(S_n \right)  %\notag\\
\geq 
\mathbb{E} \overline h * \kappa_{\epsilon^2} \left( S_n \right) - 
\int_{\abs{v} \geq \epsilon} \mathbb{E} \overline h \left( S_n -v\right) \kappa_{\epsilon^2} (v) d v.
\end{align}
For the first term, by Theorem \ref{Theor-Loc-cmpsupp} and the first inequality in Lemma \ref{smoothing-lemma-001},
\begin{align} \label{Fstar002}
 \mathbb{E} \overline h * \kappa_{\epsilon^2} \left( S_n \right)
& \geq  \frac{1}{ \sigma \sqrt{n}}  
   \int_{\bb R } \overline h * \kappa_{\epsilon^2}\left( t\right) \phi \left( \frac{t}{\sigma \sqrt{n}}\right) dt 
- \frac{c_{\epsilon}}{n^{(1+\delta)/2}}  \left\Vert \overline h \right\Vert_1   \notag\\
& \geq \frac{ 1- c \epsilon}{ \sigma \sqrt{n}}  \int_{\bb R }  h \left( t\right) \phi \left( \frac{t}{ \sigma \sqrt{n}}\right) dt 
   -  \frac{c_{\epsilon}}{n^{(1+\delta)/2}}  \left\Vert f \right\Vert_1. 
%   \notag\\
%& = \frac{ 1- \ee}{ \sigma \sqrt{n}}  \int_{\bb R }  h \left( t\right) \phi \left( \frac{t}{ \sigma  \sqrt{n}}\right) dt 
%   -  \frac{c_{\ee}}{n^{(1+\delta)/2}}  \left\Vert h \right\Vert_1,  
\end{align}
%where we used the fact that $h \leq_{\epsilon}  \overline h$.  
For the second term on the right hand side of \eqref{Fstar001},  
applying the upper bound \eqref{LLT-general001} gives that for any $v\in \bb R$,
\begin{align*}%\label{}
\mathbb{E} \overline h \left( S_n -v\right) 
& \leq  \frac{ 1 + c \epsilon }{ \sigma \sqrt{n}}  \int_{\bb R }f (t) \phi \left( \frac{t + v}{ \sigma \sqrt{n}}\right) dt 
   +  \frac{c_{\epsilon}}{n^{(1 + \delta)/2}} \left\Vert f\right\Vert_1.  
%     \notag\\
%& =   \frac{ 1 + 4 \ee }{ \sigma \sqrt{n}}  \int_{\bb R }f (t) \phi \left( \frac{t + v}{ \sigma \sqrt{n}}\right) dt 
%   +  \frac{c_{\ee}}{n^{(1 + \delta)/2}} \left\Vert f\right\Vert_1.  
\end{align*}
%again by Lemma \ref{smoothing-lemma-001} and  Theorem \ref{Theor-Loc-cmpsupp}, we have
%\begin{align} \label{Fstar003}
% \mathbb{E} f \left( x-v+S_n \right)   
%& \leq c \mathbb{E} g *\kappa_{\ee^2} \left(S_n -v\right)  \notag\\
%&\leq \int_{\bb R } g*\kappa_{\ee^2}\left( t\right) \phi \left( \frac{t+v}{\sigma \sqrt{n}}\right) dt 
%+ c (r_{n}(\veps)  +\ee ) \left\Vert g \right\Vert_1  \notag\\
%&\leq  ?? c (r_{n}(\veps)  + 1 ) \left\Vert g \right\Vert_1.  %\notag\\
%\end{align}
%\todos{Add details} 
In the same way as in the proof of \eqref{gener33-06}, one has
\begin{align} \label{Fstar004}
& \int_{\abs{v} \geq \epsilon} \mathbb{E} \overline h \left(S_n -v\right) \kappa_{\epsilon^2} (v) d v   \notag\\  
&\leq \frac{ 1 + c \ee }{ \sigma \sqrt{n}}  
 \int_{\abs{v} \geq \epsilon}  \left[ \int_{\bb R }f (t) \phi \left( \frac{t + v}{ \sigma \sqrt{n}}\right) dt  \right] 
    \kappa_{\epsilon^2} (v) d v
  +  \frac{c_{\ee}}{n^{(1 + \delta)/2}} \left\Vert f\right\Vert_1  \notag\\
& \leq  \frac{ 1 + c \ee }{ \sigma \sqrt{n}}  c \epsilon  \int_{\bb R }f (t) \phi \left( \frac{t}{ \sigma  \sqrt{n}}\right) dt
   +  \frac{c \ee^2}{ n } \|f\|_1  +  \frac{c_{\ee}}{n^{(1 + \delta)/2}} \left\Vert f\right\Vert_1   \notag\\
&  \leq   \frac{ c \ee }{ \sigma \sqrt{n}}   \int_{\bb R }f (t) \phi \left( \frac{t}{ \sigma \sqrt{n}}\right) dt
   +  \frac{c_{\ee}}{n^{(1 + \delta)/2}} \left\Vert f\right\Vert_1,
\end{align}
where in the last inequality we used the fact that $\delta \in (0,1]$. 
By collecting the bounds \eqref{Fstar001}, \eqref{Fstar002}  and \eqref{Fstar004}, 
we get the lower bound \eqref{LLT-general002}. 
%\begin{align*}%\label{}
%\mathbb{E} f\left(S_n \right)   \geq  
%\frac{ 1-\ee}{ \sigma \sqrt{n}}  \int_{\bb R }  h \left( t\right) \phi \left( \frac{t}{ \sigma \sqrt{n}}\right) dt 
%  -  \frac{ 2 \ee }{ \sigma \sqrt{n}}   \int_{\bb R }f (t) \phi \left( \frac{t}{ \sigma \sqrt{n}}\right) dt
%   -  \frac{c_{\ee}}{n^{(1+\delta)/2}}  \left\Vert f \right\Vert_1,
%\end{align*}
%which concludes the proof of the second assertion \eqref{LLT-general002}. 
%% follows 
%%and using the bound $\|h\|_1  \leq \|f\|_1$. 
\end{proof}

%Let us illustrate how to apply Theorem \ref{LLT-general} 
%with the d.R.i.\ %directly Riemann integrable 
%functions. 
%
%\begin{corollary}
%\label{Loc-DRI}
%Assume \ref{A1} and \ref{SecondMoment}.
%For any d.R.i.\  % directly Riemann integrable 
%function $f:\bb R \mapsto \bb R_+$,
%\begin{align*} %\label{}
%\lim_{n\to\infty} &\sup_{x\in \bb R } \left( \sqrt{n} \mathbb{E} f\left( x+S_n \right) 
%-\int_{\bb R } f\left( t\right) \phi \left( \frac{t-x}{ \sqrt{n}}\right) dt\right) = 0.
%  %\notag\\
%\end{align*}
%\end{corollary}

%%%%%%%%%%%%%%%%%%%%%%%%%%%%%%%%%%%%%%%%%%%%%
\subsection{Conditioned integral limit theorems with rate of convergence}
The goal of this section is to formulate effective conditioned integral limit theorems
for the random walk $x + S_n$ with explicit rate of convergence
and dependence on the starting point $x$. 
To the best of our knowledge, these results have not yet been established in the literature
and besides the fact that we need them for establishing the main results of the paper, 
they are of independent interest.  

The following result is a conditioned integral limit theorem for small starting point $x = o(\sqrt{n})$. % for sums of i.i.d.\ random variables. 
Recall that $\Phi^+(t) = (1 - e^{-t^2/2}) \mathds 1_{\{ t \geq 0 \} }$ is the Rayleigh distribution function on $\bb R$.

\begin{theorem}\label{Theor-IntegrLimTh} 
Assume \ref{SecondMoment}. 
%Let $(\alpha_n)_{n \geq 1}$ be a sequence of positive numbers satisfying $\lim_{n \to \infty} \alpha_n = 0$.  
Then, there exists $\ee_0>0$ such that for any $\ee\in(0, \ee_0)$, it holds 
uniformly in $n \geq 1$, $t \in  \bb R_+$ and $x\in [0, n^{1/2 - \ee}]$,
%\begin{align} \label{C-CLTalphan}
% \left| \mathbb{P} \left(  \frac{x+S_n }{\sigma \sqrt{n}} \leq t,  \tau_{x}>n\right) - 
%\frac{2V(x)}{ \sigma \sqrt{2\pi n}} \Phi^+(t)\right| 
%\leq c \left(  \alpha_n  + n^{-\ee} \right) \frac{1+x}{n^{1/2}}.
%%c \frac{1+x}{n^{1/2+\veps}}.
%\end{align}
%In particular, 
%there exists $\ee_0>0$ such that for any $\ee\in(0, \ee_0)$, it holds 
%uniformly in $n \geq 1$, $t \in  \bb R_+$ 
%and $x\in [0, n^{1/2 - \ee}]$, %$x \in \bb R_+$ satisfying $\frac{x}{\sqrt{n}} \leq n^{- \ee},$ 
\begin{align}\label{CLLT-bound only 001}
\left| \mathbb{P} \left(  \frac{x+S_n }{ \sigma \sqrt{n}} \leq t,  \,  \tau_{x}>n\right) 
-  \frac{2  V(x) }{ \sigma \sqrt{2 \pi  n}}   \Phi^+(t)   \right| 
 \leq  c_\ee  \frac{1 + x }{ n^{1/2 + \ee  }  }. 
\end{align}
In particular, we have, uniformly in $n \geq 1$ and $x\in [0, n^{1/2 - \ee}]$, 
\begin{align}\label{CLLT-bound only 001b}
\left| \mathbb{P} \left(  \tau_{x}>n\right) 
-  \frac{2  V(x) }{ \sigma \sqrt{2 \pi  n}}  \right| 
 \leq  c_\ee  \frac{1 + x }{ n^{1/2 + \ee  }  }.  
\end{align}
Moreover, there exists a constant $c$ such that for any $x \geq 0$,  
\begin{align}\label{CLLT-bound only 002}
\mathbb{P} \big( \tau_{x} > n \big) \leq c\frac{1+x}{\sqrt{n}}.
\end{align}
\end{theorem}

%In particular, there exists $\ee_0>0$ such that for any $\ee\in(0, \ee_0)$, we have, 
%uniformly in $n \geq 1$ and $x\in [0, n^{1/2 - \ee}]$,
%%$x \in \bb R_+$ satisfying $ \frac{x}{\sqrt{n}} \leq n^{- \ee},$ 
%\begin{align}\label{CLLT-bound only 001}
%\left| \mathbb{P} \left( \tau_{x}>n\right) 
%-  \frac{2  V(x) }{ \sigma \sqrt{2 \pi  n}}     \right| 
% \leq   c_\ee  \frac{1 + x }{ n^{1/2 + \ee  }  }.   
%\end{align}

%%%%%%%%%%%%%%%%%%%%%%%%%%%%%%%%%%%%%%%%
%%%%%%%%%%Do Not Remove it%%%%%%%%%%%%%
%%%\begin{align*}%\label{}
%%%\left| \mathbb{P} \left(  \frac{x+S_n }{ \sqrt{n}} \leq t,  \tau_{x}>n\right) 
%%%-  \frac{2  V(x) }{\sqrt{2 \pi  n}}   \Phi^+(t)   \right| 
%%% \leq  c_\ee  \left( \frac{1 }{ n^{ \ee  }  } + \frac{x^2}{ n } \right) \frac{1 + x}{\sqrt{n}}. 
%%% %c_\ee  \frac{1 + x }{ n^{1/2 + \ee  }  }.  
%%%\end{align*}
%%%In particular, there exists $\ee_0>0$ such that for any $\ee\in(0, \ee_0)$, we have, 
%%%uniformly in $n \geq 1$ and $x$ satisfying $ \frac{x}{\sqrt{n}} \leq n^{- \ee},$ 
%%%\begin{align*}%\label{}
%%%\left| \mathbb{P} \left( \tau_{x}>n\right) 
%%%-  \frac{2  V(x) }{\sqrt{2 \pi  n}}     \right| 
%%% \leq   c_\ee  \left( \frac{1 }{ n^{ \ee  }  } + \frac{x^2}{ n } \right) \frac{1 + x}{\sqrt{n}}.  
%%%\end{align*}
%%%%%%%%%%%%%%%%%%%%%%%%%%%%%%%%%%%%%%%%
%%%%%%%%%%Do Not Remove it%%%%%%%%%%%%%

%We shall complement Theorem \ref{Theor-IntegrLimTh} by a result for large $x$.
For large $x$, the following result complements 
Theorem \ref{Theor-IntegrLimTh}.
Recall that the function $\psi$ is defined by \eqref{Def-Levydens}. 

\begin{theorem}\label{CorCCLT}
Assume \ref{SecondMoment} with some $\delta>0$. 
Then, %for any $\gamma \in (0,  \frac{\delta}{2(2+\delta)})$,
there exists a constant $c_{\delta} >0$ such that for any $n\geq 1$, $x > n^{\frac{1}{2} -  \frac{\delta}{ 2(3 + \delta) }  }$ and $t\in \bb R_+$,    
%\begin{align} \label{CorCCLT01}
%\left| \mathbb{P} \left(  \frac{x+S_n }{ \sqrt{n}} > t, \tau_x >n\right) 
%     -   \int_{t}^{\infty}  \psi \left( s, \frac{x}{\sqrt{n}}   \right)  ds    \right|
%   \leq   \frac{c_{\delta} }{ n^{ \frac{\delta}{ 2(3 + \delta) } } }
%\end{align}
%and 
\begin{align} \label{CorCCLT02}
\left| \mathbb{P} \left(  \frac{x+S_n }{ \sigma \sqrt{n}} \leq  t, \tau_x >n\right) 
     -   \int_{0}^{t}  \psi \left( s, \frac{x}{\sigma \sqrt{n}}   \right)  ds    \right|
   \leq   c_{\delta}  n^{- \frac{\delta}{ 2(3 + \delta) } }.  
\end{align}
In particular, we have, uniformly in $n\geq 1$ and $x > n^{\frac{1}{2} -  \frac{\delta}{ 2(3 + \delta) }  }$, 
\begin{align} \label{CorCCLT02b}
\left| \mathbb{P} \left( \tau_x >n \right) 
     -\left( 2\Phi\left( \frac{x}{\sigma \sqrt{n}}   \right)-1 \right) \right| 
     % \int_{0}^{\infty}  \psi \left( s, \frac{x}{\sigma \sqrt{n}}   \right)  ds
   \leq c_{\delta}  n^{- \frac{\delta}{ 2(3 + \delta) } }.  
\end{align}
\end{theorem}

Combining the bounds \eqref{CLLT-bound only 001} and \eqref{CorCCLT02},
%and taking into account that
%\begin{align*} %\label{}
%\int_{0}^{t}  \psi \left( s, \frac{x}{\sigma \sqrt{n}}   \right)  ds 
%\sim \left( 2\Phi\left( \frac{x}{\sigma \sqrt{n}}   \right)-1 \right) \Phi^+(t) 
%\sim \frac{2  V(x) }{ \sigma \sqrt{2 \pi  n}}   \Phi^+(t),
%\end{align*}
 one can deduce the following:
\begin{corollary}\label{Cor-CCLT-Optimal}
Assume \ref{SecondMoment}. 
Then there exists $\ee>0$ such that for any sequence of positive numbers $(\alpha_n)_{n \geq 1}$ 
satisfying $\lim_{n \to \infty} \alpha_n = 0$, 
 it holds 
 uniformly in $n \geq 1$, $t \in  \bb R_+$ and $x\in [0, \alpha_n \sqrt{n}]$,  
\begin{align} \label{C-CLTalphan}
 \left| \mathbb{P} \left(  \frac{x+S_n }{\sigma \sqrt{n}} \leq t,  \tau_{x}>n\right) - 
\frac{2V(x)}{ \sigma \sqrt{2\pi n}} \Phi^+(t)\right| 
\leq c_{\ee} \left(  \alpha_n  + n^{-\ee} \right) \frac{1+x}{n^{1/2}}.
%c \frac{1+x}{n^{1/2+\veps}}.
\end{align}
%\end{corollary}
%\textcolor{magenta}{Moreover, in the same way,  we obtain the following bound:
%\begin{corollary}\label{Cor-CCLT-Optimal-bis}
%Assume \ref{SecondMoment}. 
In addition, there exists $\ee \in (0, \frac{1}{2})$ such that for any $\beta \in (0,  \frac{1}{2} - \ee)$,  
%for any sequence of positive numbers $(\alpha_n)_{n \geq 1}$ satisfying $\lim_{n \to \infty} \alpha_n = 0$, 
% uniformly in $n \geq 1$, $t \in  \bb R_+$ and $x\in [0, \alpha_n \sqrt{n}]$,  
uniformly in $n \geq 1$, $t \in  \bb R_+$ and $x\geq n^{\beta}$, 
\begin{align} \label{CorCCLT02bis}
\left| \mathbb{P} \left(  \frac{x+S_n }{ \sigma \sqrt{n}} \leq  t, \tau_x >n\right) 
     -   \int_{0}^{t}  \psi \left( s, \frac{x}{\sigma \sqrt{n}}   \right)  ds    \right|
   \leq  c_\ee  \frac{1 + \min\{x, n^{1/2-\ee}\} }{ n^{1/2 + \ee \beta  }  }.
   %  c_{\delta}  n^{- \frac{\delta}{ 2(3 + \delta) } }.  
\end{align}
\end{corollary}

Some comments on the precision of the above results %of the asymptotic \eqref{C-CLTalphan} 
seem to be appropriate.
%It is interesting to note that the rate in \eqref{CorCCLT02b} can be improved. Indeed, 
Nagaev \cite{Nag69, Nag70} showed that, under the assumption that $\beta_3=\bb E (|X_1|^3) < \infty$ 
%the third order moment of $X_1$ is finite 
(which corresponds to $\delta=1$), 
 the remainder term in \eqref{CorCCLT02b} can be improved, namely it is of the order $n^{-1/2}$: uniformly in $x\geq 0$,
 \begin{align} \label{CorCCLT02c}
\left| \mathbb{P} \left( \tau_x >n \right) 
     -\left( 2\Phi\left( \frac{x}{\sigma \sqrt{n}}   \right)-1 \right) \right| 
     % \int_{0}^{\infty}  \psi \left( s, \frac{x}{\sigma \sqrt{n}}   \right)  ds
   \leq c (\beta_3)^2 n^{- 1/2  }.  
\end{align}
Aleskevicene \cite{Aleskev73} improved the upper bound in \eqref{CorCCLT02c} to $c \beta_3 n^{- 1/2  }$. 
%  We conjecture that, similarly to \eqref{CorCCLT02b}, the remainder term in \eqref{CorCCLT02} 
%can also be improved to be of order $n^{-1/2}$.  

Note however that Nagaev's bound \eqref{CorCCLT02c} makes sense only when $x=x_n\to\infty$ as $n\to\infty$. 
For $x$ in compact sets it is not precise, 
since the remainder term $O(n^{-1/2})$ is of the same order as the main term 
$2\Phi ( \frac{x}{\sigma \sqrt{n}} ) - 1$.
In fact, for any fixed $x\geq 0$ % the bound \eqref{CorCCLT02b} 
the right asymptotic of the probability $\mathbb{P} \left( \tau_x >n \right)$
is not given by \eqref{CorCCLT02c} but by \eqref{CLLT-bound only 001b}. 
Borovkov \cite[Theorem 6]{Borovk62} %Khinchin ??? and Prokhorov ??? 
obtained precise asymptotics of the above probabilities under the stronger condition that $X_1$ has exponential moment
and under additional assumption that the distribution of $X_1$ has an absolutely continuous 
component.
The results of Borovkov are stated in terms of some infinite series, which makes the comparison with ours %results
tedious.  For instance, we could not deduce \eqref{CorCCLT02bis} from \cite{Borovk62}.

Note that the uniformity in $t$ and $x$ is crucial for establishing 
our main theorems in Section \ref{Sect-Mean0}. 

The proofs of Theorems \ref{Theor-IntegrLimTh}, \ref{CorCCLT} and Corollary \ref{Cor-CCLT-Optimal}
 will be given in Appendix \ref{Sect-Appendix}. 
 
%{\color{magenta}
%Because of the limitation of the length,
%the proofs of Theorems \ref{Theor-IntegrLimTh}, \ref{CorCCLT} and Corollary \ref{Cor-CCLT-Optimal} will be not given here. 
%But they can be performed in a similar way as the corresponding results in \cite{GLL18}.
% } \todos{For the published version: remove the Appendix}

%\ref{Thm-C 002}, \ref{Thm_CLLT_Drift_Negative001}, \ref{ThmCaraLargeStart} and \ref{ThmCLLTLarge32}. 

%%%%%%%%%%%%%%%%%%%%%%%%%%%%%%%%%%%%%%%%%%%%%
%%%%%%%%%%%%%%%%%%%%%%%%%%%%%%%%%%%%%%%%%%%%%
\section{Conditioned concentration bounds near the boundary} \label{sec-proof Th CondLLT Version1}
%CLLT for target functions with supports moving to infinity and with small starting point

%In this section we will prove Theorem \ref{Thm-C 002}. 
%This theorem will be deduced from some upper and lower bounds for the expectation
%$$
% \mathbb{E}\left( f\left(x+S_n \right) ;\tau_x >n\right),
%$$
%where $f$ is a non-negative measurable function on $\bb R.$ 

The goal of this section is to establish Theorems \ref{Theorem-AA001} and
\ref{Theorem Delta-002}.
%, \ref{Thm-C 002} and \ref{Thm-C 002-Delta}. 

%%%%%%%%%%%%%%%%%%%%%%%%%%%%%%%%%%%%%%%%%%
%%%%%%%%%%%%%%%%%%%%%%%%%%%%%%%%%%%%%%%%%%
%%%%%%%%%%%%%%%%%%%%%%%%%%%%%%%%%%%%%%%%%%
\subsection{Formulation of the result}
%A non-asymptotic conditioned local limit theorem of order $n^{-1}$ for small starting point}
%Theorem \ref{Thm-C 002} is a consequence of a non-asymptotic conditioned local limit theorem which we state below. 
%In the following, given any integrable function $f: \bb R \mapsto \bb R_+$ the convolution $f* \kappa_{\ee^2}$ 
%%the convolution of the function  $f \mathds 1_{[0,\infty)}$ with $\kappa_{\ee^2}$ on $\bb R$ 
%is defined as usual:
%\begin{align*} %\label{}
% f* \kappa_{\ee^2}(t) = \int_{\bb R} f(v) \kappa_{\ee^2}(t-v)dv.  %\notag\\
%\end{align*}
%Recall that, for any integrable function $f: \bb R \mapsto \bb R_+$ we denote 
%$\|f\|_1= \int_{\bb R} f(t)dt.$

%In the following, for any measurable function $h$ and any $\veps>0$ we denote by $h_{\ee}$ and $h_{-\ee}$ 
%any of its measurable upper and lower $\ee$-envelopes, i.e.\ any measurable functions $h_{\ee}$ and $h_{-\ee}$ such that
% $\sup_{|v|\leq \ee} h(u+v)\leq h_{\ee}(u)$ and $\inf_{|v|\leq \ee} h(u+v)\geq h_{-\ee}(u)$. 

%func the set of functions $\mathscr H_{\veps}$ contains all non-negative and locally 
%bounded functions $g$ such that $g$, $g_{\veps}$ and $g_{-\veps}$ 
%are measurable from $(\mathbb R, \mathscr B (\mathbb R))$ to $(\mathbb R_+,\mathscr B (\mathbb R_+))$ 
%and Lebesgue-integrable, see Section \ref{subsec-nonasymptoticLLT}.

In the proof of Theorems \ref{Theorem-AA001} and \ref{Theorem Delta-002}, 
we need a refined version of the conditioned local limit theorem formulated below, 
which is stated as upper and lower bounds.

\begin{theorem} \label{t-B 002}   
Assume \ref{A1} and  \ref{SecondMoment}. 
%Let $(\alpha_{n})_{n\geq 1}$ be any sequence of positive numbers satisfying $\lim_{n\to \infty} \alpha_{n} = 0$. 
%For any non-negative function $g:\mathbb R_+ \mapsto \mathbb R_+$ satisfying 
%$\int \left(1+x\right) g\left( x\right) dx<\infty $ 
%and any  $y\in \bb R _{+}^{\ast},$ 
%Suppose that $X_{1}$ is non-lattice. 
Then, there exist constants $c>0$ and $\ee_0 >0$ such that for any $\ee \in (0, \ee_0)$, 
one can find a constant $c_{\ee} >0$ such that 
for any sequence of positive numbers $(\alpha_{n})_{n\geq 1}$ satisfying $\lim_{n\to \infty} \alpha_{n} = 0$, 
the following holds uniformly in $x\in [0, \alpha_n \sqrt{n}]$ and $n\geq 1$: 
% there exists  
%a  sequence of positive real numbers $(r_{n}(\veps))_{n\geq 1}$ with 
%$\lim_{n\rightarrow \infty}r_{n}(\veps)=0$, 
%such that  

\noindent 1. For any integrable functions $f,g: \bb R \mapsto \bb R_+$ satisfying $f\leq_{\ee}g$,  
%and $f(x) = 0$ for $x < 0$, 
%$x\in [0, \alpha_n \sqrt{n}]$ and $n\geq 1$, 
\begin{align}\label{eqt-B 001} 
&  \mathbb{E}\left( f\left(x+S_n \right) ;\tau_x >n\right)  
 \leq   (1 +  c \ee)   \frac{2V(x) }{\sqrt{2\pi } \sigma^2 n}
\int_{\bb{R}_{+}}  g(t) \phi^+ \left( \frac{t}{ \sigma \sqrt{n}} \right)  dt   \notag\\
& \quad  + c \ee^{1/4} \frac{V(x)}{n} \int_{\mathbb R} g(t) \phi \left( \frac{t}{ \ee^{1/4} \sigma  \sqrt{n} } \right) dt   
+  c_{\ee}  \left(  \alpha_n  + n^{-\ee} \right) \frac{V(x)}{ n } \left\|  g \right\|_1.
\end{align}

\noindent 2. For any integrable functions $f,g,h: \bb R \mapsto \bb R_+$ satisfying $h\leq_{\ee}f\leq_{\ee}g$,
\begin{align} \label{eqt-B 002}
  \mathbb{E}\left( f\left(x+S_n \right) ;\tau_x >n\right)   
& \geq    \frac{2 V(x) }{\sqrt{2\pi } \sigma^2 n}
 \int_{\bb{R}_{+}}  h(t)  \phi^+ \left( \frac{t}{ \sigma \sqrt{n}} \right)   dt   \notag\\
&   -  c \ee^{1/12}  \frac{V(x) }{n}    
   \int_{\bb R}  g(t)  \left[ \phi \left( \frac{t}{ \sigma \sqrt{n}} \right)  +  \phi^+ \left( \frac{t}{ \sigma \sqrt{n}} \right)   \right]  dt  \notag\\
 &\quad      -  c_{\ee}  \left(  \alpha_n  + n^{-\ee} \right) \frac{V(x)}{ n }  \left\|  g \right\|_1.  
\end{align}
%where $c>0$ is a constant not depending on $\ee$, $n$, $f, g$ and  $x$. 
\end{theorem}

Note that the bounds in Theorem \ref{t-B 002}
 become effective when the support of the target function $f$ moves to infinity with an appropriate rate. 
As a consequence of this theorem, 
in Section \ref{SectProofThm1} we shall prove Theorems \ref{Theorem-AA001} and \ref{Theorem Delta-002}. 

\subsection{Duality identities}
In this section we establish two duality identities for the random walk $x + S_n$ jointly with the exit time $\tau_x$.
%by utilising the fact that the increments of the random walk are independent identically distributed
These identities are simple consequence of the fact that the Lebesgue measure is translation invariant. 
Recall that $S_0 = 0$ and $S_n=\sum_{i=1}^{n} X_i$ for $n \geq 1$, and its dual random walk is 
given by $S_j^*= S_{n-j} - S_n$ for $0 \leq j \leq n.$

\begin{lemma} \label{lemma-duality-lemma-2_Cor}
For any $n\geq 1$ and any bounded measurable functions $g, h: \bb R \mapsto \mathbb R$, 
%and $g:  \bb R \mapsto \mathbb R$, 
we have
\begin{align} \label{eq-duality-lemma-002_Cor}
\int_{\mathbb R_+}  h(x) \mathbb E   g(x+S_n) \mathds 1_{\{ \tau_x > n  \}} dx
=\int_{\mathbb R_+}  g(y) \mathbb E  h(y+S_n^*)  \mathds 1_{\{ \tau_y^* > n \}} dy 
\end{align}
and 
\begin{align} \label{eq-duality-lemma-002_Cor_less}
& \int_{\mathbb R}  h(x) \mathbb E   g(x+S_n)  \mathds 1_{\{ \tau_x \leq  n  \}} dx 
 = \int_{\mathbb R}  g(y) \mathbb E  h(y+S_n^*)  \mathds 1_{ \{ \tau_y^* \leq n \}} dy.
\end{align}
%and 
%\begin{align} \label{eq-duality-lemma-002_Cor_less}
%& \int_{\mathbb R}  h(x) \mathbb E   g(x+S_n)  \mathds 1_{\{ \tau_x \leq  n  \}} dx 
% = \int_{\mathbb R}  g(y) \mathbb E  h(y+S_n^*)  \mathds 1_{ \{ \tau_y^* \leq n \}} dy.
%\end{align}
%Moreover, for any $k$ and $m$ satisfying $0\leq k,m\leq n$, $k + m = n$, we have
%\begin{align} \label{eq-duality-lemma-003_Cor}
%& \int_{\mathbb R}  h(x) \mathbb E g(x+S_n) 
%   \mathds 1_{\big\{   x + S_k \geq 0, \,  \ldots, \,  x + S_n \geq 0 \big\}} dx  \notag \\
%& = \int_{\mathbb R}  g(y) \mathbb E  h(y + S_n^*)  
%\mathds 1_{\big\{y \geq 0,  \,  y + S_{1}^* \geq 0,  \,  \ldots, \,  y + S_{m}^* \geq 0 \big\}} dy.
%\end{align}
%and 
%\begin{align} \label{eq-duality-lemma-004_Cor}
%& \int_{\mathbb R}  h(x) \mathbb E g(x+S_n) 
%   \mathds 1_{\big\{  x + S_j < 0,  \, \exists  j\in [k,n] \big\}} dx  \notag \\
%& = \int_{\mathbb R}  g(y) \mathbb E  h(y + S_n^*)  
%\mathds 1_{\big\{ y + S_j < 0,  \,  \exists  j\in [0,m] \big\}} dy.
%\end{align}
\end{lemma}

\begin{proof}%\label{}
By a change of variable $x + S_n = y$, we get
\begin{align*}%\label{}
& \int_{\mathbb R_+}  h(x) \mathbb E   g(x+S_n) \mathds 1_{\{ \tau_x > n  \}} dx  \notag\\
& = \int_{\mathbb R}  h(x) \mathbb E   g(x+S_n) \mathds 1_{\{ x \geq 0, \,  x+S_1 \geq 0,  \, 
      \ldots, \, x + S_{n-1} \geq 0, \, x+S_n \geq 0  \}} dx  \notag \\
& =  \int_{\mathbb R}  g(y)  \mathbb E  h(y - S_n) 
   \mathds 1_{\{ y - S_n \geq 0, \,  y - S_n +S_1 \geq 0, \,  \ldots, \, y - S_n + S_{n-1} \geq 0, \,  y  \geq 0  \}} dy  \notag \\
& =  \int_{\mathbb R}  g(y)  \mathbb E  h(y + S_n^*) 
   \mathds 1_{\{ y + S_n^* \geq 0, \,  y + S_{n-1}^*  \geq 0, \,  \ldots, \,  y + S_1^* \geq 0, \,  y \geq 0  \}} dy  \notag\\
& =  \int_{\mathbb R_+}  g(y) \mathbb E  h(y+S_n^*)  \mathds 1_{\{ \tau_y^* > n \}} dy,   
\end{align*}
which concludes the proof of \eqref{eq-duality-lemma-002_Cor}.

%Taking $K_0 = \ldots = K_n = [0, \infty)$ in \eqref{eq-duality-lemma-002}, we get \eqref{eq-duality-lemma-002_Cor}. 
%Taking $K_i = \bb R$ for $0 \leq i \leq n$ in \eqref{eq-duality-lemma-002}, we get 
Using again the change of variable $x + S_n = y$ gives
\begin{align}\label{Eq_dual_mmm}
\int_{\mathbb R}  h(x) \mathbb E   g(x+S_n)   dx
=\int_{\mathbb R}  g(y) \mathbb E  h(y+S_n^*)    dy.
\end{align}
%By \eqref{eq-duality-lemma-002_Cor},  we have 
%\begin{align} \label{eq-Dual_000}
%& \int_{\mathbb R}  h(x) \mathbb E   g(x+S_n) \mathds 1_{\{ x \geq 0, x + S_1 \geq 0,  \ldots, x + S_n \geq  0  \}}   dx  \notag\\
%& = \int_{\mathbb R}  g(y) \mathbb E  h(y+S_n^*)  \mathds 1_{\{ y \geq 0, y + S_1^* \geq 0, \ldots,  y + S_n^* \geq 0 \}} dy 
%\end{align}
Taking the difference between \eqref{Eq_dual_mmm} and \eqref{eq-duality-lemma-002_Cor}, 
we get \eqref{eq-duality-lemma-002_Cor_less}. 
%Taking $K_i = \bb R$ for $0 \leq i \leq n$ in \eqref{eq-duality-lemma-002}, we get 
%\begin{align}\label{Eq_dual_mmm}
%\int_{\mathbb R}  h(x) \mathbb E   g(x+S_n)   dx
%=\int_{\mathbb R}  g(y) \mathbb E  h(y+S_n^*)    dy.
%\end{align}
%By \eqref{eq-duality-lemma-002_Cor},  we have 
%\begin{align} \label{eq-Dual_000}
%& \int_{\mathbb R}  h(x) \mathbb E   g(x+S_n) \mathds 1_{\{ x \geq 0, x + S_1 \geq 0,  \ldots, x + S_n \geq  0  \}}   dx  \notag\\
%& = \int_{\mathbb R}  g(y) \mathbb E  h(y+S_n^*)  \mathds 1_{\{ y \geq 0, y + S_1^* \geq 0, \ldots,  y + S_n^* \geq 0 \}} dy 
%\end{align}
%Taking the difference between \eqref{Eq_dual_mmm} and \eqref{eq-Dual_000}, 
%we get \eqref{eq-duality-lemma-002_Cor_less}. 
%%\begin{align*}%\label{}
%%& \int_{\mathbb R}  h(x) \mathbb E   g(x+S_n) \mathds 1_{\{ \min_{0 \leq i \leq n} x + S_i \leq 0  \}}   dx  \notag\\
%%& = \int_{\mathbb R}  g(y) \mathbb E  h(y+S_n^*)  \mathds 1_{\{ y > 0, y + S_1^* >0, \ldots,  y + S_n^* >0 \}} dy 
%%\end{align*}
\end{proof}

\subsection{Auxiliary statements}
We state several auxiliary statements which will be used in the proofs of the main results.
Recall that
$\phi_{v}(x)=\frac{1}{\sqrt{2\pi v}}e^{-\frac{x^2}{2 v}}$, $x\in \mathbb R$   
is the normal density of mean $0$ and variance $v$ 
and that $\phi^+_{v}(x)=\frac{x}{v} e^{-\frac{x^2}{2 v}} \mathds 1_{\{x \geq 0\}}$, $x\in \mathbb R$
is the Rayleigh density with scale parameter $\sqrt{v}$.
It holds that $\phi_1(x) = \phi(x)$ and $\phi^+_1(x) = \phi^+(x)$, $x\in \mathbb R$. 
%The standard normal density is denoted by $\phi(x)=\phi_1(x)$, $x\in \mathbb R$ and
%the Rayleigh density with scale parameter $1$ is denoted by $\phi^+(x)=\phi^+_1(x)$, $x\in \mathbb R$.   
%Consider the convolution function
%\begin{align*} %\label{}
%\phi_{v} * \phi^+_{1- v}(x)
%= \int_{\bb R } \phi_{v} \left( z\right) \phi^+_{1-v}\left(x-z\right)  dz. 
%\end{align*}
The following lemma shows that when $v$ is small, the convolution $\phi_{v} * \phi^+_{1-v}$ 
behaves like the Rayleigh density $\phi^+$.
%\todos{NEW: We have a better approximation: we have an exact expression  !!!}
\begin{lemma} \label{t-Aux lemma}
For any $v \in (0,1/2]$ and $x\geq 0$, it holds
\begin{align} \label{SmoothingBoundsRayleigh}
\sqrt{1-v} \phi^+(x) 
\leq  \phi_{v} * \phi^+_{1- v}(x) 
\leq  \sqrt{1-v} \phi^+(x) +  \sqrt{v}  e^{ -\frac{x^2}{2v} }.   
\end{align}
%In particular, there exists an absolute constant $c>0$ such that for any $\delta \in (0,\frac{1}{2})$, it holds
%\begin{align*} %\label{}
%\sup_{x\in \bb R} \left| \phi_{\delta} * \phi^+_{1-\delta}(x) - \phi^+(x) \right|  \leq c \delta^{1/2}.
%\end{align*}
\end{lemma}
\begin{proof}
%\todos{There is an error in the proof}
By definition, we have that for any $x\geq 0$, % it holds
\begin{align*}%\label{proofconv01-001}
\phi_{v} * \phi^+_{1-v}(x) 
& = \int_{\bb R}  \phi_{v}(z) \phi^+_{1-v}(x-z) dz  \notag\\  
%& = \int_{\bb R } \frac{1}{\sqrt{\delta}}\phi \left( \frac{z}{\sqrt{\delta}}\right)
 %  \frac{1}{\sqrt{1-\delta}}   \phi^+\left( \frac{x - z}{\sqrt{1-\delta}} \right)   dz  \notag \\ 
%& = \int_{\bb R } \frac{1}{\sqrt{2\pi \delta}} e^{ -\frac{z^2}{2\delta} }
 %  \frac{x-z}{\sqrt{1-\delta}}   e^{ -\frac{(z-x)^2}{2(1-\delta)} }   dz \notag \\ 
& = \frac{1}{\sqrt{1-v}}  \int_{\bb R } \frac{x-z}{\sqrt{2\pi v (1-v)}}
      e^{ -\frac{z^2}{2 v}  -\frac{(x - z)^2}{2(1-v)} } \mathds 1_{\{ x-z \geq 0 \}}   dz. 
\end{align*}
Since
\begin{align*} %\label{}
  \frac{z^2}{2 v}  + \frac{(x-z)^2}{2(1-v)} = \frac{x^2}{2}  + \frac{(z-x v)^2}{2v(1-v)},
\end{align*} %\notag\\
 we get
\begin{align*}%\label{proofconv01-002}
\phi_{v} * \phi^+_{1-v}(x)   
& = \frac{1}{\sqrt{1-v}} e^{ -\frac{x^2}{2} } \int_{\bb R } \frac{ x - z }{\sqrt{2\pi v(1-v)}}
    e^{ -\frac{(z-xv )^2}{2v (1-v)} }  \mathds 1_{\{ x-z \geq 0 \}}  dz   \notag\\ 
&  = : I_1 + I_2,   
\end{align*}
where 
\begin{align*}
I_1 & = \frac{1}{\sqrt{1-v}} e^{ -\frac{x^2}{2} } 
    \int_{\bb R } \frac{x - z}{\sqrt{2\pi v(1-v)}}  e^{ -\frac{(z-xv )^2}{2v (1-v)} }   dz 
      =  \sqrt{1-v} \,  x e^{ -\frac{x^2}{2} },   \notag \\ 
I_2 &  =  \frac{1}{\sqrt{1-v}} e^{ -\frac{x^2}{2} } \int_{\bb R } \frac{z - x}{\sqrt{2\pi v(1-v)}}
    e^{ -\frac{(z-xv )^2}{2v (1-v)} }  \mathds 1_{\{ x-z < 0 \}}  dz.   
%& \geq  %\frac{1}{\sqrt{1-v}} e^{ -\frac{x^2}{2} } \left( x - xv \right) 
%  x e^{ -\frac{x^2}{2} }\sqrt{1-v},
\end{align*}
Hence the first inequality in \eqref{SmoothingBoundsRayleigh} holds since $I_2 \geq 0$. 

To prove the second inequality in \eqref{SmoothingBoundsRayleigh}, it remains to give an upper bound for $I_2$. 
By a change of variable,  we have 
\begin{align*}%\label{}
I_2 & = \sqrt{\frac{v}{2 \pi}} e^{ -\frac{x^2}{2} } 
  \int_{x \sqrt{\frac{1-v}{v}}}^{\infty}  \left( w - x \sqrt{\frac{1-v}{v}} \right)  e^{ -\frac{w^2}{2} }  dw   \notag\\
& \leq  \sqrt{\frac{v}{2 \pi}} e^{ -\frac{x^2}{2} } 
  \int_{x \sqrt{\frac{1-v}{v}}}^{\infty}   w   e^{ -\frac{w^2}{2} }  dw 
  =  \sqrt{\frac{v}{2 \pi}}  e^{- \frac{x^2}{2v}}, 
\end{align*}
which finishes the proof of the second inequality in \eqref{SmoothingBoundsRayleigh}. 
\end{proof}

%\begin{align*}%\label{}
%I_2 &= \sqrt{\frac{v}{2 \pi}} e^{ -\frac{x^2}{2} } 
%  \int_{x \sqrt{\frac{1-v}{v}}}^{\infty}  \left( w - x \sqrt{\frac{1-v}{v}} \right)  e^{ -\frac{w^2}{2} }  dw   \notag\\
%& =  
%\sqrt{\frac{v}{2 \pi}} e^{ -\frac{x^2}{2} } 
%  \int_{x \sqrt{\frac{1-v}{v}}}^{\infty}   w   e^{ -\frac{w^2}{2} }  dw 
%- x \sqrt{\frac{1-v}{v}}   \sqrt{\frac{v}{2 \pi}} e^{ -\frac{x^2}{2} } 
%  \int_{x \sqrt{\frac{1-v}{v}}}^{\infty}    e^{ -\frac{w^2}{2} }  dw \\
% & =  \sqrt{\frac{v}{2 \pi}}  e^{- \frac{x^2}{2v}} 
% - x \sqrt{\frac{1-v}{v}}   \sqrt{\frac{v}{2 \pi}} e^{ -\frac{x^2}{2} } 
% (1-\Phi(x \sqrt{\frac{1-v}{v}})) 
% %\int_{x \sqrt{\frac{1-v}{v}}}^{\infty}    e^{ -\frac{w^2}{2} }  dw  
%\end{align*}

The following Fuk-Nagaev  inequality can be found in \cite{FN71, Nag79}. 
\begin{lemma}\label{Lem_FukNagaev}
Let  $\bb E X_1 = 0$ and  $\bb E (X_1^{2}) = \sigma^2  < \infty.$
% \ref{SecondMoment} with some $\delta>0$. 
%\todos{Check the condition}
Then, for any $u, v >0$, 
\begin{align*}%\label{}
\bb P \left(\max_{1 \leq k \leq n} |S_k| > u \right) 
\leq  2 \exp \left[ \frac{u}{v}  \left( 1 + \log \frac{n}{uv} \right)  \right]   + n \bb P \left(|X_1| > v \right). 
\end{align*}
\end{lemma}

Let $\left( B_{t}\right)_{t\geq 0}$ be a standard Brownian motion on the
probability space $\left(\Omega, \mathscr{F}, \mathbb{P} \right).$ 
For any $x \geq 0$, define the exit time
%\todos{Should we replace $x>0$ by $x \geq 0$? }
\begin{equation*}
\tau_{x}^{bm} = \inf \left\{ t \geq 0: x + \sigma B_{t} < 0 \right\}. 
\end{equation*}
%where $\sigma^2 >0$ is the variance of the random variable $X_1$. 
The following well known formulas are due to Levy \cite{Levy37} (Theorem 42.I, pp.194-195).

\begin{lemma}\label{lemma tauBM} 
%Let $x > 0$. The stopping time $\tau_{x}^{bm}$ has the following properties:
%\noindent 1. For any $n\geq 1$,
%\noindent 2. For any  $a,b$ satisfying $0\leq a<b<+\infty $ and $n\geq 1$,
For any $x \geq 0$, $n\geq 1$ and $0 \leq a < b \leq \infty $, 
\begin{align}\label{levy001b}
\bb P \left( x +  \sigma B_{n} \in \left[a, b\right],  \tau_{x}^{bm}>n \right) 
=  \frac{1}{\sigma \sqrt{ n} } \int_{a}^{b}  \psi \left( \frac{s}{\sigma \sqrt{n}},\frac{x}{\sigma \sqrt{n}} \right)
%\left( e^{-\frac{ (s-x)^2 }{ 2 n }} - e^{-\frac{ (s+x)^2 }{ 2 n }} \right) 
ds. 
\end{align}
In particular, by taking $a = 0$ and $b = \infty$, for any $x \geq 0$ and $n\geq 1$, 
\begin{align}\label{levy001a}
\bb P \left( \tau_{x}^{bm} > n \right) 
= \bb P \left( \sigma \inf_{0 \leq u \leq n} B_{u} \geq -x \right) 
%=2\Phi \left( \frac{x}{\sqrt{n}} \right)-1 
% = \Phi \left( \frac{[-x,x]}{\sigma\sqrt{n}} \right)\notag\\ 
= \frac{2}{\sigma \sqrt{2\pi n} } \int_{0}^x  e^{ - \frac{s^{2}}{2 \sigma^2 n} } ds.   
\end{align}
%Moreover, for any sequence $(\alpha_n)_{n \geq 1}$ of positive numbers satisfying $\lim_{n \to \infty} \alpha_n = 0$,  
%\begin{align}\label{Relative_Exit_BM}
%\sup_{x\in [0,\alpha_n \sqrt{n}]} \left| \frac{\mathbb{P}\left( \tau_{x}^{bm} > n \right)}{\frac{ 2 x }{\sigma \sqrt{2\pi n}}} - 1  \right| 
%= O(\alpha_n).  
%\end{align}
%In particular, for any $t \geq 0$, $x>0$ and $n \geq 1$, 
%\begin{align}\label{Levy_CondiCLT}
%\bb P \left( \frac{x +  B_{n}}{\sqrt{n}} > t,  \tau_{x}^{bm}>n \right) 
%=  \frac{1}{\sqrt{2\pi n} } \int_{t \sqrt{n}}^{\infty}  \left( e^{-\frac{ (s-x)^2 }{ 2 n }} - e^{-\frac{ (s+x)^2 }{ 2 n }}  \right) ds
%\end{align}
%and 
%\begin{align}\label{Levy_CondiCLT_leq}
%\bb P \left( \frac{x +  B_{n}}{\sqrt{n}} \leq t,  \tau_{x}^{bm}>n \right) 
%=  \frac{1}{\sqrt{2\pi n} } \int_{0}^{t \sqrt{n}}  \left( e^{-\frac{ (s-x)^2 }{ 2 n }} - e^{-\frac{ (s+x)^2 }{ 2 n }}  \right) ds.  
%\end{align}
\end{lemma}

%{\color{magenta}
%\begin{align*}
%\bb P \left( x +  \sigma B_{n} \in \left[0,  \Delta \right] + y,  \tau_{x}^{bm}>n \right) 
%& =  \frac{1}{\sigma \sqrt{ n} } \int_{y}^{y + \Delta}  \psi \left( \frac{s}{\sigma \sqrt{n}},\frac{x}{\sigma \sqrt{n}} \right) ds  \notag\\
%& =  \frac{1}{\sigma \sqrt{ n} } \int_{0}^{\Delta}  \psi \left( \frac{s' + y}{\sigma \sqrt{n}},\frac{x}{\sigma \sqrt{n}} \right) ds'   \notag\\
%& \sim  \frac{\Delta}{\sigma \sqrt{ n} } \psi \left( \frac{y}{\sigma \sqrt{n}},\frac{x}{\sigma \sqrt{n}} \right) 
%\end{align*}
%}

%Note that \eqref{levy001a} is a particular case of \eqref{levy001b} by taking $a = 0$ and $b = \infty$. 
%The conditioned central limit theorems \eqref{Levy_CondiCLT} \eqref{Levy_CondiCLT_leq} are direct consequences of %\eqref{levy001b} 
%by taking $a = t \sqrt{n}$ and $b = \infty$,  and $a = 0$ and $b = t \sqrt{n}$, respectively. 

%%%%%%%%%%%

We need the following functional central limit theorem due to Sakhanenko \cite{Sak06},
which allows us to couple out the random walk with the Brownian motion. 
%whose proof can be found in Sakhanenko \cite{Sak06}. 

\begin{lemma}\label{FCLT}
Assume  \ref{SecondMoment} with some $\delta>0$. 
Then there exists a   
%and positive sequence $(r_n)_{n\geq 1}$ satisfying $r_n \to 0$ as $n \to \infty$
construction of the random walk $(S_n)_{n\geq 0}$ on the initial probability space
together with a continuous time Brownian motion $(B_t)_{t\geq 0}$ such that for any 
$\gamma \in (0,  \frac{\delta}{2(2+\delta)})$ and $n\geq 1$,
\begin{align*}
\bb{P} \left( \sup_{0\leq t\leq 1} \left| S_{[nt]} - \sigma B_{nt} \right| > n^{1/2 - \gamma}  \right)
 \leq  \frac{c_{\gamma}}{ n^{ \delta/2 - (2+\delta) \gamma} }, 
\end{align*}
where $c_{\gamma}$ is a constant depending only on $\gamma$.
\end{lemma}

Note that $\delta$ can be greater than $1$ in Lemma \ref{FCLT}.

%%%%%%%%%%%%%%%%%%%%%%%%%%%%%%%%%%%%%%%%%%%%%%%%%
%%%%%%%%%%%%%%%%%%%%%%%%%%%%%%%%%%%%%%%%%%%%%%%%%
\subsection{Proof of the upper bound}\label{Sec-UpperBound}
The goal of this section is to prove the upper bound \eqref{eqt-B 001} of Theorem \ref{t-B 002}.

%\begin{proof}[Proof of \eqref{eqt-B 001}]
%The idea of the proof is as follows. 
%We split the trajectory of the walk $(x+S_n)_{n\geq 0}$ into two parts of length $k$ and $m$ respectively and use the Markov property. 
%%which is seen as a convolution formula. 
%The second part of the trajectory is handled using the effective local limit theorem (Theorem \ref{LLT-general}), 
%and for the first part we apply the effective conditioned central limit theorem (Theorem \ref{Theor-IntegrLimTh}). 
%Then the assertion \eqref{eqt-B 001} is obtained as a convolution of the bounds of two parts 
%%of the normal law and of the Rayleigh one
% when $\frac{m}{k}$ becomes small.  

Let $\veps >0$ be a sufficiently small constant. With $\delta = \sqrt{\veps}$, set
$m=\left[ \delta n \right]$ and $k = n-m.$ 
Note that for $n$ such that $n > n_0(\ee):= \frac{4}{\ee} $, we have
$
\frac{1}{2}\delta^{1/2} %\leq \frac{\left[ \delta n\right]}{n} 
\leq \frac{m}{k} %= \frac{\left[ \delta n\right]}{n\left( 1-\frac{\left[\delta n\right] }{n}\right)} 
\leq %\frac{\delta n}{n\left( 1-\frac{\delta n}{n}\right)}= 
\frac{\delta}{1-\delta}.
$ 
%We proceed to prove the inequality \eqref{eqt-B 001} in Theorem \ref{t-B 002}. 
%Let $k$ and $m$ be integers in $ \bb N$ depending on $n$ such that $k+m=n$ and $k,m$ go to infinity as $n\to \infty.$
 %with $2\ee$ instead of $\ee$. 
It suffices to prove \eqref{eqt-B 001} only for sufficiently large $n>n_0(\ee)$, where $n_0(\ee)$ depends on $\ee$; 
otherwise the bound becomes trivial.
We start by using the Markov property to get that 
for any starting point $x \in \mathbb R_+$, 
%by the Markov property, it holds
\begin{align} \label{JJJ-markov property}
I_n(x): & =\mathbb{E} \left( f(x+S_n ); \tau_x >n\right) \notag \\  
%&= \int_{\mathbb R_+}  \int_{\mathbb R_+} f(t)  \mathbb P \left( y+S_{m}\in dt ;\tau _{y}>m\right) 
%\mathbb{P}\left( x+S_{k}\in dy ;\tau_x >k\right) \notag \\
&= \int_{\mathbb R_+}  \mathbb{E} \left( f(t+S_{m}); \tau_t >m\right) \mathbb{P}\left( x+S_{k}\in dt,  \tau_x >k\right)
 \notag \\
&\leq \int_{\mathbb R_+}  \mathbb{E}  f(t+S_{m}) \mathbb{P}\left( x+S_{k}\in dt,  \tau_x >k\right).
\end{align}
%Using \eqref{ladderfun002} we have
%\begin{align} \label{JJJ-markov property}
%I_n(x)
%&\leq \sum_{l\in \bb Z} f_l 
%\int_{\mathbb R_+}   \mathbb{P} \left( x+S_{m} \in I_l  \right) \mathbb{P}\left( x+S_{k}\in dt;  \tau_x >k\right).
%\end{align}
By the local limit theorem (Theorem \ref{LLT-general}), 
for any  integrable function $g: \bb R\mapsto \bb R_+$ satisfying $g \geq_{\ee} f$, 
there exist constants $c, c_{\ee} >0$ such that for any $t \in \bb R$,  
%%a sequence $r_{m} =  r_{m}(\ee) \to 0$ as $m \to \infty,$ such that, for any $y\in \bb R_+$,
%\begin{align} \label{JJJJJ-1111-001}
% \mathbb E   f ( t+S_{m}) \leq 
%(1+8\ee) \sum_{l\in \bb Z} f_l  \int_{\mathbb R} I_l^{\ee}(s)  
%\frac{1}{ \sigma \sqrt{m} }\phi \left( \frac{t-s}{ \sigma \sqrt{m}}\right) ds 
%    +   \frac{c_{\ee}}{m^{(1+\delta) /2 }} \sum_{l\in \bb Z} f_l   \left\Vert  I_l^{\ee}(s) \right\Vert _{1},
%\end{align}
\begin{align} \label{JJJJJ-1111-001}
 \mathbb E   f ( t+S_{m}) \leq 
(1 + c\ee)  \int_{\mathbb R} g(s)  
\frac{1}{ \sigma \sqrt{m} }\phi \left( \frac{t-s}{ \sigma \sqrt{m}}\right) ds 
    +   \frac{c_{\ee}}{n^{1/2 + \ee }}   \left\Vert  g \right\Vert _{1}. 
\end{align}
From \eqref{JJJJJ-1111-001} and the bound \eqref{CLLT-bound only 002}, % of Theorem \ref{Theor-IntegrLimTh},  
we get that for any $x \in \mathbb R_+$,
\begin{align} \label{JJJ004}
I_{n}(x)    
& \leq  (1+ c\ee) J_n(x) 
   +  \frac{c_{\ee}}{n^{ 1/2 + \ee }}   \left\Vert  g \right\Vert _{1} \mathbb{P}\left(\tau_{x}>k\right)  \notag\\
&  \leq  (1+ c\ee)  J_n(x) +  c_{\ee} \frac{ V(x) }{ n^{ 1 + \ee } }  \left\Vert g \right\Vert_{1},  
\end{align}
where for brevity we set
\begin{align} \label{JJJ006}
J_n(x) 
: & = \int_{\bb R _{+}}   
\left[ \int_{\mathbb R} g(s)  
   \frac{1}{\sigma \sqrt{m} }\phi \left( \frac{t-s}{\sigma \sqrt{m}}\right) ds  \right] 
   \mathbb{P}\left( x+S_{k}\in dt, \tau_{x}>k\right).  %  \notag\\
\end{align}
By a change of variable, it follows that 
\begin{align*} 
J_n(x) 
& = \int_{\bb R _{+}}  \left[ \int_{\mathbb R} g(s)  \frac{1}{\sigma \sqrt{m} }
   \phi \left( \frac{t \sigma \sqrt{k}-s}{\sigma \sqrt{m}}\right) ds   \right]
      \mathbb{P}\left( \frac{x+S_{k}}{ \sigma \sqrt{k}}\in dt, \tau_{x}>k\right) \notag\\
%&= \int_{\bb R _{+}}   \int_{\mathbb R} g(\sigma \sqrt{k} s)  
%  \frac{1}{ \sqrt{m/k} }\phi \left( \frac{t-s}{  \sqrt{m/k}}\right) ds   
%   \mathbb{P}\left( \frac{x+S_{k}}{ \sigma \sqrt{k}}\in dt, \tau_{x}>k\right)   \notag\\
&= \int_{\bb R _{+}}  \varphi_n(t)  \mathbb{P}\left(  \frac{x+S_{k}}{\sigma \sqrt{k}}\in dt, \tau_{x}>k\right), 
%\notag\\ \varphi_n(y) &:= \int_{\mathbb R} g(\sqrt{k} s) \frac{1}{\sqrt{m/k}}\phi \left( \frac{y-s}{\sqrt{m/k}}\right) ds.
\end{align*}
where 
\begin{align} \label{JJJ006b}
 \varphi_n(t) &:= \int_{\mathbb R} g(\sigma \sqrt{k} s) \frac{1}{\sqrt{m/k}}\phi \left( \frac{t-s}{\sqrt{m/k}}\right) ds.
\end{align}
Using integration by parts and the conditioned integral limit theorem (Corollary \ref{Cor-CCLT-Optimal}),  
we claim that uniformly in $x\in [0, \alpha_n \sqrt{n}]$, 
%, for any $x$ satisfying $0\leq \frac{x}{\sqrt{n}} \leq  n^{-\ee},$
\begin{align} \label{New-ApplCondLT-002}
& \left| J_{n}(x)  -   \frac{2V(x)}{ \sigma \sqrt{2\pi k}}  \int_{\bb R _{+}}  \varphi_n(t) \phi^+(t) dt \right| 
 \leq  c_{\ee}  \left(  \alpha_n  + n^{-\ee} \right) \frac{V(x)}{ n } \left\|  g \right\|_1.
\end{align}
 Indeed, since the function $t \mapsto \varphi_n(t)$ is differentiable on $\bb R$ and vanishes as $t \to \pm \infty$, 
using integration by parts, we get that for any $x \in \mathbb R_+$,
\begin{align} \label{ApplCondLT-001}
J_{n}(x)
& = \int_{\bb R _{+}}  \varphi'_n(t) \ \mathbb{P}\left(  \frac{x+S_{k}}{\sigma  \sqrt{k}} > t, \tau_{x}>k\right)  dt.  
 \end{align}
Applying the conditioned integral limit theorem (Corollary \ref{Cor-CCLT-Optimal}) gives that uniformly in $x\in [0, \alpha_n \sqrt{n}]$, 
\begin{align} \label{C-th-ee001}
 \left|  \mathbb{P} \left(  \frac{x+S_{k}}{\sigma \sqrt{k}} > t,  \tau_{x}>k\right) - 
\frac{2V(x)}{ \sigma \sqrt{2\pi k}} (1-\Phi^+(t))   \right|
 \leq  c_{\ee}  \left(  \alpha_n  + n^{-\ee} \right) \frac{V(x)}{ n^{1/2} }, 
\end{align}
which, together with \eqref{ApplCondLT-001}, implies that uniformly in $x\in [0, \alpha_n \sqrt{n}]$, 
\begin{align} \label{ApplCondLT-002}
& \left| J_{n}(x)  -   \frac{2V(x)}{  \sigma \sqrt{2\pi k}}  \int_{\bb R _{+}}  \varphi'_n(t) (1- \Phi^+(t))dt \right|   \notag\\
 &  \leq  c_{\ee}  \left(  \alpha_n  + n^{-\ee} \right) \frac{V(x)}{ n^{1/2} } \int_{\bb R _{+}}  | \varphi'_n(t) | dt.
\end{align}
From \eqref{JJJ006b} and a change of variable, it is easy to see that
%%$$
%% \varphi_n(y) := \int_{\mathbb R} g(\sqrt{k} s) \frac{1}{\sqrt{m/k}}\phi \left( \frac{y-s}{\sqrt{m/k}}\right) ds.
%%$$
%\begin{align*} %\label{}
% \varphi'_n(t) 
% = \int_{\mathbb R} g(\sigma \sqrt{k} s) \frac{1}{\sqrt{m/k}}\phi' \left( \frac{t-s}{\sqrt{m/k}}\right) \frac{ds}{\sqrt{m/k}}.
% \end{align*} %\notag\\
%Then,
\begin{align} \label{JJJ-20001}
%\Vert \varphi'_n \Vert_{1} %&= \int_{\bb R _{+}}  | \varphi'_n(y) | dy \notag\\
\int_{\bb R _{+}}  | \varphi'_n(t) | dt 
& \leq  \int_{\bb R _{+}}  \left[ \int_{\mathbb R}  g (\sigma \sqrt{k} s)  
   \left| \phi'\left( \frac{ t -s }{\sqrt{m/k}}\right) \right|  \frac{ds}{\sqrt{m/k}}  \right]
\frac{dt}{\sqrt{m/k}} \notag\\
& = \int_{\bb R _{+}}  \left[ \int_{\mathbb R}  g (\sigma \sqrt{m} s)   \left| \phi'\left( t -s\right) \right| ds \right] dt  
 \leq \frac{c}{\sqrt{m}} \Vert g \Vert_{1}. 
\end{align}
Using integration by parts, we have 
\begin{align} \label{inegr by parts-999-01}
\int_{\bb R _{+}}  \varphi'_n(t) (1- \Phi^+(t))dt
=\int_{\bb R _{+}}  \varphi_n(t) \phi^+ ( t) dt.
  %\notag\\
\end{align}
Putting together %\eqref{ApplCondLT-001}, 
\eqref{ApplCondLT-002}, \eqref{JJJ-20001} and \eqref{inegr by parts-999-01}, 
we obtain \eqref{New-ApplCondLT-002}.
%it follows that,
%for $0\leq x\leq \alpha \sqrt{n},$
%\begin{align} \label{Integ_3}
%J_{n}(x) \leq \frac{2V(x)}{\sqrt{2\pi k }} \int_{\bb R _{+}}  h_m(y) \phi^+ \left( \frac{y}{\sqrt{k}} \right) \frac{dy}{\sqrt{k}} 
%+  c\| g \|_{1} \frac{1+x}{k^{1/2}}  \left( \frac{1}{k^{\delta_0}}+\alpha^2  \right).
%\end{align}
%%%%%%%%%%%%%%%%%%%%%%%%%%%%%%%%%%
%========here we use the following bound
%\begin{align} \label{JJJ004xxxxxxxxx}
%\sqrt{m} I_n(x) %\mathbf{Q}^{n} g(x) 
%& \leq \int_{\bb R _{+}}  h_m(y)  \mathbb{P}\left(  x+S_{k}\in dy, \tau_{x}>k\right) \notag \\  
%& \quad + r_{m}(\ee)  \left\Vert g_{\ee} \right\Vert_{1} 
%\mathbb{P}\left(\tau_{x}>k\right).
%\end{align}
%
%================

Combining \eqref{New-ApplCondLT-002} with \eqref{JJJ004}, 
we get that uniformly in $x\in [0, \alpha_n \sqrt{n}]$, 
\begin{align}
\label{JJJ201aaa}
I_n(x) 
%& \leq  \frac{2V(x)}{\sqrt{2\pi k}}  \int_{\bb R _{+}}  \varphi_m(y) \phi^+(y) dy
%%  \frac{2V(x)}{\sqrt{2\pi }} H_{n} 
%+   \frac{c(1+x)}{\sqrt{k m}}  \left( \frac{1}{k^{\delta_0}}+\alpha_n  \right) \| g \|_{1} \notag\\
%& \quad +  \frac{c \mathbb{P}\left(\tau_{x}>k\right) }{\sqrt{m}} (r_{m} (\ee) + \ee) 
%    \left\Vert g \right\Vert_{1}     \nonumber\\
& \leq  (1 + c \ee)   \frac{2V(x)}{ \sigma  \sqrt{2\pi k}}  \int_{\bb R _{+}}  \varphi_n(t) \phi^+(t) dt 
 + c_{\ee}  \left(  \alpha_n  + n^{-\ee} \right) \frac{V(x)}{ n } \left\|  g \right\|_1. 
 %+ \frac{c_{\gamma} V(x)}{\sqrt{k m}}   \left( \frac{1}{k^{\gamma}} +  \frac{1}{m^{\delta/2}}  \right) \left\|  g \right\|_1. 
\end{align}
%where
%\begin{align*}
%H_{n} = 
%\int_{\bb R _{+}}  h_m(y) \phi^+ \left( \frac{y}{\sqrt{k}} \right) \frac{dy}{k}.  
%\end{align*}
Denote $\delta_n=\frac{m}{n}$. 
By the definition of $\varphi_n$ (see \eqref{JJJ006b}), a change of variable and Fubini's theorem,
we derive that
\begin{align*}
& \int_{\bb R _{+}}  \varphi_n(t) \phi^+(t) dt \notag\\
&= \int_{\bb R _{+}}  \left[ \int_{\mathbb R} g(\sigma \sqrt{k} s) \frac{1}{\sqrt{m/k}}
     \phi \left( \frac{s-t}{\sqrt{m/k}}\right) ds  \right]  \phi^+(t) dt  \notag\\
&= \int_{\bb R _{+}}  \left[ \int_{\mathbb R} g(\sigma \sqrt{n}s' ) 
 \frac{1}{\sqrt{m/k}}\phi \left( \frac{s'-t'}{\sqrt{m/n}}\right) \frac{ds'}{\sqrt{k/n}}  \right]
    \phi^+ \left( \frac{t'}{\sqrt{k/n}} \right) \frac{dt'}{\sqrt{k/n}}  \notag\\
%&= \int_{\bb R _{+}}   g(\sqrt{n}s' )  \int_{\mathbb R}  \frac{1}{\sqrt{\delta_n}}\phi \left( \frac{s'-y'}{\sqrt{\delta_n}}\right)   
%\phi^+(\frac{y'}{\sqrt{1-\delta_n}}) 
%\frac{dy'} {\sqrt{1-\delta_n}} ds' \notag\\
& =  \int_{\mathbb R} g(\sigma \sqrt{n} s')  \phi_{\delta_n}*\phi^+_{1-\delta_n}(s') ds' 
 = \frac{1}{ \sigma \sqrt{n}}
\int_{\mathbb R} g( t) \phi_{\delta_n}*\phi^+_{1-\delta_n}\left(\frac{t}{ \sigma \sqrt{n}} \right) dt \notag\\
& \leq \frac{\sqrt{k}}{\sigma n}  \int_{\mathbb R} g( t) \phi^+\left(\frac{t}{\sigma \sqrt{n}} \right) dt
   +  \frac{ \sqrt{m}}{\sigma n} \int_{\mathbb R} g( t)  e^{- \frac{t^2}{2 \sigma^2 m}} dt,
\end{align*}
where for the last line we applied the upper bound in Lemma \ref{t-Aux lemma} with $v = \delta_n$ 
and the fact that $1-\delta_n=\frac{k}{n}$. 
% Using Lemma \ref{t-Aux lemma}, we obtain
%\begin{align*}
%\left| \int_{\bb R _{+}}  \varphi_m(y) \phi^+(y) dy -  \frac{1}{\sqrt{n}} \int_{\bb R}  g(t) \phi^+ \left( \frac{t}{\sqrt{n}} \right) dt  \right|
%\leq c \frac{\sqrt{m}}{n} \| g \|_1. 
%\end{align*}
Substituting this bound into \eqref{JJJ201aaa},
% and using the fact that $1-\delta_n=\frac{k}{n}$, %$V(x) \leq c(1+ x)$,
we get %that for any $x \in \mathbb R_+$,
\begin{align*}
&   I_{n}(x)  -  (1 + c \ee)  \frac{2V(x)}{\sqrt{2 \pi} \sigma^2 n}   \int_{\bb R}  g(t) \phi^+ \left( \frac{t}{\sigma \sqrt{n}} \right) dt 
   \notag\\
 & \qquad\qquad  \leq   c \ee^{1/4} \frac{V(x)}{n} \int_{\mathbb R} g( t)  e^{- \frac{t^2}{2 \sigma^2 \ee^{1/2} n}} dt   
+  c_{\ee}  \left(  \alpha_n  + n^{-\ee} \right) \frac{V(x)}{ n } \left\|  g \right\|_1, 
 %\frac{c_{\gamma} V(x)}{\sqrt{k m}}   \left( \frac{1}{k^{\gamma}} +  \frac{1}{m^{\delta/2}}  \right) \left\|  g \right\|_1. 
\end{align*} 
%Therefore, we obtain the following upper bound: 
%\begin{align*}
%& I_{n}(x)  -   (1 + c \ee)   \frac{2V(x)}{\sqrt{2 \pi} \sigma^2  n} 
%    \int_{\bb R}  g (t) \phi^+ \left( \frac{t}{\sigma \sqrt{n}} \right) dt    \notag\\
%& \qquad\qquad   \leq   c  \ee^{- 1/4} \frac{ V(x) }{n } 
% \left( \frac{ 1}{n^{\gamma}}  + \frac{\ee^{ - \delta/4}}{ n^{\delta/2}}    \right) \left\|  g \right\|_1, 
%\end{align*}
%for some sequence $(r_n''(\ee))_{n \geq 1}$ satisfying $\lim_{n \to \infty} r_n''(\ee) = 0$. 
which concludes the proof of the upper bound \eqref{eqt-B 001}.  
%\end{proof}

%%%%%%%%%%%%%%%%%%%%%%%%%%%%%%%%%%%%%%%%%%%%
%%%%%%%%%%%%%%%%%%%%%%%%%%%%%%%%%%%%%%%%%%%%
\subsection{Proof of the lower bound} \label{sec: proof lower bound}
%The goal of this section is to 
In this section we establish the lower bound \eqref{eqt-B 002}. %  of Theorem \ref{t-B 002}.
%whose proof turns out to be more delicate than that of the upper bound \eqref{eqt-B 001}.  

%\begin{proof}[Proof of \eqref{eqt-B 002}]
Let us keep the notation used in the proof of the upper bound \eqref{eqt-B 001}. 
By the Markov property we have
%Let us keep the notations of the proof of the upper bound \eqref{eqt-B 001}:
\begin{align} \label{staring point for the lower-bound-001} 
I_n(x) %: & =  \mathbb{E} \left( f(x+S_n ); \tau_x >n\right) \notag \\  
& = \int_{\mathbb R_+}  \mathbb E   f ( t +S_{m}) \mathbb{P}\left( x+S_{k}\in dt,  \tau_x >k\right)   \notag\\
 & \qquad -  \int_{\mathbb R_+}  \mathbb{E} \left( f(t+S_{m}); \tau _{t} \leq m\right)  
     \mathbb{P}\left( x+S_{k}\in dt, \tau_x >k\right)  \notag\\
& = : I_{n,1}(x) - I_{n,2}(x). 
\end{align}
%We proceed in the same way as in the proof of \eqref{eqt-B 001}. 

\textit{Lower bound of $I_{n,1}(x)$. }
Using the local limit theorem \eqref{LLT-general002} leads to 
\begin{align}\label{Pf_SmallStarting_Firstthm}
\mathbb{E}f\left( t + S_{m} \right) 
& \geq  \frac{ 1}{\sigma \sqrt{m} } \int_{\bb R }h (s) \phi \left( \frac{s - t}{\sigma \sqrt{m}}\right) ds   \notag\\
&  \quad  -  \frac{ c \ee }{\sigma \sqrt{m}}   \int_{\bb R }f (s) \phi \left( \frac{s - t}{\sigma \sqrt{m}}\right) ds
 -  \frac{c_{\ee}}{n^{ 1/2 + \ee }}  \left\Vert f \right\Vert_1. 
\end{align}
Proceeding in the same way as in the estimate of $J_n(x)$ defined by \eqref{JJJ006}, %\eqref{eqt-B 001} 
(using the lower bound in \eqref{SmoothingBoundsRayleigh} instead of the upper one, 
so that the second term on the right hand side of \eqref{eqt-B 001} does not appear), 
one has, uniformly in $x\in [0, \alpha_n \sqrt{n}]$,  
\begin{align*}%\label{}
& \frac{ 1 }{\sigma \sqrt{m} } \int_{\mathbb R_+}  \left[ \int_{\bb R }h (s) \phi \left( \frac{s - t}{\sigma \sqrt{m}}\right) ds  \right]
   \mathbb{P}\left( x+S_{k}\in dt,  \tau_x >k\right)    \notag\\
& \geq  \frac{2V(x)}{\sqrt{2 \pi} \sigma^2  n} 
    \int_{\bb R}  h (t) \phi^+ \left( \frac{t}{\sigma \sqrt{n}} \right) dt 
    -  c_{\ee}  \left(  \alpha_n  + n^{-\ee} \right) \frac{V(x)}{ n }  \left\|  h \right\|_1 
\end{align*}
and  
\begin{align*}%\label{}
&  \frac{ 1 }{\sigma \sqrt{m} }  
 \int_{\mathbb R_+}  \left[ \int_{\bb R } f (s) \phi \left( \frac{s - t}{\sigma \sqrt{m}}\right) ds  \right] 
   \mathbb{P}\left( x+S_{k}\in dt,  \tau_x >k\right)    \notag\\
& \leq  \frac{2V(x)}{\sqrt{2 \pi} \sigma^2  n} 
    \int_{\bb R}  f (t) \phi^+ \left( \frac{t}{\sigma \sqrt{n}} \right) dt 
    -  c_{\ee}  \left(  \alpha_n  + n^{-\ee} \right) \frac{V(x)}{ n }  \left\|  f \right\|_1. 
\end{align*}
By \eqref{CLLT-bound only 002}, 
these bounds together with \eqref{Pf_SmallStarting_Firstthm}  yield that uniformly in $x\in [0, \alpha_n \sqrt{n}]$, 
% the following lower bound for $I_{n,1}(x)$: 
\begin{align}\label{Lower_In_lll}
 I_{n,1}(x)  
 & \geq  \frac{2V(x)}{\sqrt{2 \pi} \sigma^2  n} 
    \int_{\bb R} \Big[  h(t) - c \ee f (t) \Big] \phi^+ \left( \frac{t}{\sigma \sqrt{n}} \right) dt   \notag\\
 & \quad   -  c_{\ee}  \left(  \alpha_n  + n^{-\ee} \right) \frac{V(x)}{ n }   \left\|  f \right\|_1. 
     %\notag\\
% & \geq  \frac{2V(x)}{\sqrt{2 \pi} \sigma^2  n} 
%    \int_{\bb R}   h(t)  \phi^+ \left( \frac{t}{\sigma \sqrt{n}} \right) dt  
%     - 3 \ee  \frac{2V(x)}{\sqrt{2 \pi} \sigma^2  n} 
%    \int_{\bb R}   f(t)  \phi^+ \left( \frac{t}{\sigma \sqrt{n}} \right) dt   \notag\\
%& \quad    -  c_{\ee} \frac{ V(x) }{n^{1+\ee} }  \left\|  f \right\|_1. 
\end{align}
%%From \eqref {Decompose_In}  and  \eqref{Upperboun_In1_aa}, 
%we get that there exists a sequence $(r_{n}')$ satisfying $\lim_{n \to \infty} r_{n}' = 0$ such that 
%\begin{align}\label{Lower_In_lll}
%&  I_{n}(x)  - (1 -\ee) \frac{2V(x)}{\sqrt{2 \pi}n} 
%    \int_{\bb R}  h (t) \phi^+ \left( \frac{t}{\sqrt{n}} \right) dt   \notag\\
%& \geq - I_{n,2}(x) 
%-   2 \ee    \int_{\bb R }f (t) \phi \left( \frac{t}{\sqrt{n}}\right) dt
% -  \frac{c (1 + x)}{n}  \left(   r_{n}'(\ee) + \ee^{1/4} \right)  \left\|  f \right\|_1. 
%\end{align}

\textit{Upper bound of $I_{n,2}(x)$. }
%It remains to deal with the second term $I_{n,2}(x)$ in \eqref{Decompose_In}. 
Splitting the integral in this term into two parts according to 
whether the value of $t$ is less or larger than $\ee^{1/6} \sqrt{n}$, %: for $x \in \bb R_+$, 
we have 
\begin{align} \label{KKK-111-001}
I_{n,2}(x) & =\int_{\mathbb R_+} \left[ \int_{\mathbb R}  f(u) 
\mathbb{P}\left( t+S_{m}\in du, \tau_t \leq m  \right)  \right]
\mathbb{P}\left( x+S_{k}\in dt, \tau_x > k  \right) \notag \\
&= K_1 + K_2,
\end{align}
where
\begin{align*} %\label{}
K_1&=\int_{0}^{ \ee^{1/6}\sqrt{n}} 
 \left[ \int_{\mathbb R}  f(u)  \mathbb{P}\left( t+S_{m}\in du, \tau_t \leq m  \right)  \right]
\mathbb{P}\left( x+S_{k}\in dt,  \tau_x > k  \right), \\
K_2&=\int_{\ee^{1/6}\sqrt{n}}^\infty 
\left[ \int_{\mathbb R}  f(u) \mathbb{P}\left( t+S_{m}\in du, \tau_t \leq m  \right)  \right]
\mathbb{P}\left( x+S_{k}\in dt,  \tau_x > k  \right).
\end{align*}
For $K_1$, we use the local limit theorem (Theorem \ref{LLT-general}) 
and the bound \eqref{CLLT-bound only 002} of Theorem \ref{Theor-IntegrLimTh} to get 
\begin{align}\label{CaraOrder12BoundK1}
K_1 
&  \leq  \int_{0}^{\ee^{1/6}\sqrt{ n}} 
\left[ \int_{\mathbb R}  f(u) \mathbb{P}\left( t+S_{m}\in du  \right)    \right]
\mathbb{P}\left( x+S_{k}\in dt,   \tau_x > k  \right)   \notag\\
& \leq  K_{11}   + \frac{c_{\ee}}{n^{ 1/2 + \ee}}  \left\Vert g\right\Vert_1   \mathbb{P} \left( \tau_x > k  \right)  \notag\\
%\int_{0}^{\ee^{1/6}\sqrt{n}}  \mathbb{P}\left( x+S_{k}\in dy ;  \tau_x > k  \right)   \notag\\
& \leq K_{11} + c_{\ee} \frac{ V(x)}{ n^{ 1 + \ee} }  \left\Vert g \right\Vert_1,   
%& \leq \frac{c}{\sqrt{m}}  (1+ r_{n}(\ee) ) \left\Vert g \right\Vert_1 
%\mathbb{P}\left( \frac{x+S_{k}}{\sqrt{k}}\leq \ee^{1/6} \sqrt{\frac{n}{k}},  \tau_x > k  \right).
%% \\&\leq \frac{c}{\sqrt{m}} \int_0^{\frac{\delta}{\sigma} \sqrt{\frac{n}{k}}}   \phi^+(t) dt + \\
\end{align}
where 
\begin{align*}%\label{}
K_{11} = \int_{0}^{\ee^{1/6}\sqrt{ n}}  \left[ \int_{\bb R} g(u)  \frac{1}{\sigma \sqrt{m}} 
    \phi \left( \frac{u-t}{\sigma \sqrt{m}} \right) du \right] 
    \mathbb{P}\left( x+S_{k}\in dt,   \tau_x > k  \right). 
\end{align*}
By Fubini's theroem, we have $K_{11} =  \int_{\bb R} g(u)  J (u)   du$, 
%\begin{align*}%\label{}
%K_{11} 
%& =  \int_{\bb R} g(u)  J (u)   du,  
%\end{align*}
where %, for $u \in \bb R$, 
\begin{align*}%\label{}
J (u) =  \int_{0}^{\ee^{1/6}\sqrt{ n}}  F_u (t)
    \mathbb{P}\left( x+S_{k}\in dt ;   \tau_x > k  \right),  \quad  u \in \bb R, 
\end{align*}
and
\begin{align*}%\label{}
F_u (t) =  \frac{1}{\sigma \sqrt{m}} \phi \left( \frac{u-t}{\sigma \sqrt{m}}   \right),
\quad  t \in [0, \ee^{1/6}\sqrt{ n}].  
%   \mathds 1_{\{ 0 \leq y \leq  \ee^{1/6}\sqrt{ n} \} }. 
\end{align*}
%so that $F'(y) = - \frac{1}{ m } \phi' \left( \frac{t-y}{\sqrt{m}}   \right)$. 
%\begin{align*}%\label{}
%& \int_{\bb R_+}   F(y)  \mathbb{P}\left( x+S_{k}\in dy ;  x  +S_{k} \in [0, \ee^{1/6}\sqrt{ n}],  \tau_x > k  \right)  \notag\\
%& = \bb E  \left( F(x+S_{k}) ;  x  +S_{k} \in [0, \ee^{1/6}\sqrt{ n}],  \tau_x > k  \right)  \notag\\
%& =   \bb E   \left( \int_{- \infty}^{x + S_k}  F'(t)  \mathds 1_{\{ x  +S_{k} \in [0, \ee^{1/6}\sqrt{ n}],  \tau_x > k \}}  dt \right)   \notag\\
%& =   \bb E   \left( \int_{\bb R}  F'(t)  \mathds 1_{\{  x + S_k > t, x  +S_{k} \in [0, \ee^{1/6}\sqrt{ n}],  \tau_x > k \}}  dt \right)   \notag\\
%& =  \int_{\bb R}  F'(t)    \bb P  \left(  x + S_k > t, x  +S_{k} \in [0, \ee^{1/6}\sqrt{ n}],  \tau_x > k \right)  dt 
%\end{align*}
Then
\begin{align}\label{CaraOrder12BoundK1aa}
J (u)
%  & =  \int_{0}^{\ee^{1/6}\sqrt{ n}}  F_u (t)
%    \mathbb{P}\left( x+S_{k}\in dt ;   \tau_x > k  \right)   \notag\\
& =  \int_{0}^{\infty}  F_u (t)
    \mathbb{P}\left( x+S_{k}\in dt,  x + S_k \leq \ee^{1/6}\sqrt{ n},  \tau_x > k  \right)    \notag\\
% & =  \int_{0}^{\infty}   F(y)  \mathbb{P}\left( x+S_{k}\in dy ;  x + S_k \leq \ee^{1/6}\sqrt{ n},   \tau_x > k  \right)  \notag\\
& =  \int_{0}^{\infty}  F_u' (t)  \bb P  \left( x + S_k \in [0, \ee^{1/6}\sqrt{ n}], x + S_k > t, \tau_x > k \right)  dt   \notag\\
%& =  \int_{0}^{\infty}   F_u' (t)  
%  \bb P  \left( \frac{x + S_k}{\sigma \sqrt{k}} \in \left[ \frac{t}{\sigma \sqrt{k}}, \frac{\ee^{1/6}}{\sigma} \sqrt{ \frac{n}{k} } \right], 
%  \tau_x > k \right)  dt  \notag\\
& =  \int_{0}^{ \ee^{1/6}\sqrt{ n }  }  
F_u' (t)  \bb P  \left( \frac{x + S_k}{\sigma \sqrt{k}} 
    \in \left[ \frac{t}{\sigma \sqrt{k}}, \frac{\ee^{1/6} \sqrt{n}}{\sigma \sqrt{k}}  \right], 
  \tau_x > k \right)  dt.  
%& =  \int_{0}^{\ee^{1/6}\sqrt{ n}}  F'(u)  \bb P  \left( x + S_k \in (u, \ee^{1/6}\sqrt{ n}],  \tau_x > k \right)  du  \notag\\
%& \leq    \int_{0}^{\infty}  F'(u)    du   \bb P  \left( x + S_k \leq \ee^{1/6}\sqrt{ n},  \tau_x > k \right).  ???
%& =   \int_{0}^{\ee^{1/6}\sqrt{ n}}  F'(u)  
%   \bb P  \left(\frac{ x + S_k }{\sqrt{k} }  >  \frac{u}{\sqrt{k}},  \tau_x > k \right)  du.  
\end{align}
Applying the conditioned integral limit theorem (Corollary \ref{Cor-CCLT-Optimal}), 
we get that uniformly in $t \in \bb R_+$ and $x\in [0, \alpha_n \sqrt{n}]$, 
\begin{align*}%\label{}
& \Bigg| \bb P  \left(\frac{ x + S_k }{\sigma \sqrt{k} }  \in  
  \left[ \frac{t}{\sigma\sqrt{k}},  \frac{\ee^{1/6} \sqrt{n}}{\sigma \sqrt{k}} \right],  \tau_x > k \right)
    \notag\\
& \qquad  -  \frac{2V(x)}{ \sigma \sqrt{2\pi k}}  
    \left[ \Phi^+ \left( \frac{\ee^{1/6} \sqrt{n}}{\sigma \sqrt{k}}  \right)   - \Phi^+ \left( \frac{t}{\sigma \sqrt{k}} \right)  \right]   \Bigg|   
     \leq  c_{\ee} \left(  \alpha_n  + n^{-\ee} \right) \frac{V(x)}{n^{1/2}}.
\end{align*}
Hence, using the fact that $F_u (\ee^{1/6}\sqrt{ n }) - F_u (0) \leq \frac{1}{\sqrt{m}} \leq \frac{c_{\ee} }{\sqrt{n}}$ gives
\begin{align}\label{Bound-Ju8}
J (u)  
& \leq   \frac{2V(x)}{\sigma \sqrt{2\pi k}}     \int_{0}^{ \ee^{1/6}\sqrt{ n }  }  F_u' (t)
\left[ \Phi^+ \left( \frac{\ee^{1/6} \sqrt{n}}{\sigma \sqrt{k}}  \right)   - \Phi^+ \left( \frac{t}{\sigma \sqrt{k}} \right)  \right] dt  \notag\\
& \quad  + c_{\ee} \left(  \alpha_n  + n^{-\ee} \right) \frac{V(x)}{n}.
%   +   c_{\ee} \frac{V(x)}{n^{1 + \ee}}. 
\end{align}
Using integration by parts and the fact that $F_u(0) \geq 0$,  we have
\begin{align}\label{Bound-Ju9}
& \int_{0}^{ \ee^{1/6}\sqrt{ n }  }  F_u' (t)
\left[ \Phi^+ \left( \frac{\ee^{1/6} \sqrt{n}}{\sigma \sqrt{k}}  \right)   - \Phi^+ \left( \frac{t}{\sigma \sqrt{k}} \right)  \right] dt  \notag\\
& = -  F_u(0) \Phi^+ \left( \frac{\ee^{1/6} \sqrt{n}}{\sigma \sqrt{k}}  \right)  
   +  \frac{1}{\sigma \sqrt{k}} \int_{0}^{ \ee^{1/6}\sqrt{ n }  }  F_u \left(t\right) \phi^+ \left( \frac{t}{\sigma \sqrt{k}} \right)  dt  \notag\\
& \leq  \frac{1}{\sigma \sqrt{k}} \int_{0}^{ \ee^{1/6}\sqrt{ n }  }   F_u \left(t\right) \phi^+ \left( \frac{t}{\sigma \sqrt{k}} \right)  dt
  = :  H(u). 
\end{align}
Elementary calculations give
\begin{align*}%\label{}
H(u)  
& =  \frac{1}{\sigma \sqrt{k}} \int_{0}^{ \ee^{1/6}\sqrt{ n} } \frac{1}{\sigma \sqrt{2 \pi m}} 
   e^{- \frac{(u-t)^2}{2 \sigma^2 m}} \frac{t}{\sigma \sqrt{k}} e^{- \frac{t^2}{2 \sigma^2 k}} dt   \notag\\
%& =  e^{- \frac{u^2}{2 \sigma^2 m}}  \int_{0}^{ \ee^{1/6}\sqrt{ n} }   \frac{1}{ \sigma^2 \sqrt{2 \pi m}}  \frac{t}{\sqrt{k}}
%     e^{- \frac{t^2}{2 \sigma^2 m} - \frac{t^2}{2 \sigma^2  k} + \frac{tu}{\sigma^2 m}}  dt   \notag\\
%& =  \frac{1}{\sigma \sqrt{k}}  e^{- \frac{u^2}{2 \sigma^2 m}}  
%  \int_{0}^{ \ee^{1/6}\sqrt{ n} }   \frac{1}{\sigma^2 \sqrt{2 \pi m}}  \frac{t}{\sqrt{k}}
%     e^{- \frac{n t^2}{2 \sigma^2 mk }  + \frac{tu}{\sigma^2 m}}  dt   \notag\\
%& =  e^{- \frac{u^2}{2 \sigma^2 m}}  e^{ \frac{k u^2}{2 \sigma^2 mn}}  
%  \int_{0}^{ \ee^{1/6}\sqrt{ n} }   \frac{1}{\sigma^2 \sqrt{2 \pi m}}  \frac{t}{\sqrt{k}}
%    e^{ - \frac{n}{2 \sigma^2 mk} (t - \frac{ku}{n})^2 } dt  \notag\\
 & = \frac{1}{\sigma \sqrt{k}} e^{- \frac{u^2}{2 \sigma^2 n}}  
   \int_{0}^{ \ee^{1/6}\sqrt{ n} }   \frac{1}{\sigma^2 \sqrt{2 \pi m}}  \frac{t}{\sqrt{k}}
    e^{ - \frac{n}{2 \sigma^2 mk} (t - \frac{ku}{n})^2 }  dt. 
\end{align*}
By a change of variable $t \sqrt{\frac{n}{\sigma^2 mk}}  = z$, it follows that
\begin{align*}%\label{}
%&  e^{- \frac{u^2}{2 \sigma^2 n}}  \int_{0}^{ \ee^{1/6}\sqrt{ n} }   \frac{1}{\sigma^2 \sqrt{2 \pi m}}  \frac{t}{\sqrt{k}}
%    e^{ - \frac{n}{2 \sigma^2 mk} (t - \frac{ku}{n})^2 }  dt   \notag\\
H(u)   & =  \frac{1}{\sigma \sqrt{k}}  e^{- \frac{u^2}{2 \sigma^2 n}}    \int_{0}^{ \ee^{1/6} n/\sqrt{\sigma^2 mk} }  
    \frac{z}{\sigma^2 \sqrt{2 \pi n}}   \sqrt{\frac{\sigma^2  mk}{n}}  
    e^{ - \frac{1}{2} (z- \sqrt{\frac{k}{\sigma^2 nm}} u)^2 }  dz   \notag\\
 & \leq  \frac{1}{\sigma \sqrt{k}}  e^{- \frac{u^2}{2 \sigma^2 n}}   \int_{0}^{ \ee^{1/6} n/\sqrt{\sigma^2 mk} }   
   \frac{z}{\sigma^2 \sqrt{2 \pi n}}   \sqrt{\frac{\sigma^2 mk}{n}}   dz  \notag\\
 & \leq  \frac{c \ee^{1/3}}{\sqrt{k}} \frac{n}{\sqrt{mk}} e^{- \frac{u^2}{2\sigma^2 n}}    
   \leq  \frac{c \ee^{1/12}}{\sqrt{k}}  \phi \left( \frac{u}{ \sigma \sqrt{n}} \right), 
\end{align*}
which, together with \eqref{Bound-Ju8} and \eqref{Bound-Ju9}, implies that
\begin{align*}%\label{}
J (u)  
%& \leq  c \frac{2V(x)}{\sigma \sqrt{2\pi k}}   \frac{\ee^{1/12}}{\sqrt{k}}   e^{- \frac{u^2}{2 \sigma^2 n}}  
%    +   c_{\gamma} \frac{V(x)}{k^{1/2 + \gamma}}  \int_{0}^{ \ee^{1/6}\sqrt{ n }  }   F_u'(t)  dt  \notag\\
& \leq  c \ee^{1/12} \frac{ V(x)}{ n }   \phi \left( \frac{u}{ \sigma \sqrt{n}} \right)  
  +   c_{\ee} \left(  \alpha_n  + n^{-\ee} \right) \frac{V(x)}{n}.
\end{align*}
%where in the last inequality we used the fact that 
%$F_u (\ee^{1/6}\sqrt{ n }) - F_u (0) \leq \frac{1}{\sqrt{m}} \leq \frac{c_{\ee} }{\sqrt{n}}$. 
%As a result, 
Since $K_{11} =  \int_{\bb R} g(u)  J (u)   du$, we obtain 
\begin{align}\label{K1-final bound}
K_{11}  %=  \int_{\bb R} g(u)  J (u)   du  
\leq   c \ee^{1/12} \frac{ V(x)}{ n }   \int_{\bb R} g(t)  \phi \left( \frac{t}{ \sigma \sqrt{n}} \right)  dt  
  +  c_{\ee} \left(  \alpha_n  + n^{-\ee} \right) \frac{V(x)}{n}  \|g\|_1.  
\end{align}

%For $K_{12}$, using the conditioned central limit Theorem \ref{Theor-IntegrLimTh}, we get
%\begin{align} \label{K1-final bound}
%K_{12}
%&\leq  \frac{c_{\ee}}{n^{(1+\delta)/2}}  \left\Vert g\right\Vert_1
%\Bigg( \frac{2V(x)}{\sqrt{2\pi k}} \int_{0}^{\ee^{1/6} \sqrt{\frac{n}{k}}} \phi^+ \left(t\right) dt 
%  + c \frac{1+x}{k^{1/2 + \gamma}}  \Bigg) \notag \\
%&\leq  \frac{c_{\ee} (1+x)}{ n^{(2+\delta)/2} }  \left\Vert g \right\Vert_1
%\Bigg(\int_{0}^{\ee^{1/6} \sqrt{\frac{n}{k}}} \phi^+ \left(t\right) dt 
%  +   \frac{1}{ n^{\gamma}}   \Bigg) \notag \\
%& \leq  \frac{c_{\ee} (1+x)}{ n^{(2+\delta)/2} }  \left\Vert g \right\Vert_1. 
%%&\leq  \frac{c(1+x)}{n}   (1+ r_{n}(\ee) ) \left\Vert g \right\Vert_1
%%\frac{1}{\ee^{\frac{1}{4}}}\Bigg(\ee^{1/3}  +   \frac{1}{k^{\delta_0}} +\alpha_n   \Bigg)  \notag \\
%%&\leq  \frac{c(1+x)}{n}   \left\Vert g \right\Vert_1
%%\left(\ee^{\frac{1}{12}}  +  r_{n}'(\ee) \right), 
%\end{align}
%%for some sequence $r_{n}'(\ee) \to 0$ as $n \to \infty$. 

We proceed to give an upper bound for $K_{2}$, which can be rewritten as
\begin{align} %\label{}  
K_{2} = \int_{\bb R}  L(t)  \mathbb{P}\left( x+S_{k}\in dt, \tau_x > k  \right),  \label{K2-b01c-001}
\end{align}
where, for $t \in \bb R$, 
\begin{align} \label{Bytheduality-001}
L(t) : =  \mathds 1_{\{ t> \ee^{1/6}\sqrt{n}  \}}  \mathbb{E}  (f(t + S_{m}); \tau_{t} \leq m ). 
\end{align}
%Suppose that $L$ is differentiable, then 
%\begin{align*}%\label{}
%K_{2,2} 
%&= \int_{\bb R}  L(y)  \mathbb{P}\left( x+S_{k}\in dy ; \tau_x > k  \right)  \notag\\
%& =  \int_{\bb R}  L'(y)  \mathbb{P}\left( x+S_{k} > y ; \tau_x > k  \right) dy  \notag\\
%& =  \int_{\bb R}  L'(y)  G(y)  dy. 
%\end{align*}
%
%\begin{align*}%\label{}
%& c_{\delta} \int_{\bb R}  \mathbb{E}  (f(y + S_{m}); m_1 < \tau_{y} \leq m )  
%  \mathds 1_{y \in [\ee^{1/6}\sqrt{n} - \delta, \ee^{1/6}\sqrt{n} + \delta]} G(y)  dy   \notag\\
%& =   c_{\delta} \int_{\bb R}  f(x)  
%  \mathbb{E}  (\mathds 1_{x + S_{m} \in [\ee^{1/6}\sqrt{n} - \delta, \ee^{1/6}\sqrt{n} + \delta]} G(x + S_{m}); \tau_{x} \leq m_2 )  dx  \notag\\
%& \leq  c_{\delta}   
%  \int_{\bb R}  f(x)  \bb P \left(  x + S_{m} \in [\ee^{1/6}\sqrt{n} - \delta, \ee^{1/6}\sqrt{n} + \delta],  \tau_{x} \leq m_2 \right)  dx
%\end{align*}
%
%
%\begin{align*}%\label{}
%& \int_{\bb R}  
%   \mathbb{E}  (f'(y + S_{m}); m_1 < \tau_{y} \leq m )  \mathds 1_{\{ y> \ee^{1/6}\sqrt{n}  \}}  G(y)  dy    \notag\\
%& =   \int_{\bb R} f'(x) 
%   \mathbb{E}  ( \mathds 1_{\{ x + S_m > \ee^{1/6}\sqrt{n}  \}}  G(x + S_m) ;  \tau_{x} \leq m_2 )    dx   \notag\\
%& \leq   \int_{\bb R} f'(x)   \mathbb{P}  ( x + S_m > \ee^{1/6}\sqrt{n};  \tau_{x} \leq m_2 )    dx
%\end{align*}
%
%
%
%
The function $t \mapsto L(t)$ is integrable on $\bb R$ since  $f$ is integrable on $\bb R$.
%Recall that $g$ is an upper $\ee$-envelope of $f$ so that $f \leq_{\ee} g$. 
% be a measurable upper $\ee$-envelope of $g$, then it is easy to see that the function
Denote %\todo[noline]{WE DO NOT NEED DUALITY !!!}
\begin{align} \label{DefM88}
M(t) :=  
\mathds 1_{\{ t+ \ee > \ee^{1/6}\sqrt{n}  \}}
\mathbb{E}  \left( g(t + S_{m});   \tau_{t-\ee} \leq m   \right),  \quad  t \in \bb R.   
\end{align}
Then %the function $M$ is an integrable upper $\ee$-envelope of $L$
$L \leq_{\ee} M$ since $f \leq_{\ee} g$. 
Using the upper bound \eqref{eqt-B 001} of  Theorem \ref{t-B 002} 
and the fact that $\left\|  M \right\|_1 \leq  \left\|  g \right\|_1$,  we obtain that 
uniformly in $x\in [0, \alpha_n \sqrt{n}]$, 
\begin{align}\label{eqt-A 001_Lower}
K_{2} 
& \leq    (1 + c \ee)   \frac{2V(x) }{\sqrt{2\pi } \sigma^2 k}
\int_{\bb{R}_{+}} M \left(t\right) \phi^+ \left( \frac{t}{\sigma \sqrt{k}} \right)  dt    \notag\\
& \quad +  c \ee^{1/4} \frac{V(x)}{n} \int_{\mathbb R} M( t)  e^{- \frac{t^2}{2 \sigma^2 \ee^{1/2} n}} dt 
 +  c_{\ee}  \left(  \alpha_n  + n^{-\ee} \right) \frac{V(x)}{ n }  \left\|  g \right\|_1.  
\end{align}
%For the second term, by a change of variable, we get 
%From \eqref{Bound_L_ee_Upper}, \eqref{Proba_001} and \eqref{Proba_002}, we get 
%%\begin{align*}%\label{}
%%\mathbb{P}  \left(  y + S_m^* + \ee > \ee^{1/6}\sqrt{n},  \min_{0 \leq j \leq m} (y  -\ee + S_j^*)  <  0 \right) 
%% \leq  c  \ee^{1/6}  + r'_n(\ee). 
%%\end{align*}
%%Substituting this into \eqref{Bound_L_ee_Upper},  we obtain
%\begin{align}\label{TypeA_Bound_Int_M}
%     \int_{\bb{R}_{+}} M \left(y\right)  dy  \leq  
%c \left( \ee^{1/6}  + r'_n(\ee) \right)  \| g \|_1. 
%\end{align}
For the first term, in view of \eqref{DefM88}, we use the duality formula (Lemma \ref{lemma-duality-lemma-2_Cor})
and the fact that the function $\phi^+$ is bounded on $\bb R$ to derive that 
\begin{align*}%\label{}
& \int_{\bb{R}_{+}} M \left(t\right) \phi^+ \left( \frac{t}{\sigma \sqrt{k}} \right)  dt   \notag\\
& =   \int_{\bb{R}_{+}}  
\mathbb{E}  \left( g(t + S_{m});   \tau_{t-\ee} \leq m   \right)
  \mathds 1_{\{ t+ \ee > \ee^{1/6}\sqrt{n}  \}} \phi^+ \left( \frac{t}{ \sigma \sqrt{k}} \right)  dt   \notag\\
& =  \int_{\bb{R}_{+}}  
\mathbb{E}  \left( g(t + \ee + S_{m});   \tau_{t} \leq m   \right)
  \mathds 1_{\{ t+ 2\ee > \ee^{1/6}\sqrt{n}  \}}   \phi^+ \left( \frac{t + \ee}{\sigma \sqrt{k}} \right)  dt  \notag\\
& =  \int_{\bb{R}_{+}}   g(t + \ee)  
 \bb E \left[ \phi^+ \left( \frac{t + S_m^* + \ee}{\sigma \sqrt{k}} \right); 
           t + S_m^* + 2\ee > \ee^{1/6}\sqrt{n}, \tau_t^* \leq m  \right]  dt  \notag\\ 
& \leq  c \int_{\bb{R}_{+}}   g(t + \ee)  
 \bb P \left(  t + S_m^* > \frac{1}{2} \ee^{1/6}\sqrt{n}, \tau_t^* \leq m  \right)  dt =: c(J_1 + J_2), 
\end{align*}
where 
\begin{align}%\label{}
J_1  & =  \int_{0}^{\ee^{1/4} \sqrt{n}}    g(t + \ee)  
 \bb P \left(  t + S_m^*  > \frac{1}{2}  \ee^{1/6}\sqrt{n}, \tau_t^* \leq m  \right) dt,   \label{CLLT-J1-6}\\
J_2 & =  \int_{\ee^{1/4}  \sqrt{n} }^{\infty}   g(t + \ee)  
 \bb P \left(  t + S_m^*  > \frac{1}{2}  \ee^{1/6}\sqrt{n}, \tau_t^* \leq m  \right)  dt.   \label{CLLT-J2-6}
\end{align}
%Since the function $\phi^+$ is bounded on $\bb R$, it follows that 
%\begin{align*}%\label{}
%& \int_{\bb{R}_{+}} M \left(t\right) \phi^+ \left( \frac{t}{\sigma \sqrt{k}} \right)  dt   \notag\\
%& \leq c \int_{\bb{R}_{+}}   g(t + \ee)  
% \bb P \left(  t + S_m^* > \frac{1}{2} \ee^{1/6}\sqrt{n}, \tau_t^* \leq m  \right)  dt   \notag\\
% & =  \int_{0}^{\ee^{1/4} \sqrt{n}}    g(t + \ee)  
% \bb P \left(  t + S_m^*  > \frac{1}{2}  \ee^{1/6}\sqrt{n}, \tau_t^* \leq m  \right) dt  \notag\\
%& \quad +  \int_{\ee^{1/4}  \sqrt{n} }^{\infty}   g(t + \ee)  
% \bb P \left(  t + S_m^*  > \frac{1}{2}  \ee^{1/6}\sqrt{n}, \tau_t^* \leq m  \right)  dt  \notag\\
%& =: J_1 + J_2. 
%\end{align*}
For $J_1$, observe that on the event $\{ \tau_t^* \leq m \}$, 
the inequality $t + S_m^* > \frac{1}{2}  \ee^{1/6}\sqrt{n}$ implies that there exists $0 \leq j \leq m$
such that $S_m^* - S_j^* > \frac{1}{2}  \ee^{1/6}\sqrt{n}$. 
Since $S_j^*= S_{m-j} - S_m$, it follows that $- S_{m-j} > \frac{1}{2}  \ee^{1/6}\sqrt{n}$ for some $0 \leq j \leq m - 1$
and hence
%Recall that $S_n=\sum_{i=1}^{n} X_i$ and its dual random walk is 
%$S_j = S_m^* - S_{m-j}^*$ observe that 
\begin{align}\label{Proba_Holder}
& \bb P \left(  t + S_m^* > \frac{1}{2}  \ee^{1/6}\sqrt{n}, \tau_{t}^* \leq m  \right)  
 \leq  \bb P \left(  \max_{1 \leq j \leq m}  |S_j| \geq   \frac{1}{2} \ee^{1/6}\sqrt{n},  \tau_{t}^* \leq m  \right). 
%& \leq \bb P^{1/2} \left(  \max_{1 \leq j \leq m}  |S_j| \geq   \frac{1}{2} \ee^{1/6}\sqrt{n}   \right) 
%   \bb P^{1/2} \left(  \tau_{u+y}^* \leq m  \right)
\end{align}
Noting that $m= [ \ee^{1/2} n ]$, 
we use the Fuk-Nagaev inequality (Lemma \ref{Lem_FukNagaev}) 
with $n$ replaced by $m$ and $u = v= \frac{1}{2} \ee^{1/6}\sqrt{ n}$, and condition \ref{SecondMoment} to get
\begin{align}\label{Proba_002}
\bb P \left(  \max_{1 \leq j \leq m}  |S_j| \geq   \frac{1}{2} \ee^{1/6}\sqrt{n}  \right)
& \leq  c  \ee^{1/6}  
%\exp\left(- \frac{u}{v} \log \left(1 + \frac{uv}{m}\right)\right) 
 + m \bb P(|X_1| > v)  
 \leq  c  \ee^{1/6}  +  \frac{c_{\ee}}{n^{2\ee}}, 
\end{align}
which, together with \eqref{CLLT-J1-6} and \eqref{Proba_Holder}, implies that 
%Therefore, we have 
\begin{align}\label{Bound_J1}
J_1 \leq   c  \ee^{1/6}  \int_{0}^{ \ee^{1/4} \sqrt{n} }    g(t + \ee)  dt
    +  \frac{c_{\ee}}{n^{2\ee}} \|g\|_1.
\end{align}
For $J_2$, we use \eqref{Proba_Holder}, \eqref{Proba_002} and H\"older's inequality to get
\begin{align}\label{Proba_Holder02}
& \bb P \left(  t + S_m^* > \frac{1}{2} \ee^{1/6}\sqrt{n}, \tau_{t}^* \leq m  \right)   \notag\\
& \leq \bb P^{1/2} \left(  \max_{1 \leq j \leq m}  |S_j| \geq   \frac{1}{2} \ee^{1/6}\sqrt{n}   \right) 
   \bb P^{1/2} \left(  \tau_{t}^* \leq m  \right)   \notag\\
& \leq  \left( c  \ee^{1/12}  +  \frac{c_{\ee}}{n^{\ee}} \right)
   \bb P^{1/2} \left( t + \min_{1 \leq j \leq m }  S_j  < 0 \right).  
\end{align}
%where $r'_n(\ee)= \ee^{-1/2}\mathbb E |X_1|^2 \mathds 1 (|X_1| \geq \frac{1}{2}\ee^{1/6} \sqrt{ n}) \to 0$ as $n\to\infty$. 
%Note that 
%\begin{align*}%\label{}
%\bb P \left( \tau_{u+y}^* \leq m \right)  =  \bb P \left( u + y + \min_{1 \leq j \leq m } S_j < 0 \right)
%\leq  \bb P \left( \max_{1 \leq j \leq m } |S_j|  > u + y \right)
%\end{align*}
Denote  
\begin{align}\label{Def_A_k_aa}
A_m = \left\{ \sup_{0 \leq t \leq 1} \left| S_{[tm]} -  \sigma B_{tm} \right| \leq  m^{1/2 - 2 \ee}  \right\} 
\end{align}
and by $A_m^c$ its complement. 
By the functional central limit theorem (Lemma \ref{FCLT})
and the identity \eqref{levy001a} of Lemma \ref{lemma tauBM}, we obtain  that for any $t \geq \ee^{1/4} \sqrt{n}$,  
\begin{align*}%\label{}
\bb P \left( t + \min_{1 \leq j \leq m } S_j < 0 \right)
& \leq   \bb P \left( t + \min_{1 \leq j \leq m } S_j < 0,  A_m  \right) 
   +  \frac{c_{\ee}}{ n^{2 \ee} }  \notag\\
& \leq   \bb P \left( t - m^{1/2 - 2 \ee} + \sigma \min_{1 \leq j \leq m }  B_j < 0  \right) 
   +  \frac{c_{\ee}}{ n^{2 \ee} }  \notag\\
& =  \frac{2}{\sigma \sqrt{2\pi m} }  \int_{ t - m^{1/2 - 2 \ee} }^{\infty}  e^{ - \frac{s^{2}}{2 \sigma^2 m} } ds 
  +  \frac{c_{\ee}}{ n^{2 \ee} }   \notag\\
 & \leq   c  \int_{ \frac{t}{2 \sigma \sqrt{m}} }^{\infty}  e^{ - \frac{u^{2}}{2} } du 
  +  \frac{c_{\ee}}{ n^{2 \ee} }   \notag\\
& \leq  c   e^{- \frac{t^2}{8 \sigma^2 \sqrt{\ee} n}}  +  \frac{c_{\ee}}{ n^{2 \ee} }, 
\end{align*}
where in the last line we used the inequality $\int_a^{\infty} e^{- \frac{u^2}{2}} du \leq \frac{1}{a} e^{- \frac{a^2}{2}}$ for $a >0$,
and the fact that $\frac{\sqrt{m}}{t} = \frac{ \sqrt{[\ee n]} }{t} \leq 1$ for $t \geq \ee^{1/4} \sqrt{n}$. 
Hence, 
\begin{align*}%\label{}
\bb P^{1/2} \left( t + \min_{1 \leq j \leq m }  S_j  < 0 \right)
\leq   c   e^{- \frac{t^2}{16 \sigma^2 \sqrt{\ee} n}}  +  \frac{c_{\ee}}{ n^{\ee} }.  
\end{align*}
This, together with \eqref{CLLT-J2-6} and \eqref{Proba_Holder02}, implies that 
\begin{align}\label{Bound_J2}
J_2 
%& \leq \left( c  \ee^{1/12}  +  \frac{c_{\ee}}{n^{\ee}} \right)
% \left[ \int_{ \ee^{1/4} \sqrt{n} }^{\infty}   g(t + \ee)    e^{- \frac{t^2}{16 \sigma^2 \sqrt{\ee} n}}  dt  
%   +  \frac{c_{\ee}}{n^{\ee}} \|g\|_1  \right]   \notag\\
& \leq  c  \ee^{1/12}  \int_{ \ee^{1/4} \sqrt{n} }^{\infty}   g(t + \ee)    e^{- \frac{t^2}{16 \sigma^2 \sqrt{\ee} n}}  dt
      +  \frac{c_{\ee}}{n^{\ee}} \|g\|_1. 
\end{align}
Putting \eqref{Bound_J1} and \eqref{Bound_J2} together 
and using the fact that $e^{\frac{t^2}{16 \sigma^2 \sqrt{\ee} n}} \leq c$ for any $t \in [0, \ee^{1/4} \sqrt{n}]$,  
we have 
\begin{align}\label{BoundMt99}
& \int_{\bb{R}_{+}} M (t) \phi^+ \left( \frac{t}{\sigma \sqrt{k}} \right)  dt   \notag\\
& \leq   c  \ee^{1/12}  
 \left[ \int_{0}^{ \ee^{1/4} \sqrt{n}}    g(t + \ee)  dt
    +  \int_{ \ee^{1/4} \sqrt{n} }^{\infty}   g(t + \ee)    e^{- \frac{t^2}{16 \sigma^2 \sqrt{\ee} n}}  dt  \right] 
     +  \frac{c_{\ee}}{n^{\ee}}  \|g\|_1      \notag\\
& \leq   c  \ee^{1/12} \int_{\bb R_+}  g(t + \ee)    e^{- \frac{t^2}{16 \sigma^2 \sqrt{\ee} n}}  dt 
     +  \frac{c_{\ee}}{n^{\ee}} \|g\|_1.   
\end{align}
For the second term on the right hand side of \eqref{eqt-A 001_Lower}, 
following the same proof of \eqref{BoundMt99} and using the fact that $e^{- \frac{t^2}{2 \sigma^2 \ee^{1/2} n}} \leq 1$,
one has
\begin{align}\label{Boundexp999}
\int_{\mathbb R} M( t)  e^{- \frac{t^2}{2 \sigma^2 \ee^{1/2} n}} dt 
 \leq  c  \ee^{1/12} \int_{\bb R_+}  g(t + \ee)    e^{- \frac{t^2}{16 \sigma^2 \sqrt{\ee} n}}  dt 
     +  \frac{c_{\ee}}{n^{\ee}} \|g\|_1.  
\end{align}
Implementing \eqref{BoundMt99} and \eqref{Boundexp999} into \eqref{eqt-A 001_Lower} gives 
\begin{align}\label{LLL2-bbb004}
K_{2} 
%& \leq  \frac{V(x)}{n} \left( c  \ee^{1/12}  +  \frac{c_{\ee}}{n^{\ee}} \right) 
%  \left[  \int_{\bb R_+}  g(t + \ee)    e^{- \frac{t^2}{16 \sigma^2 \sqrt{\ee} n}}  dt  
%     +  \frac{c_{\ee}}{n^{\ee}}  \|g\|_1  \right]   \notag\\
% &  \quad +  c  \ee^{- 1/4} \frac{ V(x) }{n } 
% \left( \frac{ 1}{n^{\gamma}}  + \frac{\ee^{ - \delta/4}}{ n^{\delta/2}}   \right) \left\|  g \right\|_1  \notag\\
 & \leq  c  \ee^{1/12}  \frac{V(x)}{n}    \int_{\bb R_+}  g(t + \ee)    e^{- \frac{t^2}{16 \sigma^2 \sqrt{\ee} n}}  dt 
    +   c_{\ee}  \left(  \alpha_n  + n^{-\ee} \right) \frac{V(x)}{ n }  \|g\|_1. 
\end{align}
From 
\eqref{KKK-111-001}, \eqref{K1-final bound} and \eqref{LLL2-bbb004},
and using the fact that $e^{- \frac{t^2}{16 \sigma^2 \sqrt{\ee} n}} \leq c \phi ( \frac{t}{ \sigma \sqrt{n}} )$, we obtain
\begin{align*}%\label{}
I_{n,2}(x) 
& \leq  c  \ee^{1/12} \frac{V(x) }{n}  \int_{\bb R} [g(t) + g(t + \ee)]  \phi \left( \frac{t}{ \sigma \sqrt{n}} \right)  dt   \notag\\
 & \qquad     +  c_{\ee}  \left(  \alpha_n  + n^{-\ee} \right) \frac{V(x)}{ n }  \|g\|_1   \notag\\
& \leq  c  \ee^{1/12} \frac{V(x) }{n}  \int_{\bb R} g(t)  \phi \left( \frac{t}{ \sigma \sqrt{n}} \right)  dt 
      +  c_{\ee}  \left(  \alpha_n  + n^{-\ee} \right) \frac{V(x)}{ n }  \|g\|_1,  
\end{align*}
where in the last line we used the Lipschitz continuity of $\phi$. 
Combining this with \eqref{staring point for the lower-bound-001} and \eqref{Lower_In_lll}, and using the fact that $f \leq_{\ee} g$,
we conclude the proof of the lower bound \eqref{eqt-B 002}. 
%we get 
%\begin{align*}%\label{}
%I_n(x) & \geq    \frac{2 V(x) }{\sqrt{2\pi } \sigma^2 n}
% \int_{\bb{R}_{+}} h \left(t\right)  \phi^+ \left( \frac{t}{\sigma \sqrt{n}} \right)   dt 
%  -  c \ee^{1/12}  \frac{V(x) }{n}    \int_{\bb R} g(t)  e^{- \frac{t^2}{2 \sigma^2 n}}  dt   \notag\\
%&  \quad -  c \ee  \frac{2V(x)}{\sqrt{2 \pi} \sigma^2 n} 
%    \int_{\bb R}   f(t)  \phi^+ \left( \frac{t}{\sigma \sqrt{n}} \right) dt   
%      - c_{\ee} \frac{ V(x) }{n^{1+\ee} }  \left\|  g \right\|_1,    
%\end{align*}
%which concludes the proof of the lower bound \eqref{eqt-B 002}. 
%\end{proof}

%Remark: we can avoid duality in the case $y=•-\eta\sqrt{n}$.

%%%%%%%%%%%%%%%%%%%%%%%%%%%%%%%%%%%%%%%%
%%%%%%%%%%%%%%%%%%%%%%%%%%%%%%%%%%%%%%%%
\subsection{Proof of Theorems \ref{Theorem-AA001} and \ref{Theorem Delta-002}} \label{SectProofThm1}
In this section we show how to deduce Theorems \ref{Theorem-AA001} and \ref{Theorem Delta-002}  from Theorem \ref{t-B 002}. 
%As to Theorem \ref{Thm-C 002} -- it is deduced from Theorem \ref{Theorem-AA001} by a change of measure formula. 

%Using Theorem \ref{t-B 002ChangeMea}, 
%now we establish Theorems \ref{Theorem-AA001} and \ref{Thm-C 002}.  
%%We first prove \eqref{demoCLLT-003_Target}. 

\begin{proof}[Proof of Theorem \ref{Theorem-AA001}]
We first prove that \eqref{demoCLLT-003_Targetaa} holds 
uniformly in $x \in [0, n^{1/2 - \ee}]$ and $y \in [\eta \sqrt{n}, \sigma \sqrt{q n \log n}]$. 
Without loss of generality, we assume that the target function $f$ is non-negative.   
%We first prove that \eqref{demoCLLT-003_Target} holds uniformly in 
%$x \in [0, n^{1/2 - \ee}]$ and $y \in [\eta \sqrt{n}, \sigma_{\lambda} \sqrt{q n \log n}]$. 
Since $f \leq_{\ee} \overline f_{\delta,\ee}$ with $\overline f_{\delta,\ee}$ defined by \eqref{Def_f_deltaee}, 
%in Section \ref{subsec-auxiliary}, 
applying the upper bound \eqref{eqt-B 001} of Theorem \ref{t-B 002} with $g = \overline f_{\delta,\ee}$ and $\alpha_n = n^{- \ee}$, 
we derive that for any $x \in [0, n^{1/2 - \ee}]$ and  $y \in \bb R_+$, 
\begin{align*}%\label{}
&   \mathbb E \left( f(x+S_n - y); \tau_x > n \right)   
 \leq   (1 + c \ee)   
   \frac{2 V(x) }{\sqrt{2\pi }  \sigma^2  n }  J_n,   
\end{align*}
where 
\begin{align*}%\label{}
J_n =  \int_{-\ee}^{\infty}  \overline f_{\delta,\ee} (t)     
\left[  \phi^+  \left( \frac{t + y}{ \sigma \sqrt{n}} \right)   +  \phi \left( \frac{t + y}{ \ee^{1/4} \sigma \sqrt{n}} \right)
    +  \frac{ c_{\ee} }{ n^{\ee} } \right]   dt. 
\end{align*}
%Note that 
%\begin{align*}%\label{}
%\int_{-\ee}^{\infty} e^{- \lambda  t}  
%\overline f_{\delta,\ee} \left(t  \right) \phi^+  \left( \frac{t + y}{ \sigma_{\lambda} \sqrt{n}} \right)    dt 
%=   \left\{ \int_{-\ee}^{n^{1/4}}  + \int_{n^{1/4}}^{\infty}  \right\} e^{- \lambda  t}  
%\overline f_{\delta,\ee} \left(t  \right) \phi^+  \left( \frac{t + y}{ \sigma_{\lambda} \sqrt{n}} \right)    dt
%\end{align*}
%Since the function $\phi^+$ is Lipschitz continuous on $\bb R_+$, there exists a constant $c>0$ such that for any 
%$t \in \bb R$ and $y \in [\delta \sqrt{n}, \sqrt{q n \log n}]$,
%\begin{align*}%\label{}
%\left| \phi^+ \left( \frac{t + y}{ \sigma_{\lambda}  \sqrt{n}} \right) -   \phi^+ \left( \frac{y}{ \sigma_{\lambda} \sqrt{n}} \right)  \right|
%\leq   c \frac{ |t| }{\sqrt{n}}, 
%\end{align*}
%which implies that 
Since the functions $\phi^+$ and $\phi$ are Lipschitz continuous on $\bb R$, there exists a constant $c>0$ such that for any
$t \geq -\ee$ and $y \in \bb R$,
\begin{align*}%\label{InequalityLipphiaa}
\left| \phi^+ \left( \frac{t + y}{ \sigma  \sqrt{n}} \right) -   \phi^+ \left( \frac{y}{ \sigma \sqrt{n}} \right)  \right|
& \leq   c \frac{ |t| }{\sqrt{n}},   \notag\\
%\end{align*}
%and 
%\begin{align*}%\label{}
\left|  \phi \left( \frac{t + y}{ \ee^{1/4} \sigma \sqrt{n}} \right) 
      -  \phi \left( \frac{y}{ \ee^{1/4} \sigma \sqrt{n}} \right) \right|
& \leq   c \frac{ |t| }{\ee^{1/4} \sqrt{n}}, 
\end{align*}
%Using \eqref{InequalityLipphiaa} we have
which implies that 
\begin{align*}%\label{}
J_n  & \leq  \left[  \phi^+  \left( \frac{y}{ \sigma  \sqrt{n}} \right)  
           + \phi \left( \frac{y}{ \ee^{1/4} \sigma \sqrt{n}} \right)   +   \frac{ c_{\ee} }{ n^{\ee} }    \right]
  \int_{-\ee}^{\infty}  \overline f_{\delta,\ee}(t)    dt   \notag\\
& \quad  +  \frac{c}{\ee^{1/4} \sqrt{n}} \int_{-\ee}^{\infty} \overline f_{\delta,\ee}(t)  |t|  dt.      
\end{align*}
%\todos{One can improve here the range of $y$}
Note that for any fixed $\eta >0$, there exists a constant $c_{\eta} >0$ such that 
uniformly in $y \in [ \eta \sqrt{n}, \sigma \sqrt{q n \log n}]$, 
\begin{align} \label{Inequa-phi-phiplus}
n^{-q/2} \leq  c_{\eta}  \phi^+  \left( \frac{y}{ \sigma  \sqrt{n}} \right),
\quad 
\phi \left( \frac{y}{ \ee^{1/4} \sigma \sqrt{n}} \right)  
\leq  c_{\eta}  \exp \Big\{ - \frac{\eta^2}{ 4 \sqrt{\ee} \sigma^2 }  \Big\} \phi^+  \left( \frac{y}{ \sigma  \sqrt{n}} \right).    
\end{align}
It follows that  uniformly in $y \in [\eta \sqrt{n}, \sigma \sqrt{q n \log n}]$,
\begin{align*}%\label{}
J_n & \leq   \phi^+  \left( \frac{y}{ \sigma  \sqrt{n}} \right) 
   \left( 1  +  c_{\eta} \exp \Big\{ - \frac{\eta^2}{ 4 \sqrt{\ee} \sigma^2 }  \Big\}  +  \frac{c_{\ee, \eta}}{ n^{ \ee - q/2} }  \right) 
    \int_{-\ee}^{\infty} \overline f_{\delta,\ee}(t)     dt   \notag\\
& \quad  +   \phi^+  \left( \frac{y}{ \sigma  \sqrt{n}} \right) 
    \frac{c_{\eta}}{ \ee^{1/4}  n^{1/2 - q/2}}  \int_{-\ee}^{\infty}  \overline f_{\delta,\ee}(t)   |t| dt.  
\end{align*}
Therefore, we get that uniformly in $x \in [0, n^{1/2 - \ee}]$ and $y \in [\eta \sqrt{n}, \sigma \sqrt{q n \log n}]$,
\begin{align*}%\label{}
 \limsup_{n \to \infty}  
\frac{ \mathbb E \left( f(x+S_n - y); \tau_x > n \right)    }{  \frac{2 V(x) }{\sqrt{2\pi } \sigma^2 n }    
   \phi^+  \left( \frac{y}{ \sigma \sqrt{n}} \right)  }    
 \leq  \left[ 1 + c_{\eta}  \exp \Big\{ - \frac{\eta^2}{ 4 \sqrt{\ee} \sigma^2 }  \Big\}   \right] 
       \int_{-\ee}^{\infty} \overline f_{\delta,\ee}(t)  dt. 
\end{align*}
This proves the upper bound by
taking first $\ee \to 0$ and then $\delta \to 0$, and using Lemma \ref{lemma-DRI-convergence}.  
The proof of the lower bound can be carried out in the same way 
using \eqref{eqt-B 002}.

In a similar way, one can also use Theorem \ref{t-B 002}  to prove that
 \eqref{demoCLLT-003_Targetaa} holds uniformly in $x \in [0, \alpha_n \sqrt{n}]$ and 
 $y \in [\eta \sqrt{n}, \eta^{-1} \sqrt{n}]$. 
The proof of  Theorem \ref{Theorem-AA001} is complete. 
\end{proof}
% holds uniformly in 
%$x \in [0, n^{1/2 - \ee}]$ and $y \in [\eta \sqrt{n}, \sigma_{\lambda} \sqrt{q n \log n}]$.
%%and the proof of Theorem \ref{Thm-C 002} is complete. 
%We next prove Theorem \ref{Thm-C 002}. 
%By a change of measure, we get
%\begin{align}\label{ChangeOfMeas0aa}
%& \mathbb E \left( f(x+S_n - y); \tau_x >n\right)  \notag\\
%& = e^{n \Lambda(\lambda)}  \bb E_{\lambda} \left( e^{- \lambda S_n} f(x+S_n - y); \tau_x >n \right)  \notag\\
%& =  e^{n \Lambda(\lambda) + \lambda (x-y) } \bb E_{\lambda} \left( e^{- \lambda (x + S_n - y)} f(x+S_n - y); \tau_x >n \right). 
%\end{align}
%Since $\int_{\bb R_+}  f(t) e^{- \lambda t} (1 + t)  dt < \infty$, 
%applying Theorem \ref{Theorem-AA001} we get Theorem \ref{Thm-C 002}. 
%%\end{proof}

%%%%%%%%%%%%%%%%%%%%%%%%%%%%%%%%%%%%%%%%%%%%%%
%%%%%%%%%%%%%%%%%%%%%%%%%%%%%%%%%%%%%%%%%%%%%%
%\subsection{Proof of Theorem \ref{Theorem Delta-002}}

We will show in Lemma \ref{Theorem harmonic func} that the harmonic function $V$ satisfies $x \leq  V(x) \leq c(1+x)$ for any $x\geq 0$. 
To prove Theorem \ref{Theorem Delta-002} we need the following refinement of this bound.

\begin{lemma}\label{Lem_V_Ineq_aa}
Assume \ref{SecondMoment} for some $\delta >0$.  
Then, for any  $\ee \in  (0, \frac{\delta}{2(2 + \delta)}]$, 
there exists a constant $c_{\ee} >0$ such that 
%There exists $\ee_0>0$ such that for any $\ee\in (0, \ee_0)$, 
for any $x \geq 0$ and $k_0 \geq 1$, 
\begin{align*}
x \leq  V(x) \leq  \left(1 +  \frac{ c_\ee }{ k_0^{  \ee} } \right)  x + c_{\ee} k_0^{1/2 - \ee}. 
\end{align*}
\end{lemma}

\begin{proof}[Proof of Theorem \ref{Theorem Delta-002}]
We first prove that \eqref{demoCLLT-003aa} holds 
 uniformly in $x \in [0, n^{1/2 - \ee}]$,  $y \in [\eta \sqrt{n}, \sigma \sqrt{q n \log n}]$
and $\Delta \in [\Delta_0,   n^{1/2 - \ee} ]$. 
Using the upper bound \eqref{eqt-B 001} 
with $f = \mathds 1_{[y, y + \Delta]}$, $g = \mathds 1_{[y-\ee, y + \Delta + \ee]}$
and $\alpha_n = n^{-\ee}$, we have, uniformly in $x \in [0, n^{1/2 - \ee}]$, 
%%We only give a proof of Theorem \ref{Thm-C 002-Delta} 
%%since Theorem \ref{Theorem Delta-002} is a particular case of Theorem \ref{Thm-C 002-Delta} by taking $\lambda = 0$. 
%Since $g = \mathds 1_{[-\ee, \Delta + \ee]}$ is a measurable upper $\ee$-envelope of 
%$f = \mathds 1_{[0, \Delta]}$, 
%using the upper bound \eqref{eqt-B 001} of Theorem \ref{t-B 002}, we have 
\begin{align*}%\label{}
&   \mathbb{P} \left( x+S_n  \in [0, \Delta] + y, \,  \tau_x > n \right)  
 \leq   (1 + c \ee)   
   \frac{2V \left( x\right) }{\sqrt{2\pi }  \sigma^2  n }  J_n,  
\end{align*}
where 
\begin{align*}%\label{}
J_n =  \int_{-\ee}^{\Delta + \ee}   %\overline f_{\delta,\ee} \left(t  \right)   
\left[  \phi^+  \left( \frac{t + y}{ \sigma \sqrt{n}} \right)   +  \phi \left( \frac{t + y}{ \ee^{1/4} \sigma \sqrt{n}} \right)
    +  \frac{ c_{\ee} }{ n^{\ee} } \right]   dt. 
\end{align*}
%\begin{align*}%\label{}
%&  \mathbb{P} \left( x+S_n  \in [0, \Delta] + y, \,  \tau_x > n \right)     \notag\\
%& \leq   (1 + c \ee)   
%   \frac{2V_{\lambda} \left( x\right) }{\sqrt{2\pi }  \sigma_{\lambda}^2  n }    
%    e^{n \Lambda(\lambda) + \lambda (x - y)}
%\int_{-\ee}^{\Delta + \ee} e^{- \lambda t}  \phi^+  \left( \frac{t + y}{ \sigma_{\lambda} \sqrt{n}} \right)    dt    \notag\\
%& \quad +  c_{\ee}  \frac{ V_{\lambda} (x) }{ n^{1 + \ee} }  
%e^{n \Lambda(\lambda) + \lambda (x - y)}   \int_{-\ee}^{\Delta + \ee} e^{- \lambda t}  dt.  
%\end{align*}
Since the functions $\phi^+$ and $\phi$ are Lipschitz continuous on $\bb R$, there exists a constant $c>0$ 
such that for any $\Delta \in [\Delta_0, n^{1/2 - \ee}]$, 
$t \in [-\ee, \Delta + \ee]$ and $y \in [\eta \sqrt{n}, \sigma \sqrt{q n \log n}]$,
\begin{align*}%\label{InequalityLipphi}
\left| \phi^+ \left( \frac{t + y}{ \sigma  \sqrt{n}} \right) -   \phi^+ \left( \frac{y}{ \sigma \sqrt{n}} \right)  \right|
& \leq   c \frac{ |t| }{\sqrt{n}}  \leq  \frac{ c }{n^{\ee}},     \notag\\
\left|  \phi \left( \frac{t + y}{ \ee^{1/4} \sigma \sqrt{n}} \right) 
      -  \phi \left( \frac{y}{ \ee^{1/4} \sigma \sqrt{n}} \right) \right|
& \leq   c \frac{ |t| }{\ee^{1/4} \sqrt{n}}  \leq   \frac{ c }{\ee^{1/4} n^{\ee}}. 
\end{align*}
%This implies that 
%\begin{align*}%\label{}
%\int_{-\ee}^{\Delta + \ee} e^{- \lambda t}  \phi^+  \left( \frac{t + y}{ \sigma_{\lambda} \sqrt{n}} \right)    dt
%%& \leq  \left[ \phi^+  \left( \frac{y}{ \sigma_{\lambda}  \sqrt{n}} \right)  +  c \frac{1 + \Delta }{\sqrt{n}}  \right]
%%  \int_{-\ee}^{\Delta + \ee} e^{- \lambda t}   dt   \notag\\
%& \leq   \left[ \phi^+  \left( \frac{y}{ \sigma_{\lambda}  \sqrt{n}} \right)  +  \frac{ c }{n^{\ee}}  \right]
%    \frac{e^{\lambda \ee} - e^{-\lambda (\Delta + \ee)} }{ \lambda }.   
%\end{align*}
%%where in the last inequality we used $\Delta \in [\Delta_0, n^{1/2 - \ee}]$. 
This implies that there exists a constant $c_{\eta} >0$ such that
 uniformly in $y \in [\eta \sqrt{n}, \sigma \sqrt{q n \log n}]$ and $\Delta \in [\Delta_0, n^{1/2 - \ee}]$, 
\begin{align*}%\label{}
%& \mathbb{P} \left( x+S_n  \in [0, \Delta] + y; \tau_x >n\right)     \notag\\
J_n & \leq   (\Delta + 2\ee)
\left[ \phi^+  \left( \frac{y}{ \sigma \sqrt{n}} \right)  
   +   \phi \left( \frac{y}{ \ee^{1/4} \sigma \sqrt{n}} \right) + \frac{c}{\ee^{1/4} n^{\ee} }  \right]
      \notag\\
& \leq  \Delta  (1 + c \ee)   \phi^+  \left( \frac{y}{ \sigma  \sqrt{n}} \right) 
   \left( 1  +  c_{\eta} \exp \left\{ - \frac{\eta^2}{ 4 \sqrt{\ee} \sigma^2 }  \right\}  
        +  \frac{c_{\eta}}{ \ee^{1/4} n^{ \ee - q/2} }  \right), 
\end{align*}
where in the last inequality we used \eqref{Inequa-phi-phiplus} and the fact that $\Delta_0 >0$ is a fixed constant. 
%we have that $(e^{-\lambda \Delta} + e^{\lambda \ee})/(1 - e^{- \lambda \Delta})$ is bounded by some constant, 
%we have uniformly in $\Delta \in [\Delta_0, n^{1/2 - \ee}]$ 
%\begin{align}\label{InequaeeDelta}
%\Delta + 2\ee \leq \Delta (1 + c \ee). 
%%e^{\lambda \ee} - e^{-\lambda (\Delta + \ee)} 
%%& =   \left( 1 - e^{-\lambda \Delta} \right)   
%%\left[ 1 +  \frac{ (e^{-\lambda \Delta} + e^{\lambda \ee}) (1 - e^{- \lambda \ee}) }{ 1 - e^{-\lambda \Delta} } \right]   \notag\\
%%& \leq  \left( 1 - e^{-\lambda \Delta} \right)  \left( 1 + c \ee \right).  
%\end{align}
Therefore, we get that uniformly in $x \in [0, n^{1/2 - \ee}]$, $y \in [\eta \sqrt{n}, \sigma \sqrt{q n \log n}]$
and $\Delta \in [\Delta_0, n^{1/2 - \ee}]$, 
\begin{align*}%\label{}
\limsup_{n \to \infty}  
\frac{ \mathbb{P} \left( x+S_n  \in [0, \Delta] + y, \,  \tau_x > n \right)   }{  \Delta \frac{2V \left( x\right) }{\sqrt{2\pi } \sigma^2 n }    
   \phi^+  \left( \frac{y}{ \sigma \sqrt{n}} \right)    }  
   \leq   1 + c_{\eta} \exp \left\{ - \frac{\eta^2}{ 4 \sqrt{\ee} \sigma^2 }  \right\}. 
\end{align*}
Since $\ee >0$ can be arbitrary  small, this proves the upper bound. 
The proof of the lower bound can be carried out in the same way 
using \eqref{eqt-B 002}. % of Theorem \ref{t-B 002}.

Similarly, one can use Theorem \ref{t-B 002} to prove that
 the asymptotic \eqref{demoCLLT-003aa} holds uniformly in $x \in [0, \alpha_n \sqrt{n}]$,  
 $y \in [\eta \sqrt{n}, \eta^{-1} \sqrt{n}]$ and $\Delta \in [\Delta_0,   \alpha_n \sqrt{n} ]$. 
Hence the proof Theorem \ref{Theorem Delta-002} is complete. 
%uniformly in $x \in [0, n^{1/2 - \ee}]$ and $y \in [\eta \sqrt{n}, \sigma_{\lambda} \sqrt{q n \log n}]$. 
\end{proof}

\section{Conditioned local limit theorem near the boundary}

The aim of this section is to establish Theorems \ref{Theorem-AA002} and \ref{Theorem-AA002bis}. 
%whose proof is based on the duality formula and Theorem \ref{t-B 002}, 
%\textcolor{magenta}{i.e.\ the effective version of Theorem \ref{Thm-C 002}. }
%\todos{I do not understand this sentence}
%%%%%%%%%%%%%%%%%%%%%%%%%%%%%%%%%%%%%%%%%%%
%%%%%%%%%%%%%%%%%%%%%%%%%%%%%%%%%%%%%%%%%%%
%\subsection{Auxiliary assertion}
%
%For any bounded measurable function $g: \bb R \to  \bb R_{+}$ and $p>0$ set 
%\begin{equation}
%\left\Vert g \right\Vert_{1,p} = \int_{ \bb R  } \left( 1+ |t|  \right)^{p} g(t) dt.
%\label{norm1,p}
%\end{equation}

%%%%%%%%%%%%%%%%%%%%%%%%%%%%%%%%%%%%%%%%%%%%%%%%%
%%%%%%%%%%%%%%%%%%%%%%%%%%%%%%%%%%%%%%%%%%%%%%%%%
%%%%%%%%%%%%%%%%%%%%%%%%%%%%%%%%%%%%%%%%%%%%%%%%%
\subsection{A non-asymptotic conditioned local limit theorem} \label{Proof of LLT with rate n3/2}

The following bounds will be used the proofs of
Theorems \ref{Theorem-AA002} and \ref{Theorem-AA002bis}.
It is an easy consequence of the duality formula (Lemma \ref{lemma-duality-lemma-2_Cor}) 
and bounds of the exit time for the dual random walk. 

\begin{lemma}\label{lemma Qng integr}
Assume \ref{SecondMoment}.  
Let $g: \bb R \to  \bb R_{+}$ be a measurable function
satisfying $ \int_{\bb R_+} g(y) (1 + y) dy <\infty$.
Then there exists a constant $c$ such that 
\begin{align*}%\label{}
\sup_{n\geq 1} \sqrt{n} \int_{\bb R _{+}}  \bb E ( g(x+S_n); \tau_x >n)  dx 
\leq c  \int_{\bb R_+} g(y) (1 + y) dy. 
\end{align*}
Moreover, 
\begin{align*}
%\sup_{n\geq 1}\sqrt{n}\left\Vert \mathbf{Q}_{n}g\right\Vert _{1}\equiv
\lim_{n \to \infty} \sqrt{n} \int_{\bb R _{+}}  \bb E ( g(x+S_n); \tau_x >n)  dx
%\mathbf{Q}_{n}g\left( y\right)dx
= \frac{2}{ \sigma \sqrt{2 \pi }} \int_{\bb R_+} g(y) V^*(y) dy.
\end{align*}
\end{lemma}

\begin{proof}
Using Lemma \ref{lemma-duality-lemma-2_Cor} with $h = 1$, we get that for any $n \geq 1$, 
\begin{align*}
\int_{\bb R _{+}}  \bb E ( g(x+S_n) ; \tau_x >n)  dx
% & = \int_{\bb R _{+}}  \left( \int_{\bb R _{+}} g(y) \bb P ( x+S_n \in dy ; \tau_x >n) \right) dx \\
& = \int_{\bb R _{+}} g(y)   \mathbb{P} \big( \tau_{y}^{\ast } > n \big) dy.
\end{align*}
%By Theorem \ref{Theor-IntegrLimTh}, there exists a constant $c$ such that, for any $y \geq 0,$ 
%\begin{align*}
%% \int_{\bb R _{+}} \mathbb{P}\left( z+S_n ^{\ast }\in dy;\tau _{z}^{\ast }>n\right) \leq 
%\mathbb{P} \big( \tau_{y}^{\ast } > n \big) \leq c\frac{1+y}{\sqrt{n}}.
%\end{align*}
The conclusion follows from \eqref{CLLT-bound only 001b}, 
%\eqref{CLLT-bound only 001} (with $t = \infty$),
 \eqref{CLLT-bound only 002} of Theorem \ref{Theor-IntegrLimTh} 
and the Lebesgue dominated convergence theorem.   
%Therefore,
%\begin{align*}
%\int_{\bb R _{+}}  \bb E ( g(x+S_n) ; \tau_x >n)  dx
%\leq \frac{c}{\sqrt{n}} \int_{\bb R _{+}} g(y) (1 + y) dy,
%\end{align*}
%as desired. 
\end{proof}

Theorems \ref{Theorem-AA002} and \ref{Theorem-AA002bis}
 will be deduced from the following more general statement 
in which we establish upper and lower bounds for the joint law of the random walk $x+S_n $ and of the event $\{\tau_x >n\}$.

\begin{theorem} \label{theorem-n3/2-upper-lower-bounds}
Assume \ref{A1} and  \ref{SecondMoment}. 
%Let $(\alpha_n)_{n \geq 1}$ be a sequence of positive numbers satisfying $\lim_{n \to \infty} \alpha_n = 0$. 
%Assume that there exists a constant $\eta >0$ such that 
%\begin{align*}%\label{}
%\int_{\bb R_+} (1 +t)^{1 + \eta} g(t)  dt < \infty. 
%\end{align*}
%Let $(\alpha_{n})_{n\geq 1}$ be any sequence of positive numbers such that $\lim_{n\rightarrow \infty} \alpha_{n} = 0$. 
%For any $\veps \in (0,\frac{1}{8})$ there exists  
%a  sequence of positive numbers $(r_{n}(\veps))_{n\geq 1}$ with 
%$\lim_{n\rightarrow \infty}r_{n}(\veps)=0$, 
%such that  %for any integrable functions $f,g$ satisfying $f\leq_{\ee}g$,
%the support of $g$ should be contained in the interval $[0, \alpha_n\sqrt{n}]$, 
%Then, there exists a constant $\ee_0 > 0$ such that for any $\ee \in (0, \ee_0)$, 
%uniformly in $x\in [0,  n^{1/2 - \ee}]$, 
Then, there exist constants $c>0$ and $\ee_0 >0$ such that for any $\ee \in (0, \ee_0)$
and any sequence of positive numbers $(\alpha_n)_{n \geq 1}$ satisfying $\lim_{n \to \infty} \alpha_n = 0$, 
one can find a constant $c_{\ee} >0$ such that 
uniformly in $x\in [0,  \alpha_n \sqrt{n}]$ and $n\geq 1$,  the following holds: 

\noindent 1. For any measurable functions $f, g: \bb R \mapsto \bb R_+$ satisfying $f \leq_{\ee} g$ and 
$\int_{\bb R_+} g(t-\ee) (1+t) dt < \infty$, 
%$\|g\|_{1,1} < \infty$, 
%$f\leq_{\ee}g$ and $g(t) = 0$ for $t \geq n^{1/2 - \ee}$,  
%and their supports are contained in the interval  $(-\infty, o(\sqrt{n}))$,
\begin{align}\label{theorem-n3/2 001}
  \mathbb{E}\left( f\left(x+S_n \right) ;\tau_x >n\right)  
& \leq   \frac{2V(x) }{\sqrt{2\pi } \sigma^3 n^{3/2} }   \int_{0}^{  \alpha_n \sqrt{n} }  g (t-\ee)   V^*(t) dt   \notag\\
&   \quad  +  \left( c \sqrt{\ee} +  c_{\ee} \alpha_n + \frac{ c_{\ee} }{ n^{\ee} }  \right)  \frac{V(x)}{n^{3/2}}  
  \int_{\bb R_+} g(t-\ee) (1+t) dt  \notag\\
& \quad  +  c  \frac{V(x) }{n}  \int_{  \alpha_n \sqrt{n} }^{\infty}    g (t - \ee)  dt.   
 \end{align}
 
\noindent 2. For any measurable functions $f, g, h: \bb R \mapsto \bb R_+$ 
satisfying $h \leq_{\ee} f \leq_{\ee} g$ and $\int_{\bb R_+} g(t-\ee) (1+t) dt < \infty$, 
%$\|g\|_{1,1} < \infty$,  
%$h\leq_{\ee}f\leq_{\ee}g$ and $g(t) = 0$ for $t \geq n^{1/2 - \ee}$,  
%and $g(t) = 0$ for $t \geq \alpha_n \sqrt{n}$,  
\begin{align}\label{theorem-n3/2 002}
  \mathbb{E}\left( f \left(x+S_n \right) ;\tau_x >n\right)  
&   \geq   \frac{2 V(x)}{\sqrt{2\pi} \sigma^3 n^{3/2}}   \int_{0}^{ \alpha_n \sqrt{n} }   h (t + \ee) V^* (t) dt   \notag\\
 &  \quad -  \left(  c\ee^{1/12} +  c_{\ee} \alpha_n + \frac{ c_{\ee} }{ n^{\ee} }  \right)  \frac{V(x)}{n^{3/2}}  
     \int_{\bb R_+} g(t-\ee) (1+t) dt  \notag\\
& \quad   - c  \ee^{1/12} \frac{V(x) }{n}  \int_{ \alpha_n \sqrt{n} }^{\infty}   h (t - \ee)  dt.   
%& \geq \frac{ 2  V(x)}{ \sqrt{2\pi}  \sigma^3  n^{3/2} }   \int_{\bb{R}_{+}}   h (t-\ee) V^*\left( t\right) dt \notag \\ 
% & \quad - c \frac{ V(x) }{ n^{3/2} } \left( \ee^{1/4} + \frac{r_n(\ee)}{n^{1/2}} \right) \left\Vert g \right\Vert_{1,1}.   
% \qquad - c \ee^{1/6}  V(x)  \left\Vert g \right\Vert^+_{1,1}  
%  - c  \frac{V(x) }{n^{1/2}}  \left\Vert g \right\Vert^+_{1,1} 
%   (r_n(\ee)  + \alpha_n), 
 \end{align}
% where $c>0$ is a constant not depending on $\ee$, $n$, $f, g$ and  $x$.
\end{theorem}

%\subsection{Proof of Theorem \ref{theorem-n3/2-upper-lower-bounds}}

\begin{proof}
We first establish the upper bound \eqref{theorem-n3/2 001}.   
Set $m=\left[ n/2 \right]$ and $k = n-m.$ 
%For any starting point $x \in \mathbb R_+$, 
By the Markov property, we have for any $x \in \mathbb R_+$, 
\begin{align} \label{JJJ-markov property_aa}
I_n(x) : = \mathbb{E}  \left( f (x+S_n );\tau_x >n\right)  
 = \int_{\mathbb R_+} I_m(t) \mathbb{P}\left( x+S_{k}\in dt,  \tau_x >k\right).  
\end{align}
When $x<0$, let $I_n(x)=0$ for any $n \geq 1$. 
Now we are going to find an $\ee$-domination for the function $t \mapsto I_m(t)$ 
given on the right hand side of \eqref{JJJ-markov property_aa}. 
Since $f\leq_{\ee} g$, it holds that for any $t\in \bb R$ and $|v|\leq \ee$, 
\begin{align}\label{DefImtbb}  
I_m(t) =  \mathbb E  \left( f \left( t+S_{m} \right) ;\tau _{t}>m\right) \leq H_m(t+v), 
\end{align}
where %, for $t \in \bb R$, 
%$y<-\ee$, $H_m(y)=0$  and, for $y\geq-\ee$, 
\begin{align} \label{DefHmt}  
H_m(t) := \mathbb E  \left( g \left( t+S_{m} \right) ;\tau _{t+\ee}>m\right)  \mathds 1_{ \{t \geq- \ee\} },  \quad  t \in \bb R. 
\end{align}
It is easy to see that $I_m \leq_{\ee} H_m$ and that both $I_m$ and $H_m$ are integrable on $\bb R$. 
Thus we can apply the upper bound \eqref{eqt-B 001} of Theorem \ref{t-B 002} to obtain
\begin{align}\label{UpperBoundhhn32}
& I_n(x) 
  \leq   (1 + c \ee)  \frac{2V(x) }{\sqrt{2\pi } \sigma^2 k}
 \int_{\bb{R}_{+}} H_{m} \left(t\right) 
 \phi^+\left( \frac{t}{\sigma \sqrt{k}} \right) dt    \notag\\
 & \quad +  c \ee^{1/4} \frac{V(x)}{n} \int_{\mathbb R} H_m( t)  \phi \left( \frac{t}{ \ee^{1/4} \sigma  \sqrt{k} } \right) dt   
  +   c_{\ee}  \left(  \alpha_n  + n^{-\ee} \right) \frac{V(x)}{ n }\left\Vert H_{m} \right\Vert_{1}  \notag\\
&  =: (1 + c \ee) J_1 + J_2 + J_3.
\end{align}
\textit{Bound of $J_1$.} 
We use a change of variable %adopting the notation $\phi^+(y) = y e^{-\frac{y^{2}}{2}}$ 
and the duality formula (Lemma \ref{lemma-duality-lemma-2_Cor}) to get
\begin{align}\label{BoundJ1rr}
J_1 
&= \frac{2V(x) }{\sqrt{2\pi } \sigma^2 k}
\int_{\bb R_+} \mathbb E ( g(t+ S_m); \tau_{t+\ee} >m) 
\phi^+\left(\frac{t}{ \sigma \sqrt{k}}\right)  dt   \notag\\
&= \frac{2V(x) }{\sqrt{2\pi } \sigma^2  k}
\int_{\bb R_+} \mathbb E ( g(t + S_m-\ee); \tau_{t} >m) \phi^+\left(\frac{ t -\ee}{\sigma \sqrt{k}} \right) dt  \notag\\
& =  \frac{2V(x) }{\sqrt{2\pi } \sigma^2  k}  
\int_{\bb R_+}   g (t-\ee)   
 \bb E \left[ \phi^+\left( \frac{t + S_m^* - \ee }{\sigma \sqrt{k}} \right);  \tau_{t}^* > m  \right]  dt. 
\end{align}
Then we split this integral into two parts: $J_1 = J_{11} + J_{12}$, where 
\begin{align}
J_{11}  & = \frac{2V(x) }{\sqrt{2\pi } \sigma^2 k}  
\int_{  \alpha_n \sqrt{n}  }^{\infty}    g (t-\ee)   
  \bb E \left[ \phi^+\left( \frac{t + S_m^* - \ee }{\sigma \sqrt{k}} \right);  \tau_{t}^* > m  \right]  dt,   \label{DefJ11rr}\\
J_{12}  & = \frac{2V(x) }{\sqrt{2\pi } \sigma^2  k}  
\int_{0}^{ \alpha_n \sqrt{n} }   g (t-\ee)   
  \bb E \left[ \phi^+\left( \frac{t + S_m^* - \ee }{\sigma \sqrt{k}} \right);  \tau_{t}^* > m  \right]  dt.   \label{DefJ12rr}
%J_{13}  & = \frac{2V(x) }{\sqrt{2\pi } \sigma^2  k}  
%\int_{ n^{1/2 - \ee} }^{ 2 \sigma n^{1/2} }   g (t-\ee)   
%  \bb E \left[ \phi^+\left( \frac{t + S_m^* - \ee }{\sigma \sqrt{k}} \right);  \tau_{t}^* > m)  \right]  dt. 
\end{align}
\textit{Bound of $J_{11}$.}
Since the function $\phi^+$ is bounded on $\bb R$, % and $\left\Vert g \right\Vert_{1,p} < \infty$ for some $p>1$, 
%by $e^{-\frac{1}{2}}$ and $\int_{\bb R_+} (1 +t) g(t)  dt < \infty$, 
we easily get 
\begin{align}\label{BoundJ11Order32}
J_{11}  \leq  c  \frac{V(x) }{n}  \int_{ \alpha_n \sqrt{n} }^{\infty}    g (t-\ee)  dt.  
%\leq   c  \frac{V(x) }{n}  \int_{ 2 n^{1/2 - \ee} }^{\infty}    g (t)  dt.    
\end{align}
%we have 
% $\sqrt{n} \int_{\alpha_n \sqrt{n}}^{\infty} g (t) dt \to 0$
%as $n \to \infty$.
%{\color{magenta}For the second integral $\int_{\alpha_n \sqrt{n}}^{\infty}$, since
%$\int_{\bb R_+} (1 +t)^{1 + \eta} g(t)  dt < \infty$, we have 
% $\sqrt{n} \int_{\alpha_n \sqrt{n}}^{\infty} g (t) dt \to 0$
%as $n \to \infty$. }
%Hence,
%\begin{align*}%\label{}
%J_1  =  \frac{2V(x) }{\sqrt{2\pi } k}   \int_{0}^{\alpha_n \sqrt{n}}  
%   g (t-\ee)   \bb E \left[ \phi^+\left( \frac{t + S_m^* - \ee }{\sqrt{k}} \right);  \tau_{t}^* > m)  \right]  dt
%   +  (1+x) \frac{o(1)}{n^{3/2}}.  
%\end{align*}
\textit{Bound of $J_{12}$.}
Since the function $\phi^+$ is differentiable on $\bb R_+$, using integration by parts leads to
\begin{align}\label{ExpectaPhiUpp}
& \bb E \left[ \phi^+\left( \frac{t + S_m^* - \ee }{\sigma \sqrt{k}} \right);  \tau_{t}^* > m  \right]   \notag\\
& =  \int_{\bb R_+} (\phi^+)'(u)  \bb P  \left( \frac{t + S_m^* - \ee }{\sigma \sqrt{k}} > u, \,  \tau_{t}^* > m \right) du   \notag\\
& =  \int_{\bb R_+} (\phi^+)'(u)  
  \bb P  \left( \frac{t + S_m^* }{\sigma \sqrt{m}} >  u_{k,m,\ee}, \, \tau_{t}^* > m \right) du,  
\end{align}
where $u_{k,m,\ee} = \frac{ u \sigma \sqrt{k} +  \ee}{\sigma \sqrt{m}}$. 
Applying the conditioned integral limit theorem \eqref{C-CLTalphan} of Corollary \ref{Cor-CCLT-Optimal}
(with $S_m$, $\tau_t$ and $V(t)$ replaced by $S_m^*$, $\tau_t^*$ and $V^*(t)$, respectively),  
we get that uniformly in $t \in [0, \alpha_n \sqrt{n}]$ and $u \geq 0$,  
\begin{align*}%\label{}
& \left|  \bb P  \left( \frac{t + S_m^* }{\sigma \sqrt{m}} 
     >   u_{k,m,\ee}, \, \tau_{t}^* > m \right)  
-  \frac{2V^*(t)}{ \sigma \sqrt{2\pi m}} \left( 1 -  \Phi^+ \left( u_{k,m,\ee} \right) \right)  \right|   \notag\\
&  \leq   c_{\ee} \left(  \alpha_n  + n^{-\ee} \right)   \frac{V^*(t)}{ n^{1/2} }. 
\end{align*}  
Substituting this into \eqref{ExpectaPhiUpp} and using the fact that $\int_{\bb R_+} |(\phi^+)'(u)| du < \infty$, 
we obtain that uniformly in $t \in [0, \alpha_n \sqrt{n}]$, 
\begin{align}\label{J12hh01}      
\Bigg| 
& \bb E \left[ \phi^+\left( \frac{t + S_m^* - \ee }{\sigma \sqrt{k}} \right);  \tau_{t}^* > m)  \right]   \notag\\
&  -  \frac{2V^*(t)}{\sigma \sqrt{2\pi m}}  \int_{\bb R_+} (\phi^+)'(u)   
    \left[ 1 -  \Phi^+ \left(  u_{k,m,\ee} \right) \right]  du 
 \Bigg|   \leq  c_{\ee} \left(  \alpha_n  + n^{-\ee} \right)  \frac{V^*(t)}{ n^{1/2} }. 
\end{align}
Since $m=\left[ n/2 \right]$, $k = n-m$ and  $u_{k,m,\ee} = \frac{ u \sigma \sqrt{k} +  \ee}{\sigma \sqrt{m}}$, 
using integration by parts and the fact that $|(\phi^+)'| \leq c$ for some constant $c>0$, we get
\begin{align}\label{J12hhrr02}
\int_{\bb R_+} (\phi^+)'(u)   \left( 1 -  \Phi^+ \left(  u_{k,m,\ee} \right) \right)  du  
&   =  \sqrt{\frac{k}{m}} \int_{\bb R_+}  \phi^+(u)  \phi^+  \left(  u_{k,m,\ee}  \right) du  \notag\\
& \leq  \sqrt{\frac{k}{m}}  \int_{\bb R_+}  \phi^+(u)   \phi^+  \left(  \frac{u \sqrt{k}}{\sqrt{m}}  \right) du  +  \frac{c \ee}{\sqrt{n}}. 
\end{align}
By a change of variable and the fact that $k + m = n$, we see that
\begin{align*} %\label{}
& \sqrt{\frac{k}{m}} \int_{\bb R_+}  \phi^+(u)   \phi^+  \left(  \frac{u \sqrt{k}}{\sqrt{m}}  \right) du
 = \frac{1}{\sqrt{m}}  \int_{\bb{R}_{+}} 
    \phi^+ \left( \frac{y}{\sqrt{k}} \right) \phi^+ \left( \frac{y}{\sqrt{m}} \right)  dy  \notag\\
%& = \frac{1}{\sqrt{k}}  \int_{\bb{R}_{+}} \frac{y}{\sqrt{k}}e^{-\frac{y^2}{2k}} \frac{y}{\sqrt{m}}e^{-\frac{y^2}{2m}}  dy \\
&=  \frac{1}{\sqrt{m}}  \int_{\bb{R}_{+}} \frac{y^2}{\sqrt{km}}e^{-\frac{ny^2}{2km}}  dy    
  = \frac{k \sqrt{m} }{n^{3/2}}\int_{\bb{R}_{+}} y^2e^{-y^2/2}  dy  
  =  \frac{k \sqrt{2\pi  m}}{ 2 n^{3/2}} .
\end{align*}
Therefore, combining this with \eqref{DefJ12rr}, \eqref{J12hh01}    and \eqref{J12hhrr02}, 
we obtain that uniformly in $t \in [0, \alpha_n \sqrt{n}]$, 
\begin{align} \label{BoundJ12Order32}
J_{12}
 &  \leq \frac{2 V(x) }{\sqrt{2\pi } \sigma^2 k}  \int_{0}^{ \alpha_n \sqrt{n}  }   g (t-\ee)   \notag\\
&  \qquad\qquad \quad \times  \left[ \frac{ 2 V^* (t) }{ \sigma \sqrt{2\pi m} }  \frac{ k \sqrt{2\pi  m} }{ 2 n^{3/2}}  
       +    \left( c \frac{\ee}{\sqrt{n}} + c_{\ee} \alpha_n  + c_{\ee} n^{-\ee} \right) \frac{V^* (t) }{\sqrt{n}}  \right]   dt \notag\\
 &  \leq  \left( 1 +   c_{\ee} \alpha_n  +  c_{\ee} n^{-\ee} \right)  \frac{ 2 V (x)}{\sqrt{2\pi} \sigma^3 n^{3/2}}    
    \int_{0}^{ \alpha_n \sqrt{n} }   g (t-\ee) V^* (t) dt.   
%  + c \ee \frac{V(x) }{n^{3/2}}  \left\Vert g \right\Vert_{1,1}.  
\end{align}
\textit{Bound of $J_2$.} The estimate of $J_2$ (cf.\ \eqref{UpperBoundhhn32}) follows the same lines as that of $J_1$. 
Proceeding in the same way as the proof of \eqref{BoundJ1rr}, \eqref{DefJ11rr} and \eqref{DefJ12rr},  one has $J_2 = J_{21} + J_{22}$, where
\begin{align*}%\label{}
J_{21}  & =  c \ee^{1/4} \frac{V(x)}{n}
\int_{  \alpha_n \sqrt{n} }^{\infty}    g (t-\ee)   
  \bb E \left[ \phi \left( \frac{t + S_m^* - \ee }{\ee^{1/4} \sigma \sqrt{k}} \right);  \tau_{t}^* > m \right]  dt,   \notag\\
J_{22}  & =  c \ee^{1/4} \frac{V(x)}{n}
\int_{0}^{ \alpha_n \sqrt{n} }   g (t-\ee)   
  \bb E \left[ \phi \left( \frac{t + S_m^* - \ee }{\ee^{1/4} \sigma \sqrt{k}} \right);  \tau_{t}^* > m  \right]  dt.   
\end{align*}
\textit{Bound of $J_{21}$.}  As in \eqref{BoundJ11Order32}, since the function $\phi$ is bounded on $\bb R$, we have 
\begin{align}\label{BoundJ21Order32rr}
J_{21}  \leq  c \ee^{1/4}  \frac{V(x) }{n}  \int_{ \alpha_n \sqrt{n} }^{\infty}    g (t-\ee)  dt.  
%\leq   c  \ee^{1/4} \frac{V(x) }{n}  \int_{ 2 n^{1/2 - \ee} }^{\infty}    g (t)  dt.  
\end{align}
\textit{Bound of $J_{22}$.}
Similarly to \eqref{ExpectaPhiUpp}, it holds that
\begin{align*}%\label{}
 \bb E \left[ \phi \left( \frac{t + S_m^* - \ee }{\ee^{1/4} \sigma \sqrt{k}} \right);  \tau_{t}^* > m  \right]   
 =  \int_{\bb R_+}  \phi'(u)  
  \bb P  \left( \frac{t + S_m^* }{\sigma \sqrt{m}} >  \tilde{u}_{k,m,\ee}, \, \tau_{t}^* > m \right) du,  
\end{align*}
where $\tilde{u}_{k,m,\ee} = \frac{ \ee^{1/4} \sigma \sqrt{k} u  +  \ee}{\sigma \sqrt{m}}$. 
Following the proof of \eqref{J12hh01}, we derive that  uniformly in $t \in [0, \alpha_n \sqrt{n}]$, 
\begin{align*}%\label{}
& \bb E \left[ \phi \left( \frac{t + S_m^* - \ee }{\ee^{1/4} \sigma \sqrt{k}} \right);  \tau_{t}^* > m  \right]   \notag\\
& \leq   \frac{2V^*(t)}{\sigma \sqrt{2\pi m}}  \int_{\bb R_+}  \phi'(u)   
    \left[ 1 -  \Phi^+ \left( \tilde{u}_{k,m,\ee} \right) \right]  du 
 +  c_{\ee} \left(  \alpha_n  + n^{-\ee} \right)  \frac{V^*(t)}{ n^{1/2} }   \notag\\
&   =   \frac{2V^*(t)}{\sigma \sqrt{2\pi m}}   \ee^{1/4} \sqrt{\frac{k}{m}}
   \int_{\bb R_+}  \phi(u)  \phi^+  \left(  \tilde{u}_{k,m,\ee} \right) du 
     +  c_{\ee} \left(  \alpha_n  + n^{-\ee} \right)  \frac{V^*(t)}{ n^{1/2} }   \notag\\
& \leq  \left( c \ee^{1/4} + c_{\ee}  \alpha_n  +  c_{\ee} n^{-\ee}  \right)  \frac{V^*(t)}{\sqrt{n}},  
\end{align*}
which implies that 
\begin{align}\label{BoundJ22Order32rr}
J_{22} \leq   \left( c \sqrt{\ee} + c_{\ee}  \alpha_n  +  c_{\ee} n^{-\ee}  \right)
   \frac{V(x)}{ n^{3/2} }   \int_{0}^{ \alpha_n \sqrt{n} }   g (t-\ee) V^*(t)  dt. 
%  +  c \ee^{1/4} \frac{ V(x) }{ n^{3/2 + \ee} }  \|g\|_{1,1}. 
\end{align}
\textit{Bound of $J_3$.}
By \eqref{DefHmt}, a change of variable and Lemma \ref{lemma Qng integr}, we have 
\begin{align} \label{BoundHmsmallx}
\left\Vert H_{m} \right\Vert_{1} 
%&= \int_{\bb R} \mathbf1_{[-\ee,\infty)}(y) \mathbb E  \left( g \left( y+S_{m} \right) ;\tau _{y+\ee}>m\right) dy \notag\\
&=\int_{\bb R_+} \mathbb E  \left( g \left( t-\ee+S_{m} \right) ;\tau _{t}>m\right) dt 
\leq \frac{c}{\sqrt{n}} \int_{\bb R_+} g(t-\ee) (1+t) dt.  
%\leq \frac{c}{\sqrt{m}} \int_{-\ee}^{\infty} g(t) (1 +t)dt. 
%\notag\\ & \leq \frac{c}{\sqrt{m}}\left\Vert g_{\ee} \right\Vert_{1,1}. %\notag\\
\end{align}
Since $m=[n/2]$, it follows from \eqref{UpperBoundhhn32} and \eqref{BoundHmsmallx} that
\begin{align}\label{BoundJ2Order32}
J_3 
%&\leq \frac{c}{k} \left( r_k(\ee) + \ee^{\frac{1}{4}} \right)  (1+x)   \frac{1}{\sqrt{m}} \int_{-\ee}^{\infty} g(y) (1 +y)dy \\
&\leq  c_{\ee}  \left(  \alpha_n  + n^{-\ee} \right) \frac{V(x)}{ n^{3/2} } \int_{\bb R_+} g(t-\ee) (1+t) dt.
\end{align}
Combining \eqref{BoundJ11Order32}, \eqref{BoundJ12Order32},  \eqref{BoundJ21Order32rr}, 
\eqref{BoundJ22Order32rr} and \eqref{BoundJ2Order32}
%we get 
%Combining \eqref{BoundJ2Order32} and \eqref{BoundJ12Order32}, we get
%\begin{align*}%\label{}
%I_n(x) \leq  \frac{2V(x) }{\sqrt{2\pi } \sigma^3  n^{3/2} }   \int_{ \bb R_+ }   g (t-\ee)   V^*(t) dt 
%    +  c \frac{V(x)}{ n^{3/2 + \ee} } \|g\|_{1,1}.  
%\end{align*}
concludes the proof of the upper bound \eqref{theorem-n3/2 001} of the theorem.

We next sketch the proof of the lower bound \eqref{theorem-n3/2 002}.
% can be proved in the same way and thus  the details are omitted.
Similarly to \eqref{UpperBoundhhn32}, we use the lower bound \eqref{eqt-B 002} of Theorem \ref{t-B 002} to get
that uniformly in $x\in [0, \alpha_n \sqrt{n}]$, 
\begin{align}\label{LowerBoundhhn32}
  I_n(x) 
 & \geq   \frac{2 V(x) }{\sqrt{2\pi } \sigma^2 k}
 \int_{\bb{R}_{+}}  L_{m} \left(t\right) 
 \phi^+\left( \frac{t}{\sigma \sqrt{k}} \right) dt    \notag\\
 & \quad  -  c \ee^{1/12}  \frac{V(x)}{n}  \int_{\mathbb R}  H_m( t)  
   \left[ \phi \left( \frac{t}{ \sigma \sqrt{k}} \right)  +  \phi^+ \left( \frac{t}{ \sigma \sqrt{k}} \right)   \right] dt   \notag\\
& \quad  -    c_{\ee}  \left(  \alpha_n  + n^{-\ee} \right) \frac{V(x)}{ n } \left\Vert  H_m \right\Vert_{1}  \notag\\
&  =:  K_1 + K_2 + K_3,  
\end{align}
where $L_m \leq_{\ee}  I_m \leq_{\ee} H_m$ with $H_m$ defined by \eqref{DefHmt} and  
\begin{align}\label{DefLmt}
L_m(t) := \mathbb E  \left( h \left( t + S_{m} \right); \tau _{t - \ee} > m \right)  \mathds 1_{ \{t \geq \ee\} },  \quad  t \in \bb R.   
\end{align}
\textit{Bound of $K_1$.} 
In the same way as in the proof of \eqref{BoundJ1rr}, 
using the duality formula (Lemma \ref{lemma-duality-lemma-2_Cor}), one has
\begin{align*}%\label{}
K_1  & =  \frac{2 V(x) }{\sqrt{2\pi } \sigma^2  k}  
\int_{\bb R_+}   h (t + \ee)   
 \bb E \left[ \phi^+\left( \frac{t + S_m^* + \ee }{\sigma \sqrt{k}} \right);  \tau_{t}^* > m  \right]  dt  \notag\\
 & \geq  \frac{2 V(x) }{\sqrt{2\pi } \sigma^2  k}  
 \int_{0}^{ \alpha_n \sqrt{n} }     h (t + \ee)   
 \bb E \left[ \phi^+\left( \frac{t + S_m^* + \ee }{\sigma \sqrt{k}} \right);  \tau_{t}^* > m  \right]  dt. 
\end{align*}
Following the proof of \eqref{BoundJ12Order32}, we get that uniformly in $x \in [0, \alpha_n \sqrt{n}]$, 
\begin{align*}%\label{}
K_1 \geq  \left( 1 - c_{\ee} \alpha_n  -  c_{\ee} n^{-\ee} \right)  \frac{2 V(x)}{\sqrt{2\pi} \sigma^3 n^{3/2}}  
    \int_{0}^{ \alpha_n \sqrt{n} }   h (t + \ee) V^* (t) dt.   
\end{align*}
\textit{Bound of $K_2$.} Following the proof of \eqref{BoundJ21Order32rr} and \eqref{BoundJ22Order32rr}, 
one can check that 
\begin{align*}%\label{}
K_2 
& \geq  - c  \ee^{1/12} \frac{V(x) }{n}  \int_{ \alpha_n \sqrt{n} }^{\infty}   h (t - \ee)  dt   \notag\\ 
& \quad     -  \left( c \ee^{1/12} + c_{\ee}  \alpha_n  +  c_{\ee} n^{-\ee}  \right) 
   \frac{V(x)}{ n^{3/2} }   \int_{0}^{ \alpha_n \sqrt{n} }   g (t-\ee) V^*(t)  dt. 
\end{align*}
\textit{Bound of $K_3$.}  From \eqref{BoundJ2Order32} we see that
\begin{align*}%\label{BoundK3Order32}
K_3  \geq  -  c_{\ee}  \left(  \alpha_n  + n^{-\ee} \right) \frac{V(x)}{ n^{3/2} }  \int_{\bb R_+} g(t-\ee) (1+t) dt.
\end{align*}
Collecting the above bounds for $K_1$, $K_2$ and $K_3$,
and using the fact that $h(t + \ee) \leq f(t) \leq g(t-\ee)$ for any $t \in \bb R$, we get the lower bound \eqref{theorem-n3/2 002}. 
\end{proof}

%%%%%%%%%%%%%%%%%%%%%%%%%%%%%%%%%%%%%%%%%%
\subsection{Proof of Theorems \ref{Theorem-AA002} and \ref{Theorem-AA002bis}} \label{SecProofThm2}
%In this section we establish Theorem \ref{Thm_CLLT_Drift_Negative001} and the asymptotic \eqref{demoCLLT-003bb}
%using Theorem \ref{theorem-n3/2ChangeMeas}. 

%\todos{We should formulate a more general result without the condition $y = o(\sqrt{n})$}
%\todos{$\sigma_{\lambda}$ will appear in the theorem: To be added}

\begin{proof}[Proof of Theorem \ref{Theorem-AA002}]
We first give a proof of \eqref{Asymn32Smallxaa}. 
By the assumption $\int_{\bb R_+}  f(t) (1 + t)^{\gamma} dt < \infty$ for some constant $\gamma >1$, 
the asymptotic \eqref{Asymn32Smallxaa} is a consequence of Theorem \ref{theorem-n3/2-upper-lower-bounds}, 
Lemma \ref{lemma-DRI-convergence} and the Lebesgue dominated convergence theorem.

We next give a proof of \eqref{demoCLLT-003_Targetbbaa}. 
%We prove that \eqref{demoCLLT-003_Targetbb} holds uniformly in $x \in [0, n^{1/2 - \ee}]$ and $y \in [\alpha_n^{-1}, \alpha_n \sqrt{n}]$. 
Applying the upper bound \eqref{theorem-n3/2 001} of Theorem \ref{theorem-n3/2-upper-lower-bounds}
with $g = \overline f_{\delta,\ee}$ (cf.\ \eqref{Def_f_deltaee})
and $\alpha_n$ replaced by $2 \alpha_n$, 
we get that uniformly in $x \in [0, \alpha_n \sqrt{n}]$ and $y \in [\alpha_n^{-1}, \alpha_n \sqrt{n}]$, 
\begin{align}\label{n32Inequaa}
& \mathbb E \left( f(x+S_n - y); \tau_x > n \right)        \notag\\
&   \leq    \frac{2 V \left( x\right) }{\sqrt{2\pi }  \sigma^3  n^{3/2} }  
    \int_{0}^{ 2 \alpha_n \sqrt{n} }  \overline f_{\delta,\ee} (t - y - \ee)   V^*(t) dt    \notag\\
& \quad  +   \left(  c \sqrt{\ee} +  c_{\ee} \alpha_n + \frac{ c_{\ee} }{ n^{\ee} }  \right)
    \frac{V(x)}{ n^{3/2} }  
    \int_{ \bb R_+ }  \overline f_{\delta,\ee} (t - y - \ee)   (1 + t) dt   \notag\\
&  \quad  +  c  \frac{V(x) }{n}  
  \int_{ 2 \alpha_n \sqrt{n} }^{\infty}   \overline f_{\delta,\ee} (t - y - \ee)  dt. 
\end{align}
We claim that as $n \to \infty$, it holds uniformly in $y \in [\alpha_n^{-1}, \alpha_n \sqrt{n}]$ that 
\begin{align}
 \int_{0}^{ 2 \alpha_n \sqrt{n} }  \overline f_{\delta,\ee} (t - y - \ee)     V^*(t) dt
& \sim  y    \int_{-\ee}^{ \infty }  \overline f_{\delta,\ee} (u)   du,   \label{Pfn32Equiv01} \\
  \int_{ \bb R_+ }  \overline f_{\delta,\ee} (t - y - \ee)   (1 + t) dt  
 &  \sim  y   \int_{-\ee}^{ \infty }  \overline f_{\delta,\ee} (u)   du,    \label{Pfn32Equiv02} \\
  \int_{ 2 \alpha_n \sqrt{n} }^{\infty}   \overline f_{\delta,\ee} (t - y - \ee)   dt
 &  \sim  \frac{y}{\sqrt{n}} o(1).       \label{Pfn32Equiv03}
\end{align}
Indeed, by a change of variable and the fact that the function $\overline f_{\delta,\ee}$ is supported on $[-\ee, \infty)$,  we have
%and that $V_{\lambda}^*(t)/t \to \infty$ as $t \to \infty$, we get
\begin{align*}%\label{}
\int_{0}^{ 2 \alpha_n \sqrt{n} }  \overline f_{\delta,\ee} (t - y - \ee)     V^*(t) dt  
& =  \int_{-\ee}^{ 2 \alpha_n \sqrt{n} - y - \ee }  \overline f_{\delta,\ee} (u)   V^*(u + y + \ee) du   \notag\\
& =   \int_{-\ee}^{ \infty }  \overline f_{\delta,\ee} (u)     V^*(u + y + \ee) du    \notag\\
& \quad  -  \int_{ 2 \alpha_n \sqrt{n} - y - \ee }^{ \infty }  \overline f_{\delta,\ee} (u)   V^*(u + y + \ee) du.
\end{align*}
Since $V^*(t)/t \to \infty$ as $t \to \infty$ 
and $\int_{-\ee}^{ \infty }  \overline f_{\delta,\ee} (u)   (1 + u) du < \infty$, 
by the Lebesgue dominated convergence theorem,  as $n \to \infty$, the first integral is equivalent to 
$y  \int_{-\ee}^{ \infty }  \overline f_{\delta,\ee} (u) du$. 
The second integral is negligible compared with the first one, using the inequality $t \leq V^*(t) \leq c(1 + t)$. 
Hence \eqref{Pfn32Equiv01} holds and the relation \eqref{Pfn32Equiv02} can be proved in the same way. 
For \eqref{Pfn32Equiv03}, by a change of variable and the fact that $\int_{\bb R_+}  \overline f_{\delta,\ee} (u)  (1 + u) du < \infty$
and $y \in [\alpha_n^{-1}, \alpha_n \sqrt{n}]$ with $\alpha_n \sqrt{n} \to \infty$ as $n \to \infty$, 
\begin{align*}%\label{}
\int_{ 2 \alpha_n \sqrt{n} }^{\infty}   \overline f_{\delta,\ee} (t - y - \ee)  dt
& \leq   \int_{ \alpha_n \sqrt{n} - \ee }^{\infty}   \overline f_{\delta,\ee} (u)  du  \notag\\
& \leq   \frac{1}{\alpha_n \sqrt{n}}  
   \int_{ \alpha_n \sqrt{n} - \ee }^{\infty}   \overline f_{\delta,\ee} (u)  (1 + u) du  \notag\\
 &  \leq   \frac{y}{ \sqrt{n}}  o(1), 
\end{align*}
which shows \eqref{Pfn32Equiv03}. 

Therefore, from \eqref{n32Inequaa}, \eqref{Pfn32Equiv01}, \eqref{Pfn32Equiv02} and \eqref{Pfn32Equiv03}, we get
that uniformly in $x \in [0, \alpha_n \sqrt{n}]$ and $y \in  [\alpha_n^{-1}, \alpha_n \sqrt{n}]$,
\begin{align*}%\label{}
\limsup_{n \to \infty}  
\frac{ \mathbb E \left( f(x+S_n - y); \tau_x > n \right)    }{  \frac{2 y V \left( x\right) }{\sqrt{2\pi } \sigma^3 n^{3/2} }  }    
  \leq  \left( 1 + c \sqrt{\ee}    \right)  
      \int_{-\ee}^{ \infty }  \overline f_{\delta,\ee} (u)  du, 
\end{align*}
which implies the desired upper bound by taking first $\ee \to 0$ and then $\delta \to 0$, and using Lemma \ref{lemma-DRI-convergence}. 
The lower bound can be proved in the same way by using 
\eqref{theorem-n3/2 002} of Theorem \ref{theorem-n3/2-upper-lower-bounds}. 
%\eqref{theorem-n3/2ChangeMeas02} of Theorem \ref{theorem-n3/2ChangeMeas}. 
The proof of \eqref{demoCLLT-003_Targetbbaa} is complete. 
\end{proof}

%%%%%%%%%%%%%%%%%%%%%%%%%%%%%%%%%%%%%%%%%%
%\subsection{Proof of Theorem \ref{Theorem-AA002bis}}

\begin{proof}[Proof of Theorem \ref{Theorem-AA002bis}]
%We next use Theorem \ref{theorem-n3/2ChangeMeas}
%to prove that \eqref{demoCLLT-003bb} holds when $y \in  [\alpha_n^{-1}, \alpha_n \sqrt{n}]$.  
Without loss of generality, we assume that $\alpha_n \geq n^{-\ee}$, so that $\alpha_n \sqrt{n} \geq \Delta$, 
where $\Delta \in [\Delta_0,  o(y) ]$. 
%Let $y$ and $\Delta$ be in the interval $[0, n^{1/2 - \ee}]$. 
Applying \eqref{theorem-n3/2 001} with $f = \mathds 1_{[y, y + \Delta]}$, $g = \mathds 1_{[y-\ee, y + \Delta+\ee]}$
and with $\alpha_n$ replaced by $3 \alpha_n$, 
and noticing that the last integral in \eqref{theorem-n3/2 001} vanishes, 
we deduce that uniformly in $x \in [0, \alpha_n \sqrt{n}]$ and $y \in [\alpha_n^{-1}, \alpha_n \sqrt{n}]$,  
\begin{align}\label{pf-thm4-upper}
& \mathbb P \left( x + S_n \in [0, \Delta] + y, \,  \tau_x >n\right)   
   \leq    \frac{2 V \left( x\right) }{\sqrt{2\pi }  \sigma^3  n^{3/2} }  
    \int_{0}^{ \Delta + 2\ee }    V^*(u + y) du   \notag\\
& \qquad\qquad\qquad  +   \left(  c \sqrt{\ee} + c_{\ee} \alpha_n + \frac{ c_{\ee} }{ n^{\ee} }  \right)
 \frac{V(x)}{ n^{3/2} }   \int_{0}^{ \Delta + 2\ee }  (u + y + 1) du.    
\end{align}
Applying Lemma \ref{Lem_V_Ineq_aa} with $k_0 = (\frac{1}{\alpha_n})^{ \frac{1}{1 - 2 \ee}}$,  
we get that uniformly in $y \in [\alpha_n^{-1}, \alpha_n \sqrt{n}]$ and $\Delta \in [\Delta_0, o(y)]$, 
\begin{align}\label{BoundInteDelta01}
\int_{0}^{ \Delta + 2\ee }   V^*(u + y) du  
& \leq   \left(1 +  c_\ee  \alpha_n^{ \frac{\ee}{1 - 2 \ee}} \right)  \int_{0}^{ \Delta + 2\ee }   (u + y )  du
   +  \frac{c_{\ee} }{\sqrt{\alpha_n}}  (\Delta + 2\ee)     \notag\\
& \leq   \left(1 +  c_\ee  \alpha_n^{ \frac{\ee}{1 - 2 \ee}} \right)  y  (\Delta + 2\ee)
   +   \frac{c_{\ee} }{\sqrt{\alpha_n}}  (\Delta + 2\ee)     \notag\\
%& =    \left(1 +  c_\ee  \alpha_n^{ \frac{\ee}{1 - 2 \ee}} \right)  
%      \Bigg\{  - \frac{1}{\lambda}  \left[  (\Delta + \ee) e^{- \lambda (\Delta + \ee)} + \ee  e^{- \lambda \ee}  \right]   \notag\\
%& \qquad\quad
%  +  \left( y + \ee + \frac{1}{\lambda} \right) \frac{1}{\lambda} \left( e^{\lambda \ee} - e^{- \lambda (\Delta + \ee)} \right) \Bigg\}
%    +  \frac{c_{\ee}}{ \lambda \sqrt{\alpha_n} } \left( e^{\lambda \ee} - e^{- \lambda (\Delta + \ee)} \right)   \notag\\
& \sim  y  (\Delta + 2\ee) 
   \leq  \left( 1 + c \ee \right)  y  \Delta,
\end{align}
where in the last line we used the fact that $\frac{1}{ \sqrt{\alpha_n} } = o(y)$ and $\Delta_0 >0$ is a fixed constant. 
Similarly, one can check that uniformly in $y \in [\alpha_n^{-1}, \alpha_n \sqrt{n}]$ and $\Delta \in [\Delta_0, o(y)]$, 
\begin{align}\label{BoundInteDelta02}
\int_{0}^{ \Delta + 2\ee }     (u + y + 1) du
\leq  \left( 1 + c \ee \right)  y \Delta. 
\end{align}
%Using the fact that $V_{\lambda}^*(t)/t \to 1$ as $t \to \infty$,
%we get that as $y \to \infty$, 
%\begin{align*}%\label{}
%\int_{y}^{y + \Delta + 2\ee}  e^{- \lambda (t  - \ee)}    V_{\lambda}^*(t) dt
%& =  e^{-\lambda y} \int_0^{\Delta + 2 \ee} e^{- \lambda (u-\ee) } V_{\lambda}^*(u + y) du  \notag\\
%& \sim  y e^{-\lambda y} \int_0^{\Delta + 2 \ee} e^{- \lambda (u-\ee) }  du
% =  y e^{-\lambda y}  \frac{e^{\lambda \ee} - e^{- \lambda (\Delta + \ee)} }{ \lambda }. 
%\end{align*}
%In the same way, it holds that as $y \to \infty$, 
%\begin{align*}%\label{}
%\int_{y}^{y + \Delta + 2\ee}   e^{- \lambda (t - \ee)}  (1 + t) dt  
%& =  e^{-\lambda y} \int_0^{\Delta + 2 \ee} e^{- \lambda (u-\ee) } (1 + u + y) du \notag\\
%& \sim y e^{-\lambda y}  \frac{e^{\lambda \ee} - e^{- \lambda (\Delta + \ee)} }{ \lambda }. 
%\end{align*}
Therefore,  substituting \eqref{BoundInteDelta01} and \eqref{BoundInteDelta02} into \eqref{pf-thm4-upper}, 
we obtain that uniformly in $x \in [0, \alpha_n \sqrt{n}]$,  $y \in [\alpha_n^{-1}, \alpha_n \sqrt{n}]$
and $\Delta \in [\Delta_0,  o(y) ]$, 
\begin{align*}%\label{}
\limsup_{n \to \infty}  
\frac{ \mathbb P \left( x + S_n \in [0, \Delta] + y, \,  \tau_x >n \right)  }{ \Delta \frac{2 y V(x) }{\sqrt{2\pi }  \sigma^3 n^{3/2} }  }
  \leq  \left( 1 + c \sqrt{\ee} \right),
\end{align*}
which concludes the proof of the upper bound by letting $\ee \to 0$. 
The lower bound can be obtained in a similar way by using \eqref{theorem-n3/2 002}.
\end{proof}

%%%%%%%%%%%%%%%%%%%%%%%%%%%%%%%%%%%%%%%%%%%%%%%%
%%%%%%%%%%%%%%%%%%%%%%%%%%%%%%%%%%%%%%%%%%%%%%%%
%%%%%%%%%%%%%%%%%%%%%%%%%%%%%%%%%%%%%%%%%%%%%%%%
\section{Conditioned concentration bounds far from the boundary}
% CLLT for target functions with supports moving to infinity and with large starting point

The goal of this section is to establish Theorems \ref{Theorem-BB001} and \ref{Theorem-BB001-Delta}. 

\subsection{Formulation of the result}
%\subsection{Definitions and a convolution result} 
%Denote $g_{\ee}(s) : = %\frac{1}{1-2 \ee} g * \kappa_{\ee^2}(s)$. 
%Denote $g_{\ee}(s) : = g * \kappa_{\ee^2}(s)$. 
Denote 
\begin{align*}%\label{}
\psi^+(s,x) =  \frac{1}{\sqrt{2\pi}} \left( e^{-\frac{(s-x)^2}{2}}- e^{-\frac{(s+x)^2}{2}} \right) \mathds 1_{\{s \geq 0\}},   \quad  s, x \in \bb R. 
\end{align*}
Note that $\psi^+(s,x) = \psi(s,x)$ for any $s \geq 0$ and $x \in \bb R$, 
where $\psi$ is defined by \eqref{Def-Levydens}. 
The following result is a conditioned local limit theorem of order $n^{-1}$ with large starting point $x$.

\begin{theorem} \label{t-C 002}
Assume \ref{A1} and  \ref{SecondMoment}. 
Then, for any $\eta > 0$, there exist constants $c, \ee_0 >0$ such that for any $\ee \in (0, \ee_0)$, 
one can find a constant $c_{\ee} >0$ such that 
uniformly in $x\in [n^{1/2 - \ee}, \eta \sqrt{n}]$ and $n\geq 1$,  the following holds: 

%Let $x_n\to\infty$. 
%Then, for any $\veps \in (0,\frac{1}{8})$,
%% $\gamma \in (0,  \frac{\delta}{2(2+\delta)})$ and $n\geq 1$, \todos{??? what is gamma in the result} 
%the following assertions hold, uniformly in $x\in (x_n,\sqrt{qn\log n})$: 

\noindent 1. For any integrable functions $f,g: \bb R \mapsto \bb R_+$ satisfying $f\leq_{\ee}g$, 
\begin{align}\label{eqt-C 001}
& \mathbb{E}\left( f(x + S_n); \tau_x >n\right)   
 \leq  \frac{1 + c\ee}{ \sigma \sqrt{n} }  
  \int_{\bb R}   g (s)  \psi \left( \frac{ s }{\sigma \sqrt{n}}, \frac{ x }{\sigma \sqrt{n}} \right)  ds   \notag\\
& \qquad\qquad\quad  + \frac{1 + c\ee}{ \sigma \sqrt{n} }  
   \int_{\bb R}   g (s)  \left[ 1 - \Phi \left( \frac{s}{\sigma \ee^{1/4} \sqrt{n}} \right) \right]  ds
      +     \frac{c_{\ee}}{  n^{ 1/2 + \ee }}  \| g \|_1. 
% &\leq    (1 + 8\ee) 
%   \int_{\bb R _{+}}   g(\sqrt{n} s) 
%   \psi \left( s, \frac{ x }{\sqrt{n}} + \frac{1}{n^{\gamma}} \right)   ds   \notag\\
%& \quad  + \frac{c_{\ee}}{  n^{(1 + \delta) /2 - \gamma (2+\delta) } }  \| g \|_1.  
\end{align}

\noindent 2. For any integrable functions $f,g,h: \bb R \mapsto \bb R_+$ satisfying $h\leq_{\ee}f\leq_{\ee}g$,
\begin{align} \label{eqt-C 002}
&  \mathbb{E}\left( f(x + S_n); \tau_x >n\right)  
  \geq  \frac{1}{\sigma \sqrt{n}}  \int_{\bb R}  \big[ h (s) - c \ee f(s) \big]
    \psi \left( \frac{s}{\sigma \sqrt{n}}, \frac{x}{\sigma \sqrt{n}}  \right)  ds  \notag\\
 & \qquad\qquad\qquad  -  \frac{1}{ \sigma \sqrt{n} }  \int_{\bb R}   h (s) 
     \left[ 1 - \Phi \left( \frac{s}{\sigma \ee^{1/4} \sqrt{n}} \right) \right]   ds  \notag\\
& \qquad\qquad    
   -  c  \frac{\ee^{1/12}}{\sqrt{n}}    \int_{\bb R}  \big[ g(t) + g(t + \ee) \big]  \phi  \left( \frac{t}{ \sigma \sqrt{n}} \right)   dt  
  -  \frac{c_{\ee} }{ n^{ 1/2 + \ee } } \| g \|_1.  
%& \geq   \frac{1 - c\ee}{ \sigma \sqrt{n} }   \int_{\bb R _{+}}   h (s)  
%  \psi \left( \frac{s}{\sigma \sqrt{n}}, \frac{x}{\sigma \sqrt{n}}  \right)  ds  
%%%  -   c \ee  \int_{\bb R _{+}}   f(\sqrt{n} s)  \psi \left( s, \frac{ x }{\sqrt{n}}   \right)   ds  
%  \notag\\
%& \quad     
%   -  c  \frac{\ee^{1/12}}{\sqrt{n}}    \int_{\bb R} g(t)  e^{- \frac{t^2}{2 \sigma^2 n}}  dt  
%  -   \frac{c_{\ee}}{  n^{ \frac{1}{2} + \frac{\delta}{ 2(3 + \delta) } }}  \| g \|_1. 
\end{align}
%where $c>0$ is a constant not depending on $\ee$, $n$, $f, g$ and  $x$. 
\end{theorem}

%%%%%%%%%%%%%%%%%%%%%%%%%%%%%%%%%%%%%%%%%%%%%%
%%%%%%%%%%%%%%%%%%%%%%%%%%%%%%%%%%%%%%%%%%%%%%
\subsection{Proof of the upper bound} 

In this section we prove the upper bound \eqref{eqt-C 001} of Theorem \ref{t-C 002}.
For any $v>0$,  we denote
\begin{align} \label{Def-psi-v}
\psi_{v}(s,x)=\frac{1}{\sqrt{2\pi v}}
   \left( e^{-\frac{(s-x)^2}{2v}}- e^{-\frac{(s+x)^2}{2v}} \right),  \quad  s, x \in \bb R. 
\end{align} 
Then $\psi_{1} = \psi$ with $\psi$  defined by \eqref{Def-Levydens}. 
We shall make use of  the following convolution result. 
Recall that $\phi_{v}(z)=\frac{1}{\sqrt{2\pi v}}e^{-\frac{z^2}{2 v}}$, $z\in \mathbb R$,
and $\Phi$ is the standard normal distribution function on $\bb R$. 

\begin{lemma} \label{convol-phi-psi-001}
%Let $v \in (0,1)$. 
For any $v \in (0,1)$ and $s, x \in \bb R$, we have
\begin{align}\label{ConvoNormalLevy01}
\int_{\bb R} \phi_{v}(s-z) \psi_{1-v}(z, x) dz 
%=  \frac{1}{\sqrt{2\pi}} \left( e^{-\frac{(s-x)^2}{2}}- e^{-\frac{(s+x)^2}{2}} \right) 
= \psi(s,x). 
\end{align}
In addition, for any $v \in (0,1/4]$ and $s, x \in \bb R$, we have
\begin{align} \label{ConvoNormalLevy02}
\left| \int_{0}^{\infty} \phi_{v}(s-z) \psi_{1-v} (z, x) dz -  \psi (s,x)  \right|  \leq  1 - \Phi \left( \frac{s}{\sqrt{v}} \right).
\end{align} 
\end{lemma}

\begin{proof}
Since for any $z, x \in \bb R$,
\begin{align*} %\label{func-psi-002}
\psi_{1-v}(z, x) = \frac{1}{\sqrt{2 \pi (1-v)}} 
   \int_{-x}^{x}\frac{y-z}{1-v} e^{-\frac{(y-z)^2}{2 (1-v)}} dy, 
\end{align*} 
by Fubini's theorem, we derive that for any $s \in \bb R$, 
\begin{align*} %\label{}
& \int_{\bb R} \phi_{v}(s-z) \psi_{1-v}(z, x) dz  =  \int_{\bb R} \phi_{v}(z) \psi_{1-v}(s-z,x) dz    \notag\\
& = \int_{\bb R} \frac{1}{\sqrt{2\pi v}}  e^{-\frac{z^2}{2v}}  \frac{1}{\sqrt{2 \pi (1-v)}}
\left[ \int_{-x}^{x} \frac{y-s+z}{1-v} e^{-\frac{(y-s+z)^2}{2(1-v)}} d y \right] dz  \notag\\
& = \frac{1}{\sqrt{2 \pi} (1-v)} \int_{-x}^{x}  \bigg[ \int_{\bb R}  \frac{z+y-s}{\sqrt{2\pi v (1-v) }}   
 e^{-\frac{z^2}{2 v} -\frac{(z+ y-s)^2}{2(1-v)}} dz \bigg] dy. 
\end{align*} 
Since
%By elementary calculations, one can verify that
\begin{align*} %\label{}
\frac{z^2}{2v} + \frac{(z + y-s)^2}{2(1-v)} = \frac{(z+(y-s)v)^2}{2 v (1-v)} + \frac{(y-s)^2}{2}, 
  \end{align*} 
we obtain
%Therefore,
\begin{align*} %\label{}
& \int_{\bb R} \phi_{v}(s-z) \psi_{1-v}(z, x) dz  \notag\\
&= \frac{1}{\sqrt{2 \pi } (1-v) }  \int_{-x}^{x} \bigg[ \int_{\bb R}  \frac{z + y-s}{\sqrt{2\pi v (1-v) }}   
e^{-\frac{(z+(y-s) v)^2}{2 v(1- v)}} dz \bigg] e^{-\frac{(y-s)^2}{2}} dy  \notag\\ 
%&=  \frac{1}{\sqrt{2 \pi (1-v)}}  \int_{-x}^{x} \bigg[ 
%\int_{\bb R}  \frac{z}{\sqrt{2\pi v}(1-v)} e^{-\frac{(z+(y-s)v)^2}{2v(1-v)}} dz \notag \\
%&\qquad\qquad\qquad  + \frac{y-s}{\sqrt{2\pi v}(1-v)} \int_{\bb R}   e^{-\frac{(z+(y-s)v)^2}{2v(1-v)}} dz\bigg] 
%e^{-\frac{(y-s)^2}{2}} dy  \notag\\ 
&= \frac{1}{\sqrt{2 \pi } (1-v) } 
    \int_{-x}^{x} \bigg[ -(y-s)v + (y-s) \bigg] e^{-\frac{(y-s)^2}{2}} dy  \notag\\ 
&=  \frac{1}{\sqrt{2\pi}} \left( e^{-\frac{(s-x)^2}{2}}- e^{-\frac{(s+x)^2}{2}} \right) =  \psi(s,x), 
%\frac{1}{\sqrt{2 \pi}} \int_{-x}^{x} (y-s) e^{-\frac{(y-s)^2}{2}} dy, 
% \notag\\ &=\sqrt{1-\delta} \psi(s,t). 
\end{align*} 
which ends the proof of \eqref{ConvoNormalLevy01}.

%The first inequality in \eqref{ConvoNormalLevy02} is a consequence of \eqref{ConvoNormalLevy01} 
%by using the fact that $\tilde{\psi}_{1-v}(z, x) = \psi_{1-v}(z, x)$ for $z \geq 0$, 
%and $\tilde{\psi}_{1-v}(z, x) < 0$ for $z <0$ and $x \geq 0$. 
To prove \eqref{ConvoNormalLevy02}, using \eqref{ConvoNormalLevy01} and the fact that $\psi_{1-v}(z, x) = - \psi_{1-v}(-z, x)$, 
we get that for any $s \in \bb R$, 
\begin{align*}%\label{}
& \int_{0}^{\infty} \phi_{v}(s-z) \psi_{1-v}(z, x) dz   \notag\\
& = \int_{\bb R} \phi_{v}(s-z) \psi_{1-v}(z, x) dz -  \int_{-\infty}^0 \phi_{v}(s-z) \psi_{1-v}(z, x) dz   \notag\\
& = \psi (s,x) + \int_{\bb R_+}  \phi_{v}(s + z)  \psi_{1-v}(z, x) dz. 
\end{align*}
Since $\psi_{1-v}(z, x) < 1$ for any $v \in (0, 1/4]$, 
and $\int_{\bb R_+}  \phi_{v}(s + z) dz = 1 - \Phi (\frac{s}{\sqrt{v}})$ for $s \in \bb R$,
the inequality \eqref{ConvoNormalLevy02} follows. 
\end{proof}

Now we give a proof of the upper bound \eqref{eqt-C 001} of Theorem \ref{t-C 002}. 

\begin{proof}[Proof of \eqref{eqt-C 001}]
As in the proof of \eqref{eqt-B 001}, 
let $\veps > 0$ be a sufficiently small constant and $\delta = \sqrt{\veps}$, 
and set $m=\left[ \delta n \right]$ and $k = n-m$. 
It suffices to prove \eqref{eqt-C 001} only for sufficiently large $n$, 
otherwise the bound becomes trivial. 
Similarly to the proof of \eqref{JJJ-markov property}, \eqref{JJJJJ-1111-001}, \eqref{JJJ004}, \eqref{JJJ006} and \eqref{ApplCondLT-001}, 
one can use the Markov property and the local limit theorem (Theorem \ref{LLT-general}) to get
\begin{align} \label{Bound_In101}
I_{n}(x) : = \mathbb{E} \left( f (x+S_n ); \tau_x >n \right) 
& \leq  (1 + c \ee) J_{n}(x)  +  \frac{c_{\ee}}{ n^{1/2 + \ee} }   \left\Vert  g \right\Vert_{1} \mathbb{P}\left(\tau_{x}>k\right)  \notag\\
& \leq   (1 + c \ee) J_{n}(x)  +  \frac{c_{\ee}}{ n^{1/2 + \ee} }   \left\Vert  g \right\Vert_{1},  
\end{align}
where %$r_{m}(\ee) \to 0$ as $m \to \infty,$ and 
\begin{align} 
J_{n}(x) &: = \int_{\bb R _{+}}  \varphi'_n(t) 
    \mathbb{P}\left(  \frac{x+S_{k}}{\sigma \sqrt{k}} > t, \tau_{x}>k\right) dt,  \label{Bound_Jnx01}  \\
\varphi_n(t) 
& := \int_{\mathbb R} g (\sigma \sqrt{k} s) \frac{1}{\sqrt{m/k}}\phi \left( \frac{t-s}{\sqrt{m/k}}\right) ds.  \label{Bound_In102}
\end{align}
%%% Denote $U(x) = 2 \Phi(x) -1 $ and 
%%% \begin{align*}%\label{}
%%%\psi^+(s, x) =  \frac{1}{\sqrt{2\pi } U(x) } 
%%%  \left( e^{- \frac{(s-x)^2}{2}} -  e^{- \frac{(s+x)^2}{2}}  \right) \mathds 1_{\{ s \geq 0 \} }. 
%%%\end{align*}
%%%Note that for any $x \geq 0$, we have  $ \int_{\bb R_+} \psi^+(s, x) ds  =  1.$
%%%Denote 
%%%\begin{align*}%\label{}
%%%\psi^+_{\delta}(s, x) = \frac{1}{\sqrt{2\pi \delta} } 
%%%  \left( e^{- \frac{(s-x)^2}{2 \delta}} -  e^{- \frac{(s+x)^2}{2 \delta}}  \right) \mathds 1_{\{ s \geq 0 \} }  \frac{1}{2 \Phi(x/\sqrt{\delta}) -1}. 
%%%\end{align*}
%%%We have for any $\delta >0$ and $x \geq 0$, 
%%%\begin{align*}%\label{}
%%%\int_{\bb R_+} \psi^+_{\delta}(s, x) ds =1. 
%%%\end{align*}
By the conditioned integral limit theorem (Theorem \ref{CorCCLT}), 
%(cf.\ \eqref{CCLT_Appro_Upper}), 
we have that for any $x \geq n^{1/2 - \ee }$,  
$t \geq 0$ and $n\geq 1$, %with $x_{\gamma}^+ = x + n^{1/2 - \gamma}$, 
\begin{align} \label{CLT002var1_Second}
&  \left| \mathbb{P} \left(  \frac{x+S_{k}}{\sigma \sqrt{k}} > t, \tau_{x}>k\right) - 
%U \left( \frac{x_{\gamma}^+}{\sqrt{k}} \right)  
 \int_{t}^{\infty}   \psi \left( s, \frac{x}{\sigma \sqrt{k}} \right)  ds   \right|
     \leq  \frac{c_{\ee} }{ n^{\ee} }.
\end{align}
From \eqref{Bound_Jnx01} and \eqref{CLT002var1_Second}, using integration by parts we derive that
%, %for any $x \in \mathbb R_+$,
%for any  $x \geq 0$,  
\begin{align} \label{ApplCondLT-002_largex} 
J_{n}(x)  -  % U \left( \frac{x_{\gamma}^+}{\sqrt{k}} \right)  
\int_{\bb R _{+}}  \varphi_n(t)  \psi \left( t, \frac{x}{\sigma \sqrt{k}} \right)  dt  
& \leq   \frac{c_{\ee} }{ n^{\ee} } \int_{\bb R _{+}}  | \varphi'_n (t) | dt  
 \leq  \frac{c_{\ee} }{ n^{1/2 + \ee} }  \| g \|_1,
\end{align}
where in the last inequality we used the bound \eqref{JJJ-20001}. 
By the definition of $\varphi_n$ (cf.\  \eqref{Bound_In102}), Fubini's theorem and a change of variable, we get
\begin{align} \label{Bound-large-x-ii}
&  \int_{\bb R _{+}}  \varphi_n(t) 
     \psi \left( t, \frac{x}{\sigma \sqrt{k}} \right)  dt    \notag\\
&=   \int_{\bb R _{+}}  \left[ \int_{\mathbb R} g (\sigma \sqrt{k} s) \frac{1}{\sqrt{m/k}}\phi \left( \frac{s-t}{\sqrt{m/k}}\right)    
   \psi \left( t, \frac{x}{\sigma \sqrt{k}} \right)     ds  \right] dt  \notag\\
&=   \int_{\bb R }   g  (\sigma \sqrt{n} s)
 \left[ \int_{\mathbb R_+} \frac{1}{\sqrt{m/n}}\phi \left( \frac{s - t}{\sqrt{m/n}}\right)   
\frac{1}{\sqrt{k/n}}   \psi \left( \frac{t}{\sqrt{k/n}}, \frac{ x }{\sigma \sqrt{k}} \right) dt \right] ds  \notag\\
  &=   \int_{\bb R }   g  (\sigma \sqrt{n} s)
\left[ \int_{\mathbb R_+}  \phi_{\delta_n} \left( s - t \right)   
  \psi_{1 - \delta_n} \left( t, \frac{ x }{\sigma \sqrt{n}} \right)  dt  \right] ds,   
\end{align}
where we denoted $\delta_n=\frac{m}{n}$. 
By \eqref{ConvoNormalLevy02}, %Lemma \ref{convol-phi-psi-001}, 
we have that for any $s \in \bb R$, 
\begin{align*}%\label{}
\int_{\mathbb R_+}  \phi_{\delta_n} \left( s - t \right)   
  \psi_{1 - \delta_n} \left( t, \frac{ x }{\sigma \sqrt{n}} \right) d t
  \leq  \psi \left( s, \frac{ x }{\sigma \sqrt{n}} \right) + 1 - \Phi \left( s \sqrt{\frac{n}{m}} \right).
\end{align*}
%Notice that $\tilde{\psi}(s, x) < 0$ when $s<0$ and $x \geq 0$, and $\tilde{\psi}(s, x) = \psi(s,x)$ when $s \geq 0$ and $x \geq 0$. 
Thus, substituting this into \eqref{Bound-large-x-ii},  
%and using the fact that 
%%$\tilde{\psi} (s, x) = \psi (s, x)$ when $s , x \geq 0$ and 
%$\psi (s, x) <0$ when $s < 0$ and $x \geq 0$,  
we get that for any $x \geq n^{1/2 - \ee }$, 
\begin{align}\label{UpperBound_Paris1}
 \int_{\bb R_+ }  \varphi_n(t)  \psi \left( t, \frac{ x }{\sigma \sqrt{k}} \right)  dt 
 &  \leq  \int_{\bb R }  g  (\sigma \sqrt{n} s)  
    \left[ \psi \left( s, \frac{ x }{\sigma \sqrt{n}} \right) + 1 - \Phi \left( s \sqrt{\frac{n}{m}} \right) \right]  ds \notag\\
 & =  \frac{1}{ \sigma \sqrt{n} }  \int_{\bb R }   g (s)  \psi \left( \frac{ s }{\sigma \sqrt{n}}, \frac{ x }{\sigma \sqrt{n}} \right)  ds   
   \notag\\
 & \quad  +  \frac{1}{ \sigma \sqrt{n} }  \int_{\bb R}   g (s)  \left[ 1 - \Phi \left( \frac{s}{\sigma \sqrt{m}} \right) \right]  ds. 
\end{align}
%By a change of variable, we have 
%\begin{align}\label{UpperBound_Paris2}
% &\int_{\bb R _{+}}   g (\sqrt{n} s')  \psi \left( s', \frac{ x_{\gamma}^+ }{\sqrt{n}} \right)  ds'   
% =  \frac{ 1 }{ \sqrt{ n } }   \int_{\bb R _{+}}   g (s)  
%   \psi \left( \frac{s}{\sqrt{n}}, \frac{ x_{\gamma}^+ }{\sqrt{n}} \right)  ds.  
%%& \qquad\qquad = % \frac{1}{\sqrt{2\pi } U( x_{\gamma}^+ /\sqrt{n}) }  
%%\frac{ \sqrt{1-\delta_n} }{ \sqrt{2\pi n } }   \int_{\bb R _{+}}   g_{\ee} (s)  
%%  \left( e^{- \frac{(s - x_{\gamma}^+ )^2}{2n}} -  e^{- \frac{(s + x_{\gamma}^+ )^2}{2n}}  \right)  ds.  
%\end{align}
Putting together \eqref{Bound_In101}, \eqref{ApplCondLT-002_largex} and \eqref{UpperBound_Paris1},
and recalling that $m=[ \ee^{1/2} n ]$, 
we conclude the proof of the upper bound \eqref{eqt-C 001}  of Theorem \ref{t-C 002}. 
\end{proof}
\subsection{Proof of the lower bound} 

The aim of this section is to establish the lower bound \eqref{eqt-C 002} 
of Theorem \ref{t-C 002}.

\begin{proof}[Proof of \eqref{eqt-C 002}]
Let us keep the notation used in the proof of the upper bound \eqref{eqt-C 001}. 
Recall that 
\begin{align}\label{TypeD_LL-Decompose_In}  
I_n(x)&:=  \int_{\mathbb R_+}  \mathbb E   f ( t +S_{m}) \mathbb{P}\left( x+S_{k}\in dt, \tau_x >k\right)   \notag\\
 & \qquad -  \int_{\mathbb R_+}  \mathbb{E} \left( f (t+S_{m}); \tau _{t} \leq m\right)  
     \mathbb{P}\left( x+S_{k}\in dt,  \tau_x >k\right)  \notag\\
& = : I_{n,1}(x) - I_{n,2}(x). 
\end{align}

\textit{Lower bound of $I_{n,1}(x)$.}
By \eqref{Pf_SmallStarting_Firstthm}, we have
\begin{align}\label{TypeD_Fstar002}
\mathbb{E}f\left( t + S_{m} \right) 
& \geq  \frac{ 1 }{\sigma \sqrt{m} }\int_{\bb R } \big[ h (s) - c \ee f(s)  \big] \phi \left( \frac{s - t}{\sigma \sqrt{m}}\right) ds  
  -  \frac{c_{\ee}}{ n^{1/2 + \ee} }  \left\Vert f \right\Vert_1. 
%   \notag\\
%&  \quad  -  \frac{ c \ee }{\sigma \sqrt{m}}   \int_{\bb R }f (s) \phi \left( \frac{s - t}{\sigma \sqrt{m}}\right) ds
% -  \frac{c_{\ee}}{ n^{1/2 + \ee} }  \left\Vert f \right\Vert_1. 
\end{align}
%For the integral terms on the right hand side of \eqref{TypeD_Fstar002},
Proceeding in the same way as the proof of the upper bound \eqref{eqt-C 001}
(replacing the function $g$ by $h - c \ee f$ and using the fact that $h \leq_{\ee} f$),  
%(in the inequality \eqref{CLT002var1_Second} we use the lower bound ), 
one has 
\begin{align}\label{LowerSecondCara001}
& \int_{\mathbb R_+}  \left[  \frac{ 1 }{\sigma \sqrt{m} }
      \int_{\bb R }  \big[ h (s) - c \ee f(s)  \big]  \phi \left( \frac{s - t}{\sigma \sqrt{m}}\right) ds \right]   
    \mathbb{P}\left( x+S_{k}\in dt,  \tau_x >k\right)   \notag\\
& \geq   \frac{1}{ \sigma \sqrt{n} }  \int_{\bb R}  \big[ h (s) - c \ee f(s) \big]
    \psi \left( \frac{s}{ \sigma \sqrt{n} }, \frac{x}{\sigma \sqrt{n}}  \right)  ds    \notag\\
&\quad -  \frac{1}{ \sigma \sqrt{n} }  \int_{\bb R}  h (s)
     \left[ 1 - \Phi \left( \frac{s}{\sigma \ee^{1/4} \sqrt{n}} \right) \right]   ds  -  \frac{c_{\ee}}{ n^{1/2 + \ee} } \| f \|_1.  
\end{align}
%and 
%\begin{align}\label{LowerSecondCara002}
%& \int_{\mathbb R_+}  \left[  \frac{ c \ee }{\sigma \sqrt{m}}   \int_{\bb R }f (s) \phi \left( \frac{s - t}{\sigma \sqrt{m}}\right) ds \right]   
%    \mathbb{P}\left( x+S_{k}\in dt ;\tau_x >k\right)   \notag\\
%& \leq   \frac{c \ee}{\sigma \sqrt{n}}  \int_{\bb R _{+}}   f(s) 
%   \psi \left( \frac{s}{\sigma \sqrt{n} } , \frac{ x }{\sigma \sqrt{n} }  \right)   ds   
%    +  \frac{c_{\ee}}{ n^{1/2 + \ee} }  \| f \|_1. 
%\end{align}
Combining \eqref{TypeD_Fstar002} and \eqref{LowerSecondCara001},
%and using the fact that $h\leq_{\ee}f\leq_{\ee}g$, 
we derive that for any $x \geq n^{1/2 - \ee }$, 
\begin{align}\label{TypeD_Bound_In1}
I_{n,1}(x)  
& \geq   \frac{1}{ \sigma \sqrt{n} }  \int_{\bb R}  \big[ h (s) - c \ee f(s) \big]
    \psi \left( \frac{s}{ \sigma \sqrt{n} }, \frac{x}{\sigma \sqrt{n}}  \right)  ds    \notag\\
&\quad -  \frac{1}{ \sigma \sqrt{n} }  \int_{\bb R}  h (s) 
     \left[ 1 - \Phi \left( \frac{s}{\sigma \ee^{1/4} \sqrt{n}} \right) \right]   ds  -  \frac{c_{\ee}}{ n^{1/2 + \ee} } \| f \|_1. 
%& \geq   \frac{1}{ \sigma \sqrt{n} }  \int_{\bb R  }  \Big[ h (s) - c \ee f(s) \Big]
%  \tilde{\psi} \left( \frac{s}{ \sigma \sqrt{n} }, \frac{x}{\sigma \sqrt{n}}  \right)  ds   
%%& \quad  -   \frac{c \ee}{ \sigma \sqrt{n} }  \int_{\bb R _{+}}   f(s) 
%%   \psi \left( \frac{s}{\sigma \sqrt{n} }, \frac{ x }{\sigma \sqrt{n}}   \right)   ds   
%  -   \frac{c_{\ee}}{ n^{1/2 + \ee} }  \| f \|_1.  
\end{align}

\textit{Upper bound of $I_{n,2}(x)$.}
It has been shown in \eqref{KKK-111-001} that 
%\begin{align} \label{TypeE_KKK-111-001}
%I_{n,2}(x) & =\int_{\mathbb R} \int_{\mathbb R}  f(t) 
%\mathbb{P}\left( y+S_{m}\in dt ; \tau_y \leq m  \right) 
%\mathbb{P}\left( x+S_{k}\in dy ; \tau_x > k  \right) \notag \\
%&= K_1 + K_2,
%\end{align}
%where
%\begin{align*} %\label{}
%K_1&= \int_{0}^{ \ee^{1/6}\sqrt{n}  } \int_{\mathbb R}  f(t) 
%\mathbb{P}\left( y+S_{m}\in dt ; \tau_y \leq m  \right) \mathbb{P}\left( x+S_{k}\in dy ;  \tau_x > k  \right), \\
%K_2&= \int_{\ee^{1/6}\sqrt{n} }^\infty \int_{\mathbb R}  f(t) \mathbb{P}\left( y+S_{m}\in dt ; \tau_y \leq m  \right) 
%\mathbb{P}\left( x+S_{k}\in dy ; \tau_x > k  \right).
%\end{align*}
\begin{align} \label{TypeE_KKK-111-001}
I_{n,2}(x) 
%& =\int_{\mathbb R_+} \left[ \int_{\mathbb R}  f(u) 
%\mathbb{P}\left( t+S_{m}\in du ; \tau_t \leq m  \right)  \right]
%\mathbb{P}\left( x+S_{k}\in dt ; \tau_x > k  \right) \notag \\
&= K_1 + K_2,  
\end{align}
where $K_1$ and $K_2$ are given by \eqref{KKK-111-001}. 
%where
%\begin{align*} %\label{}
%K_1&=\int_{0}^{ \ee^{1/6}\sqrt{n}} 
% \left[ \int_{\mathbb R}  f(u)  \mathbb{P}\left( t+S_{m}\in du ; \tau_t \leq m  \right)  \right]
%\mathbb{P}\left( x+S_{k}\in dt ;  \tau_x > k  \right), \\
%K_2&=\int_{\ee^{1/6}\sqrt{n}}^\infty 
%\left[ \int_{\mathbb R}  f(u) \mathbb{P}\left( t+S_{m}\in du ; \tau_t \leq m  \right)  \right]
%\mathbb{P}\left( x+S_{k}\in dt ; \tau_x > k  \right).
%\end{align*}
%For $K_1$, using the local limit Theorem \ref{LLT-general}, we have 
%\begin{align*} %\label{}
%K_1  
%& \leq  \int_{0}^{\ee^{1/6}\sqrt{ n}}  \left[ \int_{\bb R}  f(u) \mathbb{P}\left( t + S_{m} \in du  \right) \right]
%\bb{P} \left( x+S_{k}\in dt ;   \tau_x > k  \right) \\
%& \leq c \frac{r_{n}(\ee) + 1 }{\sqrt{m}}   \left\Vert g \right\Vert_1 
%\int_{0}^{\ee^{1/6}\sqrt{n}}  \bb{P}\left( x+S_{k}\in dy ;  \tau_x > k  \right) \\
%& \leq  c \frac{r_{n}(\ee) + 1 }{\sqrt{m}}   \left\Vert g \right\Vert_1 
%\bb{P} \left( \frac{x+S_{k}}{\sqrt{k}}\leq \ee^{1/6} \sqrt{\frac{n}{k}},  \tau_x > k  \right).  
%\end{align*}
For $K_1$, it is shown in \eqref{CaraOrder12BoundK1} and \eqref{CaraOrder12BoundK1aa} that 
%using the local limit Theorem \ref{LLT-general} and Theorem \ref{Theor-IntegrLimTh}, we get 
\begin{align} \label{K1iii}
K_1 
& \leq  \int_{\bb R} g(u) J (u) du  + c_{\ee} \frac{V(x)}{ n^{1 + \ee} }  \left\Vert g \right\Vert_1,   
\end{align}
where,  with  $F_u (t) =  \frac{1}{\sigma \sqrt{m}} \phi ( \frac{u-t}{\sigma \sqrt{m}} )$, 
%\begin{align*}%\label{}
%K_{11} = \int_{0}^{\ee^{1/6}\sqrt{ n}}  \left[ \int_{\bb R} g(u)  \frac{1}{\sqrt{m}} \phi \left( \frac{u-t}{\sqrt{m}} \right) du \right] 
%    \mathbb{P}\left( x+S_{k}\in dt ;   \tau_x > k  \right) 
%\end{align*}
%By Fubini's theroem, we have 
%\begin{align*}%\label{}
%K_{11} 
%& =  \int_{\bb R} g(u)  \left[  \int_{0}^{\ee^{1/6}\sqrt{ n}}  F_u (t)  
%    \mathbb{P}\left( x+S_{k}\in dt ;   \tau_x > k  \right)   \right]   du,  
%\end{align*}
%with  $F_u (t) =  \frac{1}{\sqrt{m}} \phi ( \frac{u-t}{\sqrt{m}} ).$
%%\begin{align*}%\label{}
%%F_u (t) =  \frac{1}{\sqrt{m}} \phi \left( \frac{u-t}{\sqrt{m}}   \right).   
%%\end{align*}
%Then
\begin{align*}%\label{}
 J (u) 
%  & =  \int_{0}^{\ee^{1/6}\sqrt{ n}}  F_u (t)
%    \mathbb{P}\left( x+S_{k}\in dt ;   \tau_x > k  \right)   \notag\\
%& =  \int_{0}^{\infty}  F_u (t)
%    \mathbb{P}\left( x+S_{k}\in dt ;  x + S_k \leq \ee^{1/6}\sqrt{ n},  \tau_x > k  \right)    \notag\\
%& =  \int_{0}^{\infty}  F_u' (t)  \bb P  \left( x + S_k \in [0, \ee^{1/6}\sqrt{ n}], x + S_k > t, \tau_x > k \right)  dt   \notag\\
%& =  \int_{0}^{\infty}   F_u' (t)  \bb P  \left( \frac{x + S_k}{\sqrt{k}} \in \left[ \frac{t}{\sqrt{k}}, \ee^{1/6}\sqrt{ \frac{n}{k} } \right], 
%  \tau_x > k \right)  dt  \notag\\
& =  \int_{0}^{ \ee^{1/6}\sqrt{ n }  }  
F_u' (t)  \bb P  \left( \frac{x + S_k}{\sigma \sqrt{k}} \in \left[ \frac{t}{\sigma \sqrt{k}},  \frac{\ee^{1/6}}{\sigma} \sqrt{ \frac{n}{k} } \right], 
  \tau_x > k \right)  dt.  
\end{align*}
Using Theorem \ref{CorCCLT},  we get that for any $x \geq n^{1/2 - \ee }$, 
\begin{align*}%\label{}
& \bb P  \left( \frac{x + S_k}{\sigma \sqrt{k}} \in \left[ \frac{t}{\sigma \sqrt{k}},  \frac{\ee^{1/6}}{\sigma} \sqrt{ \frac{n}{k} } \right],  \tau_x > k \right)  
 \leq   G(t)   +  \frac{c_{\ee} }{ n^{ \ee } }.  
\end{align*}
where 
\begin{align*}%\label{}
G(t) =   \int_{ \frac{t}{\sigma \sqrt{k}} }^{ \frac{\ee^{1/6}}{\sigma} \sqrt{ \frac{n}{k} }  } 
   \psi \left( s, \frac{x}{\sigma \sqrt{k} }   \right) ds.  
\end{align*}
%Since  $F_u' (t) \geq 0$ for any $t \in \bb R$, 
Since $G( \ee^{1/6}\sqrt{ n } ) = 0$ and $F_u ( 0 ),  G( 0 ) \geq 0$,  
it follows that 
\begin{align}\label{BoundKu88}
J (u) 
& \leq  F_u ( \ee^{1/6}\sqrt{ n } )   G( \ee^{1/6}\sqrt{ n } )  -  F_u ( 0 )   G( 0 )   \notag\\
& \quad  +  \frac{1}{\sigma \sqrt{k}}   \int_{0}^{ \ee^{1/6}\sqrt{ n }  }   F_u (t)   
     \psi \left( \frac{t}{\sigma \sqrt{k}}, \frac{x}{\sigma \sqrt{k} }   \right) dt  
      +  \frac{c_{\ee} }{ n^{ \ee } }  \int_{0}^{ \ee^{1/6}\sqrt{ n }  }   |F_u'(t)|  dt  \notag\\
& \leq  \frac{1}{\sigma \sqrt{k}}   \int_{0}^{ \ee^{1/6}\sqrt{ n }  }   F_u (t)   
     \psi \left( \frac{t}{\sigma \sqrt{k}}, \frac{x}{\sigma \sqrt{k} }   \right) dt 
       +  \frac{c_{\ee} }{ n^{ 1/2 + \ee } }, 
\end{align}
where in the last line we used the fact that $\int_{0}^{ \ee^{1/6}\sqrt{ n }  }   |F_u'(t)|  dt \leq \frac{c}{\sqrt{m}}$. 
By elementary calculations,  we see that
\begin{align*}%\label{}
\mathcal J(u): & = \int_{0}^{ \ee^{1/6}\sqrt{ n }  }  
    F_u \left(t\right)  \psi \left( \frac{t}{\sigma \sqrt{k}}, \frac{x}{\sigma \sqrt{k} }   \right)   dt   \notag\\
& =  \frac{1}{2 \pi} \int_{0}^{ \ee^{1/6}\sqrt{ n} } \frac{1}{\sigma \sqrt{ m}} e^{- \frac{(u-t)^2}{2 \sigma^2 m}} 
   \left(  e^{- \frac{(t-x)^2}{2 \sigma^2 k}}  -  e^{- \frac{(t+x)^2}{2 \sigma^2 k}}   \right) dt   \notag\\
& = \frac{1}{2 \pi}  e^{- \frac{u^2}{2 \sigma^2 m} - \frac{x^2}{2 \sigma^2 k} }  
  \int_{0}^{ \ee^{1/6}\sqrt{ n} }   \frac{1}{\sigma \sqrt{ m}}  
     e^{- \frac{n t^2}{2 \sigma^2 mk }  + \frac{tu}{\sigma^2 m}}  e^{\frac{tx}{ \sigma^2 k }}  \left(1  - e^{ - \frac{2 tx}{ \sigma^2 k }} \right)  dt. %  \notag\\
%& =  e^{- \frac{u^2}{2m}}  e^{ \frac{k u^2}{2mn}}  \int_{0}^{ \ee^{1/6}\sqrt{ n} }   \frac{1}{\sqrt{2 \pi m}}  \frac{t}{\sqrt{k}}
%    e^{ - \frac{n}{2mk} (t - \frac{ku}{n})^2 } dt  \notag\\
% & = e^{- \frac{u^2}{2n}}  \int_{0}^{ \ee^{1/6}\sqrt{ n} }   \frac{1}{\sqrt{2 \pi m}}  \frac{t}{\sqrt{k}}
%    e^{ - \frac{n}{2mk} (t - \frac{ku}{n})^2 }  dt. 
\end{align*}
Since $x \in [n^{1/2 - \ee}, \eta \sqrt{n}]$ and $t \in [0, \ee^{1/6}\sqrt{ n} ]$, 
we have $e^{\frac{tx}{\sigma^2  k }} \leq c$ for some constant $c>0$.
Using the fact that $e^{- \frac{x^2}{2 \sigma^2 k} } \leq 1$
and the inequality $1 - e^{-t} \leq t$, $t \geq 0$, we get
\begin{align*}%\label{}
\mathcal J(u) 
& \leq  c    e^{- \frac{u^2}{2 \sigma^2 m} } 
  \int_{0}^{ \ee^{1/6}\sqrt{ n} }   \frac{t}{ \sqrt{k m}}  
     e^{- \frac{n t^2}{2 \sigma^2 mk }  + \frac{tu}{\sigma^2  m}}     dt  \notag\\
& =  c  e^{- \frac{u^2}{2 \sigma^2 m}}  e^{ \frac{k u^2}{2 \sigma^2 mn}}  
   \int_{0}^{ \ee^{1/6}\sqrt{ n} }     \frac{t}{ \sqrt{km}}
    e^{ - \frac{n}{2 \sigma^2 mk} (t - \frac{ku}{n})^2 } dt  \notag\\
 & =  c   e^{- \frac{u^2}{2 \sigma^2 n}}  
   \int_{0}^{ \ee^{1/6}\sqrt{ n} }   \frac{t}{ \sqrt{k m}}
    e^{ - \frac{n}{2 \sigma^2 mk} (t - \frac{ku}{n})^2 }  dt. 
\end{align*}
By a change of variable $\frac{\sqrt{n}}{\sigma \sqrt{mk}} t  = z$, it follows that 
\begin{align*}%\label{}
\mathcal J(u)  & \leq  c  e^{- \frac{u^2}{2 \sigma^2 n}}   \frac{ \sqrt{km}}{n}   \int_{0}^{ \ee^{1/6} n/\sqrt{\sigma^2 mk} }   z     
    e^{ - \frac{1}{2} (z- \frac{\sqrt{k}}{\sigma  \sqrt{nm}} u)^2 }  dz   \notag\\
 & \leq  c  e^{- \frac{u^2}{2 \sigma^2 n}}  \frac{ \sqrt{km}}{n}   \int_{0}^{ \ee^{1/6} n/\sqrt{\sigma^2 mk} }   z  dz  
  \leq  c  \ee^{1/3}  \frac{n}{\sqrt{km }}   e^{- \frac{u^2}{2 \sigma^2 n}}     \leq  c  \ee^{1/12}   e^{- \frac{u^2}{2 \sigma^2 n}}.  
\end{align*}
In view of \eqref{BoundKu88}, this implies that
\begin{align*}%\label{}
J (u)  
& \leq  c  \frac{\ee^{1/12}}{\sqrt{n}}   e^{- \frac{u^2}{2 \sigma^2 n}}  
    +  \frac{c_{\ee} }{ n^{ 1/2 + \ee } }.  
\end{align*}
%where in the last inequality we used the fact that 
%$F_u (\ee^{1/6}\sqrt{ n }) - F_u (0) \leq \frac{1}{\sqrt{m}} \leq \frac{c_{\ee} }{\sqrt{n}}$. 
Implementing this bound into \eqref{K1iii}, we obtain
\begin{align}\label{BoundK11hh}
K_{1} 
\leq   c  \frac{\ee^{1/12}}{\sqrt{n}}    \int_{\bb R} g(t)  e^{- \frac{t^2}{2 \sigma^2 n}}  dt  
  +  \frac{c_{\ee} }{ n^{ 1/2 + \ee } }  \|g\|_1.  
\end{align}

We proceed to give an upper bound for $K_{2}$ (cf.\ \eqref{K2-b01c-001}). 
%\begin{align*} %\label{}   
%K_{2} &= \int_{\bb R}  L(y)  \mathbb{P}\left( x+S_{k}\in dy ; \tau_x > k  \right),  %\label{K2-b01c-001}
%\end{align*}
%where
%\begin{align*} %\label{Bytheduality-001}
%L(y) : =  \mathds 1_{\{ y> \ee^{1/6}\sqrt{n}  \}}
%   \mathbb{E}  (f(y + S_{m}); \tau_{y} \leq m ). 
%\end{align*}
%Recall that 
%\begin{align*} %\label{}
%M(y) :=  
%\mathds 1_{\{ y+ \ee > \ee^{1/6}\sqrt{n}  \}}
%\mathbb{E}  \left( g(y + S_{m});   \tau_{y-\ee} \leq m   \right) 
%\end{align*}
%is an integrable upper $\ee$-envelope of $L$.  
%%Denote $M_{\ee} = \frac{1}{1-2\ee} M * \kappa_{\ee^2}$.  
Applying the upper bound \eqref{eqt-C 001}  of Theorem \ref{t-C 002},  we obtain 
\begin{align}\label{TypeE_eqt-A 001_Lower}
K_{2}  
& \leq   \frac{ 1 + c \ee }{\sigma \sqrt{  k }}  
   \int_{\bb R_+}   M(s)  \psi \left( \frac{s}{\sigma \sqrt{k}}, \frac{ x }{\sigma \sqrt{k}} \right)   ds  \notag\\
& \quad + \frac{ 1 + c \ee }{\sigma \sqrt{  k }}  \int_{\bb R_+}   M(s)  \left[ 1 - \Phi \left( \frac{s}{\sigma \ee^{1/4} \sqrt{k}} \right) \right]   ds  
    +   \frac{c_{\ee} }{ n^{ 1/2 + \ee } }   \| M \|_{1},  
\end{align} 
where $s \mapsto M(s)$ is defined by \eqref{DefM88}. 
For the first term, applying the duality formula 
(\eqref{eq-duality-lemma-002_Cor_less} of Lemma \ref{lemma-duality-lemma-2_Cor}), we obtain 
\begin{align*}%\label{}
&   \int_{\bb R_+}   M(s)  \psi \left( \frac{s}{\sigma \sqrt{k}}, \frac{ x }{\sigma \sqrt{k}} \right)   ds   
  \notag\\
& =   \int_{\bb{R}_{+}}  
\mathbb{E}  \left( g(s + S_{m});   \tau_{s-\ee} \leq m   \right)
  \mathds 1_{\{ s+ \ee > \ee^{1/6}\sqrt{n}  \}} 
   \psi \left( \frac{s}{\sigma \sqrt{k}}, \frac{ x }{\sigma \sqrt{k}} \right)   ds   \notag\\
& =  \int_{\bb{R}_{+}}  
\mathbb{E}  \left( g(t + \ee + S_{m});   \tau_{t} \leq m   \right)
  \mathds 1_{\{ t+ 2\ee > \ee^{1/6}\sqrt{n}  \}}   
  \psi \left( \frac{t + \ee}{\sigma \sqrt{k}}, \frac{ x }{\sigma \sqrt{k}} \right) dt  \notag\\
& =  \int_{\bb{R}_{+}}   g(t + \ee)  
 \bb E \left[ \psi \left( \frac{t + S_m^* + \ee}{\sigma \sqrt{k}}, \frac{ x }{\sigma \sqrt{k}} \right); 
     t + S_m^* + 2\ee > \ee^{1/6}\sqrt{n}, \tau_t^* \leq m  \right]  dt. 
\end{align*}
Since the function $(t, x) \mapsto \psi(t,x)$ is bounded on $\bb R \times \bb R$, it follows that 
\begin{align}\label{BoundInteMs01}
& \int_{\bb R_+}   M(s)  \psi \left( \frac{s}{\sigma \sqrt{k}}, \frac{ x }{\sigma \sqrt{k}} \right) ds    \notag\\
& \leq c \int_{\bb{R}_{+}}   g(t + \ee)  
 \bb P \left(  t + S_m^*  > \frac{1}{2} \ee^{1/6}\sqrt{n}, \tau_t^* \leq m  \right)  dt   \notag\\
 & \leq   c  \ee^{1/12} \int_{\bb R_+}  g(t + \ee)    e^{- \frac{t^2}{16 \sigma^2 \sqrt{\ee} n}}  dt 
     +  \frac{c_{\ee}}{n^{\ee}} \|g\|_1, 
%     \left( c  \ee^{1/12}  +  \frac{c_{\ee}}{n^{\delta/4}} \right) 
%  \left[  \int_{\bb R_+}  g(t + \ee)    e^{- \frac{t^2}{8 \sqrt{\ee} n}}  dt  
%     + \frac{c_{\delta, \gamma}}{ n^{ \delta/4 - (1+\delta/2) \gamma} }  \|g\|_1  \right],  
\end{align}
where the last inequality holds due to the estimates of \eqref{CLLT-J1-6}, \eqref{CLLT-J2-6} and \eqref{BoundMt99}. 
Following the same proof of \eqref{BoundInteMs01} and using the fact that $1 - \Phi(\cdot)$ is bounded on $\bb R$, one has
\begin{align}\label{BoundInteMs02}
 \int_{\bb R_+}   M(s)  \left[ 1 - \Phi \left( \frac{s}{\sigma \ee^{1/4} \sqrt{k}} \right) \right]   ds 
 \leq  c  \ee^{1/12} \int_{\bb R_+}  g(t + \ee)    e^{- \frac{t^2}{16 \sigma^2 \sqrt{\ee} n}}  dt 
     +  \frac{c_{\ee}}{n^{\ee}} \|g\|_1.  
\end{align}
Implementing \eqref{BoundInteMs01} and \eqref{BoundInteMs02} into \eqref{TypeE_eqt-A 001_Lower} 
and using the fact that $\|M\|_1 \leq \|g\|_1$, we get 
\begin{align*}%\label{}
K_2  \leq  c   \frac{\ee^{1/12} }{\sqrt{n}} \int_{\bb R_+}  g(t + \ee)    e^{- \frac{t^2}{16 \sigma^2 \sqrt{\ee} n}}  dt 
     +  \frac{c_{\ee}}{n^{1/2 + \ee}} \|g\|_1.  
\end{align*}
Combining this with \eqref{TypeE_KKK-111-001} and \eqref{BoundK11hh}, we derive that 
\begin{align*}%\label{}
I_{n,2}(x) 
&  \leq   c  \frac{\ee^{1/12}}{\sqrt{n}}    \int_{\bb R} \Big[ g(t) + g(t + \ee) \Big]  e^{- \frac{t^2}{2 \sigma^2 n}}  dt  
  +  \frac{c_{\ee} }{ n^{ 1/2 + \ee } }   \|g\|_1.  
\end{align*}
This, together with \eqref{TypeD_Bound_In1}, 
%implies that for any $x \geq n^{1/2 - \gamma}$,  
%\begin{align}\label{TypeE_Bound_In_bb}
%I_{n}(x)  & \geq  \frac{1}{\sigma \sqrt{n}}  \int_{\bb R_+}  \Big[ h (s) - c \ee f(s) \Big]
%  \psi \left( \frac{s}{\sigma \sqrt{n}}, \frac{x}{\sigma \sqrt{n}}  \right)  ds  
%%  -   \frac{c \ee}{\sigma \sqrt{n}}  
%%  \int_{\bb R _{+}}   f(s)  \psi \left( \frac{s}{\sigma \sqrt{n}}, \frac{ x }{\sigma \sqrt{n}}   \right)   ds  
%  \notag\\
%& \quad     
%   -  c  \frac{\ee^{1/12}}{\sqrt{n}}    \int_{\bb R} g(t)  e^{- \frac{t^2}{2 \sigma^2 n}}  dt  
%  -  \frac{c_{\ee} }{ n^{ 1/2 + \ee } } \| g \|_1,  
%\end{align}
finishes the proof of the lower bound \eqref{eqt-C 002}. %  of Theorem \ref{t-C 002}.  
%Combing this with \eqref{TypeD_Bound_In1} yields the desired assertion \eqref{eqt-C 002}. 
\end{proof}

%%%%%%%%%%%%%%%%%%%%%%%%%%%%%%%%%%%%%%%%%%%%%
%%%%%%%%%%%%%%%%%%%%%%%%%%%%%%%%%%%%%%%%%%%%%
\subsection{Proof of Theorems \ref{Theorem-BB001} and \ref{Theorem-BB001-Delta}}
In this section we prove Theorems \ref{Theorem-BB001} and \ref{Theorem-BB001-Delta} by making use of 
Theorem \ref{t-C 002}.

\begin{proof}[Proof of Theorem \ref{Theorem-BB001}]
Without loss of generality, we assume that the target function $f$ is non-negative.

%We first prove that \eqref{demoCLLTLargeff} holds uniformly in 
%$x \in [\eta^{-1} \sqrt{n}, \eta \sqrt{n}]$  and $y \in [\eta \sqrt{n}, \sigma_{\lambda} \sqrt{q n \log n}]$. 
Since $f \leq_{\ee} \overline f_{\delta,\ee}$ with $\overline f_{\delta,\ee}$ defined by \eqref{Def_f_deltaee}, 
applying the upper bound \eqref{eqt-C 001} of Theorem \ref{t-C 002} with $g = \overline f_{\delta,\ee}$, 
we derive that
\begin{align*}
I_n: & =  \mathbb{E} \left( f( x + S_n - y); \, \tau_x > n \right)     \notag\\
 & \leq   \frac{ 1 +  c \ee }{\sigma \sqrt{ n }  }    
  \int_{\bb{R} }  \overline f_{\delta,\ee} (t)     
   \psi \left( \frac{t + y}{\sigma \sqrt{n}}, \frac{ x }{\sigma \sqrt{n}}  \right)    dt   \notag\\
 & \quad  +  \frac{ 1 +  c \ee }{\sigma \sqrt{ n }  }   
    \int_{\bb R}   \overline f_{\delta,\ee} (t)
  \left[ 1 - \Phi \left( \frac{t + y}{\sigma \ee^{1/4} \sqrt{n}} \right) \right]  dt 
    +  \frac{c_{\ee}}{  n^{ 1/2 + \ee }}  
    \int_{\bb R}  \overline f_{\delta,\ee} (t)    dt.  
\end{align*}
Since the function $\psi (\cdot, x)$ is Lipschitz continuous on $\bb R$, 
by elementary calculations, there exists a constant $c>0$ such that for any $t\in \bb R$,
$x \in [\eta^{-1} \sqrt{n}, \eta \sqrt{n}]$  and $y \in [\eta \sqrt{n}, \sigma \sqrt{q n \log n}]$, 
\begin{align}%\label{InequalityLipphiaa}
&\left| \psi \left( \frac{t + y}{\sigma  \sqrt{n}}, \frac{ x }{\sigma \sqrt{n}}  \right) -   
       \psi \left( \frac{y}{\sigma  \sqrt{n}}, \frac{ x }{\sigma  \sqrt{n}}  \right)  \right|
 \leq   c \frac{ |t| }{\sqrt{n}},     \label{Inequa-Psi01} \\
&  1 - \Phi \left( \frac{t + y}{\sigma \ee^{1/4} \sqrt{n}} \right)  
  \leq  e^{- c/ \sqrt{\ee}}  \psi \left( \frac{y}{\sigma \sqrt{n}}, \frac{ x }{\sigma \sqrt{n}}  \right),  
      \label{Inequa-Psi02} \\
&  \frac{c_{\ee}}{ n^{\ee} } 
 \leq  \frac{c_{\ee}}{ n^{\ee - q} }   
    \psi \left( \frac{y}{\sigma  \sqrt{n}}, \frac{ x }{\sigma  \sqrt{n}}  \right),  \label{Inequa-Psi03}
\end{align}
%Using \eqref{InequalityLipphiaa} we have
which implies that 
\begin{align*}%\label{}
I_n & \leq  \frac{ 1 +  c \ee  }{\sigma \sqrt{ n }  }      
 \psi \left( \frac{y}{\sigma  \sqrt{n}}, \frac{ x }{\sigma  \sqrt{n}}  \right)
  \left( 1 +  \frac{c_{\ee}}{ n^{\ee - q} }  \right)
  \int_{\bb{R} }  \overline f_{\delta,\ee} (t)   dt  
  +   \frac{ c }{ n }    
  \int_{\bb{R} }  \overline f_{\delta,\ee} (t)   |t| dt. 
\end{align*}
Therefore,  uniformly in $x \in [\eta^{-1} \sqrt{n}, \eta \sqrt{n}]$ and $y \in [\eta \sqrt{n}, \sigma \sqrt{q n \log n}]$,
\begin{align*}%\label{}
& \limsup_{n \to \infty}  
\frac{ \mathbb E \left( f(x+S_n - y); \tau_x > n \right)    }{  \frac{1 }{ \sigma \sqrt{n} }    
   \psi \big( \frac{y}{\sigma  \sqrt{n}}, \frac{ x }{\sigma  \sqrt{n}}  \big) }   
   \leq  (1 + c \ee )   \int_{-\ee}^{\infty} \overline f_{\delta,\ee}(t)   dt,   
\end{align*}
which proves the upper bound by
taking first $\ee \to 0$ and then $\delta \to 0$, and using Lemma \ref{lemma-DRI-convergence}.  
The proof of the lower bound can be carried out in the same way 
using the lower bound \eqref{eqt-C 002} together with the fact that 
there exists a constant $c>0$ such that 
uniformly in $x \in [\eta^{-1} \sqrt{n}, \eta \sqrt{n}]$ and $y \in [\eta \sqrt{n}, \sigma \sqrt{q n \log n}]$,
\begin{align*}%\label{}
\left| \phi \left( \frac{t + y}{\sigma  \sqrt{n}}  \right) -   
       \phi \left( \frac{y}{\sigma  \sqrt{n}}  \right)  \right|
 \leq   c \frac{ |t| }{\sqrt{n}},    \quad 
  \phi \left( \frac{y}{\sigma \sqrt{n}}  \right)  
  \leq  c \psi \left( \frac{y}{\sigma  \sqrt{n}}, \frac{ x }{\sigma  \sqrt{n}}  \right). 
\end{align*}
The proof of Theorem \ref{Theorem-BB001} is complete.
% holds uniformly in 
%$x \in [\eta^{-1} \sqrt{n}, \eta \sqrt{n}]$  and $y \in [\eta \sqrt{n}, \sigma_{\lambda} \sqrt{q n \log n}]$. 
\end{proof}

%%%%%%%%%%%%%%%%%%%%%%%%%%%%%%%%%%%%%%
%\subsection{Proof of Theorems \ref{Theorem-BB001-Delta} and \ref{ThmCaraLargeStart-bis}}

\begin{proof}[Proof of Theorem \ref{Theorem-BB001-Delta}]
Using \eqref{eqt-C 001} of Theorem \ref{t-C 002} with $f = \mathds 1_{[y, y + \Delta]}$
and $g = \mathds 1_{[y -\ee, y + \Delta + \ee]}$, 
%We apply Theorem \ref{ThmCaraLargeStarMeasure} to prove \eqref{ThmCaraLargeStartaa}.
% holds uniformly in $y \in [n^{1/2 - \ee/2}, \sigma_{\lambda} \sqrt{q n \log n}]$. 
%Since $g = \mathds 1_{[y -\ee, y + \Delta + \ee]}$ is a measurable upper $\ee$-envelope of $f = \mathds 1_{[y, y + \Delta]}$, 
% we have 
we derive that uniformly in $x \in [ \eta^{-1} \sqrt{n}, \eta \sqrt{n}]$, 
\begin{align*}%\label{}
I_n: & = \mathbb{P} \left( x + S_n \in [0, \Delta] + y, \tau_x > n \right)     \notag\\
& \leq   \frac{ 1 +  c \ee }{ \sigma \sqrt{ n } }  
   \int_{-\ee}^{\Delta + \ee} 
   \psi \left( \frac{t + y}{\sigma \sqrt{n}},   \frac{ x }{\sigma \sqrt{n}}    \right)    dt   
     \notag\\
& \quad +  \frac{ 1 +  c \ee }{\sigma \sqrt{ n }  }    
    \int_{-\ee}^{\Delta + \ee}   
       \left[ 1 - \Phi \left( \frac{t + y}{\sigma \ee^{1/4} \sqrt{n}} \right) \right]  dt 
         +  \frac{c_{\ee}}{  n^{1/2 + \ee} }   (\Delta + 2\ee). 
\end{align*}
Since the function $(s, x) \mapsto \psi(s,x)$ is Lipschitz continuous on $\bb R \times \bb R$, 
there exists a constant $c>0$ such that for any $\Delta \in [\Delta_0, n^{1/2 - \ee}]$,  
$t \in [-\ee, \Delta + \ee]$,  $x \in [ \eta^{-1} \sqrt{n}, \eta \sqrt{n}]$ and $y \in [\eta \sqrt{n}, \sigma \sqrt{q n \log n}]$,
\begin{align}\label{Inequpsiaa}
\left|   \psi \left( \frac{t + y}{\sigma \sqrt{n}}, \frac{ x }{\sigma \sqrt{n}}   \right) 
   -    \psi \left( \frac{y}{\sigma \sqrt{n}}, \frac{ x }{\sigma \sqrt{n}}  \right)  \right|
\leq  c \frac{t}{\sigma \sqrt{n}} \leq  \frac{ c }{ n^{\ee} }, 
\end{align}
which, together with the fact that $\Delta + 2\ee \leq  (1 + c \ee) \Delta$, implies that 
\begin{align*}%\label{}
\int_{-\ee}^{\Delta + \ee}  
   \psi \left( \frac{t + y}{\sigma \sqrt{n}}, \frac{ x }{\sigma \sqrt{n}}  \right)    dt
& \leq   \Delta  (1 + c \ee)   \left[ \psi \left( \frac{y}{\sigma \sqrt{n}}, \frac{ x }{\sigma \sqrt{n}}  \right) 
   +  \frac{ c }{n^{\ee}}  \right]. 
%    \frac{e^{\lambda \ee} - e^{-\lambda (\Delta + \ee)} }{ \lambda }.  
\end{align*}
Similarly, using the inequality \eqref{Inequa-Psi02}, we get
\begin{align*}%\label{}
\int_{-\ee}^{\Delta + \ee}   
       \left[ 1 - \Phi \left( \frac{t + y}{\sigma \ee^{1/4} \sqrt{n}} \right) \right]  dt
 \leq  \Delta  (1 + c \ee)   e^{- c/ \sqrt{\ee}}  \psi \left( \frac{y}{\sigma  \sqrt{n}}, \frac{ x }{\sigma  \sqrt{n}}  \right). 
\end{align*}
Therefore, taking into account \eqref{Inequa-Psi03}, 
it follows that uniformly in $\Delta \in [\Delta_0, n^{1/2 - \ee}]$,  
$x \in [ \eta^{-1} \sqrt{n}, \eta \sqrt{n}]$ and $y \in [\eta \sqrt{n}, \sigma \sqrt{q n \log n}]$,
%Since $x, y \in [n^{1/2 - \ee}, \sqrt{q n \log n}]$, it follows that 
\begin{align}\label{Uppern32Largex01}
%& \mathbb{P} \left( x+S_n  \in [0, \Delta] + y; \tau_x >n\right)     \notag\\
I_n & \leq  \Delta \frac{1 + c \ee}{ \sigma \sqrt{n} }
\left[  \psi \left( \frac{y}{\sigma \sqrt{n}}, \frac{ x }{\sigma \sqrt{n}}  \right)   
    +   \frac{c}{n^{\ee}}  + \frac{c_{\ee}}{ n^{\ee} }  \right]   \notag\\
 %   \frac{e^{\lambda \ee} - e^{-\lambda (\Delta + \ee)} }{ \lambda }   \notag\\
& \leq   \Delta \frac{1 + c \ee}{ \sigma \sqrt{n} }
      \psi \left( \frac{y}{\sigma \sqrt{n}}, \frac{ x }{\sigma \sqrt{n}}  \right)
     \left( 1  +  \frac{c_{\ee}}{ n^{\ee - q} }  \right), 
\end{align}
which implies that 
%Since $\Delta_0 >0$ is a fixed constant, it holds uniformly in $\Delta \in [\Delta_0, n^{1/2 - \ee}]$ that 
%\begin{align*}%\label{}
%e^{\lambda \ee} - e^{-\lambda (\Delta + \ee)} 
%& =   \left( 1 - e^{-\lambda \Delta} \right)   
%\left[ 1 +  \frac{ (e^{-\lambda \Delta} + e^{\lambda \ee}) (1 - e^{- \lambda \ee}) }{ 1 - e^{-\lambda \Delta} } \right]   \notag\\
%& \leq  \left( 1 - e^{-\lambda \Delta} \right)  \left( 1 + c \ee \right).  
%\end{align*}
%Therefore, combining this with \eqref{InequaeeDelta}, we get that 
uniformly in $\Delta \in [\Delta_0, n^{1/2 - \ee}]$,  
$x \in [ \eta^{-1} \sqrt{n}, \eta \sqrt{n}]$ and $y \in [\eta \sqrt{n}, \sigma \sqrt{q n \log n}]$,
\begin{align*}%\label{}
\limsup_{n \to \infty}  
\frac{ I_n }{    
    \frac{ \Delta }{\sigma \sqrt{n}}   
    \psi \left( \frac{y}{\sigma \sqrt{n}}, \frac{ x }{\sigma \sqrt{n}}  \right) }  
    \leq  1 + c \ee. 
\end{align*}
Since $\ee >0$ can be arbitrary  small, this proves the upper bound. 
%The proof of the lower bound can be carried out in the same way 
%using \eqref{ThmCaraLargeMea02}.

For the lower bound, using \eqref{eqt-C 002}
with $f = \mathds 1_{[y, y + \Delta]}$, $g = \mathds 1_{[y -\ee, y + \Delta + \ee]}$
and $h = \mathds 1_{[y + \ee, y + \Delta - \ee]}$, 
%and taking into account that $y \geq \ee$ and $\tilde{\psi}(t,x) = \psi(t,x)$ for $t, x \geq 0$, 
we get that uniformly in $x \in [ \eta^{-1} \sqrt{n}, \eta \sqrt{n}]$, 
\begin{align*}%\label{}
I_n & \geq    \frac{ 1 }{\sigma \sqrt{ n }  }    
  \int_{\ee}^{\Delta - \ee} 
   \psi \left( \frac{t + y}{\sigma \sqrt{n}},   \frac{ x }{\sigma \sqrt{n}}    \right)    dt   
    -   \frac{ c \ee }{\sqrt{ n }  }   
  \int_{0}^{\Delta} 
   \psi \left( \frac{t + y}{\sigma \sqrt{n}},   \frac{ x }{\sigma \sqrt{n}}    \right)    dt   \notag\\
 & \quad  -  \frac{ c }{ \sqrt{n}  }    
  \int_{0}^{\Delta}   
   \left[ 1 - \Phi \left( \frac{t + y}{\sigma \ee^{1/4} \sqrt{n}} \right) \right]   dt  
     -    \frac{c \ee^{1/12}}{\sqrt{n}}  
    \int_{-2 \ee}^{\Delta + \ee}    \phi \left( \frac{t + y}{\sigma \sqrt{n}} \right) dt   \notag\\
& \quad   -  \frac{c_{\ee}}{  n^{ 1/2 + \ee }}  (\Delta + 2\ee)   \notag\\
 & = : I_{n,1} - I_{n,2} - I_{n,3} - I_{n,4} - I_{n, 5}. 
\end{align*}
For $I_{n,1}$ and $I_{n,2}$, using \eqref{Inequpsiaa} and proceeding in the same way as \eqref{Uppern32Largex01}, 
we get that uniformly in $\Delta \in [\Delta_0, n^{1/2 - \ee}]$, 
%\begin{align*}%\label{}
%I_{n,1} \geq   \frac{ \Delta }{ \sigma \sqrt{n} }
%      \psi \left( \frac{y}{\sigma \sqrt{n}}, \frac{ x }{\sigma \sqrt{n}}  \right)
%     \left( 1  +  \frac{c_{\ee}}{ n^{\ee - q} }  \right)
%    (1 + c\ee).  
%\end{align*}
%%Similarly to \eqref{InequaeeDelta},  we have $( 1 + e^{-\lambda (\Delta - \ee)} )/(1 - e^{- \lambda \Delta})$ is bounded by some constant, 
%%uniformly in $\Delta \in [\Delta_0, n^{1/2 - \ee}]$. Thus
%%\begin{align}\label{InequaeeDeltaLower}
%%e^{- \lambda \ee} - e^{-\lambda (\Delta - \ee)} 
%%& =   \left( 1 - e^{-\lambda \Delta} \right)   
%%\left[ 1 -  \frac{ ( 1 + e^{-\lambda (\Delta - \ee)} ) (1 - e^{- \lambda \ee}) }{ 1 - e^{-\lambda \Delta} } \right]   \notag\\
%%& \geq  \left( 1 - e^{-\lambda \Delta} \right)  \left( 1 - c \ee \right).  
%%\end{align}
%Therefore, 
\begin{align*}%\label{}
I_{n,1} 
& \geq  \Delta  \frac{1 - c \ee }{ \sigma \sqrt{n} } 
      \psi \left( \frac{y}{\sigma \sqrt{n}}, \frac{ x }{\sigma \sqrt{n}}  \right)
     \left( 1  +  \frac{c_{\ee}}{ n^{\ee - q} }  \right),    \notag\\
%\end{align*}
%%For $I_{n,2}$, using again \eqref{Inequpsiaa} and proceeding in the same way as \eqref{Uppern32Largex01}, we get
%\begin{align*}%\label{}
I_{n,2} 
&  \leq   \Delta  \frac{c \ee }{ \sqrt{n} }  
      \psi \left( \frac{y}{\sigma \sqrt{n}}, \frac{ x }{\sigma \sqrt{n}}  \right)
     \left( 1  +  \frac{c_{\ee}}{ n^{\ee - q} }  \right).
\end{align*}
For $I_{n,3}$,  using the inequality \eqref{Inequa-Psi02} yields that 
\begin{align*}%\label{}
I_{n,3} \leq  \Delta \frac{c e^{- c/ \sqrt{\ee}} }{ \sqrt{n} }
      \psi \left( \frac{y}{\sigma \sqrt{n}}, \frac{ x }{\sigma \sqrt{n}}  \right).
\end{align*}
For $I_{n,4}$,  by \eqref{Inequa-Psi03}, 
 there exists a constant $c>0$ such that for any $\Delta \in [\Delta_0, n^{1/2 - \ee}]$,  
$t \in [-2 \ee, \Delta + \ee]$,  $x \in [ \eta^{-1} \sqrt{n}, \eta \sqrt{n}]$ and $y \in [\eta \sqrt{n}, \sigma \sqrt{q n \log n}]$,
\begin{align*}%\label{Inequpsibb}
  \phi \left( \frac{t + y}{\sigma \sqrt{n}}   \right) 
&  \leq   \phi \left( \frac{y}{\sigma \sqrt{n}}  \right)  +   \frac{ c }{ n^{\ee} }  
 \leq  \psi \left( \frac{y}{\sigma \sqrt{n}}, \frac{ x }{\sigma \sqrt{n}}  \right)
   \left( 1  +  \frac{c_{\ee}}{ n^{\ee - q} }  \right).  
\end{align*}
This, together with the fact that $\Delta + 3 \ee \leq (1 + c \ee) \Delta$,  implies that 
\begin{align*}%\label{}
I_{n,4} \leq  \Delta \frac{c \ee^{1/12} }{ \sqrt{n} }
      \psi \left( \frac{y}{\sigma \sqrt{n}}, \frac{ x }{\sigma \sqrt{n}}  \right)
     \left( 1  +  \frac{c_{\ee}}{ n^{\ee - q} }  \right).
\end{align*}
For $I_{n,5}$, using again $\Delta + 2 \ee \leq (1 + c \ee) \Delta$ and  \eqref{Inequa-Psi03}, we derive that 
uniformly in $\Delta \in [\Delta_0, n^{1/2 - \ee}]$, 
$x \in [ \eta^{-1} \sqrt{n}, \eta \sqrt{n}]$ and $y \in [\eta \sqrt{n}, \sigma \sqrt{q n \log n}]$,
\begin{align*}%\label{}
I_{n,5} & \leq  \Delta \frac{c_{\ee}}{  n^{ 1/2 + \ee }} 
  \leq   \Delta \frac{c_{\ee}}{ n^{1/2 + \ee - q} }    \psi \left( \frac{y}{\sigma \sqrt{n}}, \frac{ x }{\sigma \sqrt{n}}  \right).  
\end{align*}
Putting together the above bounds for $I_{n,1}$,  $I_{n,2}$, $I_{n,3}$, $I_{n,4}$ and $I_{n, 5}$, 
letting first $n \to \infty$ and then $\ee \to 0$, we obtain the desired lower bound. 
The proof of Theorem \ref{Theorem-BB001-Delta} is complete.  
\end{proof}

\section{Conditioned local limit theorem far from the boundary}
% CLLT with large starting point

The goal of this section is to establish 
Theorems \ref{Theorem-BB002}, \ref{Theorem-BB002bis}, \ref{Thm-Iglehart001}, \ref{Thm-Iglehart002},
\ref{Thm-LLT-taux-xsmall} and \ref{Thm-LLT-taux-xlarge}.  
Recall that $\phi^+_{v}(s) = \frac{s}{v} e^{-\frac{s^2}{2 v}} \mathds 1_{\{s \geq 0\}}$, $s\in \mathbb R$
is the Rayleigh density with scale parameter $\sqrt{v}$,
and $\psi_v$ is defined by \eqref{Def-psi-v}.
The following lemma will be used to prove Theorems \ref{Theorem-BB002} and \ref{Theorem-BB002bis}.  

\begin{lemma}\label{Lem_Densities_Raileigh_Levy}
For any  $x \geq 0$ and $v \in (0,1)$,  we have 
\begin{align*}%\label{}
\int_{\bb R_+}    \phi^+_v (s)   \psi_{1-v} \left(  s,  x \right)  ds
=  \sqrt{v} \phi^+ (x). 
%=  \frac{1}{2 \sqrt{2 }}  x  e^{- \frac{ x^2 }{4} } 
%= \frac{1}{2}  \phi^+ \left( \frac{x}{\sqrt{2}} \right).  ???
\end{align*}
\end{lemma}

\begin{proof}%\label{}
Notice that
\begin{align*}%\label{}
& \int_{\bb R_+}    \phi^+_v (s)   \psi_{1-v} \left(  s,  x \right)  ds   \notag\\
& =  \frac{1}{\sqrt{2 \pi (1-v)}} \int_{\bb R_+}    \frac{s}{ v }  e^{- \frac{s^2}{2 v}}
 \left( e^{ - \frac{ (s - x)^2 }{2 (1-v)}  } -   e^{ - \frac{ (s + x)^2 }{2 (1-v)}  }  \right)  ds     \notag\\
& =  \frac{1}{v \sqrt{2 \pi (1-v)}}  e^{- \frac{ x^2 }{2 (1-v)} }  
   \int_{\bb R_+}   s e^{- \frac{s^2}{2 v(1-v)}}  \left( e^{ \frac{s x}{1-v} } - e^{ -\frac{s x}{1-v} } \right)  ds. 
\end{align*}
For the first integral, we have that for any $x \geq 0$, 
\begin{align}\label{IntegralIdenaa}
&   \int_{\bb R_+}   s e^{- \frac{s^2}{2 v(1-v)} + \frac{s x}{1-v} }   ds
 =  e^{\frac{ v x^2}{2(1-v)} }  \int_{\bb R_+} s e^{- \frac{ (s - vx)^2 }{ 2 v(1-v) } }  ds  \notag\\
& =  e^{\frac{ v x^2}{2(1-v)} }
  \int_{ -  vx }^{\infty}  \left( t +  vx \right) e^{- \frac{t^2}{2 v(1-v)} } dt  \notag\\
%& =  e^{\frac{ v x^2}{2(1-v)} }
%\left(  \int_{ -  \frac{x}{2} }^{\infty}   v e^{- v^2} dv  
%     +  \frac{x}{2} \int_{ -  \frac{x}{2} }^{\infty}  e^{- v^2} dv  \right)  
& =  e^{\frac{ v x^2}{2(1-v)} }  \left[  v(1-v)  e^{- \frac{ v x^2}{2 (1-v)} } 
   +   vx \int_{ - vx }^{\infty}  e^{- \frac{t^2}{2 v(1-v)} }  dt \right]. 
\end{align}
For the second integral, by replacing $x$ with $-x$, we get
\begin{align}\label{IntegralIdenbb}
 \int_{\bb R_+}   s e^{- \frac{s^2}{2 v(1-v)} - \frac{s x}{1-v} }   ds
%& =  e^{\frac{ x^2}{4} }   \int_{\bb R_+} u e^{- (u +  \frac{ x }{2} )^2 }  du   \notag\\
%& =  e^{\frac{ x^2}{4} }  
%\int_{  \frac{ x }{2 } }^{\infty}  \left( v -  \frac{ x }{2 } \right) e^{- v^2} dv  \notag\\
%& =  e^{\frac{ x^2}{4} }   
%\left( \int_{  \frac{ x }{2 } }^{\infty}   v e^{- v^2} dv  
%     -  \frac{ x }{2} \int_{  \frac{ x }{2 } }^{\infty}  e^{- v^2} dv  \right)  \notag\\
 =  e^{\frac{ v x^2}{2(1-v)} }  \left[  v(1-v)  e^{- \frac{ v x^2}{2 (1-v)} } 
   -   vx \int_{ vx }^{\infty}  e^{- \frac{t^2}{2 v(1-v)} }  dt \right].      
\end{align}
Therefore, taking the difference between \eqref{IntegralIdenaa} and \eqref{IntegralIdenbb}, we obtain
\begin{align*}%\label{}
&  \int_{\bb R_+}    \phi^+_v (s)   \psi_{1-v} \left(  s,  x \right)  ds   \notag\\
%& =  \frac{1}{\sqrt{2 \pi }}  e^{- \frac{ x^2 }{2} }  \int_{\bb R_+}    u e^{-u^2}  
%   \left( e^{ u x } - e^{ -u x } \right)  du  \notag\\
&  =   \frac{1}{v \sqrt{2 \pi (1-v)}}  e^{- \frac{ x^2 }{2} } 
 \left(  vx \int_{ - vx }^{\infty}  e^{- \frac{t^2}{2 v(1-v)} }  dt  
      +   vx \int_{ vx }^{\infty}  e^{- \frac{t^2}{2 v(1-v)} }  dt  \right)   \notag\\
 & =  \frac{1}{\sqrt{2 \pi  (1-v)}} x e^{- \frac{ x^2 }{2} }     \int_{\bb R} e^{- \frac{t^2}{2 v(1-v)} }  dt   
  = \sqrt{v} \phi^+ (x), 
%  =   \frac{1}{\sqrt{2 }}  e^{- \frac{ x^2 }{4} }  \frac{ x }{2 }  = \frac{1}{2}  \phi^+ \left( \frac{x}{\sqrt{2}} \right), 
\end{align*}
which ends the proof of the lemma. 
\end{proof}

\subsection{A non-asymptotic conditioned local limit theorem} 

\begin{theorem} \label{theorem-n3/2-upper-lower-bounds_Large}
Assume \ref{A1} and  \ref{SecondMoment}. 
%Let $(\alpha_n)_{n \geq 1}$ be a sequence of positive numbers satisfying $\lim_{n \to \infty} \alpha_n = 0$. 
Then, for any $\eta >0$, there exist constants $c>0$ and $\ee_0 >0$ such that for any $\ee \in (0, \ee_0)$
and any sequence of positive numbers $(\alpha_n)_{n \geq 1}$ satisfying $\lim_{n \to \infty} \alpha_n = 0$, 
one can find a constant $c_{\ee} >0$ such that 
uniformly in $x\in [n^{1/2 - \ee}, \eta \sqrt{n}]$ and $n\geq 1$,  the following holds: 
%Let $\eta >0$ be a fixed  constant.   
%Then, there exists a constant $\ee_0 > 0$ such that for any $\ee \in (0, \ee_0)$, 
%uniformly in $x \in [n^{1/2 - \ee}, \eta \sqrt{n}]$,

\noindent 1.  For any measurable functions $f, g: \bb R \mapsto \bb R_+$ satisfying $f \leq_{\ee} g$ and 
$\int_{\bb R_+} g(t-\ee) (1+t) dt < \infty$, 
%For any non-negative integrable functions $f, g$ satisfying $f\leq_{\ee}g$
%and $g(t) = 0$ for $t \geq n^{1/2 - \ee}$,  
%%and their supports are contained in the interval  $(-\infty, o(\sqrt{n}))$,
\begin{align}\label{theorem-n3/2 001_Large}
 & \mathbb{E} \left( f( x + S_n);  \tau_x > n \right)   \notag\\
& \leq   \frac{2}{\sigma^2 \sqrt{2\pi} n }  
 \left[ \phi^+ \left( \frac{x}{\sigma \sqrt{n}} \right)  +   \left( c \ee^{1/4} +  c_{\ee} \alpha_n  \right)   \right]
  \int_{0}^{ \alpha_n \sqrt{n}  }   g (t-\ee) V^* (t) dt   \notag\\
& \quad +  \frac{c }{\sqrt{n}}  \int_{  \alpha_n \sqrt{n} }^{\infty}    g (t-\ee)  dt
   + \frac{ c_{\ee} }{n^{1 +\ee }} \int_{\bb R_+} g(t-\ee) (1+t) dt. 
 \end{align}
 
\noindent 2.  For any measurable functions $f, g, h: \bb R \mapsto \bb R_+$ 
satisfying $h \leq_{\ee} f \leq_{\ee} g$ and $\int_{\bb R_+} g(t-\ee) (1+t) dt < \infty$, 
%For any non-negative integrable functions $f,g,h$ satisfying $h\leq_{\ee}f\leq_{\ee}g$
%and $g(t) = 0$ for $t \geq n^{1/2 - \ee}$,  
%%and $g(t) = 0$ for $t \geq \alpha_n \sqrt{n}$,  
\begin{align}\label{theorem-n3/2 002_Large}
&  \mathbb{E}\left( f \left( x + S_n \right); \tau_x > n \right)   \notag\\
& \geq  \frac{2}{\sigma^2 \sqrt{2\pi} n }  
 \left[  \phi^+ \left( \frac{x}{\sigma \sqrt{n}} \right) -  \left( c \ee +  c_{\ee} \alpha_n \right) \right]  
  \int_{0}^{ \alpha_n \sqrt{n}  }   h (t+\ee)   V^*(t) dt   \notag \\ 
 & \quad  -  \frac{c }{\sqrt{n}}  \int_{  \alpha_n \sqrt{n} }^{\infty}    g (t-\ee)  dt
  -  \left( \frac{c \ee^{1/12} }{ n } +  \frac{ c_{\ee} }{n^{1 + \ee }}  \right)   \int_{\bb R_+} g(t-\ee) (1+t) dt. 
 \end{align}
% where $c>0$ is a constant not depending on $\ee$, $n$, $f, g$ and  $x$.
\end{theorem}

%\subsection{Proof of Theorem \ref{theorem-n3/2-upper-lower-bounds}}

\begin{proof}
We first prove the upper bound \eqref{theorem-n3/2 001_Large}.
Set $m=\left[ n/2 \right]$ and $k = n-m.$   
It is shown in \eqref{JJJ-markov property_aa}, \eqref{DefImtbb}  and \eqref{DefHmt} that 
\begin{align*} %\label{JJJ-markov property_aa}
I_n(x) : = \mathbb{E}  \left( f (x+S_n );\tau_x >n\right)  
 = \int_{\mathbb R_+} I_m(t) \mathbb{P}\left( x+S_{k}\in dt, \tau_x >k\right)
\end{align*}
and $I_m \leq_{\ee} H_m$, where 
\begin{align*} %\label{}
H_m(t) := \mathbb E  \left( g \left( t+S_{m} \right) ;\tau _{t+\ee}>m\right)  \mathds 1_{ \{t \geq- \ee\} },  \quad t \in \bb R. 
\end{align*}
%It is easy to see that $I_m \leq_{\ee} H_m$ and that $I_m$ and $H_m$ are integrable on $\bb R$. 
By the upper bound \eqref{eqt-C 001} of Theorem \ref{t-C 002}, 
%and using the fact that $\tilde{\psi}(t, x) = \psi(t,x)$ when $t, x \geq 0$ and $\tilde{\psi}(t, x) < 0$ when $t < 0$ and $x \geq 0$,  
we get that uniformly in $x\in [n^{1/2 - \ee}, \eta \sqrt{n}]$, 
\begin{align}\label{UpperBoundn32Largex}
I_n(x) 
\leq  (1 + c \ee)  \left( J_1 + J_2 + J_3 \right),
\end{align}
where 
\begin{align*}%\label{}
& J_1 =  \frac{1}{\sigma \sqrt{k}} \int_{ \bb{R} } H_{m} (t)  \psi \left(\frac{t}{\sigma \sqrt{k}}, \frac{x}{\sigma \sqrt{k}} \right) dt,  \notag\\
& J_2 = \frac{1}{\sigma \sqrt{k}}
 \int_{\bb{R}} H_{m} (t)  \left[ 1 - \Phi \left( \frac{t}{\sigma \ee^{1/4} \sqrt{k}} \right) \right]  dt,   \quad
 J_3 = \frac{c_{\ee}}{  n^{ 1/2 + \ee} }  \left\Vert H_{m} \right\Vert_{1}. 
\end{align*}
\textit{Bound of $J_1$.} 
As in \eqref{BoundJ1rr}, \eqref{DefJ11rr} and \eqref{DefJ12rr}, using a change of variable 
and the duality formula (Lemma \ref{lemma-duality-lemma-2_Cor}), we get
\begin{align}\label{Seconthm-Decom-J1}
J_1 
&= \frac{1}{\sigma \sqrt{k}}  \int_{- \ee}^{\infty}  \mathbb E ( g(t+ S_m); \tau_{t+\ee} >m) 
\psi \left(\frac{t}{\sigma \sqrt{k}}, \frac{x}{\sigma \sqrt{k}} \right)  dt  \notag\\
&= \frac{1}{\sigma \sqrt{k}} \int_{\bb R_+} 
   \mathbb E ( g(t + S_m-\ee); \tau_{t} >m) 
   \psi \left(\frac{ t -\ee}{\sigma \sqrt{k}},  \frac{x}{\sigma \sqrt{k}}  \right) dt  \notag\\
%&= \frac{2V(x) }{\sqrt{2\pi } k}
%\int_{\bb R_+} \phi^+\left(\frac{y-\ee}{\sqrt{k}}\right) \int_{\bb R_+}  
%g(y'-\ee)  \bb P (y + S_m \in dy'; \tau_{y} > m)  dy \\
& =  \frac{1}{\sigma \sqrt{k}}  \int_{\bb R_+}   g (t-\ee)   
 \bb E \left[ \psi \left( \frac{t + S_m^* - \ee }{\sigma \sqrt{k}},  \frac{x}{\sigma \sqrt{k}}  \right);  \tau_{t}^* > m  \right]  dt  \notag\\
& = : J_{11} + J_{12}, 
%& =  \frac{1}{\sigma \sqrt{k}}  \int_{0}^{\ee +  n^{1/2 - \ee}}   g (t-\ee)   
% \bb E \left[ \psi \left( \frac{t + S_m^* - \ee }{\sigma \sqrt{k}},  \frac{x}{\sigma \sqrt{k}}  \right);  \tau_{t}^* > m  \right]  dt,  
\end{align}
where 
\begin{align*}%\label{}
J_{11} & = \frac{1}{\sigma \sqrt{k}}  \int_{ \alpha_n \sqrt{n} }^{ \infty }   g (t-\ee)   
 \bb E \left[ \psi \left( \frac{t + S_m^* - \ee }{\sigma \sqrt{k}},  \frac{x}{\sigma \sqrt{k}}  \right);  \tau_{t}^* > m  \right]  dt,   \notag\\
J_{12}  & = \frac{1}{\sigma \sqrt{k}}  \int_{0}^{  \alpha_n \sqrt{n} }   g (t-\ee)   
 \bb E \left[ \psi \left( \frac{t + S_m^* - \ee }{\sigma \sqrt{k}},  \frac{x}{\sigma \sqrt{k}}  \right);  \tau_{t}^* > m  \right]  dt.  
\end{align*}
\textit{Bound of $J_{11}$.}
Since the function $\psi (\cdot, \cdot)$ is bounded on $\bb R \times \bb R$,  
we get 
\begin{align}\label{BoundJ11Order32Largex}
J_{11}  \leq  \frac{c }{\sqrt{n}}  \int_{ \alpha_n \sqrt{n} }^{\infty}    g (t-\ee)  dt.  
\end{align}
\textit{Bound of $J_{12}$.}
For brevity we denote $\psi'_1 (u, x) = \frac{d}{du} \psi(u,x)$. 
Following the proof of \eqref{ExpectaPhiUpp} and \eqref{J12hh01}, 
from the conditioned integral limit theorem \eqref{C-CLTalphan} of Corollary \ref{Cor-CCLT-Optimal}
and the fact that $\int_{\bb R_+} | \psi'_1 ( u,  \frac{x}{\sigma \sqrt{k}} )  | du < \infty$, 
one has that uniformly in $t \in [0, \alpha_n \sqrt{n}]$, 
%Since the function $t \mapsto \psi (t,x)$ is differentiable on $\bb R_+$,  denoting $\psi'_1 (t, x) = \frac{d}{dt} \psi(t,x)$,  we have 
%\begin{align*}%\label{ExpectaPhiUpp}
%& \bb E \left[ \psi \left( \frac{t + S_m^* - \ee }{\sigma \sqrt{k}},  \frac{x}{\sigma \sqrt{k}}  \right);  \tau_{t}^* > m)  \right]   \notag\\
%& =  \int_{\bb R_+}  \psi'_1 \left( u,  \frac{x}{\sigma \sqrt{k}} \right)  
%   \bb P  \left( \frac{t + S_m^* - \ee }{\sigma \sqrt{k}} > u;  \tau_{t}^* > m \right) du   \notag\\
%& =  \int_{\bb R_+}  \psi'_1 \left( u,  \frac{x}{\sigma \sqrt{k}} \right)  
%  \bb P  \left( \frac{t + S_m^* }{\sigma \sqrt{m}} >  \frac{ u \sigma \sqrt{k} +  \ee}{\sigma \sqrt{m}};  \tau_{t}^* > m \right) du.  
%\end{align*}
%Using the conditioned central limit theorem \eqref{C-CLTalphan} of Theorem \ref{Theor-IntegrLimTh},  
%we get that uniformly in $t \in [0, 2 \alpha_n \sqrt{n} ]$ and $u \in \bb R$,  
%\begin{align*}%\label{}
%&  \Bigg|  \bb P  \left( \frac{t + S_m^* }{\sigma \sqrt{m}} > \frac{ u \sigma \sqrt{k} +  \ee }{\sigma \sqrt{m}};  \tau_{t}^* > m \right)   \notag\\
%&  \qquad
%-  \frac{2V^*(t)}{ \sigma \sqrt{2\pi m}} 
% \left( 1 -  \Phi^+ \left( \frac{ u \sigma \sqrt{k} +  \ee }{\sigma \sqrt{m}} \right) \right)  \Bigg| 
% \leq   c  \frac{1 + t}{ m^{1/2 + \ee} }. 
%\end{align*}
%Therefore, using the fact that $\int_{\bb R_+} | \psi'_1 ( u,  \frac{x}{\sigma \sqrt{k}} )  | du < \infty$, we obtain 
\begin{align*}%\label{}
\bigg| 
& \bb E \left[ \psi \left( \frac{t + S_m^* - \ee }{\sigma \sqrt{k}},  \frac{x}{\sigma \sqrt{k}}  \right);  \tau_{t}^* > m)  \right]   \notag\\
& -  \frac{2V^*(t)}{\sigma \sqrt{2\pi m}}  \int_{\bb R_+}  \psi'_1 \left( u,  \frac{x}{\sigma \sqrt{k}} \right) 
    \left( 1 -  \Phi^+ (  u_{k,m,\ee} ) \right)  du 
 \bigg|    \leq  c_{\ee} \left(  \alpha_n  + n^{-\ee} \right)  \frac{V^*(t)}{ n^{1/2} },
\end{align*}
where $u_{k,m,\ee} = \frac{ u \sigma \sqrt{k} +  \ee}{\sigma \sqrt{m}}$. 
Using integration by parts and the fact that the function $(\phi^+)'$ is Lipschitz on $\bb R_+$
and $\int_{\bb R_+} \psi ( u,  \frac{x}{\sigma \sqrt{k}} ) du < \infty$ uniformly in $x \in [n^{1/2 - \ee}, \eta \sqrt{n}]$, 
we get
\begin{align*}%\label{}
& \int_{\bb R_+} \psi'_1 \left( u,  \frac{x}{\sigma \sqrt{k}} \right)
    \left( 1 -  \Phi^+ (  u_{k,m,\ee} ) \right)  du  \notag\\
&   =  \sqrt{\frac{k}{m}} \int_{\bb R_+}  \psi \left( u,  \frac{x}{\sigma \sqrt{k}} \right) 
  \phi^+  (  u_{k,m,\ee} ) du   \notag\\
 & \leq  \sqrt{\frac{k}{m}}  \int_{\bb R_+}  \psi \left( u,  \frac{x}{\sigma \sqrt{k}} \right)   
    \phi^+ \left( u \frac{\sqrt{k}}{\sqrt{m}} \right) du  +  \frac{c \ee}{\sqrt{n}}   \notag\\
 & =    \sqrt{\frac{n}{m}}  \int_{\bb R_+}  \psi \left( \frac{s}{\sqrt{k/n}},  \frac{x}{\sigma \sqrt{k}} \right)   
    \phi^+ \left(  \frac{s}{\sqrt{m/n}} \right) ds  +  \frac{c \ee}{\sqrt{n}}        \notag\\
 & =   \sqrt{\frac{k}{n}}  \int_{\bb R_+}  \psi_{\frac{k}{n}} \left( s,  \frac{x}{\sigma \sqrt{n}} \right)   
    \phi^+_{\frac{m}{n}} (s) ds  +  \frac{c \ee}{\sqrt{n}}        \notag\\
&  =  \frac{\sqrt{km}}{n}  \phi^+ \left( \frac{x}{\sigma \sqrt{n}} \right)  +  \frac{c \ee}{\sqrt{n}}, 
\end{align*}
where in the last equality we used Lemma \ref{Lem_Densities_Raileigh_Levy}.
%it holds that
%\begin{align*} %\label{}
%& \int_{\bb R_+}   \psi \left( u,  \frac{x}{\sigma \sqrt{k}} \right)    \phi^+  \left(  u  \right) du
% =  \frac{1}{2}  \phi^+ \left( \frac{x}{\sigma \sqrt{2k}} \right).  
%% = \frac{1}{\sqrt{k}}  \int_{\bb{R}_{+}} 
%%    \psi \left( \frac{y}{\sqrt{k}},  \frac{x}{\sqrt{k}} \right)  \phi^+ \left( \frac{y}{\sqrt{m}} \right)  dy  \notag\\
%%%& = \frac{1}{\sqrt{k}}  \int_{\bb{R}_{+}} \frac{y}{\sqrt{k}}e^{-\frac{y^2}{2k}} \frac{y}{\sqrt{m}}e^{-\frac{y^2}{2m}}  dy \\
%%&=  \frac{1}{\sqrt{k}}  \int_{\bb{R}_{+}} \frac{y^2}{\sqrt{km}}e^{-\frac{ny^2}{2km}}  dy    
%%  = \frac{\sqrt{k} m}{n^{3/2}}\int_{\bb{R}_{+}} y^2e^{-y^2/2}  dy  
%%  =  \frac{\sqrt{2\pi  k} m}{ 2 n^{3/2}} .
%\end{align*}
Thus,  we deduce that
\begin{align} \label{BoundJ12Order32_Large}
J_{12}
% &  \leq \frac{1}{\sigma \sqrt{k} }  \int_{0}^{ 2 \alpha_n \sqrt{n} }   g (t-\ee)  
%  \frac{ 2 V^* (t) }{\sigma \sqrt{2\pi m} }  \frac{1}{2}  \sqrt{\frac{k}{m}}  \phi^+ \left( \frac{x}{\sigma \sqrt{2k}} \right)  dt \notag\\
% & \quad   +   c \left( \ee +  \alpha_n  + n^{-\ee} \right) 
%    \frac{1}{\sigma \sqrt{k} }  \int_{0}^{ 2 \alpha_n \sqrt{n} }   g (t-\ee)  \frac{ V^* (t) }{\sqrt{n}}  dt    \notag\\
 &  \leq \frac{2}{\sigma^2 \sqrt{2\pi} n }  
 \left[ \phi^+ \left( \frac{x}{\sigma \sqrt{n}} \right)  +  c_{\ee} \left(   \alpha_n  +  n^{-\ee} \right)   \right]
  \int_{0}^{ \alpha_n \sqrt{n}  }   g (t-\ee) V^* (t) dt. 
\end{align}
\textit{Bound of $J_2$.} In the same way as in \eqref{Seconthm-Decom-J1}, 
using a change of variable 
and the duality formula (\eqref{eq-duality-lemma-002_Cor} of Lemma \ref{lemma-duality-lemma-2_Cor}), one has
\begin{align}\label{Pf-Def-J2-yy}
J_2
 =  \frac{1}{\sigma \sqrt{k}}  \int_{\bb R_+}   g (t-\ee)   
 \bb E \left[ 1 - \Phi \left( \frac{t + S_m^* - \ee }{ \sigma \ee^{1/4} \sqrt{k} }  \right);  \tau_{t}^* > m  \right]  dt. 
\end{align}
%Since $\Phi$ is Lipschitz continuous on $\bb R$, using \eqref{CLLT-bound only 001b}, we get 
%\begin{align}\label{Seconthm-Decom-J2}
%J_2  \leq   J_{21} + J_{22} 
%%& \leq  \frac{1}{\sigma \sqrt{k}}  \int_{\bb R_+}   g (t-\ee)   
%% \bb E \left[ 1 - \Phi \left( \frac{t + S_m^* }{ \sigma \ee^{1/4} \sqrt{k} }  \right);  \tau_{t}^* > m  \right]  dt  \notag\\
% + \frac{c \ee^{3/4}}{n}  \int_{\bb R_+}   g (t-\ee)  V^*(t) dt  +  \frac{c_{\ee}}{n^{1+\ee}}  \int_{\bb R_+}   g (t-\ee)  (1 + t) dt, 
%\end{align}
%%& = : J_{21} + J_{22}, 
%%& =  \frac{1}{\sigma \sqrt{k}}  \int_{0}^{\ee +  n^{1/2 - \ee}}   g (t-\ee)   
%% \bb E \left[ \psi \left( \frac{t + S_m^* - \ee }{\sigma \sqrt{k}},  \frac{x}{\sigma \sqrt{k}}  \right);  \tau_{t}^* > m  \right]  dt,  
We decompose $J_2$ into two parts: $J_2 = J_{21} + J_{22}$, where 
\begin{align*}%\label{}
J_{21} & = \frac{1}{\sigma \sqrt{k}}  \int_{ \alpha_n \sqrt{n} }^{ \infty }   g (t-\ee)   
 \bb E \left[ 1 - \Phi \left( \frac{t + S_m^* - \ee  }{ \sigma \ee^{1/4} \sqrt{k} }  \right);  \tau_{t}^* > m  \right]  dt,   \notag\\
J_{22}  & = \frac{1}{\sigma \sqrt{k}}  \int_{0}^{ \alpha_n \sqrt{n} }   g (t-\ee)   
 \bb E \left[ 1 - \Phi \left( \frac{t + S_m^* - \ee }{ \sigma \ee^{1/4} \sqrt{k} }  \right);  \tau_{t}^* > m  \right]  dt.  
\end{align*}
\textit{Bound of $J_{21}$.}
Since the function $1 - \Phi$ is bounded,  
we get 
\begin{align}\label{J21Order32-Large-x}
J_{21}  \leq  \frac{c }{\sqrt{n}}  \int_{ \alpha_n \sqrt{n} }^{\infty}    g (t-\ee)  dt.  
\end{align}
\textit{Bound of $J_{22}$.}
Proceeding in the same way as in \eqref{ExpectaPhiUpp} and \eqref{J12hh01}, 
from the conditioned integral limit theorem \eqref{C-CLTalphan} of Corollary \ref{Cor-CCLT-Optimal}
and the fact that $1 - \Phi (\infty) = 0$, 
one has that uniformly in $t \in [0, \alpha_n \sqrt{n}]$, 
\begin{align}\label{Pf-xlarge-LLT-ee}
& \left| 
 \bb E \left[ 1 - \Phi \left( \frac{t + S_m^* - \ee }{ \sigma \ee^{1/4} \sqrt{k} }  \right);  \tau_{t}^* > m  \right]   
 -  \frac{2V^*(t)}{\sigma \sqrt{2\pi m}}  \int_{\bb R_+}  \phi (v) 
    \Phi^+ (  v_{k,m, \ee} )   dv  \right|  \notag\\
&  \leq  c_{\ee} \left(  \alpha_n  + n^{-\ee} \right)  \frac{V^*(t)}{ n^{1/2} },
\end{align}
where $v_{k,m, \ee} = \frac{  \ee^{1/4} \sigma \sqrt{k} v  + \ee}{ \sigma \sqrt{m}}$. 
Since the function $\Phi^+$ is Lipschitz continuous, by elementary calculations, we have 
\begin{align*}%\label{}
\int_{\bb R_+}  \phi (v) \Phi^+ (  v_{k,m, \ee} )   dv
& \leq   \int_{\bb R_+}  \phi (v) \Phi^+ \left( \ee^{1/4} \frac{\sqrt{k}}{\sqrt{m}}  v  \right)   dv  +  \frac{c \ee}{\sqrt{n}}  \notag\\
& = \frac{1}{2} - \frac{1}{2}  \left( 1 + \frac{\sqrt{\ee} k}{m} \right)^{-1/2}  +  \frac{c \ee}{\sqrt{n}}  \leq c \ee^{1/4}. 
\end{align*}
%Hence  we deduce that
%\begin{align} \label{J22-Order32-Large}
%J_{22}
% \leq \frac{1}{ n }  
%  \left( c \ee^{1/4} +  c_{\ee} \alpha_n  + c_{\ee} n^{-\ee} \right)  
%  \int_{0}^{  \alpha_n \sqrt{n}  }   g (t-\ee) V^* (t) dt. 
%\end{align}
Combining this with \eqref{Pf-xlarge-LLT-ee} and \eqref{J21Order32-Large-x}, we derive that
\begin{align}\label{J2-32-Large}
J_2   & \leq   \frac{1}{ n }  
  \left( c \ee^{1/4} +  c_{\ee} \alpha_n  + c_{\ee} n^{-\ee} \right)  
  \int_{0}^{  \alpha_n \sqrt{n}  }   g (t-\ee) V^* (t) dt   \notag\\
& \quad +  \frac{c }{\sqrt{n}}  \int_{ \alpha_n \sqrt{n} }^{\infty}    g (t-\ee)  dt. 
\end{align}
\textit{Bound of $J_3$.}
By \eqref{BoundHmsmallx}, we have 
\begin{align}\label{BoundJ2Order32_Large}
J_3 
%&\leq \frac{c}{k} \left( r_k(\ee) + \ee^{\frac{1}{4}} \right)  (1+x)   \frac{1}{\sqrt{m}} \int_{-\ee}^{\infty} g(y) (1 +y)dy \\
&\leq     \frac{ c_{\ee} }{n^{1 +\ee }}  \int_{\bb R_+} g(t-\ee) (1+t) dt.
\end{align}
Putting together \eqref{BoundJ11Order32Largex}, \eqref{BoundJ12Order32_Large}, \eqref{J2-32-Large}
 and \eqref{BoundJ2Order32_Large}, 
%we get
%\begin{align*}%\label{}
%I_n(x) \leq  \frac{2}{\sigma \sqrt{2\pi} n }  
% \phi^+ \left( \frac{x}{\sigma \sqrt{n}} \right)   \int_{ \bb R_+ }   g (t-\ee)   V^*(t) dt 
%    +   \frac{ c \ee }{n}   \|g\|_{1,1}.  
%\end{align*}
we conclude the proof of the upper bound \eqref{theorem-n3/2 001_Large} of the theorem. 

We next sketch the proof of the lower bound \eqref{theorem-n3/2 002_Large}. 
%The lower bound \eqref{theorem-n3/2 002_Large} is proved in the same way and thus  the details are omitted.
Similarly to the proof of \eqref{UpperBoundn32Largex}, we use the lower bound \eqref{eqt-C 002} of Theorem \ref{t-C 002} to get
$I_n(x) \geq K_1 - K_2 - K_3$, where
\begin{align*}%\label{}
K_1  & = \frac{1 }{ \sigma \sqrt{k} }   \int_{\bb{R} }  L_{m} (t) 
  \psi \left( \frac{t}{\sigma \sqrt{k}}, \frac{x}{\sigma \sqrt{k}}  \right)  dt,     \notag\\
K_2 & =   \frac{ 1 }{ \sigma \sqrt{k} }   
     \int_{\bb{R} }  H_m(t)  \left[ 1 - \Phi \left( \frac{t}{\sigma \ee^{1/4} \sqrt{k}} \right) \right]   dt   \notag\\
K_3 & =   \frac{ c \ee }{ \sigma \sqrt{k} }   
     \int_{\bb{R} }  H_m(t)   \psi \left( \frac{t}{\sigma \sqrt{k}}, \frac{x}{\sigma \sqrt{k}}  \right)  dt    \notag\\
 & \quad  +  c   \frac{ \ee^{1/12} }{\sqrt{n}}  \int_{\mathbb R} \big[ H_m( t)  +  H_m( t + \ee) \big]
   \phi \left( \frac{t}{ \sigma \sqrt{k}} \right)   dt   
  +   \frac{ c_{\ee} }{n^{1/2 + \ee }} \left\Vert  H_m \right\Vert_{1}, 
\end{align*}
with $H_m$ and $L_m$ defined by \eqref{DefHmt} and  \eqref{DefLmt}, respectively. 
%\begin{align}\label{LowerBoundhhn32bb}
%  I_n(x) 
% & \geq   \frac{1 }{ \sigma \sqrt{k} }   \int_{\bb{R} }  L_{m} (t) 
%  \tilde{\psi} \left( \frac{t}{\sigma \sqrt{k}}, \frac{x}{\sigma \sqrt{k}}  \right)  dt    \notag\\
% & \quad  -   \frac{ c \ee }{ \sigma \sqrt{k} }   
%     \int_{\bb{R} }  H_m(t)   \tilde{\psi} \left( \frac{t}{\sigma \sqrt{k}}, \frac{x}{\sigma \sqrt{k}}  \right)  dt    \notag\\
% & \quad  -  c   \frac{ \ee^{1/12} }{\sqrt{n}}  \int_{\mathbb R} \big[ H_m( t)  +  H_m( t + \ee) \big]
%   \phi \left( \frac{t}{ \sigma \sqrt{k}} \right)   dt   
%  -     \frac{ c_{\ee} }{n^{1/2 + \ee }} \left\Vert  H_m \right\Vert_{1}  \notag\\
%&  =:  K_1 - K_2 - K_3 - K_4,  
%\end{align}
%where $H_m$ and $L_m$ are defined by \eqref{DefHmt} and  \eqref{DefLmt}, respectively. 
%%\begin{align*}%\label{}
%%L_m(t) := \mathbb E  \left( h \left( t + S_{m} \right); \tau _{t - \ee} > m \right)  \mathds 1_{ \{t \geq \ee\} },  \quad  t \in \bb R.   
%%\end{align*}

\noindent
\textit{Bound of $K_1$.} 
Note that $L_m(t) = 0$ for $t < \ee$. % and $\tilde{\psi} (t, x) = \psi (t, x)$ for $t, x \geq 0$. 
As in the estimate of $J_1$ (cf.\  \eqref{Seconthm-Decom-J1}), 
using the duality formula (Lemma \ref{lemma-duality-lemma-2_Cor}), we get
\begin{align*}%\label{}
K_1  & =  \frac{1}{\sigma \sqrt{k}}  \int_{\bb R_+}   h (t + \ee)   
 \bb E \left[ \psi \left( \frac{t + S_m^* + \ee }{\sigma \sqrt{k}},  \frac{x}{\sigma \sqrt{k}}  \right);  \tau_{t}^* > m  \right]  dt  \notag\\
 & \geq  \frac{1}{\sigma \sqrt{k}}   \int_{0}^{ \alpha_n \sqrt{n} }   h (t + \ee)   
 \bb E \left[ \psi \left( \frac{t + S_m^* + \ee }{\sigma \sqrt{k}},  \frac{x}{\sigma \sqrt{k}}  \right);  \tau_{t}^* > m  \right]  dt, 
\end{align*}
where in the last inequality we used the fact that $t + S_m^* + \ee \geq \ee$ on the event $\tau_{t}^* > m$,
so that $\psi >0$. 
Following the proof of \eqref{BoundJ12Order32_Large}, one has 
\begin{align*}%\label{}
K_1 \geq  \frac{2}{\sigma^2 \sqrt{2\pi} n }  
 \left[ \phi^+ \left( \frac{x}{\sigma \sqrt{n}} \right)  -   c_{\ee} \left(  \alpha_n  + n^{-\ee} \right)   \right]
  \int_{0}^{ \alpha_n \sqrt{n}  }   h (t + \ee) V^* (t) dt.   
\end{align*}
\textit{Bound of $K_2$.} 
By \eqref{Pf-Def-J2-yy} and \eqref{J2-32-Large}, we have
%Proceeding again in the same way as in the estimate of $J_1$ (cf.\  \eqref{Seconthm-Decom-J1}), 
%using \eqref{J2-32-Large}, we get
\begin{align*}%\label{}
& K_2   =  \frac{1}{\sigma \sqrt{k}}  \int_{\bb R_+}   g (t - \ee)   
 \bb E \left[ 1 - \Phi \left( \frac{t + S_m^* - \ee }{\sigma \ee^{1/4} \sqrt{k}}  \right);  \tau_{t}^* > m  \right]  dt  \notag\\
& \leq  \frac{1}{ n }  
  \left( c \ee^{1/4} +  c_{\ee} \alpha_n  + c_{\ee} n^{- \ee}  \right)  
  \int_{0}^{ \alpha_n \sqrt{n}  }   g (t-\ee) V^* (t) dt  
 +  \frac{c }{\sqrt{n}}  \int_{ \alpha_n \sqrt{n} }^{\infty}    g (t-\ee)  dt. 
 %  + \frac{ c_{\ee} }{n^{1 +\ee }} \int_{\bb R_+} g(t-\ee) (1+t) dt. 
\end{align*}
\textit{Bound of $K_3$.}
Since the functions $\psi$ and $\phi$ are bounded, using \eqref{BoundHmsmallx} we get 
\begin{align*}%\label{}
K_3 \leq  \left( \frac{c \ee^{1/12} }{\sqrt{n}} +  \frac{ c_{\ee} }{n^{1/2 + \ee }}  \right)   \| H_m \|_1   
   \leq  \left( \frac{c \ee^{1/12} }{ n } +  \frac{ c_{\ee} }{n^{1 + \ee }}  \right)   \int_{\bb R_+} g(t-\ee) (1+t) dt.  
%  +    \left\Vert  H_m \right\Vert_{1}
\end{align*}
Collecting the above bounds for $K_1$, $K_2$ and $K_3$,
and using the fact that $V^*(t) \leq c(1 + t)$ for $t \in \bb R_+$ and $h(t + \ee) \leq f(t) \leq g(t-\ee)$ for any $t \in \bb R$,
we conclude the proof of the lower bound \eqref{theorem-n3/2 002_Large}. 
\end{proof}

%%%%%%%%%%%%%%%%%%%%%%%%%%%%%%%%%%%%%%%%%%
\subsection{Proof of Theorems \ref{Theorem-BB002} and \ref{Theorem-BB002bis}} \label{SecProofThm4}

%\todos{We should formulate a more general result without the condition $y = o(\sqrt{n})$}
%\todos{$\sigma_{\lambda}$ will appear in the theorem: To be added}

\begin{proof}[Proof of Theorem \ref{Theorem-BB002}]
We first prove \eqref{LLT32LargeStartingaa}. 
Since $\int_{\bb R_+}  f(t) (1 + t)^{\gamma} dt < \infty$ for some constant $\gamma >1$, 
the result \eqref{LLT32LargeStartingaa} is a consequence of Theorem \ref{theorem-n3/2-upper-lower-bounds_Large}, 
Lemma \ref{lemma-DRI-convergence} and the Lebesgue dominated convergence theorem. 
%\end{proof}
%
%
%\begin{proof}[Proof of \eqref{demoCLLTLargeffbbaa}]
%We next prove that \eqref{demoCLLTLargeffbb} holds 
%uniformly in $x \in [\eta^{-1} \sqrt{n}, \eta \sqrt{n}]$ and $y \in [\alpha_n^{-1}, \alpha_n \sqrt{n}]$. 

We next prove \eqref{demoCLLTLargeffbbaa}. 
Applying the upper bound \eqref{theorem-n3/2 001_Large} of Theorem \ref{theorem-n3/2-upper-lower-bounds_Large}
with $g = \overline f_{\delta,\ee}$ and $\alpha_n$ replaced by $2 \alpha_n$, 
we get that uniformly in $x \in [\eta^{-1} \sqrt{n}, \eta \sqrt{n}]$ and $y \in [\alpha_n^{-1}, \alpha_n \sqrt{n}]$, 
\begin{align*}
& \mathbb E \left( f(x+S_n - y); \tau_x > n \right)        \notag\\
&   \leq   \frac{2  }{\sigma^2  \sqrt{2\pi} n }  
\left[  \phi^+ \left( \frac{x}{\sigma \sqrt{n}} \right)   +   \big( c \ee^{1/4} +  c_{\ee} \alpha_n  \big)   \right]  
     \int_{0}^{ 2 \alpha_n \sqrt{n} }  \overline f_{\delta,\ee} (t - y - \ee)   V^*(t) dt    \notag\\
& \quad  +  \frac{c }{\sqrt{n}}  
    \int_{ 2 \alpha_n \sqrt{n} }^{\infty}    \overline f_{\delta,\ee} (t - y - \ee)   dt   
        + \frac{ c_{\ee} }{n^{1 +\ee }}  
    \int_{\bb R_+} \overline f_{\delta,\ee} (t - y - \ee)   (1+t) dt. 
\end{align*}
Using \eqref{Pfn32Equiv01}, \eqref{Pfn32Equiv02} and \eqref{Pfn32Equiv03},
and the fact that $\phi^+ ( \frac{x}{\sigma \sqrt{n}} )$ is bounded from below by a constant $c_{\eta} >0$ 
for $x \in [\eta^{-1} \sqrt{n}, \eta \sqrt{n}]$, we get
that uniformly in $x \in [\eta^{-1} \sqrt{n}, \eta \sqrt{n}]$ and $y \in  [\alpha_n^{-1}, \alpha_n \sqrt{n}]$,
\begin{align*}%\label{}
\limsup_{n \to \infty}  
\frac{ \mathbb E \left( f(x+S_n - y); \tau_x > n \right)    }{  \frac{2 y }{\sqrt{2\pi } \sigma^2 n }    
  \phi^+ ( \frac{x}{\sigma \sqrt{n}} )  }    
  \leq  \left( 1 + c \ee^{1/4}   \right)  
      \int_{-\ee}^{ \infty }  \overline f_{\delta,\ee} (u)  du.  
\end{align*}
%Since, 
%as $n \to \infty$, uniformly in $x \in [ \eta^{-1} \sqrt{n}, \eta \sqrt{n}]$ and $y \in [\alpha_n^{-1}, \alpha_n \sqrt{n}]$,
%\begin{align}\label{Relationpsiphi} 
%  \psi \left( \frac{y}{\sigma \sqrt{n}}, \frac{ x }{\sigma \sqrt{n}}  \right)
%  \sim  
%  \frac{2xy}{ \sqrt{2\pi } \sigma^2 n } e^{- \frac{x^2}{2 \sigma^2 n} }
%  \sim 
%  \frac{2 y }{\sqrt{2\pi n} \sigma  }   
% \phi^+ \left( \frac{x}{ \sigma  \sqrt{n}} \right). 
%\end{align}
This yields the desired upper bound by taking first $\ee \to 0$ and then $\delta \to 0$, and using Lemma \ref{lemma-DRI-convergence}. 
The lower bound can be proved in the same way 
by using \eqref{theorem-n3/2 002_Large} of Theorem \ref{theorem-n3/2-upper-lower-bounds_Large}. 
The proof of \eqref{demoCLLTLargeffbbaa} is complete. 
\end{proof}

%%%%%%%%%%%%%%%%%%%%%%%%%%%%%%%%%%%%%%%%
%\subsection{Proof of Theorems \ref{Theorem-BB002bis} and \ref{ThmCLLTLarge32bis}}

\begin{proof}[Proof of Theorem \ref{Theorem-BB002bis}]
%We apply Theorem \ref{theorem-n3/2ChangeMeas_Large} to 
%prove that \eqref{ThmCaraLargeStartaabb} holds uniformly in 
%$x \in [ \eta^{-1} \sqrt{n}, \eta \sqrt{n}]$, $y \in [\alpha_n^{-1}, \alpha_n \sqrt{n}]$
%and $\Delta \in [\Delta_0, n^{1/2 - \ee} ]$.  
%%Let $\Delta>0$ be a fixed constant.  
Applying \eqref{theorem-n3/2 001_Large} of Theorem \ref{theorem-n3/2-upper-lower-bounds_Large}
with $f = \mathds 1_{[y, y + \Delta]}$, $g = \mathds 1_{[y-\ee, y + \Delta+\ee]}$
and $\alpha_n$ replaced by $2 \alpha_n$ (so that the second integral in \eqref{theorem-n3/2 001_Large} vanishes),   
we derive that uniformly in $x \in [ \eta^{-1} \sqrt{n}, \eta \sqrt{n}]$, $y \in [\alpha_n^{-1}, \alpha_n \sqrt{n}]$
and $\Delta \in [\Delta_0, o(y) ]$,  
\begin{align*}%\label{}
&  \mathbb P \left( x + S_n \in [0, \Delta] + y, \, \tau_x >n\right)   \notag\\
&   \leq    \frac{2 }{\sigma^2 \sqrt{2\pi} n }  
 \left[  \phi^+ \left( \frac{x}{\sigma \sqrt{n}} \right)   +   \big( c \ee^{1/4} +  c_{\ee} \alpha_n  \big)   \right] 
  \int_{-\ee}^{ \Delta + \ee}     V^*(t + y + \ee) dt    \notag\\
& \quad  +   \frac{ c_{\ee} }{ n^{1 + \ee} }  
    \int_{-\ee}^{ \Delta + \ee}   (1 + t + y) dt.    
\end{align*}
Since $1  \leq  c_{\eta} \phi^+ ( \frac{x}{\sigma \sqrt{n}} )$ 
holds uniformly in $x \in [ \eta^{-1} \sqrt{n}, \eta \sqrt{n}]$ for some constant $c_{\eta} >0$, 
using \eqref{BoundInteDelta01} and \eqref{BoundInteDelta02}, 
%By a change of variable and using the fact that $V_{\lambda}^*(t)/t \to 1$ as $t \to \infty$,
%we get that as $y \to \infty$, 
%\begin{align*}%\label{}
%\int_{y}^{y + \Delta + 2\ee}  e^{- \lambda (t  - \ee)}    V_{\lambda}^*(t) dt
%& =  e^{-\lambda y} \int_0^{\Delta + 2 \ee} e^{- \lambda (u-\ee) } V_{\lambda}^*(u + y) du  \notag\\
%& \sim  y e^{-\lambda y} \int_0^{\Delta + 2 \ee} e^{- \lambda (u-\ee) }  du
% =  y e^{-\lambda y}  \frac{e^{\lambda \ee} - e^{- \lambda (\Delta + \ee)} }{ \lambda }. 
%\end{align*}
%In the same way, it holds that as $y \to \infty$, 
%\begin{align*}%\label{}
%\int_{y}^{y + \Delta + 2\ee}   e^{- \lambda (t - \ee)}  (1 + t) dt  
%& =  e^{-\lambda y} \int_0^{\Delta + 2 \ee} e^{- \lambda (u-\ee) } (1 + u + y) du \notag\\
%& \sim y e^{-\lambda y}  \frac{e^{\lambda \ee} - e^{- \lambda (\Delta + \ee)} }{ \lambda }. 
%\end{align*}
we obtain that uniformly in $x \in [ \eta^{-1} \sqrt{n}, \eta \sqrt{n}]$, $y \in [\alpha_n^{-1}, \alpha_n \sqrt{n}]$
and $\Delta \in [\Delta_0, o(y) ]$,  
\begin{align*}%\label{}
\limsup_{n \to \infty}  \frac{ \mathbb P \left( x + S_n \in [0, \Delta] + y, \,  \tau_x >n\right)  }{ 
   \frac{2y \Delta }{\sigma^2 \sqrt{2\pi} n }  \phi^+ \big( \frac{x}{\sigma \sqrt{n}} \big) }
  \leq  \left( 1 + c_{\eta} \ee^{1/4} \right),
\end{align*}
which concludes the proof of the upper bound by letting $\ee \to 0$.
The lower bound can be obtained in a similar way by using \eqref{theorem-n3/2 002_Large}.
The proof of Theorem \ref{Theorem-BB002bis} is complete. 
\end{proof}

%%%%%%%%%%%%%%%%%%%%%%%%%%%%%%%%%%%%%%%%%%%
\subsection{Proof of Theorems \ref{Thm-Iglehart001} and \ref{Thm-Iglehart002}}
Theorems \ref{Thm-Iglehart001} and \ref{Thm-Iglehart002} are easy consequences of the following result,
which is deduced from  Theorems \ref{Theorem-AA002} and \ref{Theorem-BB002}.

\begin{theorem} \label{Thm_CLLT_Drift_Negative001}
Assume \ref{A1} and \ref{ExponentialMoment}. 
Let $f: \bb R_+ \mapsto \bb R$ be a directly Riemann integrable function 
satisfying $\int_{\bb R_+}  f(t) e^{-\lambda t} (1 + t)^{\gamma} dt < \infty$ for some constant $\gamma >1$. 
Then, for any sequence of positive numbers $(\alpha_n)_{n\geq 1}$  satisfying
$\lim_{n\rightarrow \infty} \alpha_n =  0$, we have,  
 as $n \to \infty$,  uniformly in $x \in [0, \alpha_n \sqrt{n}]$,  % and $y \in [0, o(\sqrt{n})]$, 
\begin{align}\label{Asymn32Smallx}
\mathbb E \left( f(x+S_n);  \,   \tau_x > n  \right)
\sim      \frac{ 2 V_{\lambda}(x) e^{n \Lambda(\lambda) + \lambda x} }{ \sqrt{2\pi }  \sigma_{\lambda}^3  n^{3/2}  } 
    \int_{\bb R_+}   f( t)  e^{-\lambda t}  V_{\lambda}^{\ast } (t) dt. 
\end{align}
Moreover, for any $\eta \geq 1$,  
we have, as $n \to \infty$,  uniformly in $x \in [\eta^{-1} \sqrt{n}, \eta \sqrt{n}]$,  
\begin{align}\label{LLT32LargeStarting}
\mathbb E \left( f(x+S_n);  \,  \tau_x > n  \right)  
  \sim  
    \frac{2 e^{n \Lambda(\lambda) + \lambda  x}  }{ \sqrt{2\pi} \sigma_{\lambda}^2 n }  
      \phi^+ \left( \frac{x}{\sigma_{\lambda} \sqrt{n}} \right)
    \int_{\bb R_+}   f( t)  e^{-\lambda t}  V_{\lambda}^{\ast } (t) dt. 
\end{align}
\end{theorem}

\begin{proof}
By a change of measure, we get
\begin{align*}%\label{ChangeOfMeas0aa}
 \mathbb E \left( f(x+S_n); \tau_x >n\right)  
& = e^{n \Lambda(\lambda)}  \bb E_{\lambda} \left( e^{- \lambda S_n} f(x+S_n); \tau_x >n \right)  \notag\\
& =  e^{n \Lambda(\lambda) + \lambda x} \bb E_{\lambda} \left( e^{- \lambda (x + S_n)} f(x+S_n); \tau_x >n \right). 
\end{align*}
Since $\int_{\bb R_+}  f(t) e^{-\lambda t} (1 + t)^{\gamma} dt < \infty$ for some constant $\gamma >1$, 
applying \eqref{Asymn32Smallxaa} of Theorem \ref{Theorem-AA002} we get \eqref{Asymn32Smallx}, 
and  applying \eqref{LLT32LargeStartingaa} of Theorem \ref{Theorem-BB002} we obtain \eqref{LLT32LargeStarting}. 
\end{proof}

\begin{proof}[Proof of Theorems \ref{Thm-Iglehart001} and \ref{Thm-Iglehart002}]
In the particular case when the random walk $S_n$ has negative drift, i.e.\, $\mu = \bb E X_1 < 0$ 
(in this case $\lambda > 0$ by condition \ref{ExponentialMoment}), 
taking $f = \mathds 1_{[0, \infty) }$ in \eqref{Asymn32Smallx} we get Theorem \ref{Thm-Iglehart001},
and taking $f = \mathds 1_{[0, \infty) }$ in \eqref{LLT32LargeStarting} we get Theorem \ref{Thm-Iglehart002}. 
\end{proof}

%%%%%%%%%%%%%%%%%%%%%%%%%%%%%%%%%%%%%%%%%%%
\subsection{Proof of Theorems \ref{Thm-LLT-taux-xsmall} and \ref{Thm-LLT-taux-xlarge}} \label{Sec: loc taux}
To prove Theorems \ref{Thm-LLT-taux-xsmall} and \ref{Thm-LLT-taux-xlarge}, we need the following lemma. 

\begin{lemma}\label{Lem-Equ-Harmonic}
Assume \ref{ExponentialMoment} for some $\lambda >0$. Then,
\begin{align}\label{Equlity-Harmonic-tt}
\int_{\bb R_+}   \bb P(t + X_1 < 0)  e^{-\lambda t}  V_{\lambda}^{\ast } (t) dt
= \bb E e^{\lambda X_1} \int_{- \infty}^0  e^{-\lambda t}  V_{\lambda}^{\ast } (t) dt  \in (0, \infty). 
\end{align}
\end{lemma}

\begin{proof}
By Fubini's theorem and a change of variable, we get
\begin{align*}%\label{}
I: & =  \int_{\bb R_+}   \bb P(t + X_1 < 0)  e^{-\lambda t}  V_{\lambda}^{\ast } (t) dt  \notag\\
& = \bb E \int_0^{\infty} \mathds 1_{\{ t + X_1 < 0 \}}  e^{-\lambda t}  V_{\lambda}^{\ast } (t) dt  \notag\\
& = \bb E \int_{X_1}^{0}   e^{-\lambda (s - X_1)}  V_{\lambda}^{\ast } (s - X_1) ds  \notag\\
& =   \int_{- \infty}^{0}  \bb E \left[  e^{-\lambda (s - X_1)}  V_{\lambda}^{\ast } (s - X_1); s - X_1 \geq 0  \right] ds. 
\end{align*}
By the harmonicity property of the function $V_{\lambda}^{\ast }$,  we have that for any $s \leq 0$, 
\begin{align*}%\label{}
\bb E_{\lambda} \left[  V_{\lambda}^{\ast } (s - X_1); s - X_1 \geq 0  \right] = V_{\lambda}^{\ast }(s).
\end{align*}
Then, by the definition of $\bb E_{\lambda}$ it follows that 
\begin{align*}%\label{}
I &=  \bb E e^{\lambda X_1}\int_{- \infty}^{0}  e^{-\lambda s}  \bb E_{\lambda} \left[  V_{\lambda}^{\ast } (s - X_1); s - X_1 \geq 0  \right] ds \\
 &= \bb E e^{\lambda X_1} \int_{- \infty}^{0}  e^{-\lambda s}  V_{\lambda}^{\ast }(s) ds, 
\end{align*}
which proves the equality in \eqref{Equlity-Harmonic-tt}. 
Using Markov's inequality, condition \ref{ExponentialMoment} and the fact that $V_{\lambda}^{\ast }$ 
is strictly positive and non-decresing on $\bb R_+$, 
it is easy to see that $I \in [0, \infty)$.  
Suppose that $I = 0$, then $\bb P(X_1 < 0) = 0$. This contradicts to the first requirement in condition \ref{ExponentialMoment}, 
so that $I \in (0, \infty)$. 
\end{proof}

\begin{proof}[Proof of Theorems \ref{Thm-LLT-taux-xsmall} and \ref{Thm-LLT-taux-xlarge}]
We only give a proof of Theorem \ref{Thm-LLT-taux-xlarge} since Theorem \ref{Thm-LLT-taux-xsmall}
is a particular case of Theorem \ref{Thm-LLT-taux-xlarge} by taking $\lambda = 0$. 

We first prove \eqref{Asym-tau-n-xsmall-changemeas}. 
By the definition of $\tau_x$, we have for any $n \geq 1$, 
\begin{align*}%\label{}
\mathbb P \left( \tau_x  =  n + 1  \right) = \mathbb P \left( x + S_{n + 1} < 0,  \tau_x  >  n  \right). 
\end{align*}
From the Markov property, it follows that 
\begin{align*}%\label{}
\mathbb P \left( x + S_{n + 1} < 0,  \tau_x  >  n  \right)
=  \mathbb E \left( f(x + S_{n});  \tau_x  >  n  \right), 
\end{align*}
where
\begin{align*}%\label{}
f(t) = \bb P (t + X_1 < 0) \mathds 1_{\{t \geq 0\}}, \quad t \in \bb R. 
\end{align*}
It is easy to see that $0 \leq f \leq 1$ and $f$ is a non-increasing function on $\bb R$.
In addition, using Markov's inequality and condition \ref{ExponentialMoment}, we get that for $t \geq 1$, 
\begin{align*}%\label{}
f(t) \leq  \bb P (|X_1| > t)  \leq  \frac{1}{t^{2 + \delta} e^{\lambda t}}  \bb E  |X_1|^{2 + \delta} e^{ \lambda X_1 }
\leq  \frac{c}{t^{2 + \delta} e^{\lambda t}}, 
\end{align*}
so that by taking $\gamma = 1 + \frac{\delta}{2}$, 
\begin{align*}%\label{}
\int_{\bb R_+}  f(t) e^{-\lambda t} (1 + t)^{\gamma} dt 
& =  \int_{0}^1  f(t) e^{-\lambda t} (1 + t)^{\gamma} dt + \int_1^{\infty}  f(t) e^{-\lambda t} (1 + t)^{\gamma} dt  
 \leq   c. 
\end{align*}
Hence the function $f$ satisfies the condition stated in Theorem \ref{Thm_CLLT_Drift_Negative001}.
Using \eqref{Asymn32Smallx} of Theorem \ref{Thm_CLLT_Drift_Negative001}, 
we obtain that as $n \to \infty$,  uniformly in $x \in [0, \alpha_n \sqrt{n}]$, 
\begin{align*}
\mathbb P \left( \tau_x  =  n + 1  \right)
\sim      \frac{ 2 V_{\lambda}(x) e^{n \Lambda(\lambda) + \lambda x} }{ \sqrt{2\pi }  \sigma_{\lambda}^3  n^{3/2}  }  
   \int_{\bb R_+}   \bb P(t + X_1 < 0)  e^{-\lambda t}  V_{\lambda}^{\ast } (t) dt, 
\end{align*}
which, together with Lemma \ref{Lem-Equ-Harmonic},
 concludes the proof of \eqref{Asym-tau-n-xsmall-changemeas}. 

The asymptotic \eqref{Asym-tau-n-xlarge-changemeas} can be obtained in the same way 
by using \eqref{LLT32LargeStarting} of Theorem \ref{Thm_CLLT_Drift_Negative001}. 
The proof of Theorems \ref{Thm-LLT-taux-xsmall} and \ref{Thm-LLT-taux-xlarge} is complete. 
\end{proof}

%%%%%%%%%%%%%%%%%%%%%%%%%%%%%%%%%%%%%%%%%%%%%
%%%%%%%%%%%%%%%%%%%%%%%%%%%%%%%%%%%%%%%%%%%%%
\section{Appendix: Proof of conditioned integral limit theorems}\label{Sect-Appendix}

\subsection{Auxiliary results}\label{Sec_PfThm_a}
The goal of this section is to prove some auxiliary results which will be used to 
establish the conditioned integral limit theorems (Theorems \ref{Theor-IntegrLimTh} and \ref{CorCCLT}). 

\begin{lemma}\label{Lem-MKillTn}
Assume \ref{SecondMoment} for some $\delta >0$. 
%that $\bb E (X_1^{2+\delta}) < \infty$ for all $\delta >0$. 
Then there exists a constant $c > 0$ such that 
for any $\varepsilon \in  [0, \frac{\delta}{2(2 + \delta)}]$, $n \geq 1$ and $x \geq n^{1/2 - \ee}$, 
\begin{align}
\mathbb{E} \left(x + S_n; \tau_x > n \right)
\leq x +  c n^{1/2 - \ee}
\leq \left(1 + \frac{c}{n^{ \ee}} \right) x.   \nonumber
\end{align}
\end{lemma}

%Under weaker moment condition, Lemma \ref{Lem-MKillTn} improves the previous results in 
%\cite[Lemma 5.2]{GLP16} for invertible matrices and \cite[Lemma 3.1]{Pha17} for positive matrices 
%by removing the constant in front of $\max\{y,0\}$. Below we give a direct proof of Lemma \ref{Lem-MKillTn},
%which simplifies the proofs in \cite{GLP16, Pha17}. 

\begin{proof}
By the optional stopping theorem, we get
\begin{align}\label{Pf_Optional_thm}
\mathbb{E} (x + S_n; \tau_x > n ) 
& = x - \mathbb{E} ( x + S_n; \tau_x \leq n )  \nonumber\\
& = x - \mathbb{E} (x + S_{\tau_x}; \tau_x \leq n ). 
\end{align}
By the definition of $\tau_x$, we have $X_{\tau_x} \leq x + S_{\tau_x} < 0$ and hence
\begin{align*}
\mathbb{E} (x + S_n; \tau_x > n ) < x + \mathbb{E} ( |X_{\tau_x} |; \tau_x \leq n ).
\end{align*}
Using H\"older's inequality, condition \ref{SecondMoment} and Markov's inequality gives
\begin{align*}%\label{}
\mathbb{E} \left( | X_k |; |X_k| >  n^{1/2 - \varepsilon} \right)
& \leq  \mathbb{E}^{\frac{1}{2 + \delta} } \left( | X_k |^{2 + \delta} \right) 
    \bb P^{ \frac{1 + \delta}{2 + \delta} } \left( |X_k| >  n^{1/2 - \varepsilon} \right)  \notag\\
& \leq  c n^{ - (1/2 - \ee) (1 + \delta)}, 
\end{align*}
from which it follows that
\begin{align*}
\mathbb{E} \left( | X_{\tau_x} |; \tau_x \leq n \right) 
&  \leq   n^{1/2 - \varepsilon} 
 +  \mathbb{E} \left( | X_{\tau_x} |;  |X_{\tau_x}| >  n^{1/2- \varepsilon},  \tau_x \leq n \right)
  \nonumber\\
&  \leq     n^{1/2- \varepsilon}
  +  \sum_{ k =1 }^n \mathbb{E} \left( | X_k |; |X_k| >  n^{1/2 - \varepsilon} \right)  \notag\\
& \leq   n^{1/2- \varepsilon}  + c n^{1 - (1/2 - \ee) (1 + \delta)}
\leq  c  n^{1/2- \varepsilon}, 
\end{align*}
where in the last inequality we used the fact that $\ee \in [0, \frac{\delta}{2(2 + \delta)}]$. 
%Since $\bb E (X_k^{2+\delta}) < \infty$ for some $\delta >0$, 
The desired result follows.  % by the Markov inequality. 
\end{proof}

For $\ee \in (0, 1/2)$ and $x \geq 0$, consider the first time 
when the absolute value of the random walk $(x + S_k)_{k \geq 1}$ exceeds the level $n^{1/2 - \ee}$:
\begin{align}\label{Def_nun}
\nu_n = \nu_{n,x,\ee} = \inf \left\{ k \geq 1: |x + S_k| > n^{1/2-\ee} \right\}. 
\end{align}
It is easy to see that the stopping time $\nu_n$ converges $\bb P$-a.s.\ to $\infty$  as $n \to \infty$.
The following result gives the tail behavior of $\nu_{n}$.   

\begin{lemma}\label{Lem-NuTzEs}
Assume \ref{SecondMoment} for some $\delta >0$. 
Then, for any $\ee \in (0,1/2)$ and $\beta>0$,
there exists a constant $c_{\ee,\beta}$ such that for any $n\geq 1$ and $x \geq 0$, 
%There exists $\ee_0>0$ such that for any $\ee\in(0, \ee_0)$, $\beta>0$, $n\geq 1$, $x \geq 0$, 
\begin{align*}
\mathbb{P} \left( \nu_{n} > \beta n^{1-\ee} \right)
\leq  c_{\ee,\beta} e^{ - c_{\ee,\beta} n^{\ee} }. 
% \exp\!\left(- c_{\ee,\delta} n^{\ee} \right).  
\end{align*}
\end{lemma}

\begin{proof}
Set $K:=[n^\ee]$ and 
split the interval $[1,\beta n^{1-\ee}]$ into $K$ subintervals of length $l:=[\beta n^{1-2\ee}]$, 
so that $Kl\leq \beta n^{1-\ee}$. 
We have for large enough $n$,
\begin{align*}
\mathbb{P} \big( \nu_{n}> \beta n^{1-\ee} \big)
\leq   \mathbb{P} \left( \max_{ 1\leq j \leq  Kl} |y + S_j|  \leq n^{1/2-\ee} \right).
\end{align*}
For each $k = 1,\ldots, K$, 
denote $A_{k,x}:=\{ \max_{ 1\leq k'\leq k }  |x + S_{k'l}|  \leq  n^{1/2-\ee} \}$.  
Using the Markov property, we have
\begin{align*}
\mathbb{P} (A_{K,x}) 
\leq \mathbb{P} (A_{K-1,x}) \sup_{x \in \bb R } \mathbb{P} (A_{1,x})
\leq \Big( \sup_{x \in \bb R } \mathbb{P} (A_{1,x}) \Big)^K. 
\end{align*}
%\begin{align}\label{estimate nu}
%\mathbb{P}_x(\nu_n>\beta n^{1-\ee}, T_y>\beta n^{1-\ee})
%\leq \mathbb{P}_x(A_{K,z}, T_y>Kl).
%\end{align}
%By the Markov property, we have
%\begin{align}\label{estimate nu n 1}
%&  \mathbb{P}_x(A_{K,z}, T_y>Kl)
%=   \int_{\mathcal{S} \times \bb R }\mathbb{P}_{x'}(A_{1,z'}, T_{y'}>l)  \nonumber\\
% & \quad  \times \mathbb{P}_x \Big( X_{(K-1)l}\in dx', z+M_{(K-1)l}\in dz', A_{K-1,z}, T_y>(K-1)l \Big). 
%\end{align}
%For $z' = y' - \psi(x')$,
%recall that $z' + M_n = y' + S_n - \psi(X_n)$ and $|\psi(X_n)| \leq a$ is uniformly bounded.  
%Hence for large enough $n$,
%\begin{align}\label{estimate nu n 2}
%\mathbb{P}_{x'}(A_{1,z'}, T_{y'}>l)
%\leq \mathbb{P}_{x'} \big( z'+M_l\leq n^{1/2-\ee} \big)
%\leq \mathbb{P}_{x'} \big( y'+S_l\leq 2n^{1/2-\ee} \big).
%\end{align}
%Note that on the event $\{T_y>(K-1)l\}$, 
%we have $z'=z+M_{(K-1)l}>0$, so that $y'>-a$ since $z' = y' - \psi(x')$.
%Therefore, for large enough $n$,
%\begin{align} \label{estimate nu n 3}
%\mathbb{P}_{x'} \big( y'+S_l\leq 2n^{1/2-\ee} \big)
%\leq \mathbb{P}_{x'} \big( S_l\leq 3n^{1/2-\ee} \big).
%\end{align}
By the central limit theorem for $S_l$,  we get
that uniformly in $x \in \bb R $, 
\begin{align*} %\label{estimate nu n 4}
\mathbb{P} (A_{1,x}) \leq 
\mathbb{P} \big( |x + S_l| \leq 2n^{1/2-\ee} \big)
\leq  \frac{1}{ \sqrt{2\pi} } 
 \int_{ - \frac{x}{ \sqrt{l} } -c_{\ee,\beta} }^{ - \frac{x}{ \sqrt{l} } + c_{\ee,\beta} } 
  e^{-\frac{u^2}{2}}du + cr_l  
 < 1, 
\end{align*}
where $c_{\ee,\beta}=\frac{ n^{1/2-\ee}}{\sqrt{l}}$. 
%where the constant $q_{\ee,\beta}<1.$ 
%From \eref{estimate nu n 1}-\eref{estimate nu n 4}, we infer that
%\begin{align}
%\mathbb{P}_x(A_{K,z}, T_y>Kl)\leq (q_{\ee,\beta}+cr_l)\mathbb{P}_x(A_{K-1,z}, T_y>(K-1)l). \nonumber
%\end{align}
%By iteration, it follows that
%\begin{align}\label{estimate nu n 5}
%\mathbb{P}_x(A_{K,z}, T_y>Kl)\leq (q_{\ee,\beta}+cr_l)^K.
%\end{align}
%Since $l=[\beta n^{1-2\ee}]$, for large enough $n$, we have $q_{\ee,\beta}+cr_l<1$.
Since $K=[n^\ee]$, the result follows. % by combining \eref{estimate nu} and \eref{estimate nu n 5}. 
\end{proof}

The following result shows that the expectation of 
the random walk $(x + S_n)_{n \geq 1}$ killed at the exit time $\tau_x$
is uniformly bounded with respect to $n$.

\begin{lemma}\label{z+Mn killed Tz}
Assume \ref{SecondMoment} for some $\delta >0$.  
%Assume that $\bb E (X_1^{2+\delta}) < \infty$ for all $\delta >0$.  
There exists $c > 0$ such that for any $x \geq 0$ and $n \geq 1$, 
\begin{align}
\mathbb{E} \left( x + S_n; \tau_x > n  \right) \leq  c (1 + x). \nonumber
\end{align}
\end{lemma}

\begin{proof}
The proof is a combination of a recursive argument and the Markov property.
For brevity, denote $V_n (x) = \mathbb{E} (x + S_n; \tau_x > n)$. 
For $\ee \in  (0, \frac{\delta}{2(2 + \delta)}]$, we have 
\begin{align}\label{z+Mn killed Tz J1 et J2}
V_n (x)
%\mathbb{E} \left( x + S_n; \tau_x > n \right)
& =  \mathbb{E} \left( x + S_n; \tau_x > n, \nu_{n} > [n^{1-\varepsilon}]  \right) \nonumber\\
  & \quad + \mathbb{E} \left( x + S_n; \tau_x >n, \nu_{n} \leq [n^{1-\varepsilon}]  \right)  \nonumber\\
&  = : J_1 + J_2.
\end{align}
\textit{Bound of $J_1$}.  
By Cauchy-Schwarz's inequality, Minkowski's inequality and Lemma \ref{Lem-NuTzEs}, 
we obtain that for any  $\ee \in  (0, \frac{\delta}{2(2 + \delta)}]$, 
\begin{align}\label{z+Mn killed Tz J1 bound}
J_1 & \leq \mathbb{E} \left( x + S_n;  \nu_{n} > [n^{1-\varepsilon}]  \right)  
     \leq  \mathbb{E}^{1/2} \left( |x + S_n|^2 \right) 
       \mathbb{P}^{1/2} \left( \nu_n > [n^{1-\varepsilon}] \right)  \nonumber\\
    & \leq  \left( x + c n^{1/2} \right) c_\ee e^{- c_\ee n^{\ee}}  
    \leq  c_\ee (1 + x) e^{- c_\ee n^{\ee}}. 
\end{align}
\textit{Bound of $J_2$.}  Using the Markov property leads to
\begin{align}\label{z+Mn killed Tz J2}
J_2  & =   \sum_{k=1}^{ [ n^{1-\varepsilon} ] } \mathbb{E} (x + S_n; \tau_x > n, \nu_n = k)  \nonumber\\
     & =  \sum_{k=1}^{ [ n^{1-\varepsilon} ] }  
    \int_{\bb R }  V_{n-k} (x')
     \mathbb{P} \left( x + S_k \in dx'; \tau_x > k, \nu_n = k \right).
\end{align}
Note that on the event $\{\tau_x > k, \nu_n = k \}$, we have $x' = x + S_k > n^{ 1/2 -\varepsilon }$.
By Lemma \ref{Lem-MKillTn}, it holds that for any  $\ee \in  (0, \frac{\delta}{2(2 + \delta)}]$, 
\begin{align*} %\label{z+Mn killed Tz J2 2}
V_{n-k} (x') = \mathbb{E}  \left(  x' + S_{n-k}; \tau_{x'} > n - k  \right)
\leq  \left(1 + \frac{c}{(n-k)^{  \ee}} \right) x'.
\end{align*}
Implementing this bound into \eref{z+Mn killed Tz J2}, we derive that %for large $n \geq 1$, 
\begin{align}\label{y + M_n killed Tz001}
J_2 
& \leq  \sum_{k=1}^{ [ n^{1-\varepsilon} ] }  
   \left(1 +  \frac{c}{(n-k)^{ \ee}}   \right)  \mathbb{E} ( x + S_k;  \tau_x > k, \nu_n=k )  \nonumber\\
& \leq  \left(1 +  \frac{c_{\ee}}{n^{ \ee} }  \right) 
    \sum_{k=1}^{ [ n^{1-\varepsilon} ] }   \mathbb{E} ( x + S_k;  \tau_x > k, \nu_n=k ). 
\end{align}
Since $\big( (x + S_n) \mathds{1}_{\{\tau_x > n\}} \big)_{n\geq1}$ is a submartingale with respect to the natural filtration 
$\mathscr F_{n} = \sigma (X_1, \ldots, X_n)$, 
we get that for any $1 \leq k \leq [ n^{1-\varepsilon} ]$,
\begin{align*}
\mathbb{E} ( x + S_k;  \tau_x > k, \nu_n = k ) 
\leq   \mathbb{E} \left( x + S_{ [ n^{1-\ee} ] };  \tau_x > [ n^{1-\ee} ], \nu_n = k  \right), 
\end{align*}
and consequently, for any  $\ee \in  (0, \frac{\delta}{2(2 + \delta)}]$, 
\begin{align}\label{Pf_Sum_killed_J2}
J_2  \leq  \left(1 +  \frac{ c_\ee }{ n^{ \ee} }  \right)   V_{[ n^{1-\ee} ]} (x). 
%    \mathbb{E} \left( x + S_{ [ n^{1-\ee} ] };  \tau_x > [ n^{1-\ee} ]  \right).  
\end{align}
Substituting \eqref{Pf_Sum_killed_J2} and \eref{z+Mn killed Tz J1 bound} into \eref{z+Mn killed Tz J1 et J2} gives 
\begin{align*}
V_{n} (x) \leq  \left(1 +   \frac{ c_\ee }{ n^{ \ee} } \right)  V_{[ n^{1-\ee} ]} (x)
       + c_{\ee} (1 + x) e^{ -c_{\ee} n^{\ee} }.
\end{align*}
%\begin{align*}
%\mathbb{E} (x + S_n; \tau_x > n)
%\leq  \mathbb{E} ( x + S_{ [n^{1-\varepsilon} ] }; \tau_x > [n^{1-\varepsilon} ])  
%  + c_{\varepsilon} e^{ -c_{\varepsilon}n^{\varepsilon} }. 
%\end{align*}
%Let $u_n = \mathbb{E} (x + S_n; \tau_x > n)$. Then, the above inequality can be rewritten as
%$u_n \leq  u_{ [ n^{1-\varepsilon} ] } + c_{\varepsilon} e^{ -c_{\varepsilon}n^{\varepsilon} }.$
%Since the sequence $(u_n)_{n \geq 1}$ is increasing, 
Applying \cite[Lemma A.1]{GLL18}, we get that for any integer $k_0 \in [1, n]$,
\begin{align*}
V_{n} (x)  \leq  \left(1 +   \frac{ c_\ee }{ k_0^{  \ee} }  \right)   V_{k_0} (x) 
       + c_{\ee} (1 + x) e^{ -c_{\ee}  k_0^{\ee} }.
\end{align*}
By Lemma \ref{Lem-MKillTn}, we have $V_{k_0} (x) \leq x + c k_0^{1/2 - \ee}$ 
and therefore,  for any  $\ee \in  (0, \frac{\delta}{2(2 + \delta)}]$,  
\begin{align}\label{Pf_Recur_V}
V_{n} (x)  
& \leq  \left(1 +  \frac{ c_\ee }{ k_0^{  \ee} } \right)  x + c k_0^{1/2 - \ee} 
       + c_{\ee} (1 + x) e^{ -c_{\ee}  k_0^{\ee} }  \notag\\
& \leq  \left(1 +  \frac{ c_\ee' }{ k_0^{ \ee} } \right)  x + c_{\ee}' k_0^{1/2 - \ee},  
\end{align}
which concludes the proof of the lemma. 
\end{proof}

%Based on Lemma \ref{z+Mn killed Tz}, 
The following result proves the existence and gives some properties of the harmonic function $V$. 

\begin{lemma} \label{Theorem harmonic func} 
Assume \ref{SecondMoment} for some $\delta >0$. 

\noindent 1. For any $x\geq 0$, the limit $\lim_{n \to \infty }\mathbb{E}\left(x+S_n ;\tau_x >n\right)$
exists and 
\begin{align}\label{Def_V_002}
V(x): = \lim_{n \to \infty } \mathbb{E} \left(x+S_n ; \tau_x  > n \right). 
\end{align}
\noindent 2.  The function $V$ is increasing on $\bb R _{+}$
and $x \leq V(x) \leq c(1 + x) $ for any $x>0$.  % satisfies the following properties: 
Moreover, $\lim_{x\to \infty }\frac{V(x)}{x} = 1.$ 
%a) $V(x) >0$ for $x>0$; 

%a) $V(\cdot)$ is increasing on $\bb R _{+}$;  
% 
%b) $x \leq V(x) \leq c(1 + x) $ for any $x>0$;
% 
%c) $\lim_{x\to \infty }\frac{V(x)}{x} = 1.$

\noindent 3. The function $V$ is harmonic in the sense that for any $x \geq 0,$ 
\begin{align}\label{V_Harmonicity}
\mathbb{E} \left( V \left( x + S_{1} \right); \tau_{x} > 1 \right) = V(x).
\end{align}
\end{lemma}

\begin{proof}%[Proof of Proposition \ref{Theorem harmonic func}]
We proceed with the first assertion. Using \eqref{Pf_Optional_thm} and Lemma \ref{z+Mn killed Tz}, we get that for $x \geq 0$, 
\begin{align}\label{Pf_V_Sec_def}
x \bb P(\tau_x > n)  +  \mathbb{E} (- S_{\tau_x}; \tau_x \leq n )
= \mathbb{E} (x + S_n; \tau_x > n )
  \leq c(1+x).
%= \mathbb{E} (x + S_n) -  \mathbb{E} (x+ S_{\tau_x}; \tau_x \leq n ) \leq c(1+x). 
\end{align}
Since $x + S_{\tau_x} \leq 0$,  we have $S_{\tau_x} \leq 0$ and hence, using \eqref{Pf_V_Sec_def},
\begin{align*}
\mathbb{E} ( - S_{\tau_x}; \tau_x \leq n ) \leq c(1+x). 
\end{align*} 
By the Lebesgue monotone convergence theorem,  
it follows that 
%the limit $\lim_{n \to \infty} \mathbb{E} ( - S_{\tau_x}; \tau_x \leq n )$ exists and 
\begin{align}\label{Pf_Integrab_Sn}
\lim_{n \to \infty} \mathbb{E} ( - S_{\tau_x}; \tau_x \leq n )
= \mathbb{E} ( - S_{\tau_x} ) \leq c(1+x). 
\end{align}
This, together with \eqref{Pf_V_Sec_def} and the fact that $\bb P(\tau_x > n) \to 0$ as $n \to \infty$, proves \eqref{Def_V_002}. 

We now prove the second assertion. Since for any $0 \leq x_1 \leq x_2$, 
\begin{align*}
\mathbb{E} (x_1 + S_n; \tau_{x_1} > n )  \leq  \mathbb{E} (x_2 + S_n; \tau_{x_2} > n ), 
\end{align*}
from \eqref{Def_V_002} it follows that $V$ is increasing on $\bb R_+$. 

Using \eqref{Pf_Optional_thm} and the fact that $x + S_{\tau_x} \leq 0$, we see that 
$\mathbb{E} (x + S_n; \tau_x > n ) \geq x$.
Hence, by \eqref{Def_V_002}, we get $V(x) \geq x$. 
Since $V(x) = - \mathbb E S_{\tau_x}$, using \eqref{Pf_Integrab_Sn} we get that $V(x) \leq c(1+x)$ for some constant $c >0$.
Therefore $x \leq V(x) \leq c(1+x)$.

From \eqref{Pf_Recur_V} and \eqref{Def_V_002}, it follows that 
\begin{align*}
\limsup_{x \to \infty} \frac{V(x)}{x} \leq  1 +  \frac{ c_\ee }{ k_0^{ \ee} }.   
\end{align*}
Taking $k_0 \to \infty$, we get $\limsup_{x \to \infty} V(x)/x \leq 1$.
This, together with the inequality $x \leq V(x)$, shows that $\lim_{x \to \infty} V(x) / x = 1$.

For the third assertion, denoting $V_n (x) = \mathbb{E} (x + S_n; \tau_x > n)$, by the Markov property, we have
\begin{align*}
V_{n+1} (x) = \int_{\bb R} V_n(x') \bb P \left( x + S_1 \in dx', \tau_x >1 \right). 
\end{align*}
The identity \eqref{V_Harmonicity} follows from the Lebesgue dominated convergence theorem, 
the inequality $V_n(x') \leq c(1+x')$  and \eqref{Def_V_002}. 
%we obtain \eqref{V_Harmonicity}. 
\end{proof}

\begin{proof}[Proof of Lemma \ref{Lem_V_Ineq_aa}]
The assertion of the lemma follows from \eqref{Pf_Recur_V} and \eqref{Def_V_002}. 
\end{proof}

Recall that $\nu_n$ is a stopping time defined by \eqref{Def_nun}. 

\begin{lemma}\label{lemma E1 and E2}
Assume \ref{SecondMoment} for some $\delta >0$. 
Then,  for any  $\ee \in  (0, \frac{\delta}{2(2 + \delta)}]$,  there exists a constant $c_{\ee} >0$ such that 
%There exists $\ee_0>0$ such that for any $\ee\in (0, \ee_0)$, 
for any $x \geq 0$ and $n \geq 1$, 
\begin{align}
 E_1:= \mathbb{E} \left( x + S_{\nu_n}; \tau_x > \nu_n, \nu_n \leq [n^{1-\ee}] \right)
  \leq c_\ee (1 + x),   \label{control of E1}
%&  E_2:= \mathbb{E} \left( x + S_{\nu_n}; \tau_x > \nu_n,  \nu_n \leq [n^{1-\ee}] \right) 
%        \xrightarrow{ n \to \infty } V(x). \nonumber
\end{align}
and 
%Moreover, for any $n\geq 1$, $\ee\in (0, \ee_0)$ and $x \geq 0$,
\begin{align}\label{control of E2}
E_1 \leq  V(x)  \leq  \left( 1 +  \frac{ c_\ee }{ n^{  \ee/2} }  \right) E_1
 + c_{\ee} (1 +x) e^{ -c_{\ee}n^{\ee} }. 
\end{align}
\end{lemma}

\begin{proof}
We first prove \eqref{control of E1}. 
Since $x + S_{\nu_n}>0$, we have
\begin{align*}
E_1 
& \leq  \mathbb{E} \left( x + S_{[n^{1-\ee}]};  \tau_{x} > [n^{1-\ee}],  \nu_n \leq [n^{1-\ee}] \right)  \nonumber\\
& \leq  \mathbb{E} \left( x + S_{[n^{1-\ee}]};  \tau_{x} > [n^{1-\ee}] \right).  
\end{align*}
Using Lemma \ref{z+Mn killed Tz}, we obtain \eqref{control of E1}.

We next prove \eqref{control of E2}. 
By Lemma \ref{Theorem harmonic func} 
and the fact that $\big( V(x + S_n) \mathds{1}_{\{ \tau_x > n\}} \big)_{n \geq 0}$ is a martingale
with respect to the filtration $\mathscr F_{n} = \sigma (X_1, \ldots, X_n)$, we get 
\begin{align}
V(x) & =  \mathbb{E} \left( V(x + S_n); \tau_x > n, \nu_n \leq [n^{1-\ee}] \right)  \nonumber\\
  & \quad   + \mathbb{E} \left( V(x + S_n); \tau_x > n, \nu_n > [n^{1-\ee}]  \right) \nonumber\\ 
  & =  \mathbb{E} \left( V(x + S_{\nu_n}); \tau_x > \nu_n, \nu_n \leq [n^{1-\ee}] \right)  \nonumber\\
  & \quad + \mathbb{E} \left( V(x + S_n); \tau_x > n, \nu_n > [n^{1-\ee}] \right).   \label{control of E2 1}
\end{align}
On the one hand, since $V(x) \geq x$, from \eqref{control of E2 1} we get
\begin{align}
V(x)  \geq \mathbb{E} \left( V(x + S_{\nu_n}); 
     \tau_x > \nu_n, \nu_n \leq [n^{1-\ee}] \right) \geq  E_1. \label{E2 lower bound}
\end{align}
On the other hand, 
using Lemma \ref{Lem_V_Ineq_aa} and the inequality $V(x) \leq c(1 + x)$,  we obtain
that for any  $\ee \in  (0, \frac{\delta}{2(2 + \delta)}]$, 
\begin{align}\label{control of E21 1}
V(x) & \leq  \left(1 +  \frac{ c_\ee }{ k_0^{  \ee} } \right)   E_1 
     +  c_{\ee} k_0^{1/2 - \ee} \mathbb{P} \left( \tau_x > \nu_n, \nu_n \leq [n^{1-\ee}] \right)   \nonumber\\
    & \quad +  c \mathbb{E} \left( 1 + x + S_n; \tau_x >n, \nu_n > [n^{1 - \ee}] \right).
\end{align}
By H\"older's inequality and Lemma \ref{Lem-NuTzEs}, we get that for any $\ee \in (0,1/2)$,  
\begin{align}\label{control of E22 all}
 \mathbb{E} \left( x + S_n; \tau_x >n, \nu_n > [n^{1 - \ee}] \right)  
& \leq c (x + n^{1/2}) \mathbb{P}^{1/2} \left( \nu_n > [n^{1-\ee}] \right)  \nonumber\\
& \leq   c_{\ee} (1 + x) e^{ -c_{\ee}n^{\ee} }. 
 % \exp\!\left(-c_{\ee}n^{\ee}\right).
\end{align}
%\begin{align}
%V(x,y)\leq   &\  \mathbb{E}_x(V(X_{\nu_n^{\ee^2}}, y+S_{\nu_n^{\ee^2}}); \tau_y>\nu_n^{\ee^2}, \nu_n^{\ee^2}\leq [n^{1-\ee}])  \nonumber\\
%  &\   +\mathbb{E}_x(V(X_n, y+S_n); \tau_y>n, \nu_n^{\ee^2}>[n^{1-\ee}]) \nonumber\\
%\leq   &\  \mathbb{E}_x(V(X_{\nu_n^{\ee^2}}, y+S_{\nu_n^{\ee^2}}); \tau_y>\nu_n^{\ee^2}, \nu_n^{\ee^2}\leq [n^{1-\ee}])  \nonumber\\
%  &\   +c\mathbb{E}_x\left(1+y + M_n; \tau_y>n, \nu_n^{\ee^2}> [n^{1-\ee}]\right) \nonumber\\
%\leq   &\  
%\left(1+\frac{c_\ee}{k_0^{4\ee}}\right)E_2
%+c_\ee k_0^{1/2-2\ee}\mathbb{P}_x\left(\tau_y>\nu_n^{\ee^2}, \nu_n^{\ee^2}\leq [n^{1-\ee}]\right) \nonumber\\
%  &\   -c_\ee \mathbb{E}_x\left(z+M_{\nu_n^{\ee^2}}; z+M_{\nu_n^{\ee^2}}<0, \tau_y>\nu_n^{\ee^2}, \nu_n^{\ee^2}\leq [n^{1-\ee}]\right) \nonumber\\
%  &\  
%+c\mathbb{E}_x\left(1+y + M_n; \tau_y>n, \nu_n^{\ee^2}> [n^{1-\ee}]\right). \nonumber
%\end{align}
Since $n^{1/2-\ee}\leq x + S_{\nu_n}$ on the event $\{\tau_x > \nu_n \}$, it holds that  
\begin{align} \label{tau n and nu n}
&   \mathbb{P} \left(\tau_x > \nu_n, \nu_n \leq [n^{1 - \ee}]\right) \nonumber\\
& \leq  \frac{1}{n^{1/2-\ee}}\mathbb{E} \left(x + S_{\nu_n}; \tau_x > \nu_n, \nu_n\leq [n^{1-\ee}]\right) 
=  \frac{1}{n^{1/2-\ee}} E_1. 
\end{align}
Combining \eref{control of E21 1}, \eref{control of E22 all} and \eqref{tau n and nu n} 
gives that for any  $\ee \in  (0, \frac{\delta}{2(2 + \delta)}]$, 
%\begin{align}
%V(x)  \leq  \left(1 +  \frac{c_{\ee}}{k_0^{\ee}}  \right) E_1 
%+  c_{\ee} k_0^{1/2}  \mathbb{P} \left( \tau_x > \nu_n, \nu_n \leq [n^{1-\ee}] \right)
%+ c_{\ee} (1+x) e^{ -c_{\ee}n^{\ee} }.  \nonumber
%\end{align}
%
%It follows that
\begin{align*}
V(x) \leq  \left( 1 +  \frac{ c_\ee }{ k_0^{  \ee} } + \frac{ c_{\ee} k_0^{1/2-\ee} }{ n^{1/2-\ee} }  \right) E_1
 + c_{\ee} (1 +x) e^{ -c_{\ee}n^{\ee} }. 
\end{align*}
By taking $k_0 = n^{1 - 2 \ee}$, this implies that 
\begin{align}\label{E2 upper bound}
V(x) \leq  \left( 1 +  \frac{ c_\ee }{ n^{  \ee (1 - \ee)} }  \right) E_1
 + c_{\ee} (1 +x) e^{ -c_{\ee}n^{\ee} }. 
\end{align}
%Putting together \eqref{E2 lower bound} and \eqref{E2 upper bound}, 
%%and taking $k_0 = n^{1/2}$, 
%we obtain 
%\begin{align}
%|V(x) - E_1| \leq  \frac{ c_\ee }{ n^{ 1/2 - 2 \ee} }   (1 + E_1). \nonumber
%\end{align}
Combining this with \eqref{E2 lower bound},  
we get \eqref{control of E2}. % follows. % by using \eref{control of E1}.
\end{proof}

\begin{lemma}\label{Lem_tau01}
Assume \ref{SecondMoment} for some $\delta >0$. 
Then,  for any  $\ee \in  (0, \frac{\delta}{2(2 + \delta)}]$,  there exists a constant $c_{\ee} >0$ such that 
%There exists $\ee_0>0$ such that for any $\ee\in (0, \ee_0)$, 
for any $x \geq 0$ and $n \geq 1$, 
%There exists $\ee_0>0$ such that for any $\ee\in (0, \ee_0)$, $x \geq 0$ and $n\geq 1$,
\begin{align*}
\mathbb{P} (\tau_x > n) \leq \frac{c_\ee}{n^{1/2-\ee}}(1 + x). 
\end{align*}
\end{lemma}

\begin{proof}
By Lemmas \ref{lemma E1 and E2} and \ref{Lem-NuTzEs}, we get
\begin{align*}
\mathbb{P} ( \tau_x > n ) 
& =  \mathbb{P} \left( \tau_x > n, \nu_n \leq [n^{1-\ee}] \right)
  + \mathbb{P} \left( \tau_x > n, \nu_n > [n^{1-\ee}] \right)  \nonumber\\
& \leq  \frac{1}{n^{1/2-\ee}}
  \mathbb{E} \left(x + S_{\nu_n}; \tau_x > \nu_n, \nu_n \leq [n^{1-\ee}] \right)
 + c_\ee e^{-c_\ee n^\ee}. 
\end{align*}
Using the bound \eqref{control of E1}, the result follows. 
\end{proof}

\begin{lemma}\label{Lem-E3-bound}
Assume \ref{SecondMoment} for some $\delta >0$. 
There exists $\ee_0>0$ such that for any $\ee\in (0, \ee_0)$, 
$x \geq 0$ and $n\geq 1$, %and $\alpha \geq n^{- \ee}$,  
\begin{align*}
E_2:  =   \mathbb{E} 
 \left( x + S_{\nu_n}; x + S_{\nu_n} >  n^{1/2-\ee/2}, \tau_x > \nu_n, \nu_n \leq [n^{1-\ee}] \right) 
  \leq  \frac{  c_\ee  (1 + x) }{ n^{\delta/2 - \ee (1 + \ee + \delta/2)} }.
\end{align*}
%%\overset{k\rightarrow\infty}{\longrightarrow}\stackrel{n\to\infty}{\longrightarrow}
%Moveover, for any $n\geq 1$, $\ee\in (0, \ee_0)$, 
%$x\in \mathcal{S}$, $y\in\bb R $ and $z=y-\psi_1(x)$, 
%\begin{align}
%E_3\leq \frac{c_\ee\left(1+y\mathds{1}_{\{y>n^{1/2-2\ee}\}}\right)^2}{n^{1/2-3\ee}}
%+\frac{c_\ee(1+\max(y,0))}{n^{2\ee}}. \nonumber
%\end{align}
\end{lemma}

\begin{proof}
Since $x + S_{\nu_n} \leq X_{\nu_n} + n^{1/2-\ee}$
 on the event $\{ x + S_{\nu_n} > n^{1/2-\ee/2} \}$,
%we have 
%%\begin{align}\label{E21_nun_aa}
%%x + S_{\nu_n} \leq X_{\nu_n} + n^{1/2-\ee}. 
%%\end{align}
%$x + S_{\nu_n} \leq X_{\nu_n} + n^{1/2-\ee}$.
we have 
\begin{align*}%\label{}
E_{2} \leq  
 \bb{E} \left( x + S_{\nu_n}; X_{\nu_n} > n^{1/2-\ee/2} - n^{1/2-\ee}, \tau_x > \nu_n, \nu_n \leq [n^{1-\ee}] \right). 
\end{align*}
Since  
%\begin{align}\label{E21_nun_bb}
%n^{1/2-\ee/2} - n^{1/2-\ee} > c_\ee n^{1/2-\ee/2}. 
%\end{align}
$n^{1/2-\ee/2} - n^{1/2-\ee} > c_\ee n^{1/2-\ee/2}$,  
it follows that 
\begin{align*}%\label{}
E_{2}  
& \leq  \sum_{k = 1}^{ [n^{1-\ee}] } 
\bb{E} \left( x + S_{k}; X_{k} > c_\ee n^{1/2-\ee/2}, \tau_x > k  \right)  \notag\\
&  \leq  \sum_{k = 1}^{ [n^{1-\ee}] } 
\bb{E} \left( x + S_{k-1} + X_k;  X_{k} > c_\ee n^{1/2-\ee/2}, \tau_x > k - 1  \right)  \notag\\
&  =   \sum_{k = 1}^{ [n^{1-\ee}] } 
\bb{E} \left( x + S_{k-1};   \tau_x > k - 1  \right)   \bb P \left( X_{k} > c_\ee n^{1/2-\ee/2} \right)    \notag\\
& \quad  +   \sum_{k = 1}^{ [n^{1-\ee}] } 
\bb{E} \left( X_k;  X_{k} > c_\ee n^{1/2-\ee/2} \right)   \bb P \left( \tau_x > k - 1  \right)  \notag\\
& =:  E_{21} + E_{22}. 
\end{align*}
For $E_{21}$,  using Lemma \ref{z+Mn killed Tz}, Markov's inequality and condition \ref{SecondMoment}, 
%the fact that $\bb E (|X_k|^{2 + \delta}) < \infty$,   
we get
\begin{align*}%\label{}
E_{21} \leq   c (1 + x)  \sum_{k = 1}^{ [n^{1-\ee}] }   \bb P \left( X_{k} > c_\ee n^{1/2-\ee/2} \right) 
  \leq  \frac{c (1 + x) }{n^{\delta/2  - \ee \delta/2}}. 
\end{align*}
For $E_{22}$, using Lemma \ref{Lem_tau01} and again Markov's inequality, we obtain 
\begin{align*}%\label{}
E_{22}  
& \leq  \bb{E} \left( X_1;  X_{1} > c_\ee n^{1/2-\ee/2} \right)   \sum_{k = 1}^{ [n^{1-\ee}] }   \bb P \left( \tau_x > k - 1  \right)   \notag\\
& \leq  \frac{c_{\ee}}{ n^{(1/2 - \ee/2)(1 + \delta)} }  \sum_{k = 1}^{ [n^{1-\ee}] }  \frac{c_\ee}{k^{1/2-\ee}}(1 + x)  \notag\\
& \leq   \frac{c_{\ee}}{ n^{(1/2 - \ee/2)(1 + \delta)} }     c_\ee  (1 + x)  n^{(1-\ee)(1/2 + \ee)}  
 \leq \frac{  c_\ee  (1 + x) }{ n^{\delta/2 - \ee (1 + \ee + \delta/2)} }.  
\end{align*}
%Since 
%\begin{align*}%\label{}
%\bb{E} \left( X_1;  X_{1} > c_\ee n^{1/2-\ee/2} \right)  
%& =   \int_{0}^{\infty}  \bb P \left( X_{1} > t,  X_{1} > c_\ee n^{1/2-\ee/2} \right) dt  \notag\\
%&  =   c_\ee n^{1/2-\ee/2}     \bb P \left( X_{1} > c_\ee n^{1/2-\ee/2} \right)
%   +      \int_{ c_\ee n^{1/2-\ee/2} }^{\infty} \bb P \left( X_{1} > t \right)  dt  \notag\\
%& \leq  \frac{c_{\ee}}{ n^{(1/2 - \ee/2)(1 + \delta)} }.  
%\end{align*}
%\begin{align*}
%E_{2}  \leq  \sum_{ k =1}^{ [n^{1-\ee}] }  \mathbb{E} \left( X_{k}; X_{k} > c_\ee n^{1/2-\ee/2} \right)
%   + n^{1/2-\ee} \sum_{ k =1}^{ [n^{1-\ee}] }  \mathbb{P} \left( X_{k} > c_\ee n^{1/2-\ee/2}  \right).
%\end{align*}
The desired result follows.
% by 
%using the Markov inequality and the fact that 
% $\bb E (X_1^{2+\delta}) < \infty$ for all $\delta >0$. 
% we deduce that for any $x \geq 0$ and $p \geq 1$,
%\begin{align}\label{E3 1}
%E_{21} \leq \frac{ c_{\ee} }{ n^{p} }. 
%\end{align}
\end{proof}

\subsection{Proof of Theorems \ref{Theor-IntegrLimTh} and \ref{CorCCLT}}

We first give a proof of Theorem \ref{Theor-IntegrLimTh}
by using the bounds shown in Section \ref{Sec_PfThm_a}
and the  functional central limit theorem (Lemma \ref{FCLT}).

\begin{proof}[Proof of Theorem \ref{Theor-IntegrLimTh}]
%\begin{theorem}
%Assume conditions A1-A4, then we have that

%1. For any $y\in\bb R $, $x\in \mathcal{S}$ and $t\geq 0$,
%\begin{align}
%\lim_{n \to \infty} \mathbb{P}_x\left(\frac{y+S_n}{\sigma\sqrt{n}}\leq t \Big| \tau_y>n\right)=\Phi^{+}(t), \nonumber
%\end{align}
%where $\Phi^+(t)=1-e^{-\frac{t^2}{2}}$ is the Rayleigh distribution function.

%2. Moreover, for any $\ee\in (0,\ee_0)$, there exists a constant $c_\ee>0$ such that for any $n\geq 1$, $y\in\bb R $, $x\in \mathcal{S}$ and $t\geq 0$,
%\begin{align}
%\left|\mathbb{P}_x\left(\frac{y+S_n}{\sigma\sqrt{n}}\leq t, \tau_y>n\right)
%-\frac{2V(x,y)}{\sqrt{2\pi n}\sigma}\Phi^{+}(t)\right|
%\leq
%\frac{c_\ee\left(1+y\mathds{1}_{\{y>n^{1/2-2\ee}\}}\right)^2}{n^{1-3\ee}}
%+\frac{c_\ee(1+\max(y,0))}{n^{1/2+\ee/8}}.  \nonumber
%\end{align}
%\end{theorem}
As in \eqref{Def_A_k_aa}, denote %$A_k = \big\{ \sup_{0\leq t\leq 1} | S_{[tk]}-\sigma B_{tk} | \leq k^{1/2 - 2 \ee} \big\}$
\begin{align}\label{Def_A_k_aabis}  
A_k = \left\{ \sup_{0 \leq t \leq 1} \left| S_{[tk]} -  B_{tk} \right| \leq  k^{1/2 - 2 \ee}  \right\} 
\end{align}
and by $A_k^c$ its complement. 
Using the Markov property, we have 
\begin{align}
\mathbb{P} \left( x + S_n \leq t \sqrt{n}, \tau_x > n \right)  =:  J_1 + J_2 + J_3, 
\end{align}
where %$A_k$ is defined in \eqref{Def_A_k_aa} and 
\begin{align*}
& J_1 = \mathbb{P} \left( x + S_n \leq t  \sqrt{n}, \tau_x > n, \nu_n > [n^{1-\ee}] \right),  \nonumber\\
& J_2 = \sum_{k = 1}^{[n^{1-\ee}]} 
    \int_{\bb R } \mathbb{P}( x' + S_{n-k} \leq t \sqrt{n}, \tau_{x'} > n-k, A_{n-k}^c)  \nonumber\\
  & \qquad\qquad\qquad\qquad  \times
\mathbb{P} \left( x + S_k \in dx', \tau_x > k, \nu_n = k  \right),  \nonumber\\
& J_3 =  \sum_{k = 1}^{[n^{1-\ee}]}
  \int_{\bb R } \mathbb{P}(x' + S_{n-k} \leq t \sqrt{n}, \tau_{x'}>n-k, A_{n-k})  \nonumber\\
  & \qquad\qquad\qquad\qquad  \times
\mathbb{P} \left( x + S_k\in dx', \tau_x > k, \nu_n = k \right). 
\end{align*}
\textit{Bound of $J_1$}. 
By Lemma \ref{Lem-NuTzEs}, it holds that
\begin{align} \label{Pf-Thm3-J1}
J_1 \leq  \mathbb{P} \left( \nu_n > [n^{1-\ee}] \right) 
\leq   c_{\ee} e^{ -c_{\ee}n^{\ee} }. 
\end{align}
\textit{Bound of $J_2$}. 
For $k\leq [n^{1-\ee}]$, we have $n-k\geq c_\ee n$. 
By Lemma \ref{FCLT},  
\begin{align}\label{KMT_Inequ}
\mathbb{P} \left( x' + S_{n-k} \leq t \sqrt{n},  \tau_{x'} > n-k, A_{n-k}^c \right)
\leq \mathbb{P} \left( A_{n-k}^c \right)
\leq  \frac{c_{\ee}}{ n^{ \delta/2 - 2 \ee (2+\delta) } }. 
\end{align}
Note that on the event $\{ \tau_x > \nu_n \}$,
we have $n^{1/2-\ee}\leq x + S_{\nu_n}$. 
By Lemma \ref{lemma E1 and E2},  we get
\begin{align} \label{Pf-Thm3-J2}
J_2  &  \leq   \frac{c_{\ee}}{ n^{ \delta /2 - 2 \ee (2+\delta) } }
   \mathbb{P} \left(\tau_x > \nu_n, \nu_n \leq [n^{1-\ee}]\right)  \nonumber\\
& \leq  \frac{c_{\ee}}{ n^{ (\delta + 1) /2 - 2 \ee (2+\delta) - \ee } }
  \mathbb{E} \left(x + S_{\nu_n}; \tau_x > \nu_n, \nu_n \leq [n^{1-\ee}]\right)  \nonumber\\
&  \leq  \frac{c_\ee( 1 + x ) }{ n^{ (\delta + 1)/2 - \ee (5 + 2\delta)  }  }.
\end{align}
\textit{Bound of $J_3$}. 
%Similarly to the bound of $I_2$ in the proof of Theorem \ref{Theor-limit tau}, 
We write $J_3 = J_{31} + J_{32}$, where 
\begin{align*}
J_{31} & =  \sum_{k = 1}^{[n^{1-\ee}]}
\int_{\bb R } \mathbb{P} (x' + S_{n-k} \leq t  \sqrt{n}, \tau_{x'} > n-k, A_{n-k} ) \nonumber\\
  &\qquad  \times
  \mathbb{P} ( x + S_k \in dx', x + S_k > n^{1/2-\ee/2}, \tau_x > k, \nu_n = k),  \nonumber\\
J_{32} & =  \sum_{k = 1}^{[n^{1-\ee}]} 
\int_{\bb R }\mathbb{P} (x' + S_{n-k} \leq t  \sqrt{n}, \tau_{x'} > n-k, A_{n-k})\nonumber\\
  & \qquad  \times
\mathbb{P} ( x + S_k \in dx', x + S_k \leq n^{1/2-\ee/2}, \tau_x > k, \nu_n = k). 
\end{align*}
%\begin{align} \label{Pf-Thm3-J3}
%J_3 & =  
%\sum_{k = 1}^{[n^{1-\ee}]}
%\int_{\bb R } \mathbb{P} (x' + S_{n-k} \leq t \sigma\sqrt{n}, \tau_{x'}>n-k, A_{n-k})\nonumber\\
%  &\qquad  \times
%\mathbb{P}( x + S_k \in dx', x + S_k > n^{1/2-\ee/8}, \tau_x > k, \nu_n = k) \nonumber\\
%& \quad +   
%\sum_{k = 1}^{[n^{1-\ee}]} 
%\int_{\bb R }\mathbb{P} (x' + S_{n-k} \leq t\sigma\sqrt{n}, \tau_{x'} > n-k, A_{n-k})\nonumber\\
%  & \qquad  \times
%\mathbb{P} ( x + S_k \in dx', x + S_k \leq n^{1/2-\ee/8}, \tau_x > k, \nu_n = k) \nonumber\\
%=  &\  J_{31} + J_{32}.  
%\end{align}
\textit{Bound of $J_{31}$}. 
Denote $x_+' = x' + (n-k)^{1/2-2\ee}$. 
Since $x' > n^{1/2-\ee/2}$, from \eqref{Def_A_k_aabis} and \eqref{levy001a}, we derive that
%We deduce from Lemma \ref{lemma expantau} %\ref{Levy lemma 2} that \todo{??? !!!}
\begin{align}\label{coupling result 1}
\mathbb{P} \left( \tau_{x'} > n-k,  A_{n-k} \right) 
\leq \mathbb{P} \left( \tau_{x_+'}^{bm} > n-k \right)
\leq c \frac{x_+'}{\sqrt{n - k}}
\leq  \frac{ c_{\ee} x' }{ \sqrt{n} }.
\end{align}
Implementing the bound \eqref{coupling result 1} into $J_{31}$, 
and using Lemma \ref{Lem-E3-bound}, we obtain % that for any $p \geq 1$, 
\begin{align}\label{Pf-Thm3-J31}
J_{31}  \leq   \frac{c_\ee}{\sqrt{n}} E_2  
\leq  \frac{  c_\ee  (1 + x) }{ n^{(1+\delta)/2 - \ee (1 + \ee + \delta/2)} }. 
\end{align}
%Using the bound \eref{coupling J2 1} for $I_{21}$, we get
%\begin{align}\label{Pf-Thm3-J31}
%J_{31} \leq I_{21} 
%\leq  \frac{  c_\ee  (1 + x) }{ n^{(1+\delta)/2 - \ee (1 + \ee + \delta/2)} }. 
%\end{align}
%%For $J_{32}$, 
%%we shall give the upper and lower bounds below. 
\textit{Upper bound of $J_{32}$}. 
Denote $x_+' = x' + (n-k)^{1/2-2\ee}$
and $t_+ = t+\frac{1}{n^{2\ee}}$.
%Notice that on the event $\{x' + S_{n-k} \leq t \sqrt{n}, \tau_{x'} > n-k, A_{n-k}\}$, we have
%$x_+' +  B_{n-k}\leq t_+ \sqrt{n}$ and $\tau_{x_+'}^{bm} > n-k$. 
By Lemma \ref{lemma tauBM},  we have 
\begin{align}
\mathcal K(x')  :  & =   \mathbb{P}(x' + S_{n-k }\leq t  \sqrt{n}, \tau_{x'} > n-k, A_{n-k})  \nonumber\\
& \leq  \mathbb{P} \left( x_+' + \sigma B_{n-k}\leq t_+ \sqrt{n}, \tau_{x_+'}^{bm} > n-k\right) \nonumber\\
&  =    \frac{1}{  \sigma \sqrt{2\pi(n-k)} } \int_0^{ \sqrt{n}t_+ }
\Big( e^{ -\frac{ ( u - x_+')^2 }{ 2 \sigma^2 (n-k) } } - e^{ -\frac{ ( u + x_+')^2}{ 2 \sigma^2 (n-k)} } \Big) du \nonumber\\
& \leq  \frac{1}{ \sigma \sqrt{2\pi(n-k)} }    \int_0^{ \sqrt{n}t_+ }  
    e^{- \frac{u^2}{2 \sigma^2 (n-k)}} \Big( e^{ \frac{u x_+'}{ \sigma^2 (n-k) } }  - e^{ - \frac{u x_+'}{ \sigma^2 (n-k) } } \Big)  du \notag\\
& =  \frac{1}{\sqrt{2\pi}}    \int_0^{  \sqrt{\frac{n}{n-k}} t_+ }
 e^{- \frac{s^2}{2} } \Big( e^{  s  \frac{ x_+' }{\sqrt{n-k}}  } - e^{ - s \frac{ x_+' }{\sqrt{n-k}}   } \Big)  ds  \notag\\
 & =  \frac{1}{\sqrt{2\pi}}   \int_0^{  \sqrt{\frac{n}{n-k}} t_+ }
 e^{- \frac{s^2}{2} }  e^{  s  \frac{ x_+' }{\sqrt{n-k}}  }  \Big( 1 - e^{ - 2 s \frac{ x_+' }{\sqrt{n-k}}   } \Big)  ds.  
\end{align}
Using the inequality $1 - e^{-z} \leq  z$ for $z \geq 0$, it follows that, with $t_+ = t+\frac{1}{n^{2\ee}}$, 
\begin{align}\label{Bound_Kx_kk}
\mathcal K(x') \leq  \frac{2x_+'}{  \sqrt{2\pi (n-k)} }   \int_0^{  \sqrt{\frac{n}{n-k}} t_+ }
s  e^{- \frac{s^2}{2} }   e^{  s  \frac{ x_+' }{\sqrt{n-k}}  }  ds.  
\end{align}
Since $k\leq [n^{1-\ee}]$ and  $x_+' = x' + (n-k)^{1/2-2\ee}$ with $x' \leq n^{1/2-\ee/2}$, 
we have 
\begin{align*}%\label{}
\frac{x_+'}{ \sqrt{n-k} } \leq  \frac{c_{\ee}}{ n^{\ee /2} }. 
\end{align*}
Hence, when $t \in [0, n^{\ee^p}]$ for some constant $p \geq 2$ sufficiently large,  
we have $s \leq \sqrt{\frac{n}{n-k}}  t_+ \leq  c_{\ee} n^{\ee^p}$,
so that 
\begin{align*}%\label{}
\exp  \left(  s  \frac{ x_+' }{\sqrt{n-k}} \right)   \leq  \exp  \left( \frac{ c_{\ee} }{ n^{\ee/2 - \ee^p} } \right)
\leq  1 +  \frac{ c_{\ee} }{ n^{\ee/2 - \ee^p} }.  
\end{align*}
%Implementing this bound into \eqref{Bound_Kx_kk}, we get 
This implies that when $t \in [0, n^{\ee^p}]$,
\begin{align}\label{Bound_Kx_tSmall}
\int_0^{  \sqrt{\frac{n}{n-k}} t_+ }
s  e^{- \frac{s^2}{2} }   e^{  s  \frac{ x_+' }{\sqrt{n-k}}  }  ds
& \leq   \left(  1 +  \frac{ c_{\ee} }{ n^{\ee/2 - \ee^p} } \right)
 \int_0^{  \sqrt{\frac{n}{n-k}} t_+ }  s  e^{- \frac{s^2}{2} }   ds   \notag\\
& \leq   \left(  1 +  \frac{ c_{\ee} }{ n^{\ee/2 - \ee^p} } \right)  \left( \Phi^+(t) +  \sqrt{\frac{n}{n-k}} t_+  - t \right)   \notag\\
& \leq  \left(  1 +  \frac{ c_{\ee} }{ n^{\ee/2 - \ee^p} } \right)  \left( \Phi^+(t) +  \frac{c_{\ee}}{ n^{\ee} } \right).    
\end{align}
When $t > n^{\ee^p}$, we have 
\begin{align}\label{Bound_Kx_tLarge}
& \int_0^{  \sqrt{\frac{n}{n-k}} t_+ }  s  e^{- \frac{s^2}{2} }   e^{  s  \frac{ x_+' }{\sqrt{n-k}}  }  ds   \notag\\
& =  \int_0^{  n^{\ee^p} }  s  e^{- \frac{s^2}{2} }   e^{  s  \frac{ x_+' }{\sqrt{n-k}}  }  ds
   +   \int_{ n^{\ee^p} }^{  \sqrt{\frac{n}{n-k}} t_+ }  s  e^{- \frac{s^2}{2} }   e^{  s  \frac{ x_+' }{\sqrt{n-k}}  }  ds   \notag\\
& \leq  \left(  1 +  \frac{ c_{\ee} }{ n^{\ee/2 - \ee^p} } \right) 
   \int_0^{  n^{\ee^p} }  s  e^{- \frac{s^2}{2} }   ds   +  c_{\ee}  e^{ - c_{\ee} n^{\ee^p} }   \notag\\
& \leq  \left(  1 +  \frac{ c_{\ee} }{ n^{\ee/2 - \ee^p} } \right)
 \int_0^{t}  s  e^{- \frac{s^2}{2} }   ds  +  c_{\ee}  e^{ - c_{\ee} n^{\ee^p} }.  
\end{align}
Combining \eqref{Bound_Kx_kk}, \eqref{Bound_Kx_tSmall} and \eqref{Bound_Kx_tLarge},  we get that uniformly in $t \in \bb R_+$, 
\begin{align*}%\label{}
\mathcal K(x') \leq  \frac{2x_+'}{  \sqrt{2\pi (n-k)} }  
\left(  1 +  \frac{ c_{\ee} }{ n^{\ee/2 - \ee^p} } \right)  \left( \Phi^+(t) +  \frac{c_{\ee}}{ n^{\ee} } \right).  
\end{align*}
Since $k\leq [n^{1-\ee}]$,
we have $\frac{ 1}{  \sqrt{ n-k }} \leq  \frac{1}{\sqrt{n}} (1 +\frac{ c_{\ee}}{ n^{\ee} } )$
and 
\begin{align*}%\label{}
\frac{x_+'}{ \sqrt{n-k} } \leq \frac{1}{\sqrt{n}} \left( x' +  n^{1/2 - 2 \ee}  \right)  \left( 1 + \frac{c_\ee}{n^{\ee}} \right), 
\end{align*}
%Since  $\sqrt{\frac{n}{n - k}} \leq  \frac{1}{ \sqrt{1 - n^{-\ee}} }$, 
which implies that
\begin{align}\label{Upper_Pf_CCLT}
\mathcal K(x') 
%& \leq \frac{2x_+'}{  \sqrt{2\pi (n-k)} }   \left( 1 +  \frac{c_\ee}{n^{\ee}}  \right)
%\int_0^{  \sqrt{\frac{n}{n-k}} t_+ }  s  e^{- \frac{s^2}{2} }  ds    \notag\\
&  \leq   \frac{2}{\sqrt{2 \pi n}}  \left( x' +  n^{1/2 - 2 \ee}   \right)  
\left(  1 +  \frac{ c_{\ee} }{ n^{\ee/2 - \ee^p} } \right)  \left( \Phi^+(t) +  \frac{c_{\ee}}{ n^{\ee} } \right). 
%   \notag\\
%%& \leq   \frac{2}{\sqrt{2 \pi n}}  \left( x' +  n^{1/2 - 2 \ee}  \right)  \left( 1 + \frac{c_\ee}{n^{\ee}} \right)
%%   \left(  1- e^{- \frac{1}{2}  (t_+)^2 }   \right)  \notag\\
%&  \leq    \frac{2}{\sqrt{2 \pi n}}  \left( x' +  n^{1/2 - 2 \ee}  \right)  \left(  1 +  \frac{ c_{\ee} }{ n^{\ee/2 - \ee^p} } \right)
%    \left[ \Phi^+ \left(  \frac{ t + n^{-2\ee} }{ \sqrt{ 1 - n^{-\ee} } } \right)  +  c_{\ee}  e^{ - c_{\ee} n^{\ee^p} }  \right]   \notag\\
%&  \leq  \frac{2}{\sqrt{2 \pi n}}  \left( x' +  n^{1/2 - 2 \ee}  \right)  \left(  1 +  \frac{ c_{\ee, p} }{ n^{\ee/2 - \ee^p} } \right)
%    \Phi^+ \left(  \frac{ t + n^{-2\ee} }{ \sqrt{ 1 - n^{-\ee} } } \right).  
\end{align}
%We want to show that uniformly in $t \in \bb R_+$, 
%\begin{align*}%\label{}
%\left(  1- e^{-\frac{n}{2(n-k)} (t_+)^2 }   \right)  = \left(  1- e^{- \frac{t^2}{2} }   \right) (1 + o(1))
%\end{align*}
%
%\begin{align*}%\label{}
%\left(  1- e^{-\frac{n}{2(n-k)} ( t + \frac{1}{n^{2\ee}} )^2 }   \right)   
%\end{align*}
Therefore,  using  the fact that $n^{1/2-\ee} \leq x + S_{\nu_n}$,  we get
\begin{align}\label{Pf_Upper_CCLT_j32}
J_{32} & \leq   \frac{2}{\sqrt{2 \pi n}}  
   \left(  1 +  \frac{ c_{\ee, p} }{ n^{\ee/2 - \ee^p} } \right)
    \left( \Phi^+(t) +  \frac{c_{\ee}}{ n^{\ee} } \right) \notag\\
&  \quad  \times  \sum_{k = 1}^{[n^{1-\ee}]} 
\mathbb{E} \left( x + S_k + n^{1/2 - 2 \ee}, x + S_k \leq n^{1/2-\ee/2}, \tau_x > k, \nu_n = k  \right)   \notag\\
& \leq  \frac{2}{\sqrt{2 \pi  n}}   
\left(  1 +  \frac{ c_{\ee, p} }{ n^{\ee/2 - \ee^p} } \right)
  \left( \Phi^+(t) +  \frac{c_{\ee}}{ n^{\ee} } \right)    \notag\\
&  \quad  \times   
   \mathbb{E} \left( x + S_{\nu_n}, x + S_{\nu_n} \leq n^{1/2-\ee/2}, \tau_x >  \nu_n,  \nu_n \leq [n^{1-\ee}]  \right)   \notag\\
& \leq    \frac{2}{\sqrt{2 \pi  n}}  \left(  1 +  \frac{ c_{\ee, p} }{ n^{\ee/2 - \ee^p} } \right)
   \left( \Phi^+(t) +  \frac{c_{\ee}}{ n^{\ee} } \right)    E_1   \notag\\
& \leq    \frac{2  V(x) }{\sqrt{2 \pi  n}}  \left(  1 +  \frac{ c_{\ee, p} }{ n^{\ee/2 - \ee^p} } \right)
   \left( \Phi^+(t) +  \frac{c_{\ee}}{ n^{\ee} } \right),  
\end{align}
where in the last inequality we used Lemma \ref{lemma E1 and E2}.

\textit{Lower bound of $J_{32}$}.
Denote $x_-' = x' - (n-k)^{1/2-2\ee}$ and $t_- = t -  n^{-2\ee}$.
By Lemma \ref{lemma tauBM},  we have that for $t \geq  n^{- 2\ee}$, 
\begin{align}
 \mathcal K(x') 
& \geq  \mathbb{P} \left( x_-' +  B_{n-k}\leq t_- \sqrt{n}, \tau_{x_-'}^{bm} > n-k\right) \nonumber\\
&  =    \frac{1}{  \sqrt{2\pi(n-k)} } \int_0^{ \sqrt{n}t_- }
\Big( e^{ -\frac{ ( u - x_-')^2 }{ 2  (n-k) } } - e^{ -\frac{ ( u + x_-')^2}{ 2  (n-k)} } \Big) du \nonumber\\
& =  \frac{1}{  \sqrt{2\pi(n-k)} }   e^{ - \frac{ (x_-')^2 }{ 2 (n-k) } }  \int_0^{ \sqrt{n}t_- }  
    e^{- \frac{u^2}{2(n-k)}} \Big( e^{ \frac{u x_-'}{n-k} }  - e^{ - \frac{u x_-'}{n-k} } \Big)  du \notag\\
& =  \frac{1}{\sqrt{2\pi}}  e^{ - \frac{ (x_-')^2 }{ 2 (n-k) } }  \int_0^{  \sqrt{\frac{n}{n-k}} t_- }
 e^{- \frac{s^2}{2} } \Big( e^{  s  \frac{ x_-' }{\sqrt{n-k}}  } - e^{ - s \frac{ x_-' }{\sqrt{n-k}}   } \Big)  ds  \notag\\
& \geq  \frac{1}{\sqrt{2\pi}}  e^{ - \frac{ (x_-')^2 }{ 2 (n-k) } }  \int_0^{  t_- }
 e^{- \frac{s^2}{2} }  e^{ - s  \frac{ x_-' }{\sqrt{n-k}}  }  \Big(  e^{ 2s \frac{ x_-' }{\sqrt{n-k}}   }  - 1 \Big)  ds.   
\end{align}
Using the inequality $e^{z} - 1 \geq  z$ for $z \in \bb R$, we get
\begin{align}\label{LowerBound_Kx_kk}
\mathcal K(x')  \geq   \frac{ 2 x_-' }{  \sqrt{2\pi(n-k)} }
e^{ - \frac{ (x_-')^2 }{ 2 (n-k) } }  \int_0^{  t_- }
 s  e^{- \frac{s^2}{2} }  e^{ - s  \frac{ x_-' }{\sqrt{n-k}}  }    ds.  
\end{align}
Since $x' \leq n^{1/2 - \ee/2}$, $x_-' = x' - (n-k)^{1/2-2\ee}$ and $k\leq [n^{1-\ee}]$, we have 
\begin{align}\label{LowerBound_Kx_g}
\left| \frac{x_-'}{ \sqrt{n-k} }  \right| \leq  \frac{c_{\ee}}{ n^{\ee /2} },  
\qquad  
e^{ - \frac{ (x_-')^2 }{ 2 (n-k) } }  \geq  \left( 1 - \frac{c_{\ee}}{ n^{\ee} } \right). 
\end{align}
Hence, when $t_- \in [0, n^{\ee^p}]$ for some constant $p \geq 2$ sufficiently large,  
we have $s \leq  n^{\ee^p}$,
so that 
\begin{align*}%\label{}
\exp  \left( - s  \frac{ x_-' }{\sqrt{n-k}} \right)   \geq  \exp  \left( - \frac{ c_{\ee} }{ n^{\ee/2 - \ee^p} } \right)
\geq  1  -  \frac{ c_{\ee} }{ n^{\ee/2 - \ee^p} }.  
\end{align*}
%Implementing this bound into \eqref{Bound_Kx_kk}, we get 
This implies that when $t_- \in [0, n^{\ee^p}]$,
\begin{align}\label{LowerBound_Kx_tSmall}
\int_0^{  t_- }   s  e^{- \frac{s^2}{2} }  e^{ - s  \frac{ x_-' }{\sqrt{n-k}}  }    ds
\geq   \left(  1 -  \frac{ c_{\ee} }{ n^{\ee/2 - \ee^p} } \right)   \int_0^{  t_- }  s  e^{- \frac{s^2}{2} }   ds.  
\end{align}
When $t_- > n^{\ee^p}$, we have 
\begin{align}\label{LowerBound_Kx_tLarge}
 \int_0^{  t_- }   s  e^{- \frac{s^2}{2} }  e^{ - s  \frac{ x_-' }{\sqrt{n-k}}  }    ds   
& \geq  \int_0^{  n^{\ee^p} }  s  e^{- \frac{s^2}{2} }  e^{ - s  \frac{ x_-' }{\sqrt{n-k}}  }   ds    \notag\\
& \geq  \left(  1 -  \frac{ c_{\ee} }{ n^{\ee/2 - \ee^p} } \right) 
   \int_0^{  n^{\ee^p} }  s  e^{- \frac{s^2}{2} }   ds      \notag\\
& \geq  \left(  1 -  \frac{ c_{\ee} }{ n^{\ee/2 - \ee^p} } \right)
 \int_0^{  t_- }  s  e^{- \frac{s^2}{2} }   ds  -  c_{\ee}  e^{ - c_{\ee} n^{\ee^p} }.  
\end{align}
Combining \eqref{LowerBound_Kx_kk}, \eqref{LowerBound_Kx_g},
 \eqref{LowerBound_Kx_tSmall} and \eqref{LowerBound_Kx_tLarge},  
we get that uniformly in $t \in \bb R_+$, 
\begin{align*}%\label{}
\mathcal K(x') \geq  \frac{ 2 x_-' }{  \sqrt{2\pi(n-k)} }
\left[  \left(  1 -  \frac{ c_{\ee} }{ n^{\ee/2 - \ee^p} } \right)
 \int_0^{  t_- }  s  e^{- \frac{s^2}{2} }   ds  -  c_{\ee}  e^{ - c_{\ee} n^{\ee^p} }  \right].  
\end{align*}
Since $k\leq [n^{1-\ee}]$ and $x_-' = x' - (n-k)^{1/2-2\ee}$,
we have 
\begin{align*}%\label{}
\frac{x_-'}{ \sqrt{n-k} } \geq \frac{1}{\sqrt{n}} \left( x' -  c_{\ee}  n^{1/2 - 2 \ee}  \right). 
\end{align*}
It follows that
\begin{align}\label{Lower_PfKx_CCLT}
\mathcal K(x') 
&  \geq   \frac{2}{\sqrt{2 \pi n}}  \left( x' -  c_{\ee}  n^{1/2 - 2 \ee}   \right)  \left(  1 -  \frac{ c_{\ee} }{ n^{\ee/2 - \ee^p} } \right)
   \left(  \int_0^{  t_- }  s  e^{- \frac{s^2}{2} }   ds  -  c_{\ee}  e^{ - c_{\ee} n^{\ee^p} }  \right)   \notag\\
&  =   \frac{2}{\sqrt{2 \pi n}}  \left( x' -  c_{\ee}  n^{1/2 - 2 \ee}  \right)  \left(  1 -  \frac{ c_{\ee} }{ n^{\ee/2 - \ee^p} } \right)
    \left[ \Phi^+ (t_-)  -  c_{\ee}  e^{ - c_{\ee} n^{\ee^p} }  \right].  
\end{align}
Following the proof of \eqref{Pf_Upper_CCLT_j32}, 
using Lemmas \ref{lemma E1 and E2} and \ref{Lem-E3-bound}, we obtain 
\begin{align}\label{Pf_Lower_CCLT_j32}
J_{32} & \geq   \frac{2}{\sqrt{2 \pi  n}}  \left(  1 -  \frac{ c_{\ee} }{ n^{\ee/2 - \ee^p} } \right)
    \left[ \Phi^+ (t_-)  -  c_{\ee}  e^{ - c_{\ee} n^{\ee^p} }  \right]   \notag\\
&  \quad  \times  \sum_{k = 1}^{[n^{1-\ee}]} 
\mathbb{E} \left( x + S_k  -  c_\ee  n^{1/2 - 2 \ee}, x + S_k \leq n^{1/2-\ee/2}, \tau_x > k, \nu_n = k  \right)   \notag\\
& \geq    \frac{2}{\sqrt{2 \pi  n}}   \left(  1 -  \frac{ c_{\ee} }{ n^{\ee/2 - \ee^p} } \right)
    \left[ \Phi^+ (t_-)  -  c_{\ee}  e^{ - c_{\ee} n^{\ee^p} }  \right]    \left(  E_1 - E_2 \right)  \notag\\
& \geq    \frac{2}{\sqrt{2 \pi  n}}  \left(  1 -  \frac{ c_{\ee} }{ n^{\ee/2 - \ee^p} } \right)
    \left[ \Phi^+ (t_-)  -  c_{\ee}  e^{ - c_{\ee} n^{\ee^p} }  \right]  
\left(   V(x) -  \frac{  c_\ee  (1 + x) }{ n^{\delta/2 - \ee (1 + \ee + \delta/2)} } \right)   \notag\\
& \geq  \frac{2 V(x)}{\sqrt{2 \pi  n}}   \left(  1 -  \frac{ c_{\ee} }{ n^{\ee/2 - \ee^p} } \right)
   \Phi^+ ( t_-)
   -  \frac{  c_\ee  (1 + x) }{ n^{\delta/2 - \ee (1 + \ee + \delta/2)} }.  
\end{align}
We conclude the proof of Theorem \ref{Theor-IntegrLimTh} by combining \eqref{Pf-Thm3-J1}, \eqref{Pf-Thm3-J2}, 
 \eqref{Pf-Thm3-J31},  \eqref{Pf_Upper_CCLT_j32} and \eqref{Pf_Lower_CCLT_j32}.  
%\eqref{Pf-Thm3-J32-upp}, \eqref{Pf-Thm3-J32-low}. 
\end{proof}

To prove Theorem \ref{CorCCLT}, we first show the following result which is based on the coupling method
using the functional central limit theorem (Lemma \ref{FCLT}).  
%The following result is a consequence of Theorem \ref{Theor-limit tau-large n-001CCLT}. 

\begin{theorem} 
\label{Theor-limit tau-large n-001CCLT}  
Assume \ref{SecondMoment} with some $\delta>0$. 
%Assume \ref{A1}. 
%Assume that $\bb E (X_1^{2+\delta}) < \infty$ for some constant $\delta >0$. 
%Let $x_{\ee}^+ = x + n^{1/2 - \ee}$ and $x_{\ee}^- = x - n^{1/2 - \ee}$. 
For any $\gamma \in (0,  \frac{\delta}{2(2+\delta)})$,
there exists a constant $c_{\gamma} >0$ such that the following assertions hold: 

1. For any $n\geq 1$, $x >0$ and $t>0$,    
\begin{align} \label{CCLT_Appro_Upper}
\mathbb{P} \left(  \frac{x+S_n }{\sigma \sqrt{n}} > t, \tau_x >n\right) 
&  \leq   
    \int_{t}^{\infty}  \psi \left( s, \frac{x}{\sigma \sqrt{n}} + \frac{1}{\sigma n^{\gamma}} \right)  ds  
   +  \frac{c_{\gamma}}{ n^{\frac{\delta}{2} - \gamma (2+\delta) } }. 
\end{align}
%where $x_{\ee}^+ = x + n^{1/2 - \ee}$.

2. For any $n\geq 1$, $x > n^{1/2 - \gamma}$ and $t>0$,  
 \begin{align} \label{CCLT_Appro_Lower}
\mathbb{P} \left(  \frac{x+S_n }{ \sqrt{n}} > t, \tau_x >n\right) 
&  \geq   
    \int_{t}^{\infty}  \psi \left( s, \frac{x}{\sigma \sqrt{n}} - \frac{1}{\sigma n^{\gamma}} \right)  ds 
    -  \frac{c_{\gamma}}{ n^{\frac{\delta}{2} - \gamma (2+\delta) } }. 
\end{align}
%where $x_{\ee}^- = x - n^{1/2 - \ee}$.

%3. For any  $n\geq 1$, $x >0$ and $t>0$,  
% %satisfying $x\geq \gamma_n n^{\frac{1}{2}},$
%\begin{align} \label{CCLT_Appro_Upper_leq}
%\mathbb{P} \left(  \frac{x+S_n }{ \sqrt{n}} \leq t, \tau_x >n\right) 
%&  \leq   
%    \int_{0}^{ t + 2 n^{ - \gamma } } 
%   \psi \left( s, \frac{x}{\sqrt{n}} + \frac{1}{ n^{ \gamma }} \right) ds 
%     +  \frac{c_{ \gamma }}{ n^{\frac{\delta}{2} - \gamma  (2+\delta) } }.  
%\end{align}
%%where $x_{\ee}^+ = x + n^{1/2 - \ee}$.
%
%
%4. For any  $n\geq 1$, $x > n^{\frac{1}{2}- \gamma}$ and $t>0$,  
% %satisfying $x\geq \gamma_n n^{\frac{1}{2}},$
%\begin{align} \label{CCLT_Appro_Lower_leq}
%\mathbb{P} \left(  \frac{x+S_n }{ \sqrt{n}} \leq t, \tau_x >n\right) 
%& \geq   \int_{0}^{t - 2 n^{- \gamma } }   
%   \psi \left( s, \frac{x}{\sqrt{n}} - \frac{1}{n^{ \gamma }} \right) ds 
%      - \frac{c_{ \gamma }}{ n^{\frac{\delta}{2} - \gamma (2+\delta) } }.   
%\end{align}
\end{theorem}

\begin{proof}%[Proof of Theorem \ref{Theor-limit tau-large n-001CCLT}]
Denote
\begin{align}\label{eq-fff-000_bb}
A_n = \left\{ \sup_{0 \leq t \leq 1} \left| S_{[tn]} - \sigma B_{tn} \right|  >  n^{1/2 - \gamma }  \right\}
\end{align}
and by $A_n^c$ the complement of $A_n$.  
By Lemma \ref{FCLT}, 
for any $\gamma \in (0,  \frac{\delta}{2(2+\delta)})$, 
there exists a constant $c_{\gamma}>0$ such that 
\begin{align} \label{boundprobabA_n001_bb}
\mathbb{P} \left( A_{n} \right) \leq  \frac{c_{\gamma}}{ n^{\frac{\delta}{2} - \gamma (2+\delta) } }.  
\end{align}
%We first prove the upper bound \eqref{CCLT_Appro_Upper}.
Since $x_{\gamma}^+ = x + n^{\frac{1}{2}- \gamma }$, 
using Lemma \ref{lemma tauBM} and \eqref{boundprobabA_n001_bb}, we get
\begin{align*}%\label{fff-001_bb}
\mathbb{P} \left(  \frac{x+S_n }{\sigma \sqrt{n}} > t, \tau_{x} > n \right)  
%& = \mathbb{P} \left( \frac{x+S_n }{ \sqrt{n}} > t, \tau_{x} > n, A_{n}^c \right) 
%    + \mathbb{P} \left( \frac{x+S_n }{ \sqrt{n}} > t, \tau_{x} > n, A_{n} \right)  \notag\\
& \leq  \mathbb{P} \left( \frac{x_{\gamma}^+ + \sigma B_{n}}{\sigma \sqrt{n}} > t, \tau_{ x_{\gamma}^+ }^{bm} > n   \right)  
   + \bb{P} ( A_{n} )  \notag\\
&  \leq  \frac{1}{\sigma \sqrt{n} } \int_{t \sigma \sqrt{n}}^{\infty}  
\psi  \left( \frac{s}{\sigma \sqrt{n}},  \frac{x_{\gamma}^+}{\sigma \sqrt{n}} \right)  ds
%   \left( e^{-\frac{ (s - x_{\gamma}^+)^2 }{ 2 n }} 
%      - e^{-\frac{ (s + x_{\gamma}^+)^2 }{ 2 n }}  \right) ds  
      + \frac{c_\gamma}{ n^{\frac{\delta}{2} - \gamma (2+\delta) } }, 
\end{align*}
which ends the proof of the upper bound \eqref{CCLT_Appro_Upper}. 

%We next prove the upper bound \eqref{CCLT_Appro_Lower}.
Since $x_{\gamma}^- = x - n^{\frac{1}{2}- \gamma}$, 
using Lemma \ref{lemma tauBM} and \eqref{boundprobabA_n001_bb}, 
we get that for any $x > n^{\frac{1}{2}- \gamma}$,  
\begin{align*}%\label{}
\mathbb{P} \left(  \frac{x+S_n }{\sigma \sqrt{n}} > t, \tau_{x} > n \right)
%& \geq  \mathbb{P} \left( \frac{x+S_n }{ \sqrt{n}} > t, \tau_{x} > n, A_{n}^c \right)   \notag\\
&  \geq   
 \mathbb{P} \left( \frac{x_{\gamma}^-  + \sigma B_{n}}{\sigma \sqrt{n}} > t, \tau_{x_{\gamma}^-}^{bm} > n,  A_{n}^c   \right)   \notag\\
& \geq  \mathbb{P} \left( \frac{x_{\gamma}^- + \sigma B_{n}}{\sigma \sqrt{n}} > t, \tau_{x_{\gamma}^-}^{bm} > n   \right)  
    - \bb P(A_{n})  \notag\\
& \geq  \frac{1}{\sigma \sqrt{n} } \int_{t \sigma \sqrt{n}}^{\infty}  
\psi  \left( \frac{s}{\sigma \sqrt{n}},  \frac{x_{\gamma}^-}{\sigma \sqrt{n}} \right)  ds 
    - \frac{c_\gamma}{ n^{\frac{\delta}{2} - \gamma (2+\delta) } },  
\end{align*}
which concludes the proof of the lower bound \eqref{CCLT_Appro_Lower}. 
\end{proof}

Using Theorem \ref{Theor-limit tau-large n-001CCLT}, we now prove Theorem \ref{CorCCLT}. 

\begin{proof}[Proof of Theorem \ref{CorCCLT}]
%We first prove \eqref{CorCCLT02}.    
%By Taylor's formula, we have, with $\theta \in (x, x + z)$, 
%\begin{align*}%\label{}
%\psi (s, x + z) - \psi (s, x) = z \psi_x (s, \theta)
%=  \frac{z}{ \sqrt{2 \pi} }  \left[ (s - \theta) e^{ -\frac{1}{2} (s - \theta)^2 }  +   (s + \theta) e^{ -\frac{1}{2} (s + \theta)^2 }  \right]
%\end{align*}
By the definition of $\psi$ (cf.\ \eqref{Def-Levydens}), we have for $t>0$, 
\begin{align*}%\label{}
&  \int_{t}^{\infty}  \psi \left( s, \frac{x}{\sigma \sqrt{n}} + \frac{1}{\sigma n^{\gamma}} \right)  ds 
      -   \int_{t}^{\infty}  \psi \left( s, \frac{x}{\sigma \sqrt{n}}   \right)  ds    \notag\\
& =   \frac{1}{ \sqrt{2 \pi} } \int_{t}^{\infty}   
  \left(   e^{ -\frac{1}{2} (s - \frac{x}{\sigma \sqrt{n}} - \frac{1}{\sigma n^{\gamma}} )^2 }  
    -   e^{ -\frac{1}{2} (s - \frac{x}{\sigma \sqrt{n}}  )^2 }  \right)  ds 
    \notag\\
& \quad       +   \frac{1}{ \sqrt{2 \pi} } \int_{t}^{\infty}   
  \left(  e^{ -\frac{1}{2} (s + \frac{x}{\sigma \sqrt{n}}  )^2 }   
     -  e^{ -\frac{1}{2} (s + \frac{x}{\sigma \sqrt{n}} + \frac{1}{\sigma n^{\gamma}} )^2 }   \right)  ds 
   \notag\\
& = :  I_1 + I_2. 
\end{align*}
For $I_1$, by a change of variable and elementary calculations, we get
\begin{align*}%\label{}
I_1 
& = \frac{1}{ \sqrt{2 \pi} } \int_{t - \frac{x}{\sigma \sqrt{n}}}^{\infty}   
  \left(   e^{ -\frac{1}{2} (s - \frac{1}{\sigma n^{\gamma}} )^2 }  -   e^{ -\frac{1}{2} s^2 }  \right)  ds   \notag\\
&  \leq  \int_{\bb R}  \left|   e^{ -\frac{1}{2} (s - \frac{1}{\sigma n^{\gamma}} )^2 }  -   e^{ -\frac{1}{2} s^2 }  \right|  ds   \notag\\
& =  \int_{|s| \leq n^{\gamma/2} }  \left|   e^{ -\frac{1}{2} (s - \frac{1}{\sigma n^{\gamma}} )^2 }  -   e^{ -\frac{1}{2} s^2 }  \right|  ds  
   +  \int_{|s|  > n^{\gamma/2}  }  \left|   e^{ -\frac{1}{2} (s - \frac{1}{\sigma n^{\gamma}} )^2 }  -   e^{ -\frac{1}{2} s^2 }  \right|  ds   \notag\\
& =:  I_{11}  + I_{12}.  
\end{align*}
For $I_{11}$, using the inequality $e^x - 1 \leq |x| + |x| e^x$, $x \in \bb R$,  we have 
\begin{align*}%\label{}
I_{11} & =   \int_{|s| \leq n^{\gamma/2} }  e^{ -\frac{1}{2} s^2 }  
   \left|   e^{ \frac{s}{\sigma n^{\gamma}} - \frac{1}{2 \sigma^2 n^{2 \gamma}} }  -   1  \right|  ds  \notag\\
 & \leq \int_{|s| \leq n^{\gamma/2} }  e^{ -\frac{1}{2}  s^2 } 
   \left|  \frac{s}{\sigma n^{\gamma}} - \frac{1}{2 \sigma^2 n^{2 \gamma}}  \right|   ds \notag\\
&   \qquad\qquad +  \int_{|s| \leq n^{\gamma/2} }  e^{ -\frac{1}{2} (s - \frac{1}{\sigma n^{\gamma}} )^2 } 
   \left|  \frac{s}{\sigma n^{\gamma}} - \frac{1}{2 \sigma^2 n^{2 \gamma}}  \right|   ds   \notag\\
& \leq  2 \int_{|s| \leq 2 n^{\gamma/2} }  e^{ -\frac{1}{2}  s^2 } 
   \left|  \frac{s}{\sigma n^{\gamma}} - \frac{1}{2 \sigma^2 n^{2 \gamma}}  \right|   ds   \notag\\
& \leq   \frac{2}{\sigma n^{\gamma} } \int_{|s| \leq 2 n^{\gamma/2} } |s| e^{ -\frac{1}{2}  s^2 } ds 
      +  \frac{1}{\sigma^2 n^{2 \gamma}}  \int_{|s| \leq 2 n^{\gamma/2} }  e^{ -\frac{1}{2}  s^2 }   ds   
      \leq  \frac{c}{ n^{\gamma} }. 
\end{align*}
For $I_{12}$, we have 
\begin{align*}%\label{}
I_{12} \leq   \int_{ n^{\gamma/2} - \frac{1}{n^{\gamma}} }^{ n^{\gamma/2} }   e^{ -\frac{1}{2} s^2 }   ds
  +   \int_{ - n^{\gamma/2} - \frac{1}{n^{\gamma}} }^{ -n^{\gamma/2} }   e^{ -\frac{1}{2} s^2 }   ds
  \leq  \frac{2}{ n^{\gamma} }. 
\end{align*}
Hence, we get $I_1 \leq    \frac{c}{ n^{\gamma} }.$ Similarly, one can also check that $I_2 \leq    \frac{c}{ n^{\gamma} }.$
Therefore, 
\begin{align}\label{UpperI001CCLT}
\int_{t}^{\infty}  \psi \left( s, \frac{x}{\sigma \sqrt{n}} + \frac{1}{\sigma n^{\gamma}} \right)  ds 
      -   \int_{t}^{\infty}  \psi \left( s, \frac{x}{\sigma \sqrt{n}}   \right)  ds  
      \leq    \frac{c}{ n^{\gamma} }. 
\end{align}
Using \eqref{CCLT_Appro_Upper} and taking $\gamma = \frac{\delta}{ 2(3 + \delta) }$, we obtain 
\begin{align}\label{UpperCCLTLarge}
\mathbb{P} \left(  \frac{x+S_n }{\sigma \sqrt{n}} > t, \tau_x >n\right) 
&  \leq   
    \int_{t}^{\infty}  \psi \left( s, \frac{x}{\sigma \sqrt{n}}  \right)  ds  
  + \frac{c}{ n^{\gamma} }  +  \frac{c_{\gamma}}{ n^{\frac{\delta}{2} - \gamma (2+\delta) } }   \notag\\
& \leq   \int_{t}^{\infty}  \psi \left( s, \frac{x}{\sigma \sqrt{n}}  \right)  ds  
  + \frac{c_{\gamma}}{ n^{ \frac{\delta}{ 2(3 + \delta) } } }.  
\end{align}

To show the lower bound, we first write for $t>0$, 
\begin{align*}%\label{}
&  \int_{t}^{\infty}  \psi \left( s, \frac{x}{\sigma \sqrt{n}} - \frac{1}{\sigma n^{\gamma}} \right)  ds 
      -   \int_{t}^{\infty}  \psi \left( s, \frac{x}{\sigma \sqrt{n}}   \right)  ds    \notag\\
& =   \frac{1}{ \sqrt{2 \pi} } \int_{t}^{\infty}   
  \left(   e^{ -\frac{1}{2} (s - \frac{x}{\sigma \sqrt{n}} + \frac{1}{\sigma n^{\gamma}} )^2 }  
   -   e^{ -\frac{1}{2} (s - \frac{x}{\sigma \sqrt{n}}  )^2 }  \right)  ds 
    \notag\\
& \quad       +   \frac{1}{ \sqrt{2 \pi} } \int_{t}^{\infty}   
  \left(  e^{ -\frac{1}{2} (s + \frac{x}{\sigma \sqrt{n}}  )^2 }   
    -  e^{ -\frac{1}{2} (s + \frac{x}{\sigma \sqrt{n}} - \frac{1}{\sigma n^{\gamma}} )^2 }   \right)  ds 
   \notag\\
& = :  J_1 + J_2. 
\end{align*}
For $J_1$, by a change of variable and elementary calculations, we get
\begin{align*}%\label{}
J_1 
& = \frac{1}{ \sqrt{2 \pi} } \int_{t - \frac{x}{\sigma \sqrt{n}}}^{\infty}   
  \left(   e^{ -\frac{1}{2} (s + \frac{1}{\sigma n^{\gamma}} )^2 }  -   e^{ -\frac{1}{2} s^2 }  \right)  ds   \notag\\
&  \geq - \int_{\bb R}  \left|   e^{ -\frac{1}{2} (s + \frac{1}{\sigma n^{\gamma}} )^2 }  -   e^{ -\frac{1}{2} s^2 }  \right|  ds   \notag\\
& =  - \int_{|s| \leq n^{\gamma/2} }  \left|   e^{ -\frac{1}{2} (s + \frac{1}{\sigma n^{\gamma}} )^2 }  
    -   e^{ -\frac{1}{2} s^2 }  \right|  ds  
   -  \int_{|s|  > n^{\gamma/2}  }  \left|   e^{ -\frac{1}{2} (s + \frac{1}{\sigma n^{\gamma}} )^2 }  
      -   e^{ -\frac{1}{2} s^2 }  \right|  ds   \notag\\
& =:  J_{11}  + J_{12}.  
\end{align*}
For $J_{11}$, using again the inequality $e^x - 1 \leq |x| + |x| e^x$, $x \in \bb R$,  we have 
\begin{align*}%\label{}
J_{11} & =  - \int_{|s| \leq n^{\gamma/2} }  e^{ -\frac{1}{2} s^2 }  
   \left|   e^{ - \frac{s}{n^{\gamma}} - \frac{1}{2 n^{2 \gamma}} }  -   1  \right|  ds  \notag\\
 & \geq  - \int_{|s| \leq n^{\gamma/2} }  e^{ -\frac{1}{2}  s^2 } 
   \left|  \frac{s}{n^{\gamma}} + \frac{1}{ n^{2 \gamma}}  \right|   ds
   -  \int_{|s| \leq n^{\gamma/2} }  e^{ -\frac{1}{2} (s + \frac{1}{n^{\gamma}} )^2 } 
   \left|  \frac{s}{n^{\gamma}} + \frac{1}{2 n^{2 \gamma}}  \right|   ds   \notag\\
& \geq  - 2 \int_{|s| \leq 2 n^{\gamma/2} }  e^{ -\frac{1}{2}  s^2 } 
   \left|  \frac{s}{n^{\gamma}} + \frac{1}{ n^{2 \gamma}}  \right|   ds   \notag\\
& \geq  -  \frac{2}{ n^{\gamma} } \int_{|s| \leq 2 n^{\gamma/2} } |s| e^{ -\frac{1}{2}  s^2 } ds
      -  \frac{2}{ n^{2 \gamma}}  \int_{|s| \leq 2 n^{\gamma/2} }  e^{ -\frac{1}{2}  s^2 }   ds   
      \geq  -  \frac{c}{ n^{\gamma} }. 
\end{align*}
For $J_{12}$, we have 
\begin{align*}%\label{}
I_{12} \leq   \int_{ n^{\gamma/2}  }^{ n^{\gamma/2} + \frac{1}{n^{\gamma}} }   e^{ -\frac{1}{2} s^2 }   ds
  +   \int_{ - n^{\gamma/2}  }^{ -n^{\gamma/2} +  \frac{1}{n^{\gamma}} }   e^{ -\frac{1}{2} s^2 }   ds
  \leq  \frac{2}{ n^{\gamma} }. 
\end{align*}
Hence, we get $I_1 \leq    \frac{c}{ n^{\gamma} }.$ Similarly, one can also check that $I_2 \leq    \frac{c}{ n^{\gamma} }.$
Therefore, 
\begin{align*}%\label{}
\int_{t}^{\infty}  \psi \left( s, \frac{x}{\sigma \sqrt{n}} - \frac{1}{\sigma n^{\gamma}} \right)  ds 
      -   \int_{t}^{\infty}  \psi \left( s, \frac{x}{\sigma \sqrt{n}}   \right)  ds 
      \geq   - \frac{c}{ n^{\gamma} }. 
\end{align*}
Using \eqref{CCLT_Appro_Lower} and taking $\gamma = \frac{\delta}{ 2(3 + \delta) }$, we obtain 
\begin{align}\label{LowerCCLTLarge}
\mathbb{P} \left(  \frac{x+S_n }{\sigma \sqrt{n}} > t, \tau_x >n\right) 
&  \geq   
    \int_{t}^{\infty}  \psi \left( s, \frac{x}{\sigma \sqrt{n}}  \right)  ds  
  - \frac{c}{ n^{\gamma} }  -  \frac{c_{\gamma}}{ n^{\frac{\delta}{2} - \gamma (2+\delta) } }   \notag\\
& \geq   \int_{t}^{\infty}  \psi \left( s, \frac{x}{\sigma \sqrt{n}}  \right)  ds  
  - \frac{c_{\gamma} }{ n^{ \frac{\delta}{ 2(3 + \delta) } } }.  
\end{align}
Combining \eqref{UpperCCLTLarge} and \eqref{LowerCCLTLarge} finishes the proof of \eqref{CorCCLT02}. 
%We next prove \eqref{CorCCLT02}. 
%Since the function $(s,x) \mapsto \psi(s,x)$ is bounded on $\bb R \times \bb R$, we have 
%\begin{align*}%\label{}
%\left| \int_{0}^{ t + 2 n^{ - \gamma } } 
%   \psi \left( s, \frac{x}{\sqrt{n}} + \frac{1}{ n^{ \gamma }} \right) ds 
%   -  \int_{0}^{ t }  \psi \left( s, \frac{x}{\sqrt{n}} + \frac{1}{ n^{ \gamma }} \right) ds  \right|
% \leq  \frac{c}{ n^{\gamma} }.
%\end{align*}
%Following the same proof of \eqref{UpperI001CCLT}, one has 
%\begin{align*}%\label{}
%\int_{0}^{ t }  \psi \left( s, \frac{x}{\sqrt{n}} + \frac{1}{ n^{ \gamma }} \right) ds
% -  \int_{0}^{ t }  \psi \left( s, \frac{x}{\sqrt{n}}   \right) ds   \leq    \frac{c}{ n^{\gamma} }. 
%\end{align*}
%The remaining part of the proof of \eqref{CorCCLT02} can be carried out in a similar way
%as that of \eqref{CorCCLT01} and we omit the details. 
\end{proof}

\begin{proof}[Proof of Corollary \ref{Cor-CCLT-Optimal}]
Note that 
\begin{align}\label{pf-Cor-29-001}
I : = \int_{0}^{t}  \psi \left( s, \frac{x}{\sigma \sqrt{n}}   \right)  ds 
=   e^{- \frac{x^2}{2 \sigma^2 n}}   \frac{1}{\sqrt{2 \pi}}  
 \int_{0}^{t}  e^{- \frac{s^2}{2}}   e^{\frac{sx}{\sigma \sqrt{n}}}
  \left( 1 - e^{- \frac{2sx}{\sigma \sqrt{n}}} \right)  ds. 
\end{align}
For the upper bound, using the inequality $1 - e^{-u} \leq u$, $u \geq 0$, we get that for any $x \geq 0$ and $t \geq 0$, 
\begin{align}\label{pf-Cor-29-002}
I & \leq    \frac{2 x}{\sigma \sqrt{2 \pi n}}  
e^{- \frac{x^2}{2 \sigma^2 n}}  \int_{0}^{t}  s  e^{- \frac{s^2}{2}}   e^{\frac{sx}{\sigma \sqrt{n}}}  ds   \notag\\
& =   \frac{2 x}{\sigma \sqrt{2 \pi n}}  
  \int_{0}^{t}  s  e^{- \frac{1}{2} (s - \frac{x}{\sigma \sqrt{n}})^2 }    ds
  =  \frac{2 x}{\sigma \sqrt{2 \pi n}}  
  \int_{ - \frac{x}{\sigma \sqrt{n}} }^{t -  \frac{x}{\sigma \sqrt{n}}}  
   \left( u +  \frac{x}{\sigma \sqrt{n}}  \right) e^{- \frac{1}{2} u^2 }    du  \notag\\
 & \leq   \frac{2 x}{\sigma \sqrt{2 \pi n}}
  \left(  e^{- \frac{x^2}{2 \sigma^2 n}}  - e^{- \frac{1}{2} (t - \frac{x}{\sigma \sqrt{n}})^2 } \right)
   +  c \frac{x^2}{n}   \notag\\
 & =  \frac{2 x}{\sigma \sqrt{2 \pi n}}   e^{- \frac{x^2}{2 \sigma^2 n}} 
  \left(  1  - e^{- \frac{t^2}{2} + \frac{tx}{\sigma \sqrt{n}} } \right)
   +  c \frac{x^2}{n}   \notag\\
 & \leq  \frac{2 V(x)}{\sigma \sqrt{2 \pi n}}  \left(  1  - e^{- \frac{t^2}{2} } \right)   +  c \frac{x^2}{n}  
  =  \frac{2 V(x)}{\sigma \sqrt{2 \pi n}}  \Phi^+(t)   +  c \frac{x^2}{n}, 
\end{align}
where in the last inequality we used the fact that $e^{- \frac{x^2}{2 \sigma^2 n}} \leq 1$, $t \geq 0$ and $x \leq V(x)$.
%When $t \leq c \sqrt{n}$, it holds that $e^{\frac{sx}{\sigma \sqrt{n}}} \leq  1 + c \alpha_n \sqrt{n}$, 
%uniformly in $s \in [0, t]$ and $x \in [n^{1/2- \ee}, \alpha_n \sqrt{n}]$.
%Hence, using the fact that $e^{- \frac{x^2}{2 \sigma^2 n}} \leq 1$ and $x \leq V(x)$, we get
%\begin{align*}%\label{}
%I \leq  \left( 1 + c \alpha_n \sqrt{n} \right)  \frac{2 x}{\sigma \sqrt{2 \pi n}}  \Phi^+(t). 
%\end{align*}
%When $t > c \sqrt{n}$, from \eqref{pf-Cor-29-001} we see that 
%\begin{align*}%\label{}
%I & \leq   \frac{2 x}{\sigma \sqrt{2 \pi n}}  
%  \int_{0}^{t}  s  e^{- \frac{1}{2} (s - \frac{x}{\sigma \sqrt{n}})^2 }    ds
%  =  \frac{2 x}{\sigma \sqrt{2 \pi n}}  
%  \int_{ - \frac{x}{\sigma \sqrt{n}} }^{t -  \frac{x}{\sigma \sqrt{n}}}  
%   \left( u +  \frac{x}{\sigma \sqrt{n}}  \right) e^{- \frac{1}{2} u^2 }    du  \notag\\
% & \leq  \frac{2 x}{\sigma \sqrt{2 \pi n}}
%  \left(  e^{- \frac{x^2}{2 \sigma^2 n}}  - e^{- \frac{1}{2} (t - \frac{x}{\sigma \sqrt{n}})^2 } \right)
%\end{align*}

For the lower bound, 
from \eqref{pf-Cor-29-001} 
and the inequalities $e^{\frac{sx}{\sigma \sqrt{n}}} \geq 1$ and $1 - e^{-u} \geq  u - \frac{u^2}{2}$, $u \geq 0$,
we obtain that for any $x \geq 0$ and $t \geq 0$, 
\begin{align}\label{pf-Cor-29-003}
I  & \geq   e^{- \frac{x^2}{2 \sigma^2 n}}   \frac{1}{\sqrt{2 \pi}}  
 \int_{0}^{t}  e^{- \frac{s^2}{2}}   
  \left( 1 - e^{- \frac{2sx}{\sigma \sqrt{n}}} \right)  ds  \notag\\
& \geq   e^{- \frac{x^2}{2 \sigma^2 n}}   
 \frac{2 x }{ \sigma \sqrt{2 \pi n}}   
 \left(  \int_{0}^{t}   s  e^{- \frac{s^2}{2}}  ds  -   \frac{x}{\sigma \sqrt{n}}  \int_{0}^{t}   s^2  e^{- \frac{s^2}{2}}  ds \right)  \notag\\
& \geq  \left( 1 -  c \frac{x^2}{n} \right)  \frac{2 x }{ \sigma \sqrt{2 \pi n}}    
    \left(  \Phi^+(t)  -  c_t \frac{x}{\sqrt{n}}  \right)   \notag\\
& \geq   \frac{2 x }{ \sigma \sqrt{2 \pi n}}    \left(  \Phi^+(t)  -  c_t \frac{x}{\sqrt{n}}  \right), 
\end{align}
where $c_t = \frac{1}{\sigma} \int_{0}^{t}   s^2  e^{- \frac{s^2}{2}}  ds$. 
Applying Lemma \ref{Lem_V_Ineq_aa} with $k_0 = n$, we get that uniformly in $x \in [n^{1/2 - \ee}, \alpha_n \sqrt{n}]$, 
\begin{align*}%\label{}
x   \geq  \frac{1}{ 1 + c_{\ee} k_0^{-\ee}}  \left( V(x) - c_{\ee} k_0^{1/2 - \ee} \right)
& \geq  V(x) - c_{\ee} \frac{V(x)}{n^{\ee}} - c_{\ee}  n^{1/2 - \ee}     \notag\\
&  \geq   V(x) \left( 1 -   \frac{ c_{\ee} }{n^{\ee}} \right). 
\end{align*}
Substituting this into \eqref{pf-Cor-29-002} gives that uniformly in $x \in [n^{1/2 - \ee}, \alpha_n \sqrt{n}]$, 
\begin{align}\label{pf-Cor-29-004}
I \geq  \frac{2 V(x) }{ \sigma \sqrt{2 \pi n}}  
\left( 1 -   \frac{ c_{\ee} }{n^{\ee}} \right)  \left(  \Phi^+(t)  -  c_t \frac{x}{\sqrt{n}}  \right)
\geq   \frac{2 V(x) }{ \sigma \sqrt{2 \pi n}}   \left(  \Phi^+(t)  -  c_t' \frac{x}{\sqrt{n}}  \right), 
\end{align}
where $c_t'$ is bounded uniformly in $t \in \bb R_+$. 
Combining \eqref{pf-Cor-29-002} and \eqref{pf-Cor-29-004},
and taking into account Theorems \ref{Theor-IntegrLimTh} and \ref{CorCCLT}, we conclude the proof of \eqref{C-CLTalphan}.

For the second assertion \eqref{CorCCLT02bis},       
applying Lemma \ref{Lem_V_Ineq_aa} with $k_0 = n^{\beta}$, 
we get that for any $\ee \in (0, 1/2)$ and $\beta \in (0, 1/2 - \ee)$, uniformly in $x \in [n^{\beta}, \alpha_n \sqrt{n}]$, 
\begin{align*}%\label{}
x \geq  \frac{1}{ 1 + c_{\ee} k_0^{-\ee}}  \left( V(x) - c_{\ee} k_0^{1/2 - \ee} \right)
& \geq  V(x) - c_{\ee}  n^{-  \ee \beta}  V(x)  - c_{\ee} n^{ \beta/2}  \notag\\
&  \geq  V(x) - c_{\ee}  n^{-  \ee \beta }  V(x)  - c_{\ee} x  n^{ - \beta/2 }   \notag\\
&  \geq  V(x)  \left( 1 -  c_{\ee}  n^{- \ee \beta } \right). 
\end{align*}
Substituting this into \eqref{pf-Cor-29-003}, we get that uniformly in $x \in [n^{\beta}, n^{1/2 - \ee}]$, 
\begin{align*}%\label{}
I  & \geq  \frac{2 V(x) }{ \sigma \sqrt{2 \pi n}}  \left( 1 -  c_{\ee}  n^{- \ee \beta } \right)  
  \left(  \Phi^+(t)  -  c_t \frac{x}{\sqrt{n}}  \right)   \notag\\
& \geq  \frac{2 V(x) }{ \sigma \sqrt{2 \pi n}}  \left( \Phi^+(t)  -  c_{\ee}  n^{- \ee \beta }  \right). 
\end{align*}
This, together with \eqref{pf-Cor-29-002} and Theorems \ref{Theor-IntegrLimTh} and \ref{CorCCLT},
finishes the proof of \eqref{CorCCLT02bis}. 
\end{proof}

%%%%%%%%%%%%%%%%%%%%%%%%%%%%%%%%%%%%%%%%%%%%%%%%
%%%%%%%%%%%%%%%%%%%%%%%%%%%%%%%%%%%%%%%%%%%%%%%%

\vskip5mm

\end{document}